\documentclass[11pt]{amsart}
\usepackage[colorlinks,citecolor=red,pagebackref,hypertexnames=false]{hyperref}

\usepackage{verbatim} 

\makeatletter
\def\@tocline#1#2#3#4#5#6#7{\relax
  \ifnum #1>\c@tocdepth 
  \else
    \par \addpenalty\@secpenalty\addvspace{#2}%
    \begingroup \hyphenpenalty\@M
    \@ifempty{#4}{%
      \@tempdima\csname r@tocindent\number#1\endcsname\relax
    }{%
      \@tempdima#4\relax
    }%
    \parindent\z@ \leftskip#3\relax \advance\leftskip\@tempdima\relax
    \rightskip\@pnumwidth plus4em \parfillskip-\@pnumwidth
    #5\leavevmode\hskip-\@tempdima
      \ifcase #1
       \or\or \hskip 1em \or \hskip 2em \else \hskip 3em \fi%
      #6\nobreak\relax
    \dotfill\hbox to\@pnumwidth{\@tocpagenum{#7}}\par
    \nobreak
    \endgroup
  \fi}
\makeatother

\usepackage{tikz,amsthm,amsmath,amstext,amssymb,amscd,epsfig,euscript, mathrsfs, dsfont,pspicture,multicol, mathtools,graphpap,graphics,graphicx,times,enumerate,subfig,sidecap,wrapfig,color}

\usepackage{xcolor,booktabs, multirow} 
\RequirePackage{caption}
\usepackage{ hyperref}
\usepackage{enumerate}

 \numberwithin{equation}{section}

\def\bR{{\mathbb{R}}}

\def\caD{{\mathcal{D}}}
\def\Ch{{\mathrm{Ch}}}

\def\oomega{\texttt{w}}

\newcommand{\pom}{{\partial\Omega}}

\def\ve{\varepsilon}

\renewcommand{\d}{{\partial}}

\def\wt{\widetilde}
\def\DS{\mathop\textup{DS}}
\def\DL{\mathop\textup{DL}}
\def\DMO{\mathop\textup{DMO}}
\def\DDMO{\mathop\textup{DDMO}}

\def\Rn1{\mathbb{R}^{n+1}}

\def\ext{\mathop\mathrm{ext}} 	
 	
\DeclareMathOperator{\diam}{diam}
\def\Cap{\textup{Cap}} 					

\def\Tan{\mathop\mathrm{Tan}} 					

\def\osc{\mathop\mathrm{osc}} 					

\def\VMO{\mathop\mathrm{VMO}} 					
\def\Lip{\mathrm{Lip}} 						

\def\dim{\mathop\mathrm{dim}} 					
\def\dist{\textup{dist}} 						
\def\supp{\mathop\mathrm{supp}}					
\def\loc{\mathop\mathrm{loc}}						
\renewcommand{\div}{\mathop\mathrm{div }}			
\def\warrow{\rightharpoonup}								

\def\Xint#1{\mathchoice
{\XXint\displaystyle\textstyle{#1}}%
{\XXint\textstyle\scriptstyle{#1}}%
{\XXint\scriptstyle\scriptscriptstyle{#1}}%
{\XXint\scriptscriptstyle\scriptscriptstyle{#1}}%
\!\int}
\def\XXint#1#2#3{{\setbox0=\hbox{$#1{#2#3}{\int}$ }
\vcenter{\hbox{$#2#3$ }}\kern-.58\wd0}}

\def\avint{\Xint-}
\theoremstyle{plain}
\newtheorem{theorem}{Theorem}

\newtheorem{lemma}[theorem]{Lemma}

\newtheorem{proposition}[theorem]{Proposition}

\theoremstyle{definition}

\newtheorem{definition}[theorem]{Definition}
\newtheorem{remark}[theorem]{Remark}

\numberwithin{equation}{section}
\numberwithin{theorem}{section}

\newtheorem{specialthm}{Theorem}

  \DeclareFontFamily{U}{mathb}{\hyphenchar\font45} 
\DeclareFontShape{U}{mathb}{m}{n}{
      <5> <6> <7> <8> <9> <10> gen * mathb
      <10.95> mathb10 <12> <14.4> <17.28> <20.74> <24.88> mathb12
      }{}
\DeclareSymbolFont{mathb}{U}{mathb}{m}{n}

\DeclareMathSymbol{\toitself}      {3}{mathb}{"FD}  

\newcommand{\vv}{\vspace{2mm}}
\newcommand{\vvv}{\vspace{4mm}}

\newcommand{\dv}{\mathop{\rm div}}
\def\R{\mathbb{R}}
\def\vphi{\varphi}
\def\om{\Omega}

\def\hm{\omega}

\begin{document}

\title[Layer potentials for elliptic operators with DMO-type coefficients]{Layer potentials for elliptic operators with DMO-type coefficients: big pieces $Tb$ theorem, quantitative rectifiability, \\ and free boundary problems}

\author[A. Merlo]{Andrea Merlo}
\address{Andrea Merlo\\ Departamento de Matem\'aticas, Universidad del Pa\' is Vasco (UPV/EHU), Barrio Sarriena s/n 48940 Leioa, Spain and\\
Ikerbasque, Basque Foundation for Science, Bilbao, Spain.}
\email{andrea.merlo@ehu.eus}

\author[M. Mourgoglou]{Mihalis Mourgoglou}
\address{Mihalis Mourgoglou\\ Departamento de Matem\'aticas, Universidad del Pa\' is Vasco (UPV/EHU), Barrio Sarriena s/n 48940 Leioa, Spain and\\
Ikerbasque, Basque Foundation for Science, Bilbao, Spain.}
\email{michail.mourgoglou@ehu.eus}

\author[C. Puliatti]{Carmelo Puliatti}
\address{Carmelo Puliatti \\Departament de Matemàtiques, Universitat Autònoma de Barcelona, 08193 Bellaterra (Barcelona), Catalonia.}
\email{carmelo.puliatti@uab.cat}

\subjclass[2020]{42B37, 42B20, 35J15, 28A75, 28A75, 33C55}
\thanks{A. M. was supported by IKERBASQUE and European Union’s Horizon 2020 research and innovation programme under the Marie Sklodowska-Curie grant agreement no 101065346. M.M. was supported by IKERBASQUE and partially supported by the grant PID2020-118986GB-I00 of the Ministerio de Economía y Competitividad (Spain), and by IT-1615-22 (Basque Government).
	C.P. was supported by the Research Council of Finland grant no.352649, and partially supported by the grant 2021-SGR-00071 (Catalonia) and PID2023-150984NB-I00 funded by MICIU/AEI/10.13039/501100011033/ FEDER, EU}
\keywords{Elliptic measure, harmonic measure, Riesz transform, single layer potential, free boundary problems, rectifiability, $Tb$ theorem.}

\newcommand{\mih}[1]{\marginpar{\color{red} \scriptsize \textbf{Mi:} #1}}
\newcommand{\car}[1]{\marginpar{\color{blue} \scriptsize \textbf{Carmelo:} #1}}
\maketitle

\begin{abstract}
	For $n \geq 2$, we consider the operator $L_A = -\mathrm{div }(A(\cdot)\nabla)$, where $A$ is a uniformly elliptic $(n+1)\times(n+1)$ matrix with variable coefficients, a Radon measure $\mu$ on $\mathbb{R}^{n+1}$, and the associated gradient of the single layer potential operator $T_\mu$. Under a Dini-type assumption on the mean oscillation of the matrix $A$, we establish the following results:
	\begin{enumerate}
	\item 
A rectifiability criterion for $\mu$  in terms of  $T_\mu$.  Under quantitative geometric and analytic assumptions within a ball $B$--including an upper $n$-growth condition on $\mu$ in $B$, a thin boundary condition, a scale-invariant decay condition expressed via a weighted sum of densities over dyadic dilations of $B$, and $L^2$ boundedness of the gradient of $T_\mu$--we show the following: if the support of $\mu$ lies very close to an $n$-plane in $B$, and $T_\mu 1$ is nearly constant on $B$ in the $L^2$ sense, then there exists a uniformly $n$-rectifiable set $\Gamma$ such that $\mu(B \cap \Gamma) \gtrsim \mu(B)$.
	\item A $Tb$ theorem for suppressed  $T_\mu$, which extends a well-known theorem of Nazarov, Treil, and Volberg, and holds also for a broader class of singular integral operators.
	\end{enumerate}
These results make it possible to prove both qualitative and quantitative one- and two-phase free boundary problems for elliptic measure, formulated in terms of (uniform) rectifiability, in bounded Wiener-regular domains.
\end{abstract}

\tableofcontents

\section{Introduction}

	\subsection{History and state of the art} In recent years, there has been significant progress in the study of the geometry of harmonic measure.
This has been made possible through the interplay of various techniques, ranging from geometric measure theory to partial differential equations. Harmonic analysis has also played a crucial role, particularly through the use of \textit{singular integral operators} (SIOs) and,  in particular, the Riesz transform, and  their deep connection to rectifiability.  A long-standing conjecture posed by Bishop in 1992 \cite[Conjecture 10]{Bis92}, and later resolved in \cite{AHMMMTV16}  by Azzam, Hofmann, Martell, Mayboroda, Tolsa, Volberg, and the second named author, concerns the following one-phase problem for harmonic measure:
\begin{specialthm}\label{theorem:Laplace-one-phase}
	Let $\Omega\subset\Rn1$, $n \geq 1$, be a bounded, open, connected set.  If $\omega^{p}_{\Omega}$ is the harmonic measure in $\Omega$ with pole at  $p\in\Omega$ and there exists $E\subset \partial \Omega$ such that $0<\mathcal H^n(E)<\infty$ and $\omega^{p}_{\Omega}|_E$ is absolutely continuous with respect to $\mathcal H^n|_E$, then $\omega^{p}_{\Omega}|_E$ is $n$-rectifiable.
\end{specialthm}

A quantitative version of Theorem~\ref{theorem:Laplace-one-phase}, for non-doubling measures on $\partial \om$ satisfying an upper growth condition, was established in \cite{MT20}  by Tolsa and the second named author  and reads as follows:
\begin{specialthm}\label{theorem:Laplace-one-phase_quant}
	Let \( n \geq 1 \), and let \( 0 < \kappa < 1 \) be a constant small enough and \( c_{db} > 1 \) another constant large enough, both depending only on \( n \). 
	Let \( \Omega \subset \mathbb{R}^{n+1} \) be an open set, and let \( \mu \) be a Radon measure supported on \( \partial \Omega \) that satisfies the growth condition \( \mu(B(x,r)) \leq c_0 r^n \) for all \( x \in \mathbb{R}^{n+1} \) and \( r > 0 \). 
	Suppose there exist \( \varepsilon, \varepsilon' \in (0,1) \) such that for every \( \mu \)-\( (2, c_{db}) \)-doubling ball \( B \) centered at \( \operatorname{supp} (\mu) \) with \( \operatorname{diam}(B) \leq \operatorname{diam}(\operatorname{supp} (\mu)) \), there exists a point \( x_B \in \kappa B \cap \Omega \) such that the following holds for any subset \( E \subset B \):
	\[
	\text{If} \quad \mu(E) \leq \varepsilon \, \mu(B), \quad \text{then} \quad \omega^{x_B}(E) \leq \varepsilon' \, \omega^{x_B}(B).
	\]
	Then the Riesz transform \( \mathcal{R}_\mu \colon  L^2(\mu) \to L^2(\mu) \) is bounded, and thus \( \mu \) is \( n \)-rectifiable. 
	
	Furthermore, if \( \mu(B(x,r)) \geq c_0^{-1} r^n \) for all \( x \in \mathbb{R}^{n+1} \) and \( r > 0 \), then \( \mu \) is uniformly \( n \)-rectifiable.
\end{specialthm}

A related two-phase variant of Bishop's conjecture, also formulated in \cite[Conjecture 10]{Bis92} in 1992 and resolved  in \cite{AMTV19}  by Azzam, Tolsa, Volberg, and the second named author (see also \cite{AMT17a}), asks whether mutual absolute continuity of harmonic measures for two disjoint domains with a common boundary portion implies rectifiability of their shared support.
\begin{specialthm}\label{theorem:Laplace-two-phase}
	Let $\Omega_1, \Omega_2 \subset \Rn1$, $n \geq 1$,  be two disjoint domains such that $\partial \Omega_1 \cap \partial \Omega_2\neq \varnothing$ and $E\subset \partial \Omega_1 \cap \partial \Omega_2$.  We denote  by $\omega_i=\omega^{x_i}_{\Omega_i}$ the harmonic measure in $\Omega_i$ with pole at  $x_i \in \Omega_i$, for $i=1,2$.  If $\omega_1|_E \ll \omega_2|_E\ll \omega_1|_E $,  then there exists a $n$-rectifiable set $F\subset E$ with  $\omega_1(E\setminus F)=0$ such that  $\omega_1|_F$ and $\omega_2|_F$ are mutually absolutely continuous with respect to $\mathcal H^n|_F$.
\end{specialthm}
A quantitative version of Theorem~\ref{theorem:Laplace-two-phase} was proved in \cite{AMT20} by Azzam, Tolsa, and the second-named author, and is stated as follows:
\begin{specialthm}\label{theorem:quantitative_two_phase_Laplace_indexed}
	Let \( n \geq 2 \). Let \( \Omega_1, \Omega_2 \subset \mathbb{R}^{n+1} \) be two bounded, open, connected, disjoint sets satisfying the capacity density condition (CDC), and set \( F \coloneqq \partial \Omega_1 \cap \partial \Omega_2 \neq \varnothing \).  Assume there exist \( \varepsilon, \varepsilon' \in (0,1) \) such that, for any ball \( B \) centered at \( F \) with \( \operatorname{diam}(B) \leq \min(\operatorname{diam}(\Omega_1), \operatorname{diam}(\Omega_2)) \), there exist \( x_i \in \tfrac{1}{4}B \cap \Omega_i \) for \( i = 1, 2 \), satisfying the following: for every measurable set \( E \subset B \),
	\begin{equation}\label{eq:hypothesis_A_infinity_type_Laplace_indexed}
		\text{if } \quad \omega^{x_1}_{\Omega_1}(E) \leq \varepsilon \, \omega^{x_1}_{\Omega_1}(B), \quad \text{then} \quad \omega^{x_2}_{\Omega_2}(E) \leq \varepsilon' \, \omega^{x_2}_{\Omega_2}(2B),
	\end{equation}
	where \( \omega^{x_i}_{\Omega_i} \) denotes harmonic measure with pole at \( x_i \) in \( \Omega_i \), for \( i = 1, 2 \).
	
	If \( \varepsilon' \) is sufficiently small, depending only on \( n \) and the CDC constant, then there exist \( \theta_i \in (0,1) \) and a uniformly \( n \)-rectifiable set \( \Sigma_B \subset \mathbb{R}^{n+1} \) such that
	\begin{equation}\label{eq:thm13_a_Laplace_indexed}
		\omega^{x_i}_{\Omega_i}(\Sigma_B \cap F \cap B) \geq \theta_i, \quad \text{for } i = 1, 2.
	\end{equation}
	Moreover, if \( x_1 \) is a corkscrew point for \( \tfrac{1}{4}B \), then there exists \( c > 0 \) such that
	\[
	\mathcal{H}^n(\Sigma_B \cap F \cap B) \geq c \, r(B)^n.
	\]
\end{specialthm}

\vv

The techniques developed in all the aforementioned works rely heavily on deep  results in the theory of singular integral operators, which are closely linked to the theory of rectifiability and its quantitative counterpart, uniform rectifiability, first introduced by David and Semmes in \cite{DS91, DS93}.

Let us now introduce some important notions that are necessary for the statement of the main results. 	We say that a non-negative Borel measure has \textit{growth of degree $d$}, or, for brevity, \textit{$d$-growth}, if there exists a constant $c_0 > 0$ such that
\[
\mu\bigl(B(x, r)\bigr) \leq c_0 \,r^d, \qquad \text{for all } x \in \mathbb{R}^{n+1}, \, r > 0,
\]
and we write $\mu \in {M}^d_+(\mathbb{R}^{n+1})$.

Such a  measure  $\mu$ is said to be \textit{$d$-Ahlfors regular} if it holds that
\[
c_0^{-1}r^d \leq \mu\big(B(x,r)\big) \leq c_0\,r^d, \quad \text{ for all } x \in \supp(\mu), \, 0 < r < \diam(\supp (\mu)).
\]
If $\mathcal{H}^d$ stands for the $d$-dimensional Hausdorff measure in $\mathbb{R}^{n+1}$, we say that a set $E \subset \mathbb{R}^{n+1}$ is \textit{$d$-Ahlfors regular} if $\mathcal{H}^d|_E$ is a $d$-Ahlfors regular measure.

A set $E\subset \Rn1$ is called \textit{$d$-rectifiable} if there exists a countable family of Lipschitz maps $f_j\colon \mathbb R^d\to \Rn1$ such that
\[
\mathcal H^d\Bigl(E\setminus \bigcup_jf_j(\mathbb R^d) \Bigr)=0.
\]
A measure $\mu$ is \textit{$d$-rectifiable} if it vanishes outside a $d$-rectifiable set $E$ and it is absolutely continuous with respect to $\mathcal H^d|_E$.

We say that a set $E\subset \Rn1$ is \textit{uniformly $d$-rectifiable} if it is $d$-Ahlfors regular and there exist $\theta,M>0$ such that, for all $x\in E$ and all $r>0$, there is a Lipschitz mapping $g$ from the ball $B_d(0,r)\subset\mathbb R^d$ to $\Rn1$ with $\Lip(g)\leq M$, such that
\[
\mathcal H^d\big(E\cap B(x,r)\cap g(B_d(0,r))\big)\geq \theta r^d.
\]
We also say that a measure $\mu$ is \textit{uniformly $n$-rectifiable} if it is $d$-Ahlfors regular and vanishes outside of a uniformly $d$-rectifiable set.

In \cite{DS91}, David and Semmes proved that a $d$-Ahlfors regular measure $\mu$ on $\Rn1$ is uniformly $d$-rectifiable if and only if \textit{all} singular integral operators with smooth, antisymmetric convolution-type kernels are bounded on $L^2(\mu)$. They also asked whether the $L^2(\mu)$-boundedness of the $d$-Riesz transform $\mathcal R^d_\mu$, associated with a $d$-Ahlfors regular measure $\mu$, implies that $\mu$ is uniformly $d$-rectifiable. This question is commonly referred to as \textit{David and Semmes' problem}, and a first positive answer was provided for $d = 1$ by Mattila, Melnikov, and Verdera in \cite{MMV96}.
The Menger curvature techniques used for the Cauchy transform, however, fail in higher dimensions. In their celebrated paper \cite{NToV14}, Nazarov, Tolsa, and Volberg gave an affirmative solution to the codimension-one case of the David and Semmes problem, corresponding to $d = n$ for any integer $n \geq 1$. All intermediate cases remain open.

Let us now discuss the singular integral operator techniques that have been employed in the study of \textit{free boundary problems} (FBPs) for harmonic measure. The one-phase problems addressed in \cite{AHMMMTV16} and \cite{MT20} make substantial use of the resolution of the David–Semmes problem in codimension one by Nazarov, Tolsa, and Volberg \cite{NToV14}, and its qualitative analogue in \cite{NToV14b}, as well as the $Tb$ theorem for suppressed kernels, due to Nazarov, Treil, and Volberg (see \cite{NTrV99} for its formulation in the context of the Cauchy transform). 

The two-phase problems treated in \cite{AMT17a}, \cite{AMTV19}, and \cite{AMT20} required an even more delicate analysis. In particular, ruling out points of zero harmonic measure density necessitated a local, quantitative rectifiability criterion, expressed in terms of suitable boundedness conditions on the $L^2$-norm and the $L^2$-mean oscillation of the Riesz transform on a flat ball, developed by Girela-Sarrión and Tolsa in \cite{GT18}. Both the $Tb$ theorem for suppressed kernels and the techniques used in its proof played a crucial role in adapting the results of \cite{GT18}, making them suitable for these applications.

\vvv

It is natural to ask whether Theorems~\ref{theorem:Laplace-one-phase}, \ref{theorem:Laplace-one-phase_quant}, \ref{theorem:Laplace-two-phase}, and \ref{theorem:quantitative_two_phase_Laplace_indexed} admit elliptic analogues. In particular, one can consider second-order uniformly elliptic operators in divergence form and their associated \textit{elliptic measures}, and study the corresponding non-variational free boundary problems under mild regularity assumptions on the coefficients of the associated elliptic matrix.

In the case of operators with Hölder continuous coefficients, such results have been obtained by Prat, Tolsa, and the third-named author in \cite{PPT21} (one-phase), and by the third-named author in \cite{Pu19} (two-phase). The validity of the suppressed $Tb$ theorem for anti-symmetric Calderón-Zygmund singular integral operators with Hölder continuous kernels enables one to pursue a strategy similar to that used in the study of free boundary problems for harmonic measure. However, the main challenge lies in extending the results of \cite{NToV14}, \cite{NToV14b}, and \cite{GT18} to the gradient of the single layer potential associated with the elliptic operator—an operator that coincides with the Riesz transform in the case of the Laplacian. This technically demanding task was successfully addressed in \cite{CMT19} by Conde-Alonso, Tolsa, and the second-named author, providing a generalization of \cite{ENV12}, as well as in \cite{PPT21} and \cite{Pu19}. In all of these works, the authors followed the overall strategy of the original papers, but the extension to the elliptic setting required overcoming substantial and technically intricate obstacles, resulting in a highly nontrivial advancement.

Finally,  let us stress out that  one-phase free boundary problems for elliptic measure in terms of (uniform) $n$-rectifiability have also been studied using approaches that avoid singular integral techniques.  In \cite{AM19}, Azzam and the second-named author proved that in uniform domains satisfying the capacity density condition (CDC) and having surface measure with positive lower density, the absolute continuity of surface measure with respect to elliptic measure—when the associated operator has $\VMO$ coefficients—implies rectifiability of the boundary. Independently and concurrently, Toro and Zhao \cite{TZ21} reached the same conclusion under the assumptions that the coefficient matrix $A$ belongs to $W^{1,1}$ and the domain is uniform with Ahlfors regular boundary.

	There is a vast literature on the quantitative one-phase problem for harmonic and elliptic measures. 
	We highlight, for instance, that for uniform domains, it was shown in \cite{HMU14} that a quantitative absolute continuity of harmonic measure with respect to surface measure, the so-called \textit{weak-$A_\infty$ property}, implies uniform rectifiability of the boundary.
	This result was later extended in \cite{HLMN17} to open sets with Ahlfors regular boundaries that satisfy the corkscrew condition.
	The weak-$A_\infty$ property of harmonic measure is particularly significant because of its connection to the solvability of the Dirichlet problem with data in $L^p$. A complete geometric characterization of open sets having this property was given in \cite{AHMMT20}, to which we refer for a further discussion.
	
	As for elliptic measure, \cite{AGMT17} established that under $L^1$-type assumptions on the coefficient matrix, if $\Omega$ is a domain with Ahlfors regular boundary satisfying the corkscrew condition, then the weak-$A_\infty$ property of elliptic measure implies the uniform rectifiability of $\partial \Omega$, thus generalizing the result of \cite{HLMN17}.
	The connection between the weak-$A_\infty$ property of elliptic measure and uniform rectifiability was also achieved in \cite{HMMTZ21} for elliptic operators whose coefficients' gradients satisfy a suitable \(L^2\)-Carleson measure condition, the so-called \emph{Dahlberg--Kenig--Pipher coefficients}, with sufficiently small Carleson constant.
	More recently, this was further extended to open sets with Ahlfors regular boundary satisfying the corkscrew condition under an $L^1$-Carleson condition on the gradient of the coefficient matrix, as shown in \cite{CHM24}.
	For further references on two-phase FBPs for harmonic measure, we refer for instance to  \cite{KT06}, \cite{KPT09}, \cite{BH16}, \cite{En16}, \cite{PT20}, \cite{TT24}, \cite{To24}.

\vv

\subsection{Motivation and setting} 
The main motivation of the present manuscript is to generalize the results of \cite{AHMMMTV16}, \cite{MT20}, \cite{AMTV19}, and \cite{AMT20} to the setting of elliptic operators with coefficients more general than H\"older continuous, using an approach grounded in the theory of singular integral operators.

Let $A(\cdot)$ be an $(n+1)\times(n+1)$-matrix whose entries $(a_{ij})_{i,j\in \{1,\ldots,n+1\}}$ are measurable real-valued functions in $L^\infty(\Rn1)$. The matrix $A(\cdot)$ is said \textit{uniformly elliptic} if there exists $\Lambda>0$ such that
\begin{align}
	\langle A(x)\xi,\xi\rangle &\geq \Lambda^{-1}|\xi|^2,\quad \qquad \text{ for all } \xi\in \Rn1 \text{ and } \mathcal L^{n+1}\text{-a.e. }x\in \Rn1,\label{eq:ellip1_scalar}\\
	\langle A(x)\xi,\eta\rangle &\leq \Lambda |\xi||\eta|,\qquad \quad \text{ for all } \xi,\eta\in \Rn1 \text{ and } \mathcal L^{n+1}\text{-a.e. }x\in \Rn1,\label{eq:ellip2_scalar}
\end{align}
where $\mathcal L^{n+1}$ denotes the Lebesgue measure on $\Rn1$.
We consider the operator
\begin{equation}\label{eq:snd_order_diff_eq}
	L_A u(x)\coloneqq -\div(A(\cdot)\nabla u(\cdot))(x) = 0, \qquad x \in \mathbb{R}^{n+1},
\end{equation}
with the identity \eqref{eq:snd_order_diff_eq} to be interpreted in the sense of distributions, see \eqref{eq:def_snd_order_elliptic_sol}.
We say that a function $\Gamma_A$ is a \textit{fundamental solution} to \eqref{eq:snd_order_diff_eq} if \(L_{A}\Gamma_A(\cdot,y) = \delta_y\) in the sense of distributions.  
We remark that, if $A$ is a uniformly elliptic matrix with real coefficients in \(L^\infty(\mathbb{R}^{n+1})\), there exists a fundamental solution associated with \(L_A\), and we refer to \cite{HK07} for its construction, see also Section \ref{section:preliminaries_and_notation}. If $A_0$ is a uniformly elliptic matrix with constant coefficients, we denote the fundamental solution to  $L_{A_0}u=0$ as $\Theta(\cdot, \cdot; A_0)$.

The key object of study in this paper is the \textit{gradient of the single layer potential}, which serves as the natural elliptic counterpart to the Riesz transform. 	For a non-negative Radon measure $\mu$ on $\Rn1$,  we define it as follows:
\begin{equation}\label{eq:single layer}
	T_\mu f(x)\coloneqq \int \nabla_1 \Gamma_A(x,y)f(y)\, d\mu(y),\qquad \text{ for }f\in L^1_{\loc}(\mu), 
\end{equation}
that we understand in the sense of the $\varepsilon$-truncations
\[
T_{\mu,\varepsilon} f(x)\coloneqq \int_{|x-y|>\varepsilon} \nabla_1 \Gamma_A(x,y)f(y)\, d\mu(y),\qquad \text{ for }f\in L^1_{\loc}(\mu).
\]

We say that the operator $T_\mu$ is \textit{bounded on} $L^2(\mu)$ if $T_{\mu, \varepsilon}$ is bounded on $L^2(\mu)$ uniformly in $\varepsilon > 0$, and we write
\[
\|T_\mu\|_{L^2(\mu)\to L^2(\mu)} \coloneqq \sup_{\varepsilon > 0}\|T_{\mu,\varepsilon}\|_{L^2(\mu) \to L^2(\mu)}.
\]

We remark that, if $A \equiv \mathrm{Id}$, we have $L_A = -\Delta$, so $\nabla_1 \Gamma_{\mathrm{Id}}$ equals the Riesz kernel up to a dimensional multiplicative constant. 
Hence, $T_\mu$ is the natural elliptic generalization of the $n$-Riesz transform $\mathcal R_\mu$; we recall that, if $\mu$ is a Radon measure on $\Rn1$, $n\geq 1$, its associated ($d$-dimensional) Riesz transform is defined as
\begin{equation}\label{eq:definition_Riesz_transform}
	\mathcal R^d_\mu f(x)=\int \frac{x-y}{|x-y|^{d+1}}f(y)\, d\mu(y),\qquad \text{ for } f\in L^1_{\loc}(\mu),
\end{equation}
to be interpreted again in the sense of truncations. For $f\equiv 1$ on $\Rn1$, we also use the notation $\mathcal R^d\mu (x)= \mathcal R^d_\mu 1(x)$ and $T\mu(x)=T_{\mu}1(x)$. In the $1$-codimensional case $d=n$, we often write $\mathcal R_\mu f=\mathcal R^{n}_\mu f$.

In order to interpret $\nabla_1 \Gamma_A$ as a Calder\'on--Zygmund-type kernel, the mere assumption that the coefficients of $A$ belong to $L^\infty(\mathbb{R}^{n+1})$ is not sufficient, as this does not necessarily guarantee local $L^\infty$-estimates. This motivates the need to impose additional regularity assumptions on $A$.

A natural class of matrices to consider is those for which weak solutions to $L_Au=0$ are continuously differentiable and  have local $L^\infty$ and  regularity estimates for $\nabla u$. These properties were proved by Dong and Kim \cite{DK17} for matrices whose mean oscillation satisfy a Dini condition, whose  definitions we recall below.  	Elliptic operators associated with $\DMO$ and $\DMO$-type matrices are of significant interest in a broad variety of PDE problems, with a considerable amount of recent research activity in the field. We cite, for instance, \cite{CD19}, \cite{CDX22}, \cite{DEK18}, \cite{DJ25},\cite{DKL23}, \cite{DJ25}, \cite{DJV24}, \cite{DL20}, \cite{Le24}, and the references therein.

Let $\kappa \geq 1$.	We say that $\theta\colon [0,\infty] \to [0,\infty]$ is a  \textit{$\kappa$-doubling} function if it satisfies
\begin{equation}\label{eq:dini1_new}
	\theta(t) \leq \kappa\, \theta(s), \qquad \text{for all }\, t > 0  \text{ and }\frac{t}{2} \leq s \leq t .
\end{equation}

A $\kappa$-doubling function $\theta$ is said to belong to the class $\DS(\kappa)$ (\textit{Dini at small scales}) if it is $\mathcal{L}^1$-measurable and satisfies
\begin{equation}\label{eq:sDini}
	\int_0^1 \theta(t)\, \frac{dt}{t} < \infty
\end{equation}
and, given $d>0$, we say that $\theta$ belongs to $\DL_{d}(\kappa)$ ($d$-{\it Dini in large scales}) if it is $\mathcal L^1$-measurable and 
\[
\int_1^\infty \theta(t)\,\frac{dt}{t^{d+1}}<\infty.
\]
In particular, if $0<d_1\leq d_2$, then $\DL_{d_1}(\kappa)\subset \DL_{d_2}(\kappa)$. Moreover, for $\theta\in \DS(\kappa)$ we define
\begin{equation}\label{eq:def_mathfrak_I}
	\mathfrak I_{\theta}(r)\coloneqq \int_0^r \theta (t)\, \frac{dt}{t}, \qquad \text{ for } r>0,
\end{equation}
and, for $d>0$ and $\theta\in \DL_d(\kappa)$,
\begin{equation}\label{eq:def_mathfrak_L}
	\mathfrak L^d_{\theta}(r)\coloneqq r^{d}\int_r^\infty \theta(t)\, \frac{dt}{t^{d+1}} , \qquad\text{ for } r>0.
\end{equation}	

\vv
{		For $x \in \mathbb{R}^{n+1}$ and $r > 0$, we denote by $B(x, r)$ the open ball centered at $x$ of radius $r$.  For $x \in \mathbb{R}^{n+1}$, $r > 0$, and an $(n+1) \times (n+1)$ matrix $A$, we write
	\[
	\bar{A}_{x, r} \coloneqq \avint_{B(x, r)} A \coloneqq \frac{1}{|B(x, r)|} \int_{B(x, r)} A(y)\, dy,
	\]
	and define its \textit{mean oscillation} function $\oomega_{A} \colon [0, \infty) \to [0, \infty)$ as 
	\[
	\oomega_{A}(r) \coloneqq \sup_{x \in \mathbb{R}^{n+1}} \avint_{B(x, r)} \bigl|A(z) - \bar{A}_{x, r}\bigr|\, dz.
	\]
	
	We remark that mean oscillation functions are $\kappa$-doubling: by \cite[p.495]{Li17}, there exists a constant $\kappa > 0$, depending only on the dimension, such that $\oomega_A$ satisfies the doubling condition \eqref{eq:dini1_new}, see also the introduction of \cite{MMPT23}.

	We say that an $(n+1)\times (n+1)$-matrix $A \in \DMO_s$ (resp. $A \in \DMO_\ell(d)$)  if $ \oomega_A \in \DS(\kappa)$ (resp. $ \oomega_A \in \DL_{d}(\kappa)$ for $d\geq 0$).	We also say that $A \in \DDMO_s$ if $A \in \DMO_s$ and $\mathfrak I_{ \oomega_A}$ satisfies \eqref{eq:sDini}, namely,
	\begin{equation}\label{eq:logdini}
		\mathfrak I_{{\mathfrak I}_{\oomega_A}}(1)\coloneqq\int_0^1 \int_0^r  \oomega_A(t)\, \frac{dt}{t}\frac{dr}{r}=-\int_0^1  \oomega_A(t)\,\log{t} \,\frac{dt}{t}<+ \infty.	
	\end{equation}
	
	\vv
	
In light of Remark \ref{rem:continuous representative}, the class of $\DDMO_s$ functions can be viewed, up to a set of measure zero,  as a subclass of the class of uniformly continuous functions whose modulus of continuity satisfies a Dini condition. 	Note also that  $ \DMO_s \varsubsetneq \VMO(\Rn1)$.
	
	\vv
	
	Let us set 
	\begin{equation}\label{eq:logdini1}
		{\mathfrak F}_{\oomega_A}(r)\coloneqq \int_0^r  \oomega_A(t)\,\Bigl(1+\log\frac{1}{t}\Bigr) \,\frac{dt}{t}={\mathfrak I}_{\oomega_A}(r)+ \mathfrak I_{{\mathfrak I}_{\oomega_A}}(r), \quad \text{ for } r>0,
	\end{equation}
and  define the main class of matrices considered in this article as
	\[
	\widetilde{\DMO}_d \coloneqq \textup{DDMO}_s \cap \textup{DMO}_\ell(d), \qquad \text{ for } d\geq 0.
	\]
	The acronyms $\DMO$ and $\DDMO$ stand for \textit{Dini mean oscillation} and \textit{double Dini mean oscillation}, respectively. The subscripts in $\DMO_s$ and $\DMO_\ell$ indicate that the corresponding Dini condition is imposed at small and large scales, respectively.
	It is important to note that the class $\widetilde{\DMO}_d$ strictly contains the family of H\"older continuous matrices; see Remark \ref{rem:remarks_on_DMO}.
	
	\vvv
	Under the assumption that $A$ is a uniformly elliptic matrix whose oscillation is Dini integrable at small scales, the gradient of the single layer potential $\nabla_1 \Gamma_A(\cdot, \cdot)$ satisfies local Calder\'on--Zygmund-type estimates in terms of the Dini integral of $\oomega_A$. These estimates allow one to study $\nabla_1 \Gamma_A$ as the kernel of a singular integral operator; see Lemma~\ref{lem:estim_fund_sol}. However, an effective analysis of this operator requires integrability of the relevant quantities at both small and large scales, which is why we also impose the additional assumption $A \in \DMO_\ell(d)$ for the particular choice of $d$ made below and $A \in  \textup{DDMO}_s$.
	
	This PDE framework is not new in the study of the gradient of the single layer potential. 	In fact, motivated by the geometric applications of studying the gradient of single layer potentials,  Molero,  Tolsa, and the second and third named authors of the present manuscript, investigated in \cite{MMPT23} the link between the $L^2$-boundedness of $T_\mu$ and the $L^2$-boundedness of the Riesz transform $\mathcal{R}_\mu$ in the $\widetilde{\mathrm{DMO}}$ setting. 	In particular, they proved that if $n \geq 2$, $A$ is a uniformly elliptic matrix with coefficients in $\widetilde{\mathrm{DMO}}_{n-1}$  and $\mu \in {M}^n_+(\mathbb{R}^{n+1})$ is a measure with compact support in $\mathbb{R}^{n+1}$, then $\mathcal{R}_\mu$ is bounded on $L^2(\mu)$ if and only if $T_\mu$ is bounded on $L^2(\mu)$. More specifically, it holds that
	\begin{equation}\label{eq:maineq}
		1 + \|\mathcal{R}_\mu\|_{L^2(\mu)\to L^2(\mu)} \approx 1 + \|T_\mu\|_{L^2(\mu)\to L^2(\mu)},
	\end{equation}
	where the implicit constant depends on $n$, $\mathrm{diam}(\mathrm{supp}(\mu))$, the growth of $\mu$, and the uniform ellipticity of $A$.
	As proved in \cite{MMPT23}, this result readily gives a $\widetilde{\mathrm{DMO}}$ variant of \cite{ENV12},  \cite{NToV14},  and \cite{NToV14b} for elliptic operators whose coefficients belong to a subclass of the Dini mean oscillation class, which strictly contains the H\"older continuous matrices.

	\vvv
}

\subsection{Main results I: A quantitative rectifiability criterion for the gradient of the single layer}
Let $\mu$ be a Radon measure on $\Rn1$. For a ball $B\subset \Rn1$ of radius $r(B)$ and $\gamma\in(0,1)$, we denote
\begin{equation*}
	\Theta^n_\mu(B)\coloneqq \frac{\mu(B)}{r(B)^n},\qquad \text{ and }\qquad
	P_{\gamma, \mu}(B)\coloneqq  \sum_{j\geq 0} 2^{-{\gamma j}}\Theta^n_{\mu}(2^j B).
\end{equation*}

\vv
Given $C>0$, we say that the ball $B$ is \textit{$C$-$P_{\gamma,\mu}$-doubling} if
\begin{equation}\label{eq:P_doubling_balls}
	P_{\gamma,\mu}(B)\leq C\, \Theta^n_\mu(B).
\end{equation}
If the value of $C$ is clear from the context, we simply say that the cube is $P_{\gamma,\mu}$\textit{-doubling}.

\vv
Given $t>0$, we say that a ball $B \subset \Rn1$ has  {\it $t$-thin boundary} if 
\begin{equation}\label{eq:small boundary}
	\mu\bigl(\bigl\{x\in 2B: \dist(x,\partial B)\leq \lambda r(B)\bigr\}\bigr)\leq t\lambda \,\mu(2B), \qquad\text{ for all } \lambda>0.
\end{equation}

\vv

Given an $n$-dimensional plane $L$ in $\Rn1$ we denote
\[
\beta^L_{\mu,1}(B)=\frac{1}{r(B)^n}\int_B\frac{\dist(x,L)}{r(B)}\, d\mu(x)\qquad\text{ and }\qquad \beta_{\mu,1}(B)=\inf_L \beta^L_{\mu,1}(B),
\]
where the infimum is taken over all hyperplanes.
Finally, for a set $E\subset\Rn1$ with $\mu(E)>0$ and $f\in L^1_{\loc}(\mu)$ we write
\[
m_{E}(f,\mu)=\frac{1}{\mu(E)}\int_E f\, d\mu.
\]

Let \( M(\mathbb{R}^{n+1}) \) be the space of real Borel measures, equipped with the total variation norm \( \|\cdot\| \). For \( \mu \in M(\mathbb{R}^{n+1}) \), we denote its variation by \( |\mu| \) and define \( \|\mu\| \coloneqq |\mu|(\mathbb{R}^{n+1}) \).

\vvv
In the next statement we denote by $T_{\mu}$ the principal values of the gradient of the single layer potential, which exist due to \cite[Theorem 1.5]{MMPT23} and the argument in \cite[Section 2.4]{GT18} (see also \cite[Section 3]{Pu19}).

\begin{theorem}\label{theorem:elliptic_GSTo}
	Let	$A$ be a uniformly elliptic matrix satisfying  $A\in \widetilde \DMO_{1-\alpha}$ for some $\alpha \in (0,1)$,  $\mu$ be a Radon measure with compact support in $\Rn1$, $n\geq 2$, such that $\diam(\supp(\mu)) \leq R$ for some $R>0$,  and denote by $T_\mu$ the gradient of the single layer potential associated with $L_A$ and $\mu$.  Let also $B\subset \Rn1$ be a ball centered at $\supp(\mu)$ of radius $r(B)$.  Assume that for some positive real numbers $ \mathfrak C_1,  \mathfrak C_2,   \mathfrak C_3,  \mathfrak C_4, \tau, \delta$, and $\lambda$, the following properties hold:
	\begin{enumerate}
		\item $r(B)\leq\lambda$.
		\item $P_{\alpha, \mu}(B)\leq \mathfrak C_1 \,\Theta^n_\mu(B)$.
		\item $\Theta^n_\mu(B(x,r)) \leq \mathfrak C_2  \,\Theta^n_\mu(B)$,   for every  $x \in B$ and $r \in (0, 2\, r(B))$.
		\item $B$ has $\mathfrak C_3$-thin boundary, in the sense of \eqref{eq:small boundary}.
		\item $T_{\mu|_{B}}$ is bounded on $L^2(\mu|_{B})$ and it holds 
		\begin{equation}\label{eq:L^2_bd_2n_B}
			\|T_{\mu|_{B}}\|_{L^2(\mu|_{B})\to L^2(\mu|_{B})}\leq \mathfrak C_4 \,\Theta^n_\mu(B).
		\end{equation}
		\item There exists some $n$-plane $L$ passing through the center of $B$ such that 
		$$\beta^L_{\mu,1}(B)\leq \,\delta\,\Theta^n_\mu(B).$$
		\item We have
		\begin{equation}\label{eq:mean_osc_small_theorem}
			\int_B \biggl|T_\mu 1 - \avint_B T_\mu 1\, d\mu\biggr|^2\, d\mu\leq\,\tau\,\Theta^n_\mu(B)^2\mu(B).
		\end{equation}
	\end{enumerate}
	Then there exists $\theta \in (0,1)$ so that, if $\delta$ and $\tau$ are small enough, depending on $n$, $\Lambda$, $\mathfrak{C}_1$, $\mathfrak{C}_2$, $\mathfrak{C}_3$, $\mathfrak{C}_4$, and $\diam(\supp(\mu))$, and if $\lambda$ is small enough, depending on $n$, $\Lambda$, $\mathfrak{C}_1$, $\mathfrak{C}_2$, $\mathfrak{C}_3$, $\mathfrak{C}_4$, $\diam(\supp(\mu))$, and $\tau$, then there exists a uniformly $n$-rectifiable set $\Gamma$ such that
	\[
	\mu(B \cap \Gamma) \geq \theta \mu(B).
	\]
	The uniform rectifiability constants of $\Gamma$ depend on all the constants above.  The same result holds for the adjoint operator $T^*$\footnote{The same result  holds for the double layer potential operator as well.}. 
\end{theorem}

\vvv

Let us point out that a  significantly weaker version of Theorem \ref{theorem:elliptic_GSTo} appeared in \cite[Corollary 1.6]{MMPT23}.  Below,  we highlight the crucial differences in the hypotheses:
\begin{itemize}
	\item Hypothesis (3) replaces the much stronger control on the density
	\[
	\Theta^n_\mu(B(x,r)) \lesssim \Theta^n_\mu(2^N B)
	\]
	for \textit{all} $x \in B$ and at \textit{all} scales $0 < r \leq 2^N r(B)$, for $N \gg 1$ large enough.
	
	\item Assumption (2) in Theorem \ref{theorem:elliptic_GSTo} is the natural $P$-doubling condition in this framework. A much more restrictive and somewhat artificial $P$-doubling assumption was formulated in \cite{MMPT23} by introducing the quantities
	\[
	\alpha_A(t) = t + t^\beta + \oomega_A(t), \qquad \text{and} \qquad \mathcal{P}^N_{\alpha_A,\mu}(B) = \sum_{j \geq N} \alpha_A(2^{-j}) \Theta_\mu(2^j B),
	\]
	for $\beta$ as in Lemma \ref{lem:estim_fund_sol}, and assuming that, for $N$ large enough,
	\begin{equation}\label{eq:P_doubling_old}
		\mathcal{P}^0_{\alpha_A,\mu}(B) \lesssim \Theta^n_\mu(B) \qquad \text{and} \qquad \mathcal{P}^N_{\alpha_A,\mu}(B) \lesssim \mathfrak{I}_{\alpha_A}(2^{-N}) \Theta^n_\mu(2^N B).
	\end{equation}
	Despite allowing recovery of the main result of \cite{Pu19} when $A$ is H\"older regular\footnote{See the discussion in \cite[pp. 12-13]{MMPT23}.}, conditions \eqref{eq:P_doubling_old} had the drawback of being unsuitable for the applications to FBPs for elliptic measure in a broader $\DMO$ setting. Indeed, in Section~\ref{section:main_lemma}, we apply Theorem \ref{theorem:elliptic_GSTo} to a sequence of arbitrarily small balls that are $P_{\alpha, \mu}$-doubling in the sense of \eqref{eq:P_doubling_balls}, which would not have been possible under the formulation of \cite[Corollary 1.6]{MMPT23}.
	
	\item We impose the thin boundary hypothesis~(4) for the ball $B$ for technical reasons, specifically to facilitate the analysis of error terms in the buffer zones arising in our perturbation argument. This is a mild assumption that poses no obstacle in applications since, by \cite[Lemma 9.43]{To14}, for any ball $B$, there exists a concentric ball $B'$ with thin boundary such that $B \subset B' \subset 1.1B$.
	
	\item The $L^2$-boundedness of $T_\mu$ was assumed to hold at the level of the ball $2^N B$, with $N \gg 1$, whereas now we only require it for $B$.
	
\end{itemize} 
\vvv

Theorem \ref{theorem:elliptic_GSTo} improves upon \cite[Corollary 1.6]{MMPT23}, whose proof involved bounding \eqref{eq:mean_osc_small_theorem} in terms of the $L^2$-mean oscillation of the Riesz transform at a local scale using a three-step perturbation argument. Specifically, it consisted of comparing $\nabla_1 \Gamma_A(x,y)$ with the gradient of the fundamental solution associated with the averaged matrices $\bar A_{x,|x-y|/2}$ and $\bar A_{x,\delta}$, where $\delta$ denotes the truncation level of $T_{\mu,\delta}$.

The proof of Theorem~\ref{theorem:elliptic_GSTo} relies on two crucial components. The first is the derivation of pointwise bounds for the perturbation kernel
\begin{equation}\label{eq:intro-perturb-kernel}
	\mathfrak{K}_A(x, y) \coloneqq \nabla_1 \Gamma_A(x, y) - \nabla_1 \Theta(x, y; \bar{A}_B),
\end{equation}
for all $x \neq y$ with $|x - y| < R$.

To obtain these bounds, we follow the approach developed in \cite{MMPT23}, where the authors employ a three-step perturbation argument. In the first step, they estimate the difference
\[
\nabla_1 \Gamma_A(x, y) - \nabla_1 \Theta\bigl(x, y; \bar{A}_{x, |x - y|/2}\bigr),
\]
as stated in \cite[Lemma 3.12]{MMPT23}. In the second step, they bound
\[
\nabla_1 \Theta\bigl(x, y; \bar{A}_{x, r/2}\bigr) - \nabla_1 \Theta\bigl(x, y; \bar{A}_{x, \delta}\bigr),
\]
for $|x - y| > \delta$, where $\delta$ denotes the truncation scale of $T_{\mu, \delta}$ (see \cite[Lemma 3.13]{MMPT23}). The third step, \cite[Lemma 3.14]{MMPT23},concerns the difference
\[
\nabla_1 \Theta\bigl(z, 0; \bar{A}_{x, \delta}\bigr) - \nabla_1 \Theta(z, 0; \bar{A}_B).
\]

A technical obstacle arises because $\lim_{\delta \to 0} \bar{A}_{x,\delta} = A(x)$ holds only for $\mathcal{L}^{n+1}$-almost every $x \in \mathbb{R}^{n+1}$, whereas the measures of interest may be supported on sets of zero Lebesgue measure. Consequently, obtaining estimates that hold for all $x \neq y$ with $|x - y| < R$ requires a different approach.

In the present paper, we avoid the comparison at the level of the truncations, which is one of the main technical challenges. Without doing so, it would be impossible to apply the $Tb$-theorem or take principal values of the gradient of the single layer potential. A key novelty of our approach is the replacement of $\bar{A}_{x,\delta}$ by the matrix $\mathcal{A}(x)$, the \emph{uniformly continuous representative} of $A$ with modulus of continuity $\mathfrak{I}_{\oomega_A}$ (see \eqref{eq:modulus_continuity_mathcal_A}). This substitution preserves the averages while enabling the application of the Lebesgue differentiation theorem at \emph{all} points in $\mathbb{R}^{n+1}$, not merely $\mathcal{L}^{n+1}$-almost everywhere.

The second crucial component is the perturbation argument for the $L^2$-mean oscillation of $T_\mu$, developed in Lemma~\ref{lem:mean oscillation perturb} via a duality argument. To carry it out, we split the perturbation operator according to three spatial regions: the \emph{local term} on the ball $B$, the \emph{far-away term} outside $2B$, and the \emph{intermediate buffer zone} $2B \setminus B$. As before, the ball $B$ is assumed to be $P_{\alpha, \mu}$-doubling with thin boundary, conditions that are essential for controlling the error terms.

\begin{enumerate}
	\item The local and main  terms can be bounded via the $L^2$-norm of the Riesz transform using Lemmas~\ref{lem:truncatedL2} and ~\ref{lem:estimate_norm_K3_new}.
	
	\item In the far-away region, $\mathbb{R}^{n+1} \setminus 2B$, neither a direct analogue of step~(1) nor a pointwise regularity estimate is available. Instead, we prove a Hörmander-type smoothness condition for the kernel \eqref{eq:intro-perturb-kernel}, as established in Lemma~\ref{lemma:hormander_perturbation}.
	A key part of the argument is a careful analysis that yields sufficiently small bounds on the right-hand side, achieved through delicate PDE estimates. These bounds form the core of the proof of Theorem~\ref{theorem:elliptic_GSTo} and critically rely on the $P_{\alpha, \mu}$-doubling property of the ball $B$, which is indispensable for our method. To be precise,  we obtain 
	\begin{equation}\label{eq:hormander_intro}
		\int_{\mathbb{R}^{n+1} \setminus 2B} \bigl| \mathfrak{K}_A(x,y) - \mathfrak{K}_A(x,y_B) \bigr| \, d\mu(x)
		\lesssim_\alpha \, \Theta^n_\mu(B)\bigl(\mathfrak{I}_{\oomega_A}(r(B)) + \mathfrak{L}^{1-\gamma}_{\oomega_A}(r(B))\bigr),
	\end{equation}
	where $y_B$ is the center of $B$. We remark that, assuming $A\in\widetilde \DMO_{1-\alpha}$, the quantity on the right-hand side of \eqref{eq:hormander_intro} is infinitesimal as $r(B) \to 0$.  A bound analogous to \eqref{eq:hormander_intro} can also be proven for the kernel $\mathfrak{K}_A^*$ of the adjoint of the perturbation operator.

	\item The study of the perturbation operator in $2B \setminus B$ is delicate as well, since pointwise bounds are unsuitable in this region. Hence,  in Lemma \ref{lem:L1 buffer},  we prove $L^1$ estimates, which are based on a dyadic decomposition of $\supp(\mu)$ at the scale of the ball $B$, together with a thin-boundary and doubling properties for $B$ with respect to $\mu$.
\end{enumerate}

Hence, via a duality argument, gathering all these estimates allows us to control the $L^2$ mean oscillation of the integral operators associated with the kernels of the perturbations in terms of $\mathfrak I_{\oomega_A}(r(B))$ times the $L^2$ norm of the Riesz transform at the scale of the ball $B$, along with an error term involving $\mathfrak F_{\oomega_A}$, $\mathfrak L_{\oomega_A}$, and the $P_{\gamma, \mu}$-doubling character of $B$; see Lemma~\ref{lem:mean oscillation perturb}.
This, in turn, yields a bound on the $L^2$ mean oscillation of $T$ and its adjoint $T^*$ in terms of the analogous quantity for the Riesz transform. Therefore, by invoking \eqref{eq:maineq} and the main result of \cite{GT18}, we conclude the proof of Theorem \ref{theorem:elliptic_GSTo}.

We note that the inequality \eqref{eq:hormander_intro} is also of independent interest, as it differs from the more typical scenario where one obtains a modulus of continuity for such kernels--which is not possible in our setting. A related condition was recently investigated in a different context in \cite{Su21}, as an endpoint case of the results in \cite{GS19}. Moreover, if the measure \(\mu\) has  $n$-growth,   then \eqref{eq:hormander_intro} yields the \(L^\infty(\mu)\) version of the condition in scales smaller than $1$, which corresponds to the classical Hörmander condition for the kernel \(\mathfrak K_A\) and the measure \(\mu\).

\vv

Let us say a few words about the proof of \cite[Corollary 1.6]{MMPT23}.  The goal in \cite{MMPT23} was again to show that the perturbation operator with kernel
\(
	\mathfrak{K}_A(x,y) 
\)
has small $L^2$-mean oscillation, and to combine this with \eqref{eq:L^2_bd_2n_B} to deduce the corresponding mean oscillation bound for the Riesz transform (under the assumption $\bar{A}_B = \mathrm{Id}$).  As mentioned above, a key ingredient in the proof of Theorem~\ref{lem:mean oscillation perturb} is having the regularity estimates for the kernel $\mathfrak{K}_A(\cdot, \cdot)$, which constitutes one of the main novelties of the present manuscript. In contrast, the authors of \cite{MMPT23} treated the terms $\nabla_1 \Gamma_A(\cdot, \cdot)$ and $\nabla_1 \Theta(\cdot, \cdot; \bar{A}_B)$ separately, thereby losing valuable information arising from the cancellation in the kernel $\mathfrak{K}_A$. This is the reason why it was important to impose the density and  the generalized Poisson-doubling (associated to $\oomega$) assumptions related to the ball $2^N B$ for $N$ large enough, so that they can create the desired smallness.

\vv

\subsection{Main results II: The big pieces $Tb$ theorem for the gradient of the single layer}
A crucial component in applications to free boundary problems is the {\textit{big pieces $Tb$ theorem}}, which ensures the $L^2$-boundedness of singular integral operators under suitable conditions on a class of testing functions $b$. More specifically, for applications to two-phase problems, the quantitative rectifiability theorem must be formulated, as in Theorem~\ref{teo1}, under an appropriate assumption on the \textit{maximal truncation} $T_*$ of the gradient of the single layer potential.

Deriving this from Theorem~\ref{theorem:elliptic_GSTo} requires a {big pieces $Tb$ theorem} and corresponding estimates for a suitable \textit{suppressed operator}. These results are known for singular integral operators satisfying the weak boundedness property, whose kernels are Hölder continuous—or even uniformly continuous with a modulus of continuity obeying a double Dini condition (see \cite{MV21} for the most general result in this direction). Nevertheless, although our operators satisfy the weak boundedness property, the modulus of continuity of their kernels is merely Dini integrable, rather than double Dini.

We also remark that proving a suitable $Tb$ theorem is essential to our approach. An alternative strategy to deduce Theorem~\ref{teo1} from Theorem~\ref{theorem:elliptic_GSTo}, avoiding reliance on the $Tb$ theorem, could involve a perturbation argument at the level of big pieces of $\mu|_B$, along the lines of \cite{MMPT23}. However, adapting the method of \cite{MMPT23} would require a two-weight analogue of the deep results established in \cite{DT24}, which would constitute a significant achievement in its own right. The suppression method plays a crucial role both in the proof of the big pieces $Tb$-theorem and in its application to Theorem~\ref{teo1}, as it enables the formulation of the relevant estimates over full balls. 	

\vv

For the sake of simplicity, we now state the big pieces $Tb$ theorem specifically for the gradient of single layer potentials. A more general version, applicable to a broader class of singular integral operators--including the gradients of single layer potentials associated with elliptic matrices in $\widetilde{\DMO}_{n-1}$--will be presented in Section~\ref{big piecesTb}.

In particular, we prove a version of the $Tb$ theorem for a broad class of operators $\widetilde{T}$, which we refer to as \emph{good singular integral operators} (\emph{good} SIOs). These operators satisfy the kernel bounds from Lemma~\ref{lem:estim_fund_sol}, and are \emph{almost antisymmetric} in the following sense: for any choice of a \emph{suppression function} $\Xi$ and any compactly supported measure $\mu \in M^n_+(\mathbb{R}^n)$, if we denote by $\widetilde{T}_{\Xi,\mu}$ the associated suppressed operator (see Section~\ref{sec:aux_operator}), then its adjoint can be expressed as $-\widetilde{T}_{\Xi,\mu}$ plus a remainder term $K$. This term $K$ is bounded on $L^2$, uniformly in $\Xi$, and corresponds to a suppressed kernel whose deviation from that of $\widetilde{T}_{\Xi,\mu}$ is controlled by a Dini-integrable modulus of continuity. For the precise definition of \emph{good} SIOs, we refer the reader to Definition~\ref{eq:good_SIO_definition}.

\vv

Let $\mathcal D_0$ be the lattice of (standard) dyadic cubes in $\R^d$ and for every \( w \in \mathbb{R}^{d} \), we consider the translated dyadic lattice  
\[
\mathcal D(w) = w + \mathcal  D_0.
\]
\begin{theorem}\label{thm:big piecesTb-aux_intro}
	Let $n\geq 2$, and let $A$ be a $(n+1)\times(n+1)$ uniformly elliptic matrix satisfying $A\in\widetilde \DMO_{n-1}$, and let $T$ denote its associated gradient of the single layer potential. Assume that $\mu$ is a finite Borel measure on $\R^d$, $d\geq 2$,  supported on a compact set $F \subset \R^d$ so that $\diam(\supp(\mu))\leq R$.  Suppose also that $\nu=b\, \mu$, for a function $b$ such that $\|b\|_{L^\infty(\mu)}\leq c_b$.  For every $w\in \R^d$,  let $\mathfrak T_{\caD(w)}\subset \R^d$ be the union of the maximal dyadic cubes in $\caD(w)$  for which 
	\[
	|\nu(Q)| \leq c_{acc}\,  \mu(Q). 
	\]
	We are also given a measurable set $H  \subset \R^d$ satisfying the following properties:
	\begin{itemize}
		\item 	If a ball $B(y,r)$ is such that $\mu(B(y,r))>c_0\, r^n$,  then $B(y,r)\subset H$.
		\item 	There exists $\delta_0>0$ such that $\mu(H \cup \mathfrak T_{\caD(w)})\leq \delta_0\, \mu(F)$, for all $w \in \R^d$.
		\item	We have that 
		\[
		\int_{\R^d \setminus H}  T_{ *}\nu\, d\mu\leq c_* \mu(F), \qquad \text{ for all }x\in \R^d,
		\]
		where
		\begin{equation}\label{eq:max_aux_truncated}
			T_{ *}\nu(x)=\sup_{\varepsilon>0}\bigl|T_{\nu, \varepsilon}1(x)\bigr|.
		\end{equation}
	\end{itemize}
	Then, there exist 
	$$G \subset F \setminus \bigcap_{w \in \R^d} \Big( H \cup \mathfrak T_{\mathcal D(w)} \Big)$$  and  $\hat  c>0$ depending on $n,$ $d,$ $c_0,$ $c_b,$ $c_*,$ $c_{acc},$ $c_K,$ $c_\Lambda,$ $\delta_0$ so that the following hold:
	\begin{enumerate}
		\item $\mu(G)\geq \hat c\, \mu(F)$.
		\item The measure $\mu|_G$ has $n$-growth with constant $c_0$, namely
		\begin{equation}\label{eq:growth_G_B-1}
			\mu(G \cap B(x,r))\leq c_0 \, r^n,\qquad \qquad \text{ for }x\in \R^d,  r>0.
		\end{equation}
		\item $ T_{\mu}$ is bounded on $L^2(\mu|_G)$ with constant depending only on $n$, $c_0,$ $c_b,$ $c_*,$ $c_{acc},$ $c_K,$ $c_\Lambda,$  $\delta_0$.
	\end{enumerate} 
\end{theorem}

\vvv

For a general overview of $Tb$ theorems, we refer the reader to \cite{Vo03} and \cite{To14}. Nazarov, Treil, and Volberg first established an analogous result for the Cauchy transform in \cite{NTrV99}; their methods extend to higher dimensions and apply to more general Calderón–Zygmund operators that are antisymmetric (see also \cite{Vo03}). As mentioned above, the most general result to date was obtained by Martikainen and Vuorinen in \cite{MV21}, for operators satisfying the weak boundedness property whose kernels are uniformly continuous with a modulus of continuity obeying a double Dini condition. 

Our result is, on the one hand, more general in terms of the assumptions on the modulus of continuity of the kernels--we require only Dini continuity, whereas \cite{MV21} assumes the stronger double Dini condition. On the other hand, it is more restrictive, as the class of operators we consider is a strict subclass of those satisfying the weak boundedness condition and constitutes an $L^2$ perturbation of more ``typical'' singular integral operators, in the sense of standard kernel bounds. However, the results in \cite{MV21} do not apply to the type of problems we aim to address. This is why we had no other choice but to prove Theorem~\ref{thm:big piecesTb-aux_intro}. Our approach combines these techniques with structural properties of \emph{good} SIOs, along with careful estimates of the error terms arising from the lack of perfect antisymmetry and the $\DMO$-type regularity of the kernel of ${T}_{\mu}$.

\vv

\subsection{Main results III: Free boundary problems for elliptic measure}

Our final results, which constitute the main motivation for this paper, are the generalizations of Theorems~\ref{theorem:Laplace-one-phase}, \ref{theorem:Laplace-one-phase_quant}, \ref{theorem:Laplace-two-phase}, and \ref{theorem:quantitative_two_phase_Laplace_indexed} to the $\widetilde{\DMO}$ setting. Before stating these theorems, we recall some definitions.

\vv

Let  $\Omega\subset\Rn1$ be an  open set.  We say that a point $\xi \in \partial \om$ is {\it Wiener regular}, if 
\begin{equation}\label{eq:Wiener}
	\int_0^1 \frac{\Cap\bigl(\overline{B(x,r)}\cap \Omega^c, B(x,2r)\bigr)}{r^{n-1}}\,\frac{dt}{t}=\infty,
\end{equation}
where $\Cap(\cdot, \cdot)$ denotes the variational $2$-capacity of the condenser $(\cdot,\cdot)$, see \cite[p.~27]{HKM93}. We say that $\om$ is {\it Wiener regular}, if every $\xi \in \partial \om$ is {\it Wiener regular}.
\vv

We say that an open set $\Omega\subset\Rn1$ satisfies the \textit{capacity density condition}, abbreviated as \textit{CDC}, if, for any $x\in \partial \Omega$ and $0<r<\diam (\Omega)$, 
\begin{equation}\label{eq:CDC_definition}
	\Cap\bigl(\overline{B(x,r)}\cap \Omega^c, B(x,2r)\bigr)\gtrsim r^{n-1}.
\end{equation}
We remark that, by Wiener's criterion, if $\Omega\subset \Rn1$ satisfies the CDC, then it is Wiener regular.
\vv

We say that $\Omega\subset \Rn1$ satisfies the \textit{corkscrew condition} if there exists a uniform constant $c\in(0,1/2)$ such that, for every ball $B$ centered at $\partial \Omega$ with $0<r_B\leq \diam(\Omega)$, there exists a ball $B(x_B,c\, r_B)$ such that $B(x_B,c\, r_B)\subset \Omega\cap B$. The point $x_B$ is called a \textit{corkscrew point} relative to $B$.

\vv

The elliptic measure associated with $L_A$ in $\Omega$ with pole $p \in \Omega$ is denoted by $\omega^{L_A,p}_\Omega$, and its definition is given in subsection~\ref{sec:elliptic}. Let us state the generalization of Theorem \ref{theorem:Laplace-one-phase}.

\begin{theorem}\label{theorem:main_one_phase}
	Let $\Omega\subset\Rn1$, $n \geq 2$, be an open, connected Wiener regular set.  Let $A$ be a uniformly elliptic matrix satisfying $A\in\widetilde \DMO_{n-1}$. If $\omega^{L_A,p}_{\Omega}$ is the elliptic measure in $\Omega$ with pole at  $p\in\Omega$ and there exists a compact set $E\subset \partial \Omega$ such that $0<\mathcal H^n(E)<\infty$ and $\omega^{p}_{\Omega}|_E$ is absolutely continuous with respect to $\mathcal H^n|_E$, then $\omega^{p}_{\Omega}|_E$ is $n$-rectifiable.
\end{theorem}

The proof of Theorem \ref{theorem:main_one_phase} follows the general scheme developed in \cite{AHMMMTV16}, particularly the use of singular integral techniques and the elliptic version of the David–Semmes problem established in \cite{MMPT23}.  There are three crucial components to generalize this theorem from the Laplace to more general elliptic operators: i) The suppressed $Tb$ theorem for the gradient of the single layer potential,  which we prove in the present paper (see Theorem \ref{thm:big piecesTb-aux_intro}); ii) The fact that any $L_A$-solution is $C^1$ and satisfies  Cauchy-type estimates, which  is the reason why we choose $\DMO$-coefficients;  iii) Calder\'on-Zygmund type estimates for the gradient of the single layer potentials, which is why we choose $\widetilde \DMO_{n-1}$-coefficients.

\vv

A quantitative version of Theorem \ref{theorem:main_one_phase}, which applies to non-doubling measures supported on a compact subset of $\partial \Omega$, satisfying an upper growth condition, and extends Theorem \ref{theorem:Laplace-one-phase_quant}, is given below.

\begin{theorem}\label{theorem:elliptic-one-phase_quant}
	Let \( n \geq 2 \), and let \( 0 < \kappa < 1 \) and \( c_{db} > 1 \) be constants depending only on \( n \), with \( \kappa \) sufficiently small and \( c_{db} \) sufficiently large.   Let \( \Omega \subset \mathbb{R}^{n+1} \) be an open set, and let \( \mu \) be a Radon measure whose support is contained  in a compact subset of \( \partial \Omega \), satisfying the growth condition
	\[
	\mu(B(x,r)) \leq c_0 \, r^n \quad \text{for all } x \in \mathbb{R}^{n+1}, \; r > 0.
	\]
	Suppose there exist \( \varepsilon, \varepsilon' \in (0,1) \) such that for every \( \mu \)-(2, \( c_{db} \))-doubling ball \( B \) centered on \( \operatorname{supp} \mu \), with \( \operatorname{diam}(B) \leq \operatorname{diam}(\operatorname{supp} (\mu)) \), there exists a point \( x_B \in \kappa B \cap \Omega \) such that for every Borel set \( E \subset B \),  it holds
	\begin{equation}\label{eq:quant_mutual_ac}
		\mu(E) \leq \varepsilon \, \mu(B) \quad \Rightarrow \quad \omega^{x_B}(E) \leq \varepsilon' \, \omega^{x_B}(B).
	\end{equation}
	Then the operator \( T_\mu \), defined as the gradient of the single layer potential associated with \( \mu \), is bounded on \( L^2(\mu) \). In particular, \( \mu \) is \( n \)-rectifiable. 
	
	Furthermore, if \( \mu(B(x,r)) \geq c_0^{-1} r^n \) for all \( x \in \mathbb{R}^{n+1} \) and \( r > 0 \), then \( \mu \) is uniformly \( n \)-rectifiable.
\end{theorem}

The strategy for proving Theorem \ref{theorem:elliptic-one-phase_quant} is based on the quantitative framework developed in \cite{MT20}, which refines the method introduced in \cite{AHMMMTV16}.  Although \cite{MT20} is substantially more technical than \cite{AHMMMTV16}, the additional tools required to extend its results to the class of elliptic operators considered in this paper are precisely those needed to generalize the approach of \cite{AHMMMTV16}.

\vvv

We now state a qualitative two-phase problem  under the additional condition \eqref{eq:extra_cond_A_DMO}, whose role is discussed below.

\begin{theorem}\label{theorem:main_two_phase_qualitative}
	Let $\Omega_1, \Omega_2 \subset \Rn1$, $n \geq 2$,  be two disjoint Wiener-regular domains such that $\partial \Omega_1 \cap \partial \Omega_2\neq \varnothing$ and let $E\subset \partial \Omega_1 \cap \partial \Omega_2$ be a compact set.  Let $A$ be a uniformly elliptic matrix in $\widetilde \DMO_{1-\alpha}$, for some $\alpha\in(0,1)$, and such that there exists $\gamma \in (0,1]$ for which
	\begin{equation}\label{eq:extra_cond_A_DMO-gen}
	m_{\oomega}\coloneqq	\limsup_{t\to 0^+}\frac{ \oomega_A(t)}{t^\gamma}<+\infty.
	\end{equation}
	We denote  by $\omega_i=\omega^{L_A, x_i}_{\Omega_i}$ the elliptic measure in $\Omega_i$ with pole at  $x_i \in \Omega_i$, for $i=1,2$.  If $\omega_1|_E \ll \omega_2|_E\ll \omega_1|_E $,  then there exists a $n$-rectifiable set $F\subset E$ with  $\omega_1(E\setminus F)=0$ such that  $\omega_1|_F$ and $\omega_2|_F$ are mutually absolutely continuous with respect to $\mathcal H^n|_F$.
	
	The theorem holds, in particular, when
		\begin{equation}\label{eq:extra_cond_A_DMO}
	M_{\oomega}\coloneqq	\limsup_{t\to 0^+}\frac{\mathfrak I_{\oomega_A}(t)}{\oomega_A(t)}<+\infty.
	\end{equation}
\end{theorem}

The study of two-phase problems for elliptic measure is more intricate due to the need to analyze two distinct sets: the set of points with positive density, denoted by \( G_{pd} \), and the set of points with zero density, denoted by \( G_{zd} \). For \( G_{pd} \), we primarily employ methods from the one-phase problem. By utilizing the \( \widetilde{\DMO} \) version of the 1-co-dimensional David–Semmes problem \cite{NToV14}, as developed in \cite{MMPT23}, we establish the rectifiability of the elliptic measure on this set.

The more challenging aspect is demonstrating that the elliptic measure on \( G_{zd} \) is \( n \)-rectifiable, specifically showing that \( \omega_1(G_{zd}) = 0 \). The key advancement enabling us to address this is Theorem~\ref{teo1}, which provides the essential tool to formulate an elliptic analogue of Theorem~\ref{theorem:Laplace-two-phase} within the \( \widetilde{\DMO} \) framework. Our approach adapts strategies initially developed for harmonic measure in \cite{AMT17a, AMTV19}.

A critical component in applying Theorem~\ref{teo1} is the detailed blow-up analysis introduced by Azzam and the second-named author in \cite{AM19} for operators with \( \VMO \) coefficients\footnote{This analysis was initially established by Kenig, Preiss, and Toro \cite{KPT09} in NTA domains for the Laplacian and later extended to CDC and Wiener regular domains in \cite{AMT17a, AMTV19}.}. This method is vital for proving that all tangent measures of the elliptic measure are flat and that the measure on the set of mutual absolute continuity has co-dimension one. The former ensures that for \( \omega \)-almost every point \( \xi \), there exists a scale \( r_0 \) such that any ball \( B(\xi, r) \) with \( r < r_0 \) has small \( \beta_{\omega,1} \) coefficients. The latter allows for identifying a sequence of \( P \)-doubling balls with arbitrarily small radii that are also flat—precisely the geometric conditions required to apply Theorem~\ref{teo1} (see items (3) and (7)).

A particularly delicate step in the proof is establishing the $L^2$-mean oscillation estimate \eqref{eq:oscillation_zero_density}, which is the only part of the argument that requires the additional assumption \eqref{eq:extra_cond_A_DMO}. Since we work with sequences of balls whose radii shrink to zero while their centers remain far from the pole of the elliptic measure, it becomes essential to control the density at \emph{all} scales. This is precisely what the $P$-doubling condition ensures.

After several intermediate estimates, we arrive at the following key point: if $\xi \in G_{zd}$ and $r_0$ is the corresponding scale, then for any $r < r_0$, we would have to prove that
\[
	\mathfrak{I}_{\oomega_A}(r/r_0)\, \Theta^n_{\omega^p}(B(\xi, r_0)) \leq \tau\, \Theta^n_{\omega^p}(B(\xi, r)),
\]
for $\tau>0$ small enough.
When the coefficients are $\alpha$-Hölder continuous, this expression is controlled via a $P_{\alpha/2}$-doubling condition, yielding
\[
\mathfrak{I}_{\oomega_A}(r/r_0)\, \Theta^n_{\omega^p}(B(\xi, r_0)) \lesssim \left( {r}/{r_0} \right)^{\alpha/2} \Theta^n_{\omega^p}(B(\xi, r)),
\]
which, for sufficiently small $r$, gives the desired smallness for estimate (8). In the general case, we would require a bound of the form
\[
\mathfrak{I}_{\oomega_A}(r/r_0)\, \Theta^n_{\omega^p}(B(\xi, r_0)) \lesssim \theta(r/r_0) \, \Theta^n_{\omega^p}(B(\xi, r)),
\]
for some continuous and increasing function $\theta\colon (0,1) \to (0,\infty)$. However, obtaining such a bound when the modulus of continuity is not Hölder appears to be quite challenging,
 and it is the reason why we ask the additional condition \eqref{eq:extra_cond_A_DMO-gen}.
Interestingly, we observe that \eqref{eq:extra_cond_A_DMO-gen} is connected to the more restrictive condition \eqref{eq:extra_cond_A_DMO} on the Dini integral of the modulus. Indeed, in Lemma~\ref{lem:holder_bound_I_omega} we observe that \eqref{eq:extra_cond_A_DMO} allows us to show that for sufficiently small $t > 0$, we have that
\(
	\mathfrak{I}_{\oomega_A}(t) \lesssim t^\gamma,
\)
\textit{for any $\gamma \in \bigl(0, \tfrac{1}{2 M_\oomega} \bigr)$.} Therefore,  we may proceed as in the case of H\"older continuous coefficients and conclude.   Whether Theorem~\ref{theorem:main_two_phase_qualitative} holds without this assumption remains an open question. 

Conditions of the type $\mathfrak{I}_{\oomega_A}(t) \leq c\, \oomega_A(t)$ are common in the PDE literature on elliptic operators with coefficients satisfying conditions involving moduli of continuity or for generalized H\"older and Campanato spaces; see, for instance, \cite[Theorem 1.4]{MMM20}.  To our knowledge, Lemma~\ref{lem:holder_bound_I_omega} is new: it asserts that even when, a priori, the modulus decays more slowly than Hölder, a scale-invariant inequality of this form implies sub-Hölder growth of the Dini integral for some small exponent depending on $c$.  Moreover, by the same argument of Lemma~\ref{lem:holder_bound_I_omega}, we observe that the condition $\mathfrak I_{\oomega_A}(t)\lesssim t^\gamma$ is equivalent to $\oomega_A\lesssim t^\gamma$, so it corresponds to a sub-H\"older growth condition for $\oomega_A$.

We also note that proving upper $n$-growth of elliptic measure in the two-phase setting is accomplished using the elliptic Alt–Caffarelli–Friedman-type monotonicity formula from \cite{AGMT17}, see Theorem \ref{theorem:ACF_DMO}. To apply this result, we replace $A$ with its uniformly continuous representative $\mathcal{A}$, which induces the same elliptic measure and Green's function (see Lemma~\ref{lem:reduction_to_UC}) and has the Dini-integrable modulus of continuity $\mathfrak{I}_{\oomega_A}$.
We also remark that in Theorem \ref{theorem:ACF_DMO} we further assume that $\mathcal{A}$ coincides with the identity matrix at the center of the ball. This assumption on $\mathcal{A}$ is not an obstacle in our application: a standard change of variables argument analogous to \cite{AGMT17} and \cite[Section 12]{Pu19} allows us to reduce to this case.

\vvv

A quantitative version  of Theorem  \ref{theorem:quantitative_two_phase_Laplace_indexed} is stated below.

\begin{theorem}\label{theorem:quantitative_two_phase}
	Let \( n \geq 2 \), and let $A$ be a uniformly elliptic matrix satisfying $A\in\widetilde \DMO_{1-\alpha}$ for some $\alpha\in (0,1)$ satisfying also 
	\begin{equation}\label{eq:small_BMO_condition}
		\sup_{r \in (0,1)}	\sup_{x\in \Rn1}  \avint_{B(x,r)} \Big| A(y) - \avint_{B(x,r)} A \Big| \, dy  \leq \eta, 
	\end{equation}
	for $\eta$ sufficiently small.  Let \( \Omega_1, \Omega_2 \subset \mathbb{R}^{n+1} \) be two bounded, open, connected, disjoint sets satisfying the CDC, and set \( F \coloneqq \partial \Omega_1 \cap \partial \Omega_2 \neq \varnothing \).  Assume there exist \( \varepsilon, \varepsilon' \in (0,1) \) such that, for any ball \( B \) centered at \( F \) with \( \operatorname{diam}(B) \leq \min(\operatorname{diam}(\Omega_1), \operatorname{diam}(\Omega_2)) \), there exist \( x_i \in \tfrac{1}{4}B \cap \Omega_i \) for \( i = 1, 2 \), satisfying the following: for every measurable set \( E \subset B \),
	\begin{equation}\label{eq:hypothesis_A_infinity_type}
		\text{if } \quad \omega_1(E) \leq \varepsilon \, \omega_1(B), \quad \text{then} \quad \omega_2(E) \leq \varepsilon' \, \omega_2(2B),
	\end{equation}
	where \( \omega_i= \omega^{L_A,x_i}_{\Omega_i} \) denotes elliptic measure with pole at \( x_i \) in \( \Omega_i \), for \( i = 1, 2 \).
	
	If \( \varepsilon' \) is sufficiently small, depending only on \( n \) and the CDC constant, then there exist \( \theta_i \in (0,1) \) and a uniformly \( n \)-rectifiable set \( \Sigma_B \subset \mathbb{R}^{n+1} \) such that
	\begin{equation}\label{eq:thm13_a}
		\omega_i(\Sigma_B \cap F \cap B) \geq \theta_i, \quad \text{for } i = 1, 2.
	\end{equation}
	Moreover, if \( x_1 \) is a corkscrew point for \( \tfrac{1}{4}B \), then there exists \( c > 0 \) such that
	\[
	\mathcal{H}^n(\Sigma_B \cap F \cap B) \geq c \, r(B)^n.
	\]
\end{theorem}

\vvv

The proof of Theorem~\ref{theorem:quantitative_two_phase} builds upon the strategy developed in \cite{AMT20}, though, apart from Theorems \ref{teo1} and \ref{thm:big piecesTb-aux_intro},  it requires a highly non-trivial extension of results from potential theory for elliptic PDEs with merely bounded coefficients.  The method is a quantification of the one in Theorem  \ref{theorem:main_two_phase_qualitative} and a key aspect of the argument is a refined analysis of the mutual boundary $\partial \Omega_1 \cap \partial \Omega_2$ at the level of a ball $B$, which is necessary in order to apply Theorem~\ref{theorem:quantitative_two_phase}.

One of the core components of the proof is a \textit{compactness argument} for  elliptic measure in domains satisfying the CDC.  Although inspired by earlier work,  it is based on the following  proposition, whose proof is  significantly more delicate in the elliptic setting and is of independent interest.

\begin{proposition}\label{limlem}
	Let \( \{ \Omega_j \}_{j \geq 1} \subset \mathbb{R}^{n+1} \) be a sequence of CDC domains with the same CDC constants such that \( 0 \in \partial \Omega_j \) and \( \inf_j \operatorname{diam}(\partial \Omega_j) > 0 \). Suppose there exists a fixed ball \( B_0 = B(x_0, r_0) \subset \Omega_j \) so that $\delta_{\om_j}(x_0) \approx r_0$ for all \( j \geq 1 \). 
	
	Assume further that \( \{A_j\}_{j \in \mathbb{N}} \) is a sequence of Borel measurable, uniformly elliptic coefficient matrices, all with the same ellipticity constants,  for which  there exists a  uniformly elliptic matrix $A$ (with not necessarily constant coefficients) such that
	$$
	A_j \to A  \quad \text{ in } L^1_{\loc}(\Rn1) \text{ as } j \to \infty.
	$$
	Then there exists a connected open set \( \Omega_{\infty}^{x_0} \subset \mathbb{R}^{n+1} \) containing \( B_0 \),  such that, after passing to a subsequence, the following properties hold:
	
	\begin{enumerate}
		\item $\Omega_{\infty}^{x_{0}}$ is also a CDC domain with the same constants. 
		\item $G^{A^T_j}_{\Omega_{j}} (\cdot,x_0) \to G^{A^T}_{\Omega_{\infty}^{x_{0}}}(\cdot,x_0)$ locally uniformly in $\R^{n+1}\setminus \{x_{0}\}$ as $j \to \infty$.
		\item   $\displaystyle \omega_{\Omega_{j}}^{{A_j},x_{0}}\warrow \omega^{A,x_{0}}_{\Omega_{\infty}^{x_{0}}}$ as $j \to \infty$,
	\end{enumerate}
\end{proposition}

\vv
Our strategy proceeds as follows. We begin by fixing a ball \( B \) centered at the boundary and apply a compactness argument to obtain a flat ball \( B' \) of comparable radius, located away from the poles. In the context of the two-phase problem, this step serves as the quantitative analogue of the blow-up analysis presented in \cite{AM19}. Importantly, we emphasize that the compactness argument does \textit{not} rely on the \(\mathrm{DMO}\)-type condition on \( A \); it requires only the small \(\mathrm{BMO}\) norm assumption, which is the quantitative analogue of $\VMO$. The $\widetilde{\DMO}$ hypothesis is required to carry out the remaining parts of the proof.

Next, fixing this ball \( B' \), we partition the set \(\partial \Omega_1 \cap \partial \Omega_2 \cap B'\) into points exhibiting \emph{small density} and points where the density of the elliptic measure \(\omega_1\) is \emph{bounded from below}. The latter set is handled similarly to the one-phase case, by employing the elliptic analogue of the David--Semmes problem in codimension~1, \cite[Corollary 1.3]{MMPT23}. For the former set, we apply Theorem~\ref{theorem:elliptic_GSTo}, once a suitable \( L^2 \)-mean oscillation estimate for the gradient of the single layer potential \( T\omega_1 \) is established.

We remark that a limiting argument of the type employed here was previously used in \cite{HMMTZ21} in quantitative one-phase free boundary problems for elliptic measure in the setting of uniform domains with Ahlfors regular boundaries and elliptic operators with Dahlberg--Kenig--Pipher coefficients with small Carleson constant.

\vv

\subsection{Outline of the paper}

In \textit{Section \ref{section:preliminaries_and_notation}}, we collect the general notation and recall basic properties of Dini-type functions and matrices. We also review the definitions and main properties of solutions to the equation $L_A u = 0$, of elliptic measure, of Green's functions, and we gather the main pointwise estimates for the gradient of fundamental solutions if $A$ has Dini mean oscillation.

In \textit{Section \ref{sec:aux_operator}}, we define the suppression of Calder\'on–Zygmund type kernels and study the main properties of the associated operator $T_{\Xi,\nu}$. If the suppression is adapted to the growth of the measure and to the truncations of the SIO, we establish pointwise estimates for the maximal truncations $T_{\Xi, *}\nu$.

Then, we begin \textit{Section \ref{big piecesTb}} by discussing the class of operators for which we prove the big pieces $Tb$ theorem, and we later present its proof. The choice of the kernels is motivated by our applications: in \textit{Section \ref{sec:gradient_SLP_good}} we first define the main kernels involved in the perturbation argument and their associated $L^2$-boundedness properties. Hence, we show that the gradient of the single layer potential satisfies the required assumptions if the operator $L_A$ is associated with a uniformly elliptic matrix in $\widetilde \DMO_{n-1}$.

In \textit{Section \ref{sec:mean_oscillation_perturbation}} we prove the estimates for the operators involved in the perturbation argument described earlier in the introduction. In particular, in Lemma \ref{lem:mean oscillation perturb}, we bound the $L^2$-mean oscillation of the perturbations at the level of a ball $B$ in terms of the $L^2$-norm of $\mathcal R_\mu$. We can readily apply this result to prove Theorem \ref{theorem:elliptic_GSTo}. Finally in  \textit{Section \ref{section:limits}} we demonstrate Proposition \ref{limlem} along with some potential theoretic results, and in \textit{Section \ref{section:main_lemma}}, we apply all our results to free boundary problems for elliptic measures.

\vv

\section{Preliminaries and notation}\label{section:preliminaries_and_notation}

\subsection{General notation}

\begin{itemize}
	\item For $\lambda>0$ and an open ball $B=B(x,r)$, we define its dilation $\lambda B\coloneqq B(x,\lambda r)$.
	\item Given $A\subset \R^d$, we denote by $\chi_A$ its characteristic function. 
	\item We say that $\Omega\subset\Rn1$ is a domain if it is an open and connected subset of $\Rn1$.
	\item We endow  the space of matrices $\mathbb R^{n_1\times n_2}$ with the norm $|A|\coloneqq \max_{i,j}|a_{ij}|,$ for $A=(a_{ij})_{i,j}\in \mathbb R^{n_1\times n_2}$.
	\item We write $a\lesssim b$ if there exists $C>0$ such that $a\leq Cb$, and $a\lesssim_t b$ to specify that the constant $C$ depends on
	the parameter $t$. We write $a\approx b$ to mean $a\lesssim b\lesssim a$, and define $a\approx_t b$ similarly. 
\end{itemize}

\vv

\subsection{Dini functions, integral operators, and Dini-type matrices}
Let $\theta$ be a $\kappa$-doubling function in the sense of \eqref{eq:dini1_new} for $\kappa>0$.
{For $\eta \in \left(0, \frac{1}{2}\right)$ denote by $N_\eta$ the positive integer such that $2^{-N_\eta-1} \leq  \eta < 2^{-N_\eta}$. 
	We now list some properties of Dini integrals; for the detailed proofs we refer to \cite[Section 2]{MMPT23}.	
	First, if $r>0$ and $\eta r\leq t\leq r$, we have that
	\begin{align*}
		\int_{\eta r}^r \theta(t) \,\frac{dt}{t}  \leq \kappa  \, \frac{\kappa^{N_\eta+1} -1}{\kappa-1}\, \theta(\eta r)\eqqcolon C(\kappa, \eta)\, \theta(\eta r),
	\end{align*}
	which yields
	\begin{equation*}
		\int_0^R \theta(t)\, \frac{dt}{t}=\sum_{j=0}^\infty \int_{\eta^{j+1} R}^{\eta^{j}R} \theta(t)\, \frac{dt}{t}\leq C(\kappa, \eta)\, \sum_{j=0}^\infty \theta (\eta^{j+1}R).
	\end{equation*}
	Conversely, for $R>0$ it holds that
	\begin{equation}\label{eq:mod_cont_sum}
		\sum_{j=0}^\infty \theta\bigl(\eta^{j}R\bigr) \lesssim  \max(1,(\kappa/2)^{N_\eta}) \int_0^R \theta(t)\, \frac{dt}{t}.
	\end{equation}
}

In particular, $\theta$ belongs to $\DS(\kappa)$ if and only if the doubling property \eqref{eq:dini1_new} holds and $\sum_{j=0}^\infty \theta(2^{-j})<+\infty$. 
One can analogously show that, if $\theta$ verifies  \eqref{eq:dini1_new}, and $0<d\leq n$, we have
\begin{equation}\label{eq:mod_cont_sum_2}
	\sum_{k=1}^\infty \frac{\theta(2^kR)}{(2^kR)^{d}}\lesssim \int_R^\infty \theta(t)\, \frac{dt}{t^{d+1}}, \qquad \qquad R>0.
\end{equation}
{Moreover, by the doubling property of $\theta$,
	\begin{equation}\label{eq:omega<dini}
		\theta(r) \leq \kappa \int_{r/2}^r \theta(t) \, \frac{dt}{t} \leq \kappa\, \mathfrak I_{\theta}(r),\qquad r>0.
\end{equation}}
\vvv

{
	\begin{lemma}\label{lem:mod_cont_large_DS1}
		Assuming that for  fixed $d>0$
		\[
		\mathfrak L^d_{\theta}(t)= t^{d}\int_t^\infty \theta(s)\, \frac{ds}{s^{d+1}} < +\infty \qquad \text{ for }\,\,t>0,
		\]
		then $\mathfrak L^d_\theta$ is  a $2^d$-doubling  function.
		
		Moreover, if $\mathfrak I_\theta(1)<\infty$ and $\mathfrak L^d_{\theta}(1)<\infty$, then $\mathfrak L^d_\theta\in\DS(2^d)$ and 
		\begin{equation}\label{eq:split_I_L_theta_merged}
			\mathfrak I_{\mathfrak L^d_\theta}(r)=\frac{1}{d} \bigl(\mathfrak I_\theta(r)+\mathfrak L^d_\theta(r)\bigr)=\mathfrak L^d_{\mathfrak I_\theta}(r), \qquad\text{ for }r>0.
		\end{equation}
	\end{lemma}
	\begin{proof}
		If $t/2 \leq s \leq t$, then 
		\begin{align*}
			\mathfrak L^d_{\theta}(t)& = t^d \int_t^\infty\theta(r)\, \frac{dr}{r^{d+1}} \leq t^d \int_s^\infty\theta(r)\, \frac{dr}{r^{d+1}}\leq  2^d s^d \int_s^\infty\theta(r)\, \frac{dr}{r^{d+1}}=2^d \mathfrak L^d_\theta(s).
		\end{align*}
		which proves that $\mathfrak L^d_\theta$ is $2^d$-doubling.
		To show \eqref{eq:split_I_L_theta_merged}, we use Fubini's theorem and, if  $\mathfrak I_\theta(1)<\infty$ and $\mathfrak L^d_{\theta}(1)<\infty$, we have
		\begin{equation*}
			\begin{split}
				\mathfrak I_{\mathfrak L^d_{\theta}}(r)&=\int_0^r \mathfrak L^d_\theta(t)\, \frac{dt}{t}=\int_0^r t^{d}\int_t^1\theta(s)\, \frac{ds}{s^{d+1}}\, \frac{dt}{t} + \int_0^1 t^{d}\int_1^\infty\theta(s)\, \frac{ds}{s^{d+1}}\, \frac{dt}{t}\\
				&=\frac{1}{d}\int_0^r \theta(s)\, \frac{ds}{s} + \frac{1}{d}\int_r^\infty \theta(s)\, \frac{ds}{s^{d+1}}=\frac{1}{d} \bigl(\mathfrak I_\theta(r)+\mathfrak L^d_\theta(r)\bigr),
			\end{split}
		\end{equation*}
		which implies that $\mathfrak L^d_\theta\in\DS(2^d)$.
		To prove the second equality in \eqref{eq:split_I_L_theta_merged}, we apply Fubini's theorem again, and we obtain
		\begin{equation}\label{eq:L_I_theta}
			\begin{split}
				\mathfrak L^d_{\mathfrak I_{\theta}}(r)=	r^d\int_r^\infty \mathfrak I_{\theta}(t)\, \frac{dt}{t^{d+1}}&=\frac{1}{d}\int_0^r \theta(t)\, \frac{dt}{t} + \frac{r^d}{d}\int_r^\infty \theta(t)\, \frac{dt}{t^{d+1}}=\frac{1}{d}\bigl(\mathfrak I_\theta(r) + \mathfrak L^d_\theta(r)\bigr).\qedhere
			\end{split}
		\end{equation}
	\end{proof}
	
}

\vv

\begin{remark}\label{rem:rem_right_cont}
	By the previous lemma,   if $\theta \in \DS(\kappa)\cap \DL_d(\kappa)$ and $\mathfrak I_{\theta} \in \DS(\kappa)$,  we have that
	$$
	\mathfrak F_\theta(R)  +	\mathfrak L^d_\theta(R)    \to    0, \quad \text{ as } R\to 0.
	$$
\end{remark}

\vv

\begin{definition}\label{def:theta_d_kernel}
	Let $\theta$ be a $\kappa$-doubling function and $0<d\leq n+1$.
	We say that a function $K\colon \Rn1\times \Rn1\setminus\{(0,0)\}\to \R$ is a \textit{$(\theta,d)$-kernel} if it is continuous  and there exists $C>0$ such that
	\[
	|K(x,y)|\leq C \frac{\theta(|x-y|)}{|x-y|^{d}} \quad \text{ for }x\neq y.
	\]
\end{definition}

The latter estimate for {$(\theta,d)$-kernels} is directly connected with the Dini integral.
\begin{lemma}\label{lem:lem_estim_dini_integr_growth}
	Let $\theta\in \DS(\kappa)$.
	Let $0<d\leq n+1$, and assume that $\mu\in M^d_+(\Rn1)$ with $d$-growth constant $c_0>0$.
	\begin{enumerate}
		\item [(a)]For $\rho>0$, we have
		\begin{equation}\label{eq:lem_estim_dini_integr}
			\int_{B(x,\rho)}\frac{\theta(|x-z|)}{|x-z|^{d}}\, d\mu(z) \lesssim_{c_0,\kappa} \int^\rho_0 \theta(t)\, \frac{dt}{t},
		\end{equation} 
		and the right hand side of \eqref{eq:lem_estim_dini_integr} tends to $0$ as $\rho\to 0$.
		\item [(b)] If $K$ is a $(\theta,d)$-kernel for {$\theta\in \DS(\kappa)$}, then its associated integral operator
		\[
		Tf(x)\coloneqq \int K(x,y)f(y)\, d\mu(y), \qquad x\in \Rn1
		\]
		is bounded on $L^2(\mu)$. More specifically, if $C>0$ is as in Definition \ref{def:theta_d_kernel} and $R\coloneqq \diam(\supp (\mu))$, we have
		\begin{equation}\label{eq:bound_Schurs_Test}
			\|T\|_{L^2(\mu)\to L^2(\mu)}\lesssim_{\kappa, C}c_0 \int_0^R \theta(t)\, \frac{dt}{t}.
		\end{equation}
	\end{enumerate}
\end{lemma}
\begin{proof}
	For a proof, we refer to \cite[Lemmas 2.4 and 2.5]{MMPT23}.
\end{proof}

\vv

We now make two remarks concerning matrices in $\DS(\kappa)$ and $\widetilde{\mathrm{DMO}}$.

\begin{remark}\label{rem:continuous representative}
	If $A\in\DS(\kappa)$ then, by \cite[Appendix A]{HwK20}, it has a uniformly continuous representative $\mathcal A$ such that 
	\begin{equation}\label{eq:modulus_continuity_mathcal_A}
		|\mathcal A(x)- \mathcal A(y)|\lesssim_n \mathfrak I_{\oomega_A}(|x-y|), \qquad \text{ for }x,y\in\Rn1.
	\end{equation}
	Thus, if $\Omega\subset \Rn1$ with $|\Omega|>0$  and $x\in \Rn1$, then
	\begin{equation}\label{eq:pw_diff_average_unif_cont}
		\begin{split}
			|\mathcal A(x)-\bar A_\Omega|=|\mathcal A(x)-\bar {\mathcal A}_\Omega| 
			\lesssim \mathfrak I_{\oomega_A}\bigl(\diam (\Omega) + \dist(x, \Omega) \bigr).
		\end{split}
	\end{equation}
\end{remark}
\begin{remark}\label{rem:remarks_on_DMO}
	If a matrix $A$ is $\alpha$-H\"older continuous for some $\alpha \in (0,1)$, then $\oomega_A(t) \lesssim t^\alpha$. Hence, it belongs to $\widetilde{\DMO}$.
	We also observe that, if $A$ is uniformly continuous with a Dini modulus of continuity, then it is of Dini mean oscillation.
	
	Conversely, it can be shown that the class $\widetilde{\DMO}$ is strictly larger than that of Dini continuous continuous matrices, and hence of H\"older continuous matrices. As a variant of the example in \cite[p.~418]{DK17}, if we define $A$ such that $a_{ij}(x) = \delta_{ij}$ for $|x| > 1$, and
	\[
	a_{ij}(x) \coloneqq \delta_{ij} \left(1 + (-\log|x|)^{-\gamma - 1} \right), \qquad \text{for }\quad  0 < |x| \ll 1, \quad 0 < \gamma < \tfrac{1}{2},
	\]
	then it can be shown that $\oomega_A(r) \approx (-\log r)^{-\gamma - 2}$ for $r \ll 1$. {Thus, $A\in\DDMO_s$, although it is not H\"older continuous.}
\end{remark}

\vv

We conclude this subsection by showing that the condition \eqref{eq:extra_cond_A_DMO} translates to H\"older-type bounds for $\mathfrak I_{\oomega_A}$ in sufficiently small scales.

\begin{lemma}\label{lem:holder_bound_I_omega}
	Let $A\in\DS(\kappa)$, and assume that there exists $C>0$ such that
	\begin{equation}\label{eq:holder_bound_integra_oscillation_hp}
	M_\oomega\coloneqq	\limsup_{s\to 0^+}\frac{\mathfrak I_{\oomega_A}(s)}{\oomega_A(s)}<\infty.
	\end{equation}
	Hence,  there exists $s_0$ depending on $M_\oomega$ such that, for every $\gamma \in \bigl(0,\frac{1}{2M_\oomega} \bigr)$, we have that 
	\begin{equation}\label{eq:holder_bound_integra_oscillation}
	\oomega(s)\lesssim 	\mathfrak I_{\oomega_A}(s)\leq\,  (s_0^{-\gamma} 	\mathfrak I_{\oomega_A}(s_0))\, s^\gamma\eqqcolon  C_{s_0}\, s^\gamma,  \qquad \qquad \text{ for all } s\in (0,s_0).
	\end{equation}
\end{lemma}

\begin{proof}
	We define 
	\[
		F_\gamma(s)\coloneqq s^{-\gamma} \, \mathfrak I_{\oomega_A}(s), \qquad\text{ for } s\in (0,1].
	\]
	Note that as $\mathfrak I_{\oomega_A}(1)<\infty$, then $\mathfrak I_{\oomega_A}$ is absolutely continuous in $(0,1]$ and so differentiable $\mathcal L^1$-a.e.   in $s\in (0,1]$.  	If $s\in (0,1)$ is a point of differntiability of $\mathfrak I_{\oomega_A}$, we have that
	\begin{equation}\label{eq:derivative-dini}
		\begin{split}
			F'_\gamma(s)&= s^{-\gamma}\frac{\oomega_A(s)}{s}  -\gamma\, s^{-\gamma-1}\mathfrak I_{\oomega_A}(s) =s^{-\gamma-1}\bigl(\oomega_A(s)-\gamma \, \mathfrak I_{\oomega_A}(s)\bigr)\
		\end{split}
	\end{equation}
	
As $M_\oomega<\infty$,  there is $s_0 \in (0,1)$,  such that $\mathfrak I_{\oomega_A}(s) <2 M_\oomega \oomega(s)$ for every $s<s_0$.  Thus, if $\gamma < \frac{1}{2M_\oomega}$,  in light of \eqref{eq:derivative-dini},  we get  that 
$$
F'_\gamma(s ) = s^{-\gamma-1}  (1- \gamma \, 2M_\oomega   ) \oomega_A(s)  >0, \quad \text{ for every } s \in (0, s_0), 
$$ 
 which, in turn, implies  that $F_\gamma$ is increasing in $(0,s_0)$.  Therefore,  $F(s) < F(s_0)$ for every  $s \in (0,s_0)$, and so,  thanks to \eqref{eq:omega<dini},  we conclude that \eqref{eq:holder_bound_integra_oscillation} holds.
\end{proof}

\vvv

\subsection{Uniformly elliptic operators, elliptic measure, and monotonicity formula}\label{sec:elliptic}
Let  $\Omega\subset \R^{n+1}$ be  an open set. Assume that  $A(\cdot)$ is a $(n+1)\times (n+1)$-uniformly elliptic matrix with $L^\infty(\om)$ real  coefficients and $L_A\coloneqq -\div A\nabla$ is the associated second order differential operator in divergence form. If $f\in L^1_{\loc}(\Omega)$, we say that function $u\in W^{1,2}_{\loc}(\Omega)$ is a \textit{weak solution} to the equation $L_A u=f$ in $\Omega$ if 
\begin{equation}\label{eq:def_snd_order_elliptic_sol}
	\int_\Omega A\nabla u \cdot \nabla \Phi = \int_\Omega f\, \Phi, \qquad \text{ for all }\Phi \in C^\infty_0(\Omega).
\end{equation}
If $u$ is a weak solution to $L_Au=0$ in $\Omega$, we also say that it is \textit{$L_A$-harmonic}.

We say that $u\in W^{1,2}_{\loc}(\Omega)$ is a \textit{subsolution} (resp., \textit{supersolution}) to $L_Au=0$ in $\Omega$ if $\int A\nabla u\nabla \Phi\leq 0$ (resp., $\int A\nabla u\nabla \Phi\leq 0$) for all $\Phi \in C^\infty_0(\Omega)$, $\Phi\geq 0$. In this case, we also say that $u$ is \textit{$L_A$-subharmonic} (resp., \textit{$L_A$-superharmonic}), and if the matrix $A$ is clear from the context we often write $L=L_A$.

\vv
If $\Omega\subset \Rn1$, $n\geq 2$, is an open set and $2^*\coloneqq \frac{2(n+1)}{n-1}$, we define $Y^{1,2}(\Omega)$ as the space of all functions  $u\in L^{2^*}(\Omega)$ that are weakly differentiable and  whose weak derivatives belong to $L^2(\Omega)$. We endow $Y^{1,2}(\Omega)$ with the norm
\[
\|u\|_{Y^{1,2}(\Omega)}\coloneqq \|u\|_{L^{2^*}(\Omega)} + \|\nabla u\|_{L^2(\Omega)},\qquad \text{ for } u\in Y^{1,2}(\Omega).
\]	
We denote by $Y^{1,2}_0(\Omega)$ the closure of $C^\infty_c(\Omega)$ in $Y^{1,2}(\Omega)$, and remark that $Y^{1,2}(\Rn1)=Y^{1,2}_0(\Rn1)$ (see for instance \cite[p. 46]{MZ97}).

\vv

Following  \cite[Chapter 9]{HKM93}, we introduce some important notions and results in potential theory for general elliptic operators in divergence form.

Given a domain $\Omega \subset \R^{n+1}$, we denote by $\partial \Omega$ its boundary, with the convention that it includes the point at infinity if $\Omega$ is unbounded.  

\begin{definition}
	For a (possibly unbounded) domain $\Omega \subset \R^{n+1}$ and a function $f\colon \partial \Omega \to [-\infty, \infty]$, we define the {\it upper class} $\mathcal{U}_f$ of $f$ to be the class of functions that are $L_A$-superharmonic, bounded from below, and lower semicontinuous at all points on $\partial \Omega$. Similarly, we define the {\it lower class} $\mathcal{L}_f$ of $f$ to be the class of functions $u$ that are $L_A$-subharmonic, bounded from above, and upper semicontinuous at all points on $\partial \Omega$. We define 
	$$ \overline{H}_f(x) = \inf\{ u(x) : u \in \mathcal{U}_f \} \quad \textup{and} \quad \underline{H}_f(x) = \sup\{ u(x) : u \in \mathcal{L}_f \} $$
	to be the {\it upper} and {\it lower Perron solutions} of $f$ in $\Omega$. If $\mathcal{U}_f = \varnothing$ (resp.\ $\mathcal{L}_f = \varnothing$), then we set $\overline{H}_f = \infty$ (resp.\ $\underline{H}_f = -\infty$).
	
	A function $f\colon \partial \Omega \to [-\infty, \infty]$ is called {\it resolutive} if the upper and lower Perron solutions of $f$ coincide. In this case, we denote $H_f \coloneqq \overline{H}_f = \underline{H}_f$, and have that it is harmonic in $\Omega$.
\end{definition}

Note that if the complement of $\Omega$ has positive capacity, then every continuous function is resolutive.

Let  $f\in C(\partial \Omega)$.   A point $x_0\in \partial \Omega$ is \textit{Wiener regular} if
\[
	\lim_{x\to x_0}H_f(x)= f(x_0).
\]
We say that a domain $\Omega$ is Wiener regular if all the points of $\partial \Omega\setminus \{\infty\}$ are Wiener regular.

\vv
If $\Omega\subset \R^{n+1}$ is a Wiener regular domain and $f\in C(\partial \Omega)$, for every fixed $x\in \Omega$ the map $f\mapsto {H}_f(x)$ defines a bounded linear functional on $C(\partial \Omega)$. Thus,  Riesz representation theorem and the maximum principle yield the existence of a probability measure $\omega^x$ on $\partial \Omega$ such that 
\[
{H}_f(x)=\int_{\partial \Omega}f\, d\omega^x, \qquad \text{ for all }x\in \Omega.
\]
The measure $\omega^x$ is called the \textit{ellipic measure} associated with the operator $L$ and with pole $x$. We also use the notation $\omega^x=\omega^{L, x}_\Omega$ in case we want to specify the operator and the domain explicitly.

\vv

We say that the {\it variational  Dirichlet problem} with data $\varphi \in C_c(\partial \om)$ is  solvable  if for every $\Phi \in Y^{1,2}(\om)\cap C(\overline \om)$ such that $\Phi|_{\partial \om} = \varphi$,   there exists $u \in Y^{1,2}(\om)$ such that $L_A u=0$ and $u-\Phi \in Y_0^{1,2}(\om)$\footnote{By Lax-Milgram,  this problem is always uniquely solvable.  See e.g.  \cite[Subsection 2.4]{AGMT17}.}.

Let $\Sigma\subset \bar \Omega$, and $u\in Y^{1,2}(\Omega)$. We say that $u$ \textit{vanishes on }$\Sigma$ if $u$ is a limit in $Y^{1,2}(\Omega)$ of a sequence of functions in $C^\infty_c(\bar \Omega\setminus \Sigma)$. In this case, we often write that $u=0$ on $\Sigma$.

A point \( x_0 \in \partial \Omega \setminus \{\infty\} \) is said to be \emph{Sobolev \( L_A \)-regular} if, for every \( \Phi \in Y^{1,2}(\Omega) \cap C(\overline{\Omega}) \),  the solution of the variational Dirichlet problem with data $\Phi$, denoted by $h$,   satisfies
\[
\lim_{x \to x_0} h(x) = \Phi(x_0).
\]
We say that \( \Omega \) is \emph{Sobolev \( L_A \)-regular} if every point in \( \partial \Omega \setminus \{\infty\} \) is Sobolev \( L_A \)-regular.

We remark that a point is $L_A$-Sobolev regular if and only if it is Wiener regular (i.e., it satisfies \eqref{eq:Wiener}). 
According to \cite[Theorem 9.20]{HKM93}, see also the remarks in \cite[Lemma 2.3]{AGMT17} regarding the case if $\Omega$ is unbounded, if \( x_0 \in \partial \Omega \setminus \{\infty\} \) is Sobolev \( L_A \)-regular, then it is also {Wiener regular.}

\begin{theorem}\label{theorem:resolutiveHf}
	If $\Rn1 \setminus \om$ has positive capacity and $f \in Y^{1,2}(\Rn1) \cap C(\overline \om)$,  then $f$ is resolutive and $H_f - f \in Y^{1,2}_0(\om)$. 
\end{theorem}

\begin{proof}
	The proof follows from the proof of \cite[Theorem 9.28]{HKM93} by using Tietze extension theorem.
\end{proof}

For $x\in \Omega$, we write $\delta_\Omega(x)\coloneqq \dist(x,\partial \Omega)$,  understanding that $\delta_\Omega(x)=\infty$ if $\Omega=\Rn1$.

\vv
We recall below a classical interior regularity result for solutions to the equation $L_A u = f -\dv g$, where $f \in L^{q/2}$ and $g \in  L^q$ with $q > n+1$. For a more general version, see for instance \cite[Theorem 8.24]{GT01}.
\begin{theorem}\label{thm:C_alpha_interior-inhom}
	Let $\Omega\subset \Rn1$ be a domain, $n\geq 2$. We assume that $A$ is an $(n+1)\times (n+1)$ uniformly elliptic matrix with coefficients in $L^\infty(\Omega)$, and that $f\in L^{\frac{q}{2}}(\Omega)$ and $g \in  L^q(\om)$ with $q > n+1$.  Hence, if $B\subset \Rn1$ is a ball such that $2B\subset\Omega$ and $u\in W^{1,2}(\Omega)$ is a weak solution to $L_Au=f - \div g$ in $2B$, there exists $\alpha\in (0,1)$ depending on $n$ and $\Lambda$ such that we have
	\begin{equation}\label{eq:C_alpha_interior}
		\|u\|_{C^\alpha(\overline{B})}\lesssim \|u\|_{L^2(2B)}+ \|f\|_{L^{\frac{q}{2}}(2B)} + \|g\|_{L^{q}(2B)},
	\end{equation}
	where the implicit constant depends on $n$, $\Lambda$, $q$, and $r(B)$.
\end{theorem}

\vv

For $u\colon \Rn1\to \mathbb R$, $x\in \Rn1$, and $r>0$, we define the oscillation
\[
\osc_{B(x,r)}u\coloneqq \sup_{B(x,r)}u - \inf_{B(x,r)}u.
\]
The boundary oscillation of solutions to $L_Au=0$  can be controlled as in the next statement.  
\begin{theorem}\label{thm:boundary oscillation}
	Let $n\geq 2$, $\Omega\subset \Rn1$ be a domain, and let $B_{r_0}$ be a ball of radius $r_0 \in (0, \diam(\om))$ such that $B_{r_0}\cap \Omega\neq \varnothing$. Let $r\leq r_0/2$, and $B_r=B(\xi,r)$ be a ball centered at $\xi\in \partial \Omega$. Let $A$ be a $(n+1)\times (n+1)$ uniformly elliptic matrix with coefficients in $L^\infty (\Omega)$, $f\in L^{\frac{q}{2}}(\Omega)$, and $g \in  L^q(\om)$ with $q > n+1$.  If  $u$ is a solution to $L_Au=f -\dv g$ in $\om$ and $\varphi \in Y^{1,2}(\Omega)\cap C(\bar \Omega)$ is such that $u-\varphi$ vanishes on $\partial \Omega\cap B_r$, then there exists $C>0$ depending on $\Lambda$ such that, for any $0\leq \rho\leq r/2$, we have that 
	\begin{align}\label{eq:boundary oscillation}
		\osc_{B(\xi,\rho)\cap \Omega}u\lesssim \osc_{\partial \Omega\cap B(\xi,\rho)}\varphi &+ \exp\Bigl(-\frac{1}{C}\int^r_{2\rho}\frac{\Cap\bigl(\overline{B(\xi,s)}\setminus \Omega\bigr)}{s^{n-1}}\, \frac{ds}{s}\Bigr)\Bigl(\osc_{B(\xi,r)\cap \partial \Omega}u \Bigr)\\
		&+   \left( \frac{\rho}{r} \right)^{2\left(1-\frac{n+1}{q}\right)} \| f\|_{L^{\frac{q}{2}}(\om) } + \left( \frac{\rho}{r} \right)^{1-\frac{n+1}{q}}  \| g\|_{L^{q}(\om)}.  \notag
	\end{align}
	
	If, additionally,  $\Omega$ satisfies the CDC and $f=g=0$,   then there exists $\alpha\in (0,1)$ depending on $n$, and $\Lambda$ such that, for $0<r<r_0/2$ and $x,y\in B(\xi,r/2)$ we have
	\begin{equation}\label{eq:Holder-boundary-cdc}
		|u(x)-u(y)|\lesssim \Bigl(\frac{|x-y|}{r}\Bigr)^\alpha\Bigl(\frac{1}{r^{n+1}}\int_{\Omega\cap B_r}|u|^2\Bigr)^{1/2}+ \osc_{B(\xi,2|x-y|)\cap \partial \Omega}\varphi,
	\end{equation}
	where the implicit constant depends on the CDC constant, $n$, and $\Lambda$.
	
In fact, if  $\Omega$ satisfies the CDC and $u$ is a  solution to $L_Au=0$ in $B_{2r} \cap \Omega$, continuous on $B_{2r} \cap \overline{\Omega}$,
	and such that $u \equiv 0$ on $\partial \Omega \cap B_{2r}$, then,  there exist constants $\alpha > 0$ and $c>0$ such that
	\begin{equation}
		\label{holder}
		|u(y)| \leq c \left( \frac{\delta_\Omega(y)}{r} \right)^\alpha \sup_{B_{2r}} |u|, 		\quad \text{for all } y \in B_{r},
	\end{equation}
	where $C$ and $\alpha$ depend on $n$, $\Lambda$, the CDC constant, and $\delta_\Omega(y) = \dist(y, \Omega^c)$. In particular, $u$ is locally $\alpha$-H\"older continuous on $ B_r \cap \overline \om$.
\end{theorem}
	
	\begin{proof}
For a proof within the framework of more general elliptic operators with lower-order terms in CDC domains, we refer, for instance, to \cite[pp. 207-208]{GT01}; for general domains,  the iteration argument in  \cite[pp. 207-208]{GT01} is not accurately  written.  One can follow the iteration in  \cite[Theorem 4.14]{Mou23} with 
$$
k(r)\coloneqq  r^{2\bigl(1-\frac{n+1}{q}\bigr)} \| f\|_{L^{\frac{q}{2}}(\om) } +r^{1-\frac{n+1}{q}}  \| g\|_{L^{q}(\om)}.
$$

Finally, a proof of \eqref{holder} can be found in greater generality  e.g.  in \cite[Theorem 4.22]{MZ97}
	\end{proof}

\vv

\begin{lemma}
	\label{l:holdermeasure}
	Let $A$ be a $(n+1)\times (n+1)$ uniformly elliptic matrix with coefficients in $L^\infty(\Omega)$. Suppose $\Omega\subsetneq \bR^{d+1}$ is an open set satisfying the CDC condition \eqref{eq:CDC_definition}, and let $x\in \partial \Omega$. Then, there is $\alpha>0$ so that for all $0<r<\diam   \Omega$ we have 
	\begin{equation}\label{e:wholder}
		\omega_{\Omega}^{L_A,y}(B(x,2r)^c)\leq c\, \Bigl(\frac{\lvert x-y\rvert}{r}\Bigr)^\alpha,\qquad\text{for all }y\in \Omega\cap B(x,r), 
	\end{equation}
	where $c$ and $\alpha$ depend on the CDC constant of $\Omega$ and the ellipticity constants of $L_A$. 
\end{lemma}

\begin{proof} Let $f \in C(\partial \om)$ so that $f=1$ on $\partial \om \setminus B(x,2r)$, $f=0$ on $B(x,r) \cap \partial \om$, and $0 \leq f \leq 1$ on $\partial \om$.  Then, if $u_f(x)= \int_{\partial \om} f \,d\omega_{\om}^{L_A,x}$,  by \eqref{holder}, we have that 
	$$
	\omega_{\Omega}^{L_A,y}(B(x,2r)^c)\leq u_f(y)\leq c\, \Bigl(\frac{\lvert x-y\rvert}{r}\Bigr)^\alpha\, \sup_{z\in B(x,2 r) \cap \om} u(z) \leq c\, \Bigl(\frac{\lvert x-y\rvert}{r}\Bigr)^\alpha,
	$$
	for all $y \in B(x,r)$, since, by maximum principle, $u_f \leq 1$ in $\om$. 
\end{proof}

\vv

The lemma below is often called Bourgain's Lemma, as he proved a similar estimate for harmonic measure in \cite{Bo87}. 

\begin{lemma}[{\cite[Lemma 11.21]{HKM93}}]\label{l:bourgain}
	Let $A$ be a $(n+1)\times (n+1)$ uniformly elliptic matrix with coefficients in $L^\infty(\Omega)$.
	Let $\Omega\subset \bR^{n+1}$ be any domain satisfying the CDC condition \eqref{eq:CDC_definition},  $x_{0}\in \d\Omega$, and $r>0$ so that $\Omega\backslash B(x_{0},2r)\neq\varnothing$. Then 
	\begin{equation}\label{e:bourgain}
		\hm_{\Omega}^{L_A, x}(B(x_{0},2r))\geq c >0 \;\; \mbox{ for all }x\in \Omega\cap B(x_{0},r),
	\end{equation}
	where $c$ depends on $d$ and the CDC constant.
\end{lemma}
\vv

\begin{lemma}
	\label{l:nullboundary}
	Let $\omega$ be the harmonic measure for some domain $\om$ in $\R^{n+1}$, $n \geq 1$. Then there are at most countably many balls $B_{1},B_{2},...$ for which $\omega(\d B_{j})>0$. 
\end{lemma}

\begin{proof}
The proof is identical to the one 	of \cite[Lemma 2.10]{AMT20} using \cite[Theorem 2.27]{HKM93}, \cite[Theorem 10.1]{HKM93}, and   \cite[Theorem 11.15]{HKM93}.  We omit the details. 
\end{proof}

\vvv

Finally, we state an elliptic version of the Alt–Caffarelli–Friedman monotonicity formula.

\begin{theorem}\label{theorem:ACF_DMO}
	Let $n\geq 2$ and suppose that $A$ is a uniformly elliptic matrix with coefficients in $L^\infty(\Rn1)$ which are uniformly continuous with modulus of continuity $w$. Let $B\coloneqq B(x,R)$ be a ball in $\Rn1$, and assume that $A(x)=\mathrm{Id}$. Let $u_1, u_2\in W^{1,2}(B(x,R))\cap C(B(x,R))$  be nonnegative $L_A$-subharmonic functions such that $u_1(x)=u_2(x)=0$ and $u_1\cdot u_2\equiv 0$. We set
	\begin{equation}\label{eq:monotonicity_formula}
		J(x,r)\coloneqq \Biggl(\frac{1}{r^2}\int_{B(x,r)}\frac{|\nabla u_1(y)|^2}{|y-x|^{n-1}}\, dy\Biggr) \cdot \Biggl(\frac{1}{r^2}\int_{B(x,r)}\frac{|\nabla u_2(y)|^2}{|y-x|^{n-1}}\, dy\Biggr).
	\end{equation}
	Then $J(x,r)$ is an absolutely continuous function of $r\in (0,R)$ and
	if, in addition, there exists  $\alpha\in (0,1]$ such that 
	\[
	u_i(y)\leq C_1 \Bigl(\frac{|y-x|}{r}\Bigr)^\alpha\, \|u_i\|_{L^\infty(B(x,r))}\, \qquad i=1,2,
	\]
	for all $0<r\leq R$ and $y\in B(x,r)$, then,  there exists a dimensional constant $c$ such that,  for any  $0<r_1 < r_2<R$,  
	\begin{equation}\label{eq:almost monotone}
		J(x,r_1) \leq  J(x,r_2)  \, e^{c \int_{r_1}^{r_2} w(t) \frac{dt}{t}}.
	\end{equation}
\end{theorem}

\begin{proof}
	This is a particular case of \cite[Theorem 1.5]{AGMT17}, which was formulated in a more general framework of uniformly elliptic operators in divergence form with lower order terms.
\end{proof}

\vv
\subsection{Green's function and fundamental solution}\label{sec:green}

In this subsection, we collect several results concerning the Green's function that will be used throughout the paper.
\begin{definition}[Green's function]
	Let $n\geq 2$, and $\Omega\subseteq \Rn1$ be a domain. Let $A$ be a $(n+1)\times (n+1)$ uniformly elliptic matrix with coefficients in $L^\infty (\Rn1)$, and denote $L_A\coloneqq -\div (A(\cdot)\nabla)$.
	We say that a function $G\colon \Omega\times \Omega\setminus \{(x,y)\in\Omega\times\Omega: x=y\}\to \mathbb R$ is a \textit{Green's function} of $L_A$ in $\Omega$ if it satisfies the following properties:
	\begin{enumerate}
		\item $G(\cdot, y)\in W^{1,1}_{\loc}(\Omega)$ and $L_AG(\cdot,y)=\delta_y$  for all $y\in\Omega$, namely it holds that
		\[
		\int_\Omega A \nabla_1  G (\cdot, y)\cdot \nabla \phi= \phi(y), \qquad \text{ for all } \phi\in C^\infty_c(\Omega).
		\]
		\item $G\in Y^{1,2}(\Omega\setminus B(y,r))$ for all $y\in\Omega$ and $r>0$, and $G(\cdot, y)$ vanishes on $\partial \Omega$.
		\item For any $f\in L^\infty_c(\Omega)$, the function 
		\[
		u(x)\coloneqq \int_\Omega G(y,x) f(y)\, dy
		\]
		belongs to $Y^{1,2}_0(\Omega)$ and satisfies $L^*_A u= f$, namely
		\[
		\int_\Omega A^T\nabla u\cdot \nabla \phi =\int f\, \phi, \qquad \text{ for all }\phi\in C^\infty_c(\Omega).
		\]
	\end{enumerate}
	
	The Green's function may be denoted by \( G_A \) when it is necessary to emphasize the associated elliptic operator, and by \( G_{\Omega}^A \) when it is also important to indicate the domain in which it is defined.

		\vv

	If $\om =\Rn1$, then the Green's function is denoted by $\Gamma(\cdot, \cdot)$ (or $\Gamma_A(\cdot,\cdot)$) and is called the {\bf  fundamental solution} for $A$. 
\end{definition}

\vvv

If $f\colon \Omega\to \mathbb R$ is a measurable function on an open set $\Omega\subset \Rn1$, we define
\[
d_{f,\Omega}(t)\coloneqq \bigl|\{x\in\Omega: |f(x)|>t\}\bigr|, \qquad \text{ for }t>0.
\]
Hence, for $p\in (0,\infty)$ and $q\in (0,\infty)$, we define the $(p,q)$-{\it Lorentz semi-norm} 
\[
\|f\|_{L^{p,q}(\Omega)}\coloneqq p^{\frac{1}{q}}\Bigl(\int^\infty_0 \bigl(t \, d_{f,\Omega}(t)^{\frac{1}{p}}\bigr)^q\, \frac{dt}{t}\Bigr)^{\frac{1}{q}},
\]
and understand that $f$ belongs to the $(p,q)$ {\it Lorentz space} $L^{p,q}(\Omega)$ if $	\|f\|_{L^{p,q}(\Omega)}<\infty$. \vvv

If $f \in L^\infty(\om)$ and it is compactly supported in $\Rn1$, it is clear that 
\begin{equation}\label{eq:lorentz<bounded}
	\| f \|_{L^{\frac{n+1}{2},1}(\om)} \lesssim \diam(\supp(f))^2 \,  \|f\|_{L^\infty(\om)}.
\end{equation}

We now gather the existence results for Green's functions, and some relevant properties.
\begin{lemma}\label{lemgreen*}
	Let $A$ be a $(n+1)\times(n+1)$ uniformly elliptic matrix with coefficients in $L^\infty(\Rn1)$ and let $\Omega\subseteq \R^{n+1}$ be an open, connected set. 
	There exists a unique Green's function $G\colon \Omega\times \Omega\setminus\{(x,y)\in\Omega\times\Omega:x=y\}\to \R$ associated with $L_A$, which
	satisfies the following properties:
	\begin{align}
		\|G(\cdot,y)\|_{Y^{1,2}(\Omega\setminus B(y,r))}&\lesssim r^{1-\frac{n+1}{2}}, \qquad \text{ for } r>0, \label{eq:bd_g_1}\\
		\|G(\cdot,y)\|_{L^p(B(y,r))}&\lesssim_p r^{1-n+\frac{n+1}{p}}, \qquad \text{ for  } r<\delta_\Omega(y), p\in \Bigl[1,\frac{n+1}{n-1}\Bigr],\label{eq:bd_g_2}\\
		\|\nabla G(\cdot,y)\|_{L^p(B(y,r))}&\lesssim_p r^{-n+\frac{n+1}{p}}, \qquad \text{ for  } r<\delta_\Omega(y), p\in \Bigl[1,\frac{n+1}{n}\Bigr].\label{eq:bd_g_3}
	\end{align}
	
	If $f\in L^{\frac{n+1}{2}, 1}(\Omega)$, we have that the function $u(x)=\int_\Omega G(y,x)f(y)\, dy$ satisfies
	\begin{equation}\label{eq:Green_function_global bound}
		\sup_\om |u| \lesssim \|f\|_{L^{\frac{n+1}{2}, 1}(\Omega)}.
	\end{equation}

	Furthermore, there are positive constants $C$ and $c$ depending on $n$ and $\Lambda$ such that for all $x,y\in\Omega$ with $x\neq y$, it holds:
	\begin{equation}\label{eq:bound_above_green}
		0\leq G(x,y)\leq C\,|x-y|^{1-n},
	\end{equation}
	
	\vspace{-4mm}
	\begin{equation}\label{eq:bound_below_green}
		G(x,y)\geq c\,|x-y|^{1-n},  \quad \text{ if } 2  |x-y|  < \dist(\{x,y\}, \pom),
	\end{equation}
	
	\vspace{-3mm}
	$$G(x,\cdot)\in C(\Omega \setminus \{x\}) \cap W^{1,2}_{\loc}(\Omega \setminus \{x\}) 
	\quad \mbox{ and } \quad G(x,\cdot)|_{\partial\Omega} \equiv 0,$$
	
	\vspace{-3mm}
	$$G(x,y) = G^*(y,x),$$
	where $G^*$ is the Green function associated with the operator $L^*_A = -\mathrm{div} A^T\nabla$.

	Finally,  if we assume that $\om$ is Wiener regular,  then $G(x,\cdot)\in C(\overline \Omega \setminus \{x\})$ and  for every $\vphi\in Y^{1,2}(\om) \cap C(\overline \om)$,
	\begin{equation}\label{eqgreen*23}
		\int_{\partial\Omega} \vphi\,d\omega^x - \vphi(x) = - \int_\Omega A^T(y)\nabla_yG(x,y)\cdot \nabla\vphi(y)\,dy,
		\quad\mbox{ for all  $x\in \overline \Omega$.}
	\end{equation}
\end{lemma}

\begin{proof}
	The existence of a Green's function for bounded domains was proved in \cite{GW82}.
	For the existence of $G$ in general unbounded domains and the bounds \eqref{eq:bd_g_1}, \eqref{eq:bd_g_2}, \eqref{eq:bd_g_3}, we refer to \cite[Lemma 5.4]{HK07}, where existence of Green's functions is studied in the more general case of elliptic systems.  For a proof of \eqref{eq:Green_function_global bound},  we refer to \cite[Theorem 6.1-(2)]{Mou23}, for the global bound \eqref{eq:bound_above_green} to \cite{KK10}, for \eqref{eq:bound_below_green} to \cite{GW82}.   
	
	Finally, \eqref{eqgreen*23} was proved for $\vphi \in C^\infty_c(\Rn1)$ and $\mathcal{L}^{n+1}$-a.e.  $x \in \om$ in \cite[Lemma 2.6]{AGMT17}. However, note that the function  on the right hand-side of \eqref{eqgreen*23} is the weak solution to the inhomogeneous problem $L w =\dv (A \nabla \vphi)$  in $\om$ and $w \in Y^{1,2}_0(\om)$. Thus,  since $\om$ is Wiener regular, by Theorems \ref{thm:C_alpha_interior-inhom} and \ref{thm:boundary oscillation},  \eqref{eq:Green_function_global bound}, and \eqref{eq:lorentz<bounded},  it is  continuous on $\overline \om$.   Since the functions on the left hand side of \eqref{eqgreen*23} are  continuous on $\overline \om$ as well and $\om$ is connected, we have that  \eqref{eqgreen*23} holds everywhere on $\overline \om$.  
\end{proof}

\vv

\begin{lemma}\label{eq:Green-repr-wiener}
Let \( \Omega \subset \mathbb{R}^{n+1} \) be a Wiener regular domain, and let \( A \) be a uniformly elliptic \( (n+1) \times (n+1) \) matrix with coefficients in \( L^\infty(\Omega) \). 
If   $G(\cdot, \cdot)$ and $\Gamma(\cdot, \cdot)$  are the respective Green's function in $\om$ and  fundamental solution    associated with $A$ constructed in Lemma \ref{lemgreen*}, then it holds that
\begin{equation}\label{eq;green ident-Wiener} 
	G(x,y) = \Gamma(x,y) - H_{\Gamma(\cdot,y)|_{\pom}}(x),  \qquad \text{ for all } x, y \in \om \text{ with } x \neq y. 
\end{equation}
\end{lemma}
\begin{proof}
Let us fix $x, y \in \om$ and a small constant $\ve>0$ so that $\ve \ll \min\{ \rho_k, s_j  \} \leq \max\{ \rho_k, s_j  \} \ll \min\{ \delta_\om(x), \delta_\om(y), |x-y| \}$.  Let us also denote as $G_k^x(\cdot)\coloneqq G_A^{\rho_k}(x,\cdot)$ and $\Gamma_j^y(\cdot)\coloneqq \Gamma^{s_j}_A(\cdot,y)$ the approximating Green's functions and fundamental solutions as constructed in \cite[(4.3)]{HK07} and \cite[(3.4)]{HK07} respectively.   

Let us  define $\eta_\ve^y(\cdot)\coloneqq \eta(|\cdot-y|/\ve)$ where $\eta=1$ in $\Rn1 \setminus B(0,2)$, $\eta=0$ in $\overline{B(0,1)}$,  $0\leq \eta \leq 1$, and $|\nabla \eta| \lesssim 1$ in $\Rn1$.  Set now $\psi^y_\ve \coloneqq  \Gamma_j^y\eta_\ve^y$ and notice that as $ \Gamma_j^y \in Y^{1,2}(\Rn1) \cap C(\Rn1 \setminus \{y\})$, then $\psi_\ve^y \in Y^{1,2}(\Rn1) \cap C(\Rn1 )$ and it vanishes at infinity.  Solve now the variational Dirichlet problem for $L_A$ in $\om$ and find a $L_A$-solution $u_\ve \in Y^{1,2}(\om) \cap C(\overline \om)$ such that $u_\ve - \psi_\ve^y \in Y^{1,2}_0(\om)$.  Then,  if we set $B_k\coloneqq B(x,\rho_k)$, by \cite[(4.3)]{HK07},  we have that 
\begin{align*}
\int_\om A^T(z)\nabla G_k^x (z) \cdot\nabla (u_\ve - \psi_\ve^y)(z)\,dz = \avint_{B_k} (u_\ve-\psi_\ve^y)(z)\,dz
\end{align*}
Now, using that $G_k^x \in Y^{1,2}_0(\om)$ and that $L_A u_\ve=0$ in $\om$, we obtain that
$$
\int_\om A^T(z)\nabla G_k^x (z) \cdot \nabla u_\ve(z)\,dz  = \int_\om A(z) \nabla u_\ve(z) \cdot \nabla G_k^x (z) \,dz  =0.
$$
Note that 
\begin{align*}
\int_\om A^T(z)\nabla G_k^x (z) \cdot \nabla \psi_\ve^y(z)\,dz  &= \int_\om A^T(z)\nabla G_k^x (z) \cdot \nabla \Gamma_j^y(z) \, \eta_\ve^y(z)\,dz\\
&\qquad  + \int_\om A^T(z)\nabla G_k^x (z)\cdot  \nabla  \eta_\ve^y(z) \,\Gamma_j^y(z) \,dz \eqqcolon  J^\ve_1 +J^\ve_2.
\end{align*}
Since  $G_k^x \in Y^{1,2}_0(\om)$ and  $\Gamma_j^y \in Y^{1,2}(\Rn1)$,  dominated convergence and \cite[(3.5)]{HK07} yield 
$$
\lim_{\ve \to 0} J^\ve_1 = \int_\om A^T(z)\nabla G_k^x (z)\cdot  \nabla \Gamma_j^y(z)\,dz=  \int_\om A(z)\nabla \Gamma_j^y(z)\cdot \nabla G_k^x (z) \,dz= G_k^x(y).
$$
By \cite[(3.19)]{HK07},  if $A(y,\ve,2\ve)\coloneqq B(y,2\ve) \setminus B(y, \ve) $, we have that 
$$
|J^\ve_2| \lesssim \ve^{-n}  \int_{A(y,\ve,2\ve)} |\nabla G_k^x| \lesssim \ve\,\mathcal{M}(|\nabla G_k^x|\, \chi_\om)(y),
$$ 
where $\mathcal M$ stands for the centered Hardy-Littlewood function in $\Rn1$.  Since $G_k^x \in Y^{1,2}_0(\om)$,  it holds that $\mathcal{M}(|\nabla G_k^x|\, \chi_\om)(y)<\infty$ for $\mathcal{L}^{n+1}$-a.e. $y \in \om$, and thus, 
$
\lim_{\ve \to 0}  J^\ve_2 = 0.
$

It remains to prove that 
\begin{equation}\label{eq:green-ident-1}
 \lim_{\ve \to 0}  \avint_{B_k} (u_\ve-\psi_\ve^y)(z)\,dz = \avint_{B_k} \left ( \int_\pom \Gamma_j^y(\xi)\,d\omega^z (\xi)- \Gamma_j^y(z)\right)\,dz.
\end{equation}
By Theorem \ref{theorem:resolutiveHf} and the fact that $\psi_\ve^y(z) = \Gamma_j^y(z)$ for any $z \not\in B(y,2\ve)$,  we have that 
$$
u_\ve(z) = \int_\pom \Gamma_j^y(\xi)\,d\omega^z (\xi) \qquad \text { and } \qquad \avint_{B_k} \psi_\ve^y =  \avint_{B_k} \Gamma_j^y,
$$
which proves  \eqref{eq:green-ident-1}.  We have shown that 
$$
 -G_k^x(y)=\avint_{B_k} \left ( \int_\pom \Gamma_j^y(\xi)\,d\omega^z (\xi)- \Gamma_j^y(z)\right)\,dz.
$$
By picking a  sequence $s_j$ so that $\lim_{j \to \infty} \Gamma_j^y(\cdot) = \Gamma(\cdot, y)$  locally uniformly away from $y$ (thus  in $B_k$) as guaranteed by \cite[(3.41)]{HK07},  we take limits on both sides of the identity above and infer
$$
- G_k^x(y)=\avint_{B_k} \left ( \int_\pom \Gamma(\xi,y)\,d\omega^z (\xi)- \Gamma(z,y)\right)\,dz.
$$
Now we pick a sequence $\rho_k$ so that $G_k^x(y) \to G(x,x)$ locally uniformly away from $y$ (see \cite[(3.41)]{HK07}) and deduce  that 
$$
 G(x,y)= \Gamma(x,y) - \int_\pom \Gamma(\xi,y)\,d\omega^x(\xi),\qquad \text{ for }\mathcal L^{n+1}\text{-a.e. }y\in \Omega \text{ and all }x\in \Rn1. 
$$

 Since $G(x,y)=G_{A^T}(y,x)$, $\Gamma(x,y)=\Gamma_{A^T}(y,x)$,  and $\int_\pom \Gamma(\xi,y)\,d\omega^x(\xi)$  are continuous in $y$ away from a fixed $x \in \om$,  and $\om$ is connected,  we conclude the proof of the lemma.
\end{proof}

\vv

For a real, elliptic $(n+1)\times (n+1)$-matrix $A_0$ with constant coefficients, we denote by $\Theta(x,y; A_0)=\Gamma_{A_0}(x,y)$ the fundamental solution of $L_{A_0}$. In particular, we recall that
\[
\Theta(x,y;A_0)=\Theta(x-y,0;A_0)
\]
and, for $n\geq 2$,
\begin{equation}\label{eq:fund_sol_const_matrix}
	\Theta(z,0;A_0)=\Theta(z,0;A_{0,s})=
		\frac{-\omega_n^{-1}}{(n-1)\sqrt{\det A_{0,s}}}\frac{1}{\langle A_{0,s}^{-1}z, z\rangle^{(n-1)/2}},
\end{equation}
where $A_{0,s}\coloneqq\frac12(A_0+A_0^T)$ is the symmetric part of $A_0$ {and $\omega_n$ is the surface measure of the unit sphere $\mathbb S^n\coloneqq \partial B(0,1)$. Moreover, it holds
	\begin{equation}\label{eq:gradient_fund_sol_const_matrix}
		\nabla_1 \Theta(z,0;A_0)=\frac{\omega_n^{-1}}{\sqrt{\det A_{0,s}}}\frac{A_{0,s}^{-1}\,z}{\langle A_{0,s}^{-1}z, z\rangle^{(n+1)/2}}, \qquad \text{ for }z\in \mathbb R^{n+1}\setminus\{0\}.
	\end{equation}	Therefore, it is straightforward to see that
	\begin{equation}\label{eq:gradient_fund_sol_const_matrix-1}
		\nabla_1 \Theta(y,x;A_0)	=	\nabla_1 \Theta(y-x,0;A_0)=	-\nabla_1 \Theta(x-y,0;A^T_0)=-\nabla_2 \Theta(x,y;A_0^T).
	\end{equation}
	Moreover, one can easily prove that for any  integer $k \geq 0$,
	\begin{equation}\label{eq:gradient_fund_sol_const_matrix-2}
		\bigl|\nabla_1^{(k)} \Theta(z,0;A_0)\bigr| \lesssim_{k,n,\Lambda} |z|^{-(n+k-1)},  \qquad \text{ for }z\in \mathbb R^{n+1}\setminus\{0\}.
	\end{equation}

	\vvv
	Finally, we state a lemma that relates the Green's functions and elliptic measures associated with matrices that agree $\mathcal L^{n+1}$-almost everywhere.
	
	\begin{lemma}\label{lem:reduction_to_UC}
		Let $A$ and $A'$ be two uniformly elliptic matrices with coefficients in $L^\infty(\Rn1)$ such that $A(y)=A'(y)$ for $\mathcal L^{n+1}$-almost every $y\in\Rn1$. Let also $\Omega\subset\R^{n+1}$ be a Wiener regular domain. The following statements hold:
		\begin{enumerate}
			\item If we denote by $G_A$ and $G_{A'}$ the Green functions associated with $L_A$ and $L_{A'}$ relative to  $\Omega$, then $G_A(\cdot, \cdot)=G_{A'}(\cdot, \cdot)$ on $\Omega\times\Omega\setminus \{(x,y)\in\Omega\times\Omega:x=y\}$.
			\item If $p\in \Omega$, then $\omega^{L_A,p}_{\Omega}=\omega^{L_{A'},p}_{\Omega}$.
		\end{enumerate}
	\end{lemma}
	\begin{proof}
		By \cite[(4.38)]{HK07}, for every $\phi \in C^\infty_c(\Omega)$ and $x \in \Omega$, we have that
		\begin{equation*}
			\phi(x) = \int_\Omega A \nabla_1 G_A(\cdot, x) \cdot \nabla \phi = \int_\Omega A' \nabla_1 G_{A'}(\cdot, x) \cdot \nabla \phi.
		\end{equation*}
		Thus, since $A(y) = A'(y)$ for $\mathcal{L}^{n+1}$-almost every $y$,
		\begin{equation}\label{eq_red1}
			\int_{\Omega} A\,  \nabla_1 \bigl(G_A(\cdot, x) - G_{A'}(\cdot, x)\bigr) \cdot \nabla \phi = 0, \qquad \text{for every } \phi \in C^\infty_c(\Omega).
		\end{equation}
		Hence, the property $G_A(\cdot, x)|_{\partial \Omega} \equiv G_{A'}(\cdot, x)|_{\partial \Omega} \equiv 0$ from Lemma~\ref{lemgreen*} and \eqref{eq_red1} implies that, for every $x \in \Omega$, we have $G_A(\cdot, x) \equiv G_{A'}(\cdot, x)$ $\mathcal{L}^{n+1}$-almost everywhere on $\Omega \setminus \{x\}$.
		The continuity of $G_A(\cdot, \cdot)$ and $G_{A'}(\cdot, \cdot)$ on $\Omega \times \Omega \setminus \{(x,y) \in \Omega \times \Omega : x = y\}$, see Lemma~\ref{lemgreen*}, yields that $G_A(\cdot, \cdot) = G_{A'}(\cdot, \cdot)$.

		We now prove property (2). We observe that, for any $\varphi, \Phi \in C^\infty_c(\mathbb{R}^{n+1})$ such that $\Phi \equiv \varphi$ on $\partial \Omega$ and $\mathcal{L}^{n+1}$-almost every $x \in \Omega$, 
		\begin{equation}\label{eq:red_2}
			\begin{split}
				\int_{\partial \Omega} \varphi\, d\omega^{L_A, p}_{\Omega} - \Phi(x) 	&\overset{\eqref{eqgreen*23}}{=} -\int_{\Omega} A(y) \, \nabla_y G_A(y, p) \cdot \nabla \Phi(y)\, dy \\
				&= -\int_{\Omega} A'(y) \, \nabla_y G_{A'}(y, p) \cdot \nabla \Phi(y)\, dy \overset{\eqref{eqgreen*23}}{=} \int_{\partial \Omega} \varphi\, d\omega^{L_{A'}, p}_{\Omega} - \Phi(x),
			\end{split}
		\end{equation}
		where the second equality holds because of (1) and the fact that $A$ agrees with $A'$ $\mathcal{L}^{n+1}$-almost everywhere. This readily implies that $\omega^{L_A, p}_{\Omega} = \omega^{L_{A'}, p}_{\Omega}$.
	\end{proof}

\vv

\subsection{The gradient of the fundamental solution.}
We first recall a crucial property of solutions to elliptic operators in divergence form associated with $\DMO_s$ matrices.
\begin{remark}[see {\cite[Remark 3.8]{MMPT23}}] \label{remark:rem_grad_bound}
	Let $\eta\in (0,1/2)$ be a proper absolute constant chosen as in \cite[p. 424]{DK17}.
	Let $0<r<R_0$, $x_0\in \Rn1$,  and $g\colon B(x_0, (N +1) R_0)\to \R$, where {$N\coloneqq 3 (\tfrac{4}{3})^{N_\eta}$ $N_\eta$ is such that $2^{-N_\eta-1}\leq \eta< 2^{-N_\eta}$}. For $0<t<r$ and $1\leq k\leq N$, we define  
	\begin{equation*}
		\begin{split}
			\mathring\oomega^{x_0, k r}_g(t) &= \sup_{w\in B(x_0, k  r)} \avint_{B(w, t r)}\bigl|g(x) - \bar g_{w,tr}\bigr|\, dx.
		\end{split}
	\end{equation*}
	
	Let  $A$ be a uniformly elliptic matrix such that  $A\in \DMO_s$ and assume that, for $N$ as above,
	\[
	\int_0^1 \mathring\oomega^{x_0, N R_0}_g(t)\, \frac{dt}{t}<\infty.
	\] 
	If  $u$ is a weak solution to $-\div (A(x)\nabla u)=-\div g$ in $B(x_0, (N+1) R_0)$, we obtain that $u\in C^1(\overline{B(x_0, r)})$ and it satisfies the estimate
	\begin{equation}\label{eq:Linftyest-r}
		\|\nabla u\|_{L^\infty(B(x_0,2r))}\lesssim_{R_0} \avint_{B(x_0, 4 r)} |\nabla u(x)|\, dx+ \int_0^1 \mathring\oomega^{x_0, N r}_g(t)\, \frac{dt}{t}.
	\end{equation}
	
	Furthermore, the implicit constant in \eqref{eq:Linftyest-r} blow up logarithmically as $R_0 \to \infty$. If $0 < R_0 < 1$, it only depends on ellipticity, dimension, and the Dini Mean Oscillation condition.
	
\end{remark}
}

As a consequence of Remark \ref{remark:rem_grad_bound}, we state the Calder\'on-Zygmund-type properties of gradients of fundamental solutions associated with $\DMO_s$ matrices in the following lemma; see \cite[Lemma 3.9]{MMPT23} for the proof.

\begin{lemma}\label{lem:estim_fund_sol}
Let $A(\cdot)=(a_{ij})$	be a uniformly elliptic matrix in $\Rn1$, $n\geq 2$, satisfying $A\in \DMO_s$, and let $R>0$.
There exist $C=C(n,\Lambda,R)>0$ and $\beta=\beta(n)>0$ such that the fundamental solution $\Gamma_A$ satisfies the following properties:
\begin{enumerate}
	\item \(|\nabla_1\Gamma_A(x,y)| + |\nabla_2\Gamma_A(x,y)|\leq C {|x-y|^{-n}}\)  for $x,y\in \Rn1$, $0<|x-y|<R$.
	\item \(|\nabla_1\nabla_2\Gamma_A(x,y)|\leq C |x-y|^{-(n+1)}\)  for $x,y\in \Rn1$, $0<|x-y|<R$.
	\item 
	We have
	\begin{align}\label{eq:continuityGamma}
		|\nabla_1 \Gamma_A(x,y) - \nabla_1\Gamma_A(x,z)| &+ |\nabla_1\Gamma_A(y,x) - \nabla_1\Gamma_A(z,x)| \notag\\
		& \leq C \left( \frac{|y-z|^\beta}{|x-y|^\beta}+ \int_0^{\frac{|y-z|}{|x-y|}} \oomega_A(t)\,\frac{dt}{t}\right)  |x-y|^{-n},
	\end{align}
	for  $2|y-z|\leq |x-y|<R$.
\end{enumerate}
\end{lemma}

\vv

We present a representation formula for the difference of fundamental solutions associated with two possibly different matrices.

{\begin{lemma}[{\cite[Lemma 3.10]{MMPT23}}]\label{lem:MMPT_lemma_310}
	If $A_1$ and $A_2$	are uniformly elliptic matrices in $\Rn1$, $n\geq 2$, so that  $A_1,  A_2\in \DMO_s$, then for $R>0$ and all $x, y \in \Rn1$ such that $0<|x-y|<R$, it holds that
	\begin{equation}\label{eq:representation-perturbation}
		\Gamma_{A_1}(x,y)-\Gamma_{A_2}(x,y) =\int \nabla_2\Gamma_{A_1}(x,z)\cdot\bigl(A_1(z)-   A_2(z)\bigr)\nabla_1\Gamma_{A_2}(z,y)\, dz.
	\end{equation}
\end{lemma}

\vvv

\section{Suppressed singular integral operators}\label{sec:aux_operator}

Throughout this section, we fix $d\in \mathbb N$, $d\geq 3$, and $n\in\{1,\ldots, d\}$.
Let $\Xi\colon \R^d \to \mathbb R$ be a non-negative function that satisfies the $1$-Lipschitz condition 
\[
|\Xi(x)-\Xi(y)|\leq |x-y|,\, \qquad \text{ for } x, y\in \R^d,
\]
and  $\mu$ be a Radon measure on $\R^d$.

For any kernel $k(x,y)$, we define the \textit{suppressed kernel} 
	\begin{equation}\label{eq:suppressed_kernel}
		{k}_{\Xi}(x,y)\coloneqq k(x,y) \frac{|x-y|^{2n}}{ |x-y|^{2n}+ \Xi(x)^n\Xi(y)^n},
	\end{equation}
the associated truncated operator
	\begin{equation}\label{eq:SIO operator}
T_{\Xi, \mu, \varepsilon}f(x)\coloneqq \int_{|x-y|>\varepsilon	} k_{\Xi}(x,y)f(y)\, d\mu(y), \qquad \text{ for } \varepsilon>0,
	\end{equation}
and the maximal suppressed operator
\[
T_{\Xi, \mu, *}f(x)\coloneqq \sup_{\varepsilon >0}\bigl|T_{\Xi, \mu, \varepsilon}(x)\bigr|.
\]

We understand that the suppressed operator $T_{\Xi,\mu}$ is bounded on $L^2(\mu)$ if $T_{\Xi, \mu, \varepsilon}$ are bounded operators on $L^2(\mu)$ uniformly on the truncation $\varepsilon>0$, and we set
\[
	\bigl\|T_{\Xi, \mu}\bigr\|_{L^2(\mu)\to L^2(\mu)}\coloneqq \sup_{\varepsilon>0}	\bigl\|T_{\Xi, \mu, \varepsilon}\bigr\|_{L^2(\mu)\to L^2(\mu)}.
\]

We also use the notation $T_{\Xi,\varepsilon}\mu\coloneqq T_{\Xi, \mu, \varepsilon}1$ and $T_{\Xi, *}\mu\coloneqq T_{\Xi, \mu, *}1$.

\vv

For $x,x',y\in \Rn1$, we set
\(
	m(x,x',y)\coloneqq\max(|x-y|,|x'-y|).
\)
\begin{lemma}\label{lemma:CZ_estimates_suppressed}
Let us assume that $k\colon \R^d \times \R^d \to \mathbb R$ satisfies the following bounds:
	There exist $R \in (0, \infty]$,  $c_1>0$, $ c_2>0$, and $\theta\colon[0,\infty) \to [0,\infty)$ such that
	\begin{align}
		|k(x,y)|    &\leq c_1  \, |x-y|^{-n},\qquad\textup{for} \,\,x,y\in \R^d,  0< |x-y|<R \quad \textup{and};\label{eq:CZ_1}\\
		|k(y,x)-k(y,x')|&\leq c_2\, \theta \Big(\frac{|x-x'|}{m(x,x',y)}\Big) {m(x,x',y)^{-n}},\label{eq:CZ_2}
	\end{align}
		for $x,  x', y \in \R^d$ satisfying $2|x-x'|\leq \max(|x-y|,|x'-y|)<R$.

	Then,  for $\tilde\theta(t)\coloneqq \theta(t)+t$, it holds that
	\begin{align}
		|k_{\Xi}(x,y)|    &\leq c_1\,  |x-y|^{-n},\qquad\textup{for} \,\,x,y\in \R^d,  0< |x-y|<R\quad \textup{and};\label{eq:CZ_1_supp}\\
		|k_{\Xi}(y,x)-k_{\Xi}(y,x')|& \leq \left( c_2+ c_1\,  n \, 2^{n+2}\right) \, \tilde\theta \Big(\frac{|x-x'|}{m(x,x',y)}\Big) {m(x,x',y)^{-n}},\label{eq:CZ_2_supp}
	\end{align}
	for $x,  x', y \in \R^d$ satisfying $2|x-x'|\leq \max(|x-y|,|x'-y|)<R$.
	
	Moreover, it holds that
	\begin{equation}\label{eq:bound_t_delta_min_suppressions}
		|k_{\Xi}(x,y)|\leq c_1\,   \min\bigl(\Xi(x)^{-n}, \Xi(y)^{-n}\bigr).
	\end{equation}
\end{lemma}

\begin{proof}
	The bound \eqref{eq:CZ_1_supp} readily follows from observing that
	\[
		|k_{\Xi}(x,y)|\leq |k(x,y)|\leq c_1\, |x-y|^{-n}.
	\]
	
	To prove the other estimates for $k_{\Xi}$, we adapt the proof of \cite[Lemma 8.2]{Vo03}.  The definition of $k_{\Xi}$ and \eqref{eq:CZ_1} imply that
	\begin{equation}\label{eq:bound_t_delta_Theta_1}
		|k_{\Xi}(x,y)|\leq \frac{|k(x,y)||x-y|^{2n}}{\Xi(x)^n\Xi(y)^n}\leq c_1\,  \frac{|x-y|^n}{\Xi(x)^n\Xi(y)^n}.
	\end{equation}
	If $|x-y|\leq \Xi(x)/2$, by \eqref{eq:bound_t_delta_Theta_1} we obtain
	\begin{equation}\label{eq:bound_t_delta_Theta_2}
	|k_{\Xi}(x,y)|\leq c_1\,  2^{-n} \, \Xi(y)^{-n}.
	\end{equation}
	Moreover, if $|x-y|\leq \Xi(x)/2$,  since $\Xi$ is $1$-Lipschitz, we have that 
	\[
		\Xi(y)\geq \Xi(x)-|x-y|\geq \Xi(x)/2,
	\]
	and so, by \eqref{eq:bound_t_delta_Theta_2}, we infer that
	\begin{equation}\label{eq:bound_t_delta_Theta_3}
		|k_{\Xi}(x,y)|\leq c_1\, \Xi(x)^{-n}.
	\end{equation}
	This concludes the proof of \eqref{eq:bound_t_delta_min_suppressions} in the case $|x-y|\leq \Xi(x)/2$.
	
	Assume now $ |x-y|>\Xi(x)/2$.  If $\Xi(x) < 2|x-y| \leq \Xi(y)$, then by the same argument as above, we obtain \eqref{eq:bound_t_delta_min_suppressions}, while if $2 |x-y|>\max\{\Xi(x),\Xi(y)\}$, the kernel bounds \eqref{eq:CZ_1} yield
	\[
	|k_{\Xi}(x,y)|\leq |k(x,y)|\leq c_1\,  \frac{1}{|x-y|^n}\leq \frac{c_1\,  2^{-n}}{\max\{\Xi(x),\Xi(y)\}^n}, 
	\]
	concluding the proof of \eqref{eq:bound_t_delta_min_suppressions}.

	We are left with the proof of \eqref{eq:CZ_2_supp}. Let $x,x',y\in\R^d$ be such that $|x-x'|\leq |x-y|/2$ and define 
	$$
	F_\Xi(x,y)\coloneqq 1+ |x-y|^{-2n}\Xi(x)^n\Xi(y)^n.
	$$ 
	It holds that
	\begin{equation}\label{eq:bound_F_theta_easy}
		F_\Xi(x,y)\geq \max\bigl(1,|x-y|^{-2n}\Xi(x)^n\Xi(y)^n\bigr).
	\end{equation}
	Without loss of generality, we further assume that $\Xi(x')\leq \Xi(x)$. Then, we write
	\begin{equation}
		\begin{split}
			&\bigl|k_{\Xi}(y,x)-k_{\Xi}(y,x')\bigr|\leq \frac{\bigl|k(y,x)-k(y,x')\bigl|}{F_\Xi(x,y)} \\
			&\qquad\qquad\qquad +\bigl|k_{\Xi}(y,x')\bigr|\frac{\bigl||x-y|^{-2n}\Xi(x)^n-|x'-y|^{-2n}\Xi(x')^n\bigr|\Xi(y)^n}{F_\Xi(x,y)F_\Xi(x',y)}\\
			&\qquad \leq \frac{\bigl|k(y,x)-k(y,x')\bigl|}{F_\Xi(x,y)} + \bigl|k(y,x')\bigr|\frac{\bigl||x-y|^{-2n}-|x'-y|^{-2n}\bigr|\, \Xi(x')^n\Xi(y)^n}{F_\Xi(x,y)\, F_\Xi(x',y)}\\
			&\qquad\qquad \qquad + \bigl|k(y,x')\bigr|\frac{|x-y|^{-2n}\bigl(\Xi(x)^n-\Xi(x')^n\bigr)\Xi(y)^n}{F_\Xi(x,y)F_\Xi(x',y)}\eqqcolon \mathscr A_1 + \mathscr A_2 + \mathscr A_3.
		\end{split}
	\end{equation}
	
	By \eqref{eq:CZ_2} and $F_\Xi(x,y)\geq 1$,  we obtain
	\[
		\mathscr A_1 \leq c_2\,\theta \Big(\frac{|x-x'|}{m(x,x',y)}\Big) \,{m(x,x',y)^{-n}}.
	\]
	
	To bound the term $\mathscr A_2$, we observe that 
	\[
		\bigl||x-y|^{-2n}- |x'-y|^{-2n}\bigr|= \bigl||x-y|^{-n}- |x'-y|^{-n}\bigr|\, \bigl(|x-y|^{-n} +|x'-y|^{-n} \bigr).
	\]
	Hence
	\begin{equation}\label{eq:bd_A_2split}
		\begin{split}
			\mathscr A_2&=\Bigl||x-y|^{-n}- |x'-y|^{-n}\Bigr| \frac{\bigl|k(y,x')\bigr|\, |x'-y|^{-n}\, \Xi(x')^n\Xi(y)^n}{F_\Xi(x,y)\, F_\Xi(x',y)}\\
			&\qquad \qquad   + \bigl||x-y|^{-n}- |x'-y|^{-n}\bigr| \frac{\bigl|k(y,x')\bigr|\, |x-y|^{-n}\, \Xi(x')^n\Xi(y)^n}{F_\Xi(x,y)\, F_\Xi(x',y)}\eqqcolon \mathscr A_{2,1} + \mathscr A_{2,2}.
		\end{split}
	\end{equation}
	The function $\phi(x,y)=|x-y|^{-n}$ satisfies the standard Calder\'on-Zygmund-type bounds

	\begin{equation}\label{eq:CZ_potential}
		|\phi(x,y)-\phi(x',y)|\leq n \, 2^{n+1}  \,\frac{|x-x'|}{m(x,x',y)^{n+1}}
	\end{equation}
		for $2|x-x'|\leq \max(|x-y|,|x'-y|).$
	Moreover, since $F_\Xi(x,y)\geq 1 $ and $F_{\Xi}(x',y)\geq |x'-y|^{-2n}\Xi(x)^n \Xi(y)^n$, we have that 
	\begin{equation}\label{eq:bd_A_21_by_1}
		\frac{\bigl|k(y,x')\bigr|\, |x'-y|^{-n}\, \Xi(x')^n\Xi(y)^n}{F_\Xi(x,y)\, F_\Xi(x',y)}\leq \bigl|k(y,x')\bigr||x'-y|^n \overset{\eqref{eq:CZ_1}}{\leq} c_1.
	\end{equation}
	Gathering \eqref{eq:CZ_potential} and	\eqref{eq:bd_A_21_by_1},   we infer that
	\begin{equation}\label{eq:bound_matshcr_A_2_1_I}
		\mathscr A_{2,1}\leq  c_1\,  n \, 2^{n+1}  \, \frac{|x-x'|}{m(x,x',y)^{n+1}}.
	\end{equation}
	To estimate $\mathscr A_{2,2}$, we first observe that if $|x'-y|\leq |x-y|$,  then	$\mathscr A_{2,2}\leq \mathscr A_{2,1}$. 
	If $|x-y|< |x'-y|$, then \eqref{eq:CZ_1},  \eqref{eq:bound_F_theta_easy},   and the assumption $\Xi(x')\leq \Xi(x)$ yield
	\begin{equation}\label{eq:estim_mathscrA22_II}
		\begin{split}
			\mathscr A_{2,2}&\leq c_1\, \bigl||x-y|^{-n}- |x'-y|^{-n}\bigr| \frac{ |x-y|^{-2n}\, \Xi(x)^n\Xi(y)^n}{F_\Xi(x,y)\, F_\Xi(x',y)}\\
			& \leq c_1\, \ \bigl||x-y|^{-n}- |x'-y|^{-n}\bigr|\overset{\eqref{eq:CZ_potential}}{\leq} c_1\,  n \, 2^{n+1}  \, \frac{|x-x'|}{m(x,x',y)^{n+1}}.
		\end{split}
	\end{equation}
	Hence,  combining \eqref{eq:bd_A_2split}, \eqref{eq:bound_matshcr_A_2_1_I},  and \eqref{eq:estim_mathscrA22_II},   we get that
	\[
	\mathscr A_{2}\leq c_1\,  n \, 2^{n+1}  \,  \frac{|x-x'|}{m(x,x',y)^{n+1}}.
	\]

	Finally, the Lipschitz assumption on $\Xi$ and the hypothesis $\Xi(x')\leq \Xi(x)$ yield
	\begin{equation}\label{eq:estim_A_3}
		\mathscr A_3 \leq n\, |x-x'|\bigl|k(y,x')\bigr|\, \frac{|x-y|^{-2n}\Xi(x)^{n-1}\Xi(y)^n}{F_\Xi(x,y)F_\Xi(x',y)}.
	\end{equation}
		Recall also that the hypothesis $2|x-x'|\leq \max(|x-y|,|x'-y|)$ and the triangle inequality give $\tfrac{1}{2}|x-y|\leq |x'-y|\leq 2 |x-y|.$
	Thus, if $|x-y|\geq \Xi(x)$, the Lipschitz property of $\Xi$ implies that $\Xi(y)\leq 2|x-y|$, so we further bound \eqref{eq:estim_A_3} as
	\begin{equation*}
		\begin{split}
			\mathscr A_3&\leq\frac{ n \,2^n\,  c_1\,|x-x'|}{|x'-y|^{n}|x-y|^{2n}}	|x-y|^{2n-1} \leq \frac{ n \,2^n\,  c_1\, |x-x'|}{|x'-y|^n|x-y|}\lesssim \frac{ n \,2^n\,  c_1\, |x-x'|}{m(x,x',y)^{n+1}},
		\end{split}
	\end{equation*}
	where we also used the bound $|F_\Xi(x,y)|\geq 1$.	If $|x-y|\leq \Xi(x)$,  we readily have
	\begin{equation*}
		\begin{split}
			\mathscr A_3&\leq	n\,\bigl|k(y,x')\bigr|\frac{|x-x'|}{|x-y|}\, \frac{|x-y|^{-2n}\Xi(x)^{n}\Xi(y)^n}{F_\Xi(x,y)F_\Xi(x',y)} 
			\overset{\eqref{eq:bound_F_theta_easy}}{\leq}	\frac{c_1\,n\,|x-x'|}{|x'-y|^n|x-y|}\leq c_1\, 	n \,2^n\,  \frac{|x-x'|}{m(x,x',y)^{n+1}}.
		\end{split}
	\end{equation*}
	Gathering the estimates for $\mathscr A_1, \mathscr A_2$, and $\mathscr A_3$, we  obtain \eqref{eq:CZ_2_supp}.
\end{proof}
\vvv

\begin{lemma}\label{lemma:estimates_suppressed_kernel_2}
	Let $k$ and $k_{\Xi}$ be as in Lemma \ref{lemma:CZ_estimates_suppressed}. The following properties hold:
	\begin{enumerate}
		\item If either $\Xi(x)=0$ or $\Xi(y)=0$,  
		\[
			k_{\Xi}(x,y)=k(x,y).
		\]
		\item If $0< \Xi(x)\leq \varepsilon$, and $\sigma$ is a complex Radon measure on $\R^d$ with compact support so that $\diam(\supp(\sigma))\leq R$, then for every $x \in \supp(\sigma)$,  
		\begin{equation}\label{eq:comp_supp_no_supp_kernel}
			\bigl|T_{\Xi,  \varepsilon}\sigma(x)-T_{\varepsilon}\sigma(x)\bigr|\leq 2^{2n-1}\sup_{r\geq \varepsilon}\frac{|\sigma|(B(x,r))}{r^n}.
		\end{equation}
	\end{enumerate}
\end{lemma}
\begin{proof}
	Item (1) is a direct consequence of the definition of the suppressed kernel. The proof of \eqref{eq:comp_supp_no_supp_kernel} is an adaptation of \cite[Lemma 8.3]{Vo03} and \cite[Lemma 5.4]{To14}, which we report here for the reader's convenience.
		It holds that
		\begin{equation}\label{eq:pw_diff_suppressed_not_suppressed}
			\begin{split}
				|k_\Xi(x,y)-k(x,y)|&\overset{\eqref{eq:suppressed_kernel}}{=} \Bigl| \frac{k(x,y)|x-y|^{2n}}{ |x-y|^{2n}+ \Xi(x)^n\Xi(y)^n}-k(x,y)\Bigr|\\
				&=\Bigl| \frac{k(x,y)\, \Xi(x)^n\Xi(y)^n}{ |x-y|^{2n}+ \Xi(x)^n\Xi(y)^n}\Bigr| \leq \frac{\Xi(x)^n\bigl(\Xi(x)+|x-y|\bigr)^n}{|x-y|^{3n}}\\
				&\leq 2^{n-1} \,\frac{\Xi(x)^{2n}}{|x-y|^{3n}} +2^{n-1} \,  \frac{\Xi(x)^n}{|x-y|^{2n}}.
			\end{split}
		\end{equation}
	Hence, by the hypothesis on $\Xi$, we obtain
	\begin{equation*}
		\begin{split}
			\bigl|&T_{\Xi,  \varepsilon}\sigma(x)-T_{\varepsilon}\sigma(x)\bigr|\overset{\eqref{eq:pw_diff_suppressed_not_suppressed}}{\leq} 2^{n-1} \,\int_{|x-y|>\varepsilon}\Bigl(\frac{\Xi(x)^{2n}}{|x-y|^{3n}} + \frac{\Xi(x)^n}{|x-y|^{2n}}\Bigr)\, d|\sigma|(y)\\
			&\quad\leq 2^{n-1} \, \int_{|x-y|>\varepsilon}\Bigl(\frac{\varepsilon^{2n}}{|x-y|^{3n}} + \frac{\varepsilon^n}{|x-y|^{2n}}\Bigr)\, d|\sigma|(y)\leq 2^{2n-1}\sup_{r\geq \varepsilon}\frac{|\sigma|(B(x,r))}{r^n}.\qedhere
		\end{split}
	\end{equation*}
\end{proof}

\vvv

In the rest of the section, we present some technical lemmas that are used in the proof of Theorem \ref{thm:big piecesTb-aux}.

\vv

Let $\mu$ be a non-negative finite Borel measure supported on a compact set $F \subset \R^d$ such that $\diam(\supp (\mu)) \leq R$ for some $R>0$,  and let $\nu=b\, \mu$ be a complex measure for some $b\in L^\infty(\R^d)$ with $\|b\|_\infty\leq c_b$.  Assume also that there exist 	two sets $H \subset \R^d$ and $E \subset \R^d$ such that:
\begin{itemize}
\item 	If a ball $B(y,r)$ is such that $\mu(B(y,r))>c_0\, r^n$,  then $B(y,r)\subset H$.
\item 	There exists $\delta_0>0$ such that $\mu(H \cup E)\leq \delta_0\, \mu(F)$, for all $w \in \R^d$.
\item  We have that 
			\[
				\int_{\R^d \setminus H}  T_{ *}\nu\, d\mu\leq c_* \mu(F), \qquad \text{ for all }x\in \R^d,
			\]
\end{itemize}

\vvv

For $\alpha >0$ we define
\begin{equation}\label{eq:S0}
	S_0=\bigl\{x\in \R^d: T_{*}\nu(x)>\alpha \bigr\}.
\end{equation}
We set
\[
	e(x)\coloneqq
	\begin{cases}
		\sup\bigl\{\varepsilon>0: \bigl|T _{\varepsilon}\nu(x)\bigr|>\alpha\bigr\}, \qquad \qquad &\text{ for }x\in S_0,\\
		0, \qquad \qquad &\text{ for }x\not \in S_0,
	\end{cases}
\]
and define the exceptional set
\(
	S= \bigcup_{x\in S_0} B\bigl(x,e(x)\bigr).
\)

\vv

\begin{lemma}\label{lem:lemma_choice_alpha}
	If $y\in S\setminus H$,  then $T_*\nu(y)> \alpha - 2^{n+1}\,c_0\,  c_b\, ( \,c_1+  c_2\, \mathfrak I_{\theta}(1))$, for $c_1,c_2$ as in Lemma \ref{lemma:CZ_estimates_suppressed}. Hence,  if $\alpha > 2^{n+2}\,c_0\,  c_b\, (c_1+  c_2\, \mathfrak I_{\theta}(1))$,  it holds that 
\begin{equation}\label{eq:measure S minus H }
 \mu\bigl(S\setminus H\bigr)\leq \frac{2 \,c_*}{\alpha}\,  \, \mu(F).
	\end{equation}
\end{lemma}

\begin{proof}
	The proof follows the argument of \cite[Lemma 5.2]{To14}, but we include it here to highlight the necessary modifications.

	If \( y \in S \setminus H \), there exists \( x \in S_0 \) such that \( y \in B(x, e(x)) \). We claim that
	\begin{equation}\label{eq:claim_52}
		\begin{split}
			\bigl|T_{\varepsilon_0(x)}\nu(x)-T_{\varepsilon_0(x)}\nu(y)\bigr|\leq 2^n\,c_0\,  c_b\, (2 \,c_1+  c_2\, \mathfrak I_{\theta}(1)).
		\end{split}
	\end{equation}
	
	To prove \eqref{eq:claim_52}, we denote $B_y\coloneqq B(y,2\varepsilon_0(x))$ and we split
	\begin{equation*}
		\begin{split}
			\bigl|T_{\varepsilon_0(x)}\nu(x)-T_{\varepsilon_0(x)}\nu(y)\bigr|&\leq \bigl|T_{\varepsilon_0(x)}(\nu|_{B_y})(x)-T_{\varepsilon_0(x)}(\nu|_{B_y})(y)\bigr|\\
			&\qquad +\bigl|T_{\varepsilon_0(x)}(\nu|_{B^c_y})(x)-T_{\varepsilon_0(x)}(\nu|_{B^c_y})(y)\bigr|\eqqcolon \mathfrak J_1 + \mathfrak J_2.\\
		\end{split}
	\end{equation*}
	
	Since $B_y\not\subseteq H$, we have that $\mu(B_y)\leq c_0 (2\varepsilon_0(y))^n$ and so, by \eqref{eq:CZ_1}, 
	\begin{equation}
			\mathfrak J_1 \leq \bigl|T_{\varepsilon_0(x)}(\nu|_{B_y})(x)|+ |T_{\varepsilon_0(x)}(\nu|_{B_y})(y)\bigr| \leq 2\,c_1\,c_b\,\frac{\mu(B_y)}{\varepsilon_0(x)^n}\leq 2^{n+1}\, c_0 \,c_1\, c_b.
	\end{equation}
	
Using \eqref{eq:CZ_2},  we can estimate $\mathfrak J_2$ as follows.
	\begin{equation}
		\begin{split}
			\mathfrak J_2&\leq \int_{\R^d\setminus B_y} \bigl|k(x,z)- k(y,z)\bigr|\, d|\nu|(z) \leq c_2 \, c_b \int_{\R^d\setminus B_y} \theta \left(\frac{|x-y|}{|y-z|}\right) \frac{1}{|y-z|^n} \, \,d\mu(z)\\
			&=c_2 \, c_b  \sum_{k=1}^\infty \int_{B(y, 2^{k+1}e(x)) \setminus (y, 2^{k}e(x))} \theta \left(\frac{|x-y|}{|y-z|}\right) \frac{1}{|y-z|^n} \, \,d\mu(z)\\
			&\leq c_2 \, c_b \sum_{k=1}^\infty \theta(2^{-k}) \frac{\mu(B(y, 2^{k+1}e(x)) ) }{(2^k e(x))^n} \leq 2^{n}\, c_0 \, c_2\, c_b\, \mathfrak J_\theta(1).
		\end{split}
	\end{equation}
	
	To prove \eqref{eq:measure S minus H },  note that \( T^* \nu(y) > \alpha/2 \) for all \( y \in S \setminus H \) if \( \alpha \geq  2^{n+2}\,c_0\,  c_b\, (c_1+  c_2\, \mathfrak J_{\theta}(1)) \). Then, by Chebyshev's inequality,  
\[
\mu(S \setminus H) \leq 2 \int |T^* \nu| \, d\mu \leq \frac{2c_*}{\alpha}\, \mu(F).\qedhere
\]
\end{proof}

\vv

\begin{lemma}\label{lemma:lem_55_T}
	Assume that  $x\in \R^d$, $r_0>0$, $\delta>0$, and $\varepsilon>0$ are such that
	\[
	\mu(B(x,r))\leq c_0 \, r^n \qquad \text{ for } r\geq r_0
	\] 
	and 
	\[
	\bigl|T_{\varepsilon}\nu(x)\bigr|\leq \alpha \qquad \text{ for } \varepsilon\geq r_0.
	\]
	If $\Xi(x)\geq r_0$,  for all $\varepsilon>0$ we have that
\begin{align}\label{eq:4.24}
	\bigl|T_{\Xi,\varepsilon} \nu(x)\bigr|\leq \alpha + (c_1+ 2^{2n-1} ) \, c_b \, c_0.
\end{align}
\end{lemma}

\begin{proof}
The argument closely follows the proof of \cite[Lemma 5.5]{To14}, but we include it here for completeness.  
Applying Lemma \ref{lemma:estimates_suppressed_kernel_2}-(2), we obtain that if \( \varepsilon \geq \Xi(x) \) (and consequently \( \varepsilon \geq r_0 \)), then  
\begin{align*}
|T_{\Xi,\varepsilon} \nu (x)| &\leq | T_{\varepsilon} \nu (x)| + 2^{2n-1} \, \sup_{r \geq \varepsilon} \frac{|\nu|(B(x, r))}{r} \leq \alpha +  2^{2n-1} \, c_b \,c_0.
\end{align*}
In the case where \( \varepsilon < \Xi(x) \), we instead have
\begin{align}\label{eq:4.25}
|T_{\Xi, \varepsilon} \nu (x)| \leq c_b \int_{B(x,\Xi(x))} |k_{\Xi} (x, y)| \, d\mu(y) 
+ \Bigl| \int_{\R^d \setminus B(x,\Xi(x))} k_{\Xi} (x, y) \, d\nu(y) \Bigr|.
\end{align}
We can show that the first integral on the right-hand side is bounded by $c_1 \, c_b \, c_0$  using
\[
|k_{\Xi} (x, y)| \leq c_1 \, \Xi(x)^{-n} \quad \text{and} \quad \mu(B(x, \Xi(x))) \leq c_0\, \Xi(x)^n,
\]
where the latter holds due to the assumption \( \Xi(x) \geq r_0 \).   For the second integral in \eqref{eq:4.25}, note that it coincides with \( |T_{\Xi, \Xi(x)} \nu (x) | \), which was previously shown to be at most \( \alpha +  2^{2n-1} \, c_b c_0 \).  
Combining these estimates, we conclude that \eqref{eq:4.24} follows in this case as well.
\end{proof}

\vvv

Given $x\in \R^d$ we define
	\[
		\mathcal R(x)=\sup\bigl\{r>0: \mu(B(x,r))>c_0 \,r^n\bigr\},
	\]
	where we understand that $\mathcal R(x)=0$ if the set on the right hand side is empty, and we set
	\[
		H_1 \coloneqq \bigcup_{x\in \R^d, \,\mathcal R(x)>0} B\bigl(x,\mathcal R(x)\bigr) \subset H.
	\]
	
Let   \( \Xi \colon \mathbb{R}^{d} \to [0,+\infty) \) be a Lipschitz function with Lipschitz constant 1 such that  
\[
\Xi(x) \geq \dist(x, \mathbb{R}^{d} \setminus (H \cup S))
\]
for all \( x \in \mathbb{R}^{d} \).	In fact, we might choose \( \Xi(x) = \text{dist}(x, C \setminus (H \cup S)) \), for example.  Observe that \( H \cup S \) contains all non-Ahlfors balls and all the balls \( B(x, e(x)) \), where \( x \in F \), and so
\[
\Xi(x) \geq \max \{ R(x), e(x) \}.
\]

From the definition of \( S \) and the preceding lemma with \( r_0 = \max \{ \mathcal R(x), e(x)\}\),  we deduce the following lemma.
\begin{lemma}\label{theorem:suppressed_Tb}
Let   \( \Xi \colon \mathbb{R}^{d} \to [0,+\infty) \) be a Lipschitz function with Lipschitz constant 1 such that  
\[
\Xi(x) \geq \dist(x, \mathbb{R}^{d} \setminus (H \cup S))
\]
for all \( x \in \mathbb{R}^{d} \).  Then,  it holds that   
\[
 |T_{\Xi,*} \nu (x) |\leq\alpha + (c_1+ 2^{2n-1} ) \, c_b \, c_0 \qquad \text{ for } x \in F.   
\]
If there exists an operator $\widetilde T$ such that the operator $\Lambda\coloneqq T_\Xi -  \widetilde T_\Xi$ has kernel $\ell(\cdot,\cdot)$ which satisfies the bound $|\ell(x,y)| \leq c_\Lambda \frac{\varphi(|x-y|)}{|x-y|^n}$, then 
\[
 |\widetilde T_{\Xi,*} \nu (x) |\leq\alpha + (c_1+ 2^{2n-1} ) \, c_b \, c_0 + 2^{n+1}\,c_{\Lambda}  \, c_0 \, \mathfrak I_{\varphi}(R)  \qquad \text{ for } x \in F.   
\]
\end{lemma}

\vvv

\section{The big pieces ${Tb}$ theorem for ``good" SIOs}\label{big piecesTb} 
The goal of this section is to prove a $L^2$-boundedness criterion for the operator $ T_{\Xi, \mu}$, under proper assumptions on the function $\Xi$. A precise statement requires some additional notation so we defer it to the end of the next subsection, see Theorem \ref{theorem:suppressed_Tb}.

More specifically, we focus on the {\it big pieces} $Tb$ theorem for \textit{Singular Integral Operators} (SIOs) that satisfy the following assumptions.
\begin{definition}[``Good" SIOs]\label{eq:good_SIO_definition}
    Let $k\colon \R^d\times\R^d \to \R$, with $d \geq 2$, be a kernel satisfying the Calderón-Zygmund-type conditions \eqref{eq:CZ_1} and \eqref{eq:CZ_2} for some constants $R>0$,  $c_1, c_2 > 0$ and a function $\theta\colon [0,\infty) \to [0,\infty)$.   Assume that \( \mu \in M^n_+(\mathbb{R}^{d}) \) with growth constant \( c_0 > 0 \) and compact support such that $\diam(\supp(\mu))\leq R$ and  let $\Xi\colon \R^d\to \R$ be a non-negative $1$-Lipschitz function.  We say that the integral operator $T$ with kernel $k$ is a {\it ``good" singular integral operator (SIO)} if the following hold:
       \begin{enumerate} 
 \item  There exist a singular integral operator $\widetilde T$ with kernel $\tilde k$ satisfying \eqref{eq:CZ_1} and \eqref{eq:CZ_2}  for $\theta(t)=t$ and $R' \geq R$,  and an integral operator $K$ such that 
        \begin{equation}\label{eq:adjpoint of T}
   \widetilde{T}^*_{\Xi, \mu}      = -   \widetilde T_{\Xi, \mu}  + K_{\mu},
        \end{equation}
        where $T_{\Xi, \mu}$ is the operator associated with the suppressed kernel $k_\Xi$ (see \eqref{eq:suppressed_kernel}).
   \item  $K_{\mu}$  is bounded on $L^2(\mu)$ and there exists a constant \( c_{K} > 0 \) such that
        \begin{equation}\label{eq:L2 bound of K}
            \|K_{\mu}\|_{L^2(\mu) \to L^2(\mu)} \leq c_{K},
        \end{equation}
     where $c_K$ only depends on $R$ and $c_0$.
        \item If we set $\ell(x,y)\coloneqq k(x,y)-\tilde k(x,y)$,  there exists $\vphi\colon [0,\infty) \to [0,\infty)$ so that  $\mathfrak J_\varphi(1)<\infty$ and  it satisfies the bound 
              \begin{equation}\label{eq:boundedness_kernel_ell_good_SIO}
              	|\ell(x,y)| \leq c_\Lambda \frac{\varphi(|x-y|)}{|x-y|^n}, \qquad \text{ for all } x, y \in \R^d \text{ such that } 0<|x-y|<R.
              \end{equation}
    \end{enumerate}
\end{definition}

\vv
Let $\mathcal D_0$ be the lattice of (standard) dyadic cubes in $\R^d$ and we consider the translated dyadic lattice  
\[
\mathcal D(w) = w + \mathcal  D_0, \qquad \text{ for }w \in \mathbb{R}^{d}.
\]

	The main result of this section is the following ``big pieces" $Tb$ theorem. 
	\begin{theorem}\label{thm:big piecesTb-aux}
Let $T$ be a ``good" SIO and assume that $\mu$ is a finite Borel measure on $\R^d$, $d\geq 2$,  supported on a compact set $F \subset \R^d$ so that $\diam(\supp(\mu))\leq R$.  Suppose also that $\nu=b\, \mu$, for a function $b$ such that $\|b\|_{L^\infty(\mu)}\leq c_b$.  For every $w\in \R^d$,  let $\mathfrak T_{\caD(w)}\subset \R^d$ be the union of the maximal dyadic cubes in $\caD(w)$  for which 
\[
	|\nu(Q)| \leq c_{acc}\,  \mu(Q). 
\]
We are also given a measurable set $H  \subset \R^d$ satisfying the following properties.
\begin{itemize}
	\item 	If a ball $B(y,r)$ is such that $\mu(B(y,r))>c_0\, r^n$,  then $B(y,r)\subset H$.
	\item 	There exists $\delta_0>0$ such that $\mu(H \cup \mathfrak T_{\caD(w)})\leq \delta_0\, \mu(F)$, for all $w \in \R^d$.
	\item	We have that 
			\[
				\int_{\R^d \setminus H}  T_{ *}\nu\, d\mu\leq c_* \mu(F), \qquad \text{ for all }x\in \R^d,
			\]
			where
			\begin{equation}\label{eq:max_aux_truncated}
				T_{ *}\nu(x)=\sup_{\varepsilon>0}\bigl|T_{\nu, \varepsilon}1(x)\bigr|.
			\end{equation}
\end{itemize}
	There exist 
	$$G \subset F \setminus \bigcap_{w \in \R^d} \Big( H \cup \mathfrak T_{\mathcal D(w)} \Big)$$  and  $\hat  c>0$ depending on $n,$ $d,$ $c_0,$ $c_b,$ $c_*,$ $c_{acc},$ $c_K,$ $c_\Lambda,$ $\delta_0$ so that the following hold:
	\begin{enumerate}
	\item $\mu(G)\geq \hat c\, \mu(F)$.
		\item The measure $\mu|_G$ has $n$-growth with constant $c_0$, namely
		\begin{equation}\label{eq:growth_G_B-1}
			\mu(G \cap B(x,r))\leq c_0 \, r^n\qquad \qquad \text{ for }x\in \R^d,  r>0.
		\end{equation}
		\item $ T_{\mu}$ is bounded on $L^2(\mu|_G)$ with constant depending only on $n$, $c_0,$ $c_b,$ $c_*,$ $c_{acc},$ $c_K,$ $c_\Lambda,$  $\delta_0$.
	\end{enumerate} 
\end{theorem}
 We dedicate the remainder of this section to proving Theorem \ref{thm:big piecesTb-aux}. The proof originates from Nazarov, Treil, and Volberg's article \cite{NTrV99},  but we follow Tolsa's exposition in \cite[Chapter 5]{To14} and adopt its framework to facilitate comparison with his presentation.
In particular, we note that the \textit{suppressed $Tb$ theorem}, Theorem \ref{theorem:suppressed_Tb}, corresponds to \cite[Corollary 5.33]{To14}. The elaborate argument requires some adjustments; therefore, we only provide an outline and emphasize the necessary modifications, with special attention given to handling the lack of antisymmetry of $T_\mu$.

\vv
\subsection{The  suppressed $Tb$ theorem for ``good" SIOs}
We assume without loss of generality that $F\subset \tfrac{1}{8}[0,2^N]^{d}$ for some $N$ big enough. We also set 
\(
	\Omega=[-2^{N-4}, 2^{N-4}]^{d}
\)
and 
\[
	Q^0(w)\coloneqq w + [0,2^N]^{d} \subset \mathcal D (w), \qquad \text{ for } w\in \Omega.
\]
Notice that for each \( w \in \Omega \), we have  
\(
F \subset \frac{1}{4} Q_0(w).
\)
\vvv

We choose 
\[
	\alpha= \max \{4\,c_*/(1-\delta_0) ,2^{n+2}\,c_0\,  c_b\, (c_1+  c_2\, \mathfrak J_{\theta}(1))\}
\]
 in the definition of $S_0$ in \eqref{eq:S0} and so,  by Lemma \ref{lem:lemma_choice_alpha}, we deduce that
 \[
	\mu(H \cup \mathfrak T_{\mathcal D(w)}) +  \mu(S \setminus H) \leq \delta_1\coloneqq\frac{\delta_0+1}{2}, \quad \text{ for all } w \in \R^d.
 \]
We define
$W_{\caD (w)}\coloneqq H\cup S \cup T_{\caD(w)}$ and, for  $w\in \R^d$, 
\[
	\Phi_{\caD(w)}(x)\coloneqq   \dist\bigl(x,\R^d\setminus W_{\caD (w)} \bigr).
\]
It is easy to see that
	\(
		\Phi_{\caD(w)}(x)\geq \max\bigl(e(x), \mathcal R(x)\bigr),\) for $x\in \R^d$.
	
For a fixed \( \varepsilon_0 >  0 \), we also define  
\begin{equation}\label{eq:Phiepsilon}
\Phi_{\ve_0}(x) = \varepsilon_0 + \inf_{\substack{B \subset \Omega_2, \\ P^{\Omega^2}(B) = \beta}} \sup_{(w_1, w_2) \in B} \Phi_{(w_1, w_2)}(x)\eqqcolon \varepsilon_0 + \Phi(x),
\end{equation}
where  $W_{\mathcal D(w_1)}$ and $W_{\mathcal D(w_2)}$ are the total exceptional sets corresponding to two dyadic lattices $\mathcal D(w_1)$ and $\mathcal D(w_2)$, $P^\Omega$ is the uniform probability on $\Omega$, and
\[
\Phi_{(w_1, w_2)}(x) = \text{dist} \left( x, F \setminus \left( W_{\mathcal D(w_1)} \cup W_{\mathcal D(w_2)}\right) \right).
\]
Notice that \( \Phi_{\ve_0} \) is a 1-Lipschitz function such that \( \Phi_{\ve_0}(x) = \varepsilon_0 \) for all \( x \in G \). Moreover,  
\[
\Phi_{\ve_0}(x) \geq \max \left( \text{dist}(x, \mathbb{R}^{d} \setminus H), e(x), \varepsilon_0 \right) \geq \max \left( \mathcal R(x), e(x) \right).
\]

\vv

We have introduced all the necessary objects which are involved in the statement of the main result of this section.
\begin{theorem}\label{theorem:suppressed_Tb}
	Under  the assumptions of Theorem \ref{thm:big piecesTb-aux}, if $\Phi_{\ve_0}\colon \R^d\to \mathbb R$ is the function defined in \eqref{eq:Phiepsilon},  there exists $C_{\mathrm{op}}>0$ such that
	\begin{equation}\label{eq:bd_L2_norm_suppressed}
		\|T_{\Phi_{\ve_0}, \mu}\|_{L^2(\mu)\to L^2(\mu)}\leq C_{\mathrm{op}},
	\end{equation}
	where $C_{\mathrm{op}}$ depends on  $n,  d,  c_0, c_b, c_{acc}, c^*, \delta_0,  \,c_K\),  and  $c_\Lambda$, but not  on ${\ve_0}$.
\end{theorem}

\vv
\begin{remark}\label{rem:Tb}
Note that \eqref{eq:bd_L2_norm_suppressed} implies that, for any fixed $\ve>0$ and $\ve_0>0$,  
\begin{equation}\label{eq:bd_L2_norm_suppressed-bis}
\int \Big|  \int_{|x-y|>\ve}  k_{\Phi_{\ve_0} }(x,y) f(y) \,d\mu(y)   \Big|^2 \,d\mu(x)  \leq C_{\mathrm{op}}\, \|f\|^2_{L^2(\mu)}
	\end{equation}
with bounds uniform in $\ve_0$ and $\ve$.  If $f \in L^\infty_c(\mu)$,   it is easy to see that, by the kernel bounds of $k_{\Phi_{\ve_0} }(x,y)$ given in Lemma \ref{lemma:CZ_estimates_suppressed} and  dominated convergence, 
$$
\int_{|x-y|>\ve}  k_{\Phi_{\ve_0} }(x,y) f(y) \,d\mu(y) \overset{\ve_0 \to 0}{\longrightarrow} \int_{|x-y|>\ve}  k_{\Phi}(x,y) f(y) \,d\mu(y)
$$
Thus,  by another application of the dominated convergence theorem, we can deduce that 
\[
	\int \Big|  \int_{|x-y|>\ve}  k_{\Phi_{\ve_0} }(x,y) f(y) \,d\mu(y)   \Big|^2 \,d\mu(x) \overset{\ve_0 \to 0}{\longrightarrow}   \int \Big|  \int_{|x-y|>\ve}  k_{\Phi}(x,y) f(y) \,d\mu(y)   \Big|^2 \,d\mu(x),
\]
which, by \eqref{eq:bd_L2_norm_suppressed-bis} and density of $L^\infty_c(\mu)$ in $L^2(\mu)$,  yields that
	\begin{equation}\label{eq:bd_L2_norm_suppressed-0}
		\|T_{\Phi, \mu}\|_{L^2(\mu)\to L^2(\mu)}\leq C_{\mathrm{op}}.
	\end{equation}
\end{remark}

\vvv

\subsection{Martingale difference decomposition}

Given $Q\in \caD$ we denote as $\Ch(Q)$ the collection of its $2^{n+1}$ dyadic children.

\begin{definition}[``Terminal'' and ``transit'' cubes]
	A cube $Q\in \mathcal D$ contained in $Q^0_{\mathcal D}\eqqcolon Q^0_{\mathcal D}(y)$ with $\mu(Q)\neq 0$ is said to be \textit{terminal} if $Q\subset H \cup \mathfrak T_{\mathcal D}$ and \textit{transit} otherwise. We denote the families of terminal and transit cubes as $\mathcal D^{ter}$ and $\mathcal D^{tr}$, respectively.
\end{definition}

For $f\in L^1_{\loc}(\mu)$ and a cube $Q$ in $\R^d$ with $\mu(Q)\neq 0$, we denote
\begin{equation}\label{eq:def_Xi_f}
	\langle f\rangle_Q \coloneqq m_{\mu, Q}(f)=\frac{1}{\mu(Q)}\int_Q f\, d\mu \qquad \text{ and } \quad E f \coloneqq \frac{\langle f\rangle_{Q^0_{\caD}}}{\langle b\rangle_{Q^0_{\caD}}}\, b.
\end{equation}

For $Q\in \caD^{tr}$ and $f\in L^1_{\loc}(\mu)$ we define
\[
	\Delta_Q f(x)\coloneqq 
	\begin{cases}
		0 &\qquad\text{ for } x\in \R^d\setminus \bigcup_{P\in \Ch(Q)} P,  \\
		\Bigl(\frac{\langle f\rangle_P}{\langle b\rangle_P}-\frac{\langle f\rangle_Q}{\langle b\rangle_Q}\Bigr)\,b(x) &\qquad\text{ for } x\in P \text{ for some } P \in \Ch(Q)\cap \caD^{tr},\\
		f(x)-\frac{\langle f\rangle_Q}{\langle b\rangle_Q}\,b(x)&\qquad  \text{ for } x\in P \text{ for some } P \in \Ch(Q)\cap \caD^{ter}.
	\end{cases}
\]

In particular, see \cite[Lemma 5.10 and Lemma 5.11]{To14}, $\Delta_Q^2=\Delta_Q$, $\Delta_Q \Delta_R=0$ if $R\in \caD^{tr}$ with $R\neq Q$, and a function $f\in L^2(\mu)$ can be decomposed as
\begin{equation}\label{eq:dec_f_mart}
	f= E f + \sum_{Q\in \caD^{tr}} \Delta_Q f,
\end{equation}
where the sum is unconditionally convergent in $L^2(\mu)$. Furthermore, there exists a constant $\tilde c=\tilde c(c_b, c_{acc})>0$ such that 
\[
	\tilde c^{-1}\, \|f\|_{L^2(\mu)}\leq \|E f\|^2_{L^2(\mu)} + \sum_{Q\in \caD^{tr}}\|\Delta_Q f\|^2_{L^2(\mu)}\leq \tilde c \, \|f\|_{L^2(\mu)}.
\]

\vvv

Let $M>0$. We say that $V\subset \R^d$ has $M$-thin boundary if 
\[
	\mu\bigl(\{x\in\R^d:\dist(x,\partial V)\leq r\}\bigr)\leq M r\qquad \text{ for all } r>0.	
\]

The main step to prove the $L^2$-boundedness of $T_{\Phi, \mu}$ in Theorem \ref{theorem:suppressed_Tb} is to argue via duality and to test the operator on ``good'' functions.
\begin{definition}[``Good'' and ``bad'' cubes]\label{def:good_bad_cubes}
	Let $x_1, x_2\in \Omega$, $m$ be a positive integer, and $M>0$. We denote $\caD_j=\caD_j(x_j)$, for $j=1,2$. We say that $Q\in \caD_1^{tr}$ is \textit{bad with respect to $\caD_2$} if one of the following two properties holds:
	\begin{enumerate}
		\item There exists $R\in \caD^{tr}_2$ such that $\dist(Q, \partial R)\leq \ell(Q)^{1/4}\ell(R)^{{3/4}}$ and $\ell(R)\geq 2^m \ell(Q)$.
		\item There exists $R\in \caD^{tr}_2$ such that $2^{-m}\ell(Q)\leq \ell(R)\leq 2^m\ell(Q)$, $\dist(Q,R)\leq 2^m \ell(Q)$, and at least one of the children of $R$ does not have $M$-thin boundary.
	\end{enumerate}
	We say that $Q\in \caD_1^{tr}$ is \textit{good with respect to $\caD_2$} if it is not bad with respect to $\caD_2$.
\end{definition}

\vv
\begin{definition}[``Good'' and ``very good'' functions]
	Let $\caD_1$ and $\caD_2$ be as in Definition \ref{def:good_bad_cubes}. We say that $f\in L^2(\mu)$ is a \textit{$\caD_1$-good function with respect to $\caD_2$} if $\Delta_Q f=0$ for every cube $Q\in \caD_1^{tr}$ which is $\caD_1$-bad with respect to $\caD_2$.
	
	We say that $f$ is a \textit{very $\caD_1$-good function with respect to $\caD_2$} if it is a $\caD_1$-good function with respect to $\caD_2$ and $\Delta_Q f=0$ for all but a finite number of cubes $Q\in \caD^{tr}_1$.
\end{definition}

Let $\varepsilon_b\in (0,1)$. In analogy with \cite[Lemma 5.28]{To14}, we choose $m$ and $M$ big enough that, for any $Q\in \caD(x_1)$, it holds that
\[
	\mathcal L^{d}\Bigl(\bigl\{x_2\in \Omega: Q\in \caD(x_1) \text{ is bad with respect to }\caD(x_2)\bigr\}\Bigr)\leq \varepsilon_b\,  \mathcal L^{d}(\Omega).
\]
\vv

Most of the remaining part of the section consists in justifying the following lemma, that should be compared to \cite[Lemma 5.13]{To14}.
\begin{lemma}\label{lem:main_lemma_good_functions}
	Let $w_1, w_2\in \Omega$, and $\varepsilon>0$. Let $\caD_1=\caD(w_1)$ and $\caD_2=\caD(w_2)$ be two dyadic lattices, and assume that $\Xi\colon \R^d\to \mathbb R$ is a $1$-Lipschitz function such that $\Xi (x)\geq \varepsilon$ for all $x\in \R^d$ and 
	\[
		\Xi(x)\geq \max \bigl(\Phi_{\caD_1}(x), \Phi_{\caD_2}(x)\bigr),\qquad \text{ for all }x\in\R^d.
	\]
	
	If $f$ is $\caD_1$-good with respect to $\caD_2$ and $g$ is $\caD_2$-good with respect to $\caD_1$, then there is a positive constant $c=c(c_0,c_b, c_*, \delta_0, \varepsilon_b)$ such that
	\begin{equation}\label{eq:bound_supp_good_functions}
		\bigl|\langle T_{\Xi, \mu}f, g\rangle\bigr|\leq c\, \|f\|_{L^2(\mu)}\|g\|_{L^2(\mu)}.
	\end{equation}
	In particular, the value of $c$ does not depend on $\varepsilon$.
\end{lemma}
\vv

\subsection{The proof of Lemma \ref{lem:main_lemma_good_functions}}
It suffices to prove \eqref{eq:bound_supp_good_functions} under the further assumptions that $f$ is very $\caD_1$-good with respect to $\caD_2$ and that $g$ is very $\caD_2$-good with respect to $\caD_1$, in the sense of Definition \ref{def:good_bad_cubes}. This reduction can be done provided that the estimates do not depend on the number of non-zero terms $\Delta_Q f$ and $\Delta_R g$, for $Q\in \caD_1^{tr}$ and $R\in \caD_2^{tr}$. For more details, we refer to \cite[p. 154]{To14}.

We claim that \eqref{eq:bound_supp_good_functions} holds for $\widetilde T$, namely, under the assumptions of Lemma \ref{lem:main_lemma_good_functions}, we have
	\begin{equation}\label{eq:bound_supp_good_functions_tilde}
		\bigl|\langle \widetilde T_{\Xi, \mu}f, g\rangle\bigr|\lesssim c\, \|f\|_{L^2(\mu)}\|g\|_{L^2(\mu)}.
	\end{equation}
The bound \eqref{eq:bound_supp_good_functions_tilde} readily implies \eqref{eq:bound_supp_good_functions}: since $T$ is a good SIO, in the sense of Definition \ref{eq:good_SIO_definition}, if $\Lambda =T_{\Xi}- \widetilde T_{\Xi}$ denotes the operator with kernel $\ell(x,y)$, we obtain
\[
	\bigl|\langle  T_{\Xi, \mu}f, g\rangle\bigr|\leq \bigl|\langle \widetilde T_{\Xi, \mu}f, g\rangle\bigr| +\bigl|\langle \Lambda f, g\rangle\bigr|\lesssim \|f\|_{L^2(\mu)}\|g\|_{L^2(\mu)},
\]
where the $L^2(\mu)$-boundedness of $\Lambda$ follows from property \eqref{eq:boundedness_kernel_ell_good_SIO} and Lemma \ref{lem:lem_estim_dini_integr_growth}. Hence, we devote the rest of the subsection to the proof of \eqref{eq:bound_supp_good_functions_tilde}.

The use of the suppression function $\Xi$ in Lemma \ref{lem:main_lemma_good_functions} implies the $L^2(\mu)$-boundedness of $\widetilde T_{\Xi, \mu}$, with a norm that may depend on $\varepsilon$. Hence, the decomposition \eqref{eq:dec_f_mart} yields
\[
\widetilde	T_{\Xi, \mu}f = \widetilde T_{\Xi, \mu}(E f)+ \sum_{Q\in \caD_1^{tr}} \widetilde T_{\Xi, \mu}(\Delta_Q f),
\]
to be interpreted again it the $L^2(\mu)$-sense.

Thus,
\begin{equation}\label{eq:eq_split_main}
	\begin{split}
		\langle &\widetilde  T_{\Xi, \mu}f, g\rangle=
		 \langle  \widetilde T_{\Xi, \mu}(E f), g \rangle + \sum_{Q\in \caD_1^{tr}}\langle  \widetilde T_{\Xi, \mu}(\Delta_Q f), g\rangle	 \\
			&=\langle  \widetilde T_{\Xi, \mu}(E f),   g \rangle + \langle \widetilde T_{\Xi, \mu}( f),E g \rangle -\langle \widetilde T_{\Xi, \mu}(Ef), E g \rangle  + \sum_{Q\in \caD_1^{tr},\, R\in \caD_2^{tr}}\langle  \widetilde T_{\Xi, \mu}(\Delta_Q f), \Delta_R g\rangle\\
			&\eqqcolon \mathfrak A_1 + \mathfrak A_2 + \mathfrak A_3 + \mathfrak A_4.
	\end{split}
\end{equation}

The term $\mathfrak A_1$ can be estimated as for the Cauchy transform (see \cite[p.155,  (5.29)]{To14}), deducing
\begin{equation}\label{eq:bound_A_1}
	\mathfrak A_1 \lesssim \|f\|_{L^2(\mu)} \|g\|_{L^2(\mu)}.
\end{equation}

In light of \eqref{eq:adjpoint of T}, we can write
\[
	\mathfrak A_2= - \bigl\langle f, \widetilde T_{\Xi, \mu}(E g) \bigr\rangle +\bigl\langle E f, K_{\mu}(E g) \bigr\rangle,
\]
so \eqref{eq:bound_A_1} and \eqref{eq:L2 bound of K}  yield that
\[
	|\mathfrak A_2|\lesssim \|f\|_{L^2(\mu)} \|g\|_{L^2(\mu)}.
\]

\vvv
To estimate $\mathfrak A_3$, we observe that the  inclusion $\supp(\mu)\subset S^0\cap Q^0_{\caD_1}\cap Q^0_{\caD_2}$ entails
\[
	\langle b\rangle_{Q^0_{\caD_1}}=\langle b\rangle_{Q^0_{\caD_2}}=\langle b\rangle_{S^0},
\]
so \eqref{eq:def_Xi_f} and \eqref{eq:adjpoint of T} give
\begin{equation*}
	\begin{split}
		-\mathfrak A_3&=\frac{\langle f\rangle_{Q^0_{\caD_1}}\langle g\rangle_{Q^0_{\caD_2}}}{\langle b\rangle^2_{S^0}}\,\langle \widetilde T_{\Xi, \mu}b, b \rangle  =\frac{\langle f\rangle_{Q^0_{\caD_1}}\langle g\rangle_{Q^0_{\caD_2}}}{\langle b\rangle^2_{S^0}}\,\langle K_{ \mu} b, b \rangle 
	\end{split}
\end{equation*}
and

\begin{equation}
	|\mathfrak A_3|\lesssim \frac{\bigl|\langle f\rangle_{Q^0_{\caD_1}}\langle g\rangle_{Q^0_{\caD_2}}\bigr|}{\langle b\rangle^2_{S^0}}\,  \|b\|^2_{L^2(\mu)}\leq \|f\|_{L^2(\mu)} \|g\|_{L^2(\mu)}.
\end{equation}

\vvv
We are left with the estimate of $\mathfrak A_4$. We say that $Q\in \caD_1$ and $R\in \caD_2$ are \textit{distant} cubes if
\[
	\dist(Q,R)\geq \min \bigl(\ell(Q), \ell(R)\bigr)^{1/4}\cdot \max\bigl(\ell(Q), \ell(R)\bigr)^{3/4}.
\]
As in \cite[(5.32)]{To14}, we split the sum in the definition of $\mathfrak A_4$ as
\begin{equation}
	\begin{split}
		\mathfrak A_4 &= \sum_{Q\in \caD_1^{tr}, R\in \caD_2^{tr}}\langle \widetilde T_{\Xi, \mu}(\Delta_Q f), \Delta_R g\rangle\\
		&= \sum_{\substack{Q\in \caD_1^{tr}, R\in \caD_2^{tr},\\ Q,R \text{ distant }}} + \sum_{\substack{Q\in \caD_1^{tr}, R\in \caD_2^{tr},\\ Q\cap R = \varnothing\\ Q,R \text{ not distant }}} + \sum_{\substack{Q\in \caD_1^{tr}, R\in \caD_2^{tr},\\ Q\cap R \neq \varnothing\\ 2^{-m}\ell(R)\leq \ell(Q)\leq 2^m \ell(R)}} + \sum_{\substack{Q\in \caD_1^{tr}, R\in \caD_2^{tr},\\ Q\cap R \neq  \varnothing\\ \ell(Q)<2^{-m}\ell(R) \text{ or }\ell(Q)<2^{m}\ell(R)}}\\
		&\eqqcolon S_1 + S_2 + S_3 + S_4.
	\end{split}
\end{equation}

\subsubsection*{Estimate of $S_1+S_4$}
The triangle inequality gives
\begin{equation*}
	\begin{split}
		|S_1|&\leq \sum_{\substack{Q\in \caD_1^{tr}, R\in \caD_2^{tr},\\ Q,R \text{ distant }\\ \ell(Q)\leq \ell(R)}}  \bigl|\langle \widetilde  T_{\Xi, \mu}(\Delta_Q f), \Delta_R g\rangle\bigr| + \sum_{\substack{Q\in \caD_1^{tr}, R\in \caD_2^{tr},\\ Q,R \text{ distant }\\ \ell(R)\leq \ell(Q)}} \bigl|\langle  \widetilde T_{\Xi, \mu}(\Delta_Q f), \Delta_R g\rangle\bigr|\eqqcolon S_{1,1} + S_{1,2}.
	\end{split}
\end{equation*}

The estimate
\begin{equation}\label{eq:S_1_1}
	S_{1,1}\lesssim \|f\|_{L^2(\mu)} \|g\|_{L^2(\mu)}
\end{equation}
can be obtained via an argument identical to the Cauchy transform's. The term $S_{1,2}$ can be bounded observing that  \eqref{eq:L2 bound of K} yields
\begin{equation}\label{eq:estim_s12_s11}
	\begin{split}
		S_{1,2}&\leq S_{1,1} + \sum_{\substack{Q\in \caD_1^{tr}, R\in \caD_2^{tr},\\ Q,R \text{ distant }\\ \ell(Q)\leq \ell(R)}} \bigl|\langle K_{\mu}(\Delta_Q f), \Delta_R g\rangle\bigr|\\
		&\overset{\eqref{eq:S_1_1}}{\lesssim} \|f\|_{L^2(\mu)} \|g\|_{L^2(\mu)} + \sum_{\substack{Q\in \caD_1^{tr}, R\in \caD_2^{tr},\\ Q,R \text{ distant }\\ \ell(Q)\leq \ell(R)}} \|\Delta_Q f\|_{L^2(\mu)} \|\Delta_R g\|_{L^2(\mu)}.
	\end{split}
\end{equation}

Thus, by orthogonality, we have
\begin{equation*}
	\begin{split}
		\sum_{\substack{Q\in \caD_1^{tr}, R\in \caD_2^{tr},\\ Q,R \text{ distant }\\ \ell(Q)\leq \ell(R)}} \|\Delta_Q f\|_{L^2(\mu)} \|\Delta_R g\|_{L^2(\mu)}&\leq \sum_{Q\in \caD_1^{tr},\, R\in \caD_2^{tr}}\|\Delta_Q f\|_{L^2(\mu)} \|\Delta_R g\|_{L^2(\mu)}\leq \|f\|_{L^2(\mu)}\, \|g\|_{L^2(\mu)}
	\end{split}
\end{equation*}
which, together with \eqref{eq:estim_s12_s11}, implies that
\(
	S_{1,2}\lesssim \|f\|_{L^2(\mu)}\, \|g\|_{L^2(\mu)}.
\)
\vv

To bound $S_4$, we split the sum as
\begin{equation*}
	\begin{split}
		S_4&=\sum_{\substack{Q\in \caD_1^{tr}, R\in \caD_2^{tr},\\ Q\cap R \neq  \varnothing\\ \ell(Q)<2^{-m}\ell(R) \text{ or }\ell(Q)<2^{m}\ell(R)}}\langle  \widetilde T_{\Xi, \mu}(\Delta_Q f), \Delta_R g\rangle\\
			&=\sum_{\substack{Q\in \caD_1^{tr}, R\in \caD_2^{tr},\\ Q\cap R \neq  \varnothing \text{ and }  \ell(Q)<2^{-m}\ell(R) }}+ \sum_{\substack{Q\in \caD_1^{tr}, R\in \caD_2^{tr},\\ Q\cap R \neq  \varnothing \text{ and } \ell(R)<2^{m}\ell(Q)}}\eqqcolon S_{4,1} + S_{4,2}.
	\end{split}
\end{equation*}

The term $S_{4,1}$ can be analyzed as in \cite[Sections 5.8.2 and 5.8.3]{To14}. The first subsection applies directly to our setting without modification. The only adjustment needed in Section 5.8.3 is in Lemma 5.21, as the equality in (5.58) relies on the antisymmetry of the operator.

	The absence of this property for $\widetilde T_{\Xi, \mu}$ can be addressed by using \eqref{eq:adjpoint of T} and the identity $\Delta_Q = \Delta_Q^2$.  To see this, let $R\in \caD_2^{ter}$, denote as $\widehat R$ the dyadic parent of $R$, and by $W(R)$ a suitable Whitney decomposition of $R$ into dyadic cubes in $\caD_2$, see \cite[p.168]{To14}. For $R\in \caD_2^{ter}$ and $S\in W(R)$ we define $g_{R,S}\coloneqq \chi_{2S}\, \Delta_{\widehat R}g$.
	If $z_Q$ is the center of a cube $Q$, we have that
	\begin{equation}
		\begin{split}
			&\sum_{\substack{Q\in \caD_1^{tr},\\ \ell(Q)\leq 2^{-m}\ell(R)\\ Q\subset R, \,z_Q\in S}}\bigl|\bigl\langle \widetilde T_{\Xi, \mu}(\Delta_Q f), g_{R,S}\bigr\rangle\bigr|\\
			&\leq \sum_{\substack{Q\in \caD_1^{tr},\\ \ell(Q)\leq 2^{-m}\ell(R)\\ Q\subset R,\, z_Q\in S}}\bigl|\bigl\langle \Delta^2_Q f, \widetilde T_{\Xi, \mu}(g_{R,S})\bigr\rangle\bigr| + \sum_{\substack{Q\in \caD_1^{tr},\\ \ell(Q)\leq 2^{-m}\ell(R)\\ Q\subset R, \,z_Q\in S}}\bigl|\bigl\langle \Delta^2_Q f, K_{\mu}(g_{R,S})\bigr\rangle\bigr|\eqqcolon\mathcal S_T + \mathcal S_K.
		\end{split}
	\end{equation}
	
	The arguments in \cite[pp. 170-171]{To14} give
	\[
		\mathcal S_T \lesssim \Biggl(\sum_{\substack{Q\in \caD_1^{tr},\, Q\subset 2S}}\|\Delta_Q f\|^2_{L^2(\mu)}\Biggr)^{1/2}\, \|g_{R,S}\|_{L^2(\mu)}.
	\]
	To prove the analogous estimate for $\mathcal S_K$ it is enough to observe that we have 
	\begin{equation*}
		\begin{split}
			\mathcal S_K \leq \Biggl(\sum_{\substack{Q\in \caD_1^{tr},\, Q\subset 2S}}\|\Delta_Q f\|^2_{L^2(\mu)}\Biggr)^{1/2}\Biggl(\sum_{\substack{Q\in \caD_1^{tr}, \, Q\subset 2S}}\|\Delta^*_Q K_{\mu}(g_{R,S})\|^2_{L^2(\mu)}\Biggr)^{1/2}
		\end{split}
	\end{equation*}
	and that orthogonality and the $L^2(\mu)$-boundedness of $K$ imply that
	\begin{equation*}
		\begin{split}
			&\sum_{\substack{Q\in \caD_1^{tr},  \,Q\subset 2S}}\|\Delta^*_Q K_{\mu}(g_{R,S})\|^2_{L^2(\mu)} = \sum_{\substack{Q\in \caD_1^{tr},  \,Q\subset 2S}}\bigl\|\Delta^*_Q \bigl(\chi_{2S}K_{\mu}(g_{R,S})\bigr)\bigr\|^2_{L^2(\mu)}\\
			&\qquad\qquad\leq \|\chi_{2S}K_{\mu}(g_{R,S})\|^2_{L^2(\mu)}\leq \|K_{\mu}(g_{R,S})\|^2_{L^2(\mu)}\lesssim \|g_{R,S}\|^2_{L^2(\mu)},
		\end{split}
	\end{equation*}
	where in the last inequality we used \eqref{eq:L2 bound of K}.  The rest of the proof works unchanged. Hence, we can conclude that
\[
	S_{4,1} + S_{4,2}\lesssim \|f\|_{L^2(\mu)} \|g\|_{L^2(\mu)}.
\]
where the bound of $S_{4,2}$ follows using the same arguments as in \eqref{eq:estim_s12_s11}. 

\vv

We have motivated how to prove that
\[
	S_1 + S_4 \lesssim \|f\|_{L^2(\mu)} \|g\|_{L^2(\mu)}.
\]
\vv

\subsubsection*{Estimate of $S_2+S_3$} The calculations in \cite[Section 5.9]{To14} repeat nearly unchanged, except for Lemma 5.24, which requires additional considerations.

\begin{lemma}\label{lemma:lemma524}
	Let $Q\in \caD^{tr}_1$, $R\in \caD^{tr}_2$, and assume that every cube in $\Ch(Q)$ and $\Ch(R)$ has $M$-thin boundary. Then
	\begin{equation}\label{eq:lemma524_1}
		\bigl|\langle \widetilde T_{\Xi, \mu}(\Delta_Q f), \Delta_R g \rangle\bigr|\lesssim \|\Delta_Q f\|_{L^2(\mu)} \|\Delta_R g\|_{L^2(\mu)},
	\end{equation}
	so 
	\[
		S_2 + S_3 \lesssim \|f\|_{L^2(\mu)}\|g\|_{L^2(\mu)}.
	\]
\end{lemma}
\begin{proof}
	For $P\in \Ch(Q)$ and $S\in \Ch(R)$ we set
	\[
		f_{Q,R}\coloneqq \chi_P \Delta_Q f \qquad \text{ and }\qquad g_{R,S}\coloneqq \chi_S \Delta_R g,
	\]
	so that
	\begin{equation}\label{eq:lem524_dec}
		\begin{split}
			\bigl\langle \widetilde T_{\Xi, \mu}(\Delta_Q f), \Delta_R g \bigr\rangle = \sum_{P\in \Ch(Q), S\in \Ch(R)}\bigl\langle \widetilde T_{\Xi, \mu}f_{P,Q}, g_{R,S} \bigr\rangle.
		\end{split}
	\end{equation}

	The argument for the Cauchy transform applies verbatim to the estimate in \eqref{eq:lem524_dec} except for the final paragraph of \cite[p.175]{To14}, which relies on antisymmetry. However, this part can be easily adapted to the current case.
	Indeed, assume  that the cubes in $\Ch(Q)$ and $\Ch(R)$ have $M$-thin boundary and that $P\in \Ch(Q)$ and $S\in \Ch(S)$ are transit cubes. By definition of the martingale difference, there are $c_P, c_S\in \mathbb C$ such that 
	\[
		\Delta_Q f (x)= \biggl(\frac{\langle f\rangle_P}{\langle b\rangle_P }- \frac{\langle f\rangle_Q}{\langle b\rangle_Q}\biggr)\, b (x)\eqqcolon c_P\,  b(x) \qquad \text{ for } x\in P
	\]
	and
	\[
		\Delta_R g (x)= \biggl(\frac{\langle g\rangle_S}{\langle b\rangle_S }- \frac{\langle g\rangle_R}{\langle b\rangle_R}\biggr)\, b (x)\eqqcolon c_S\, b(x)\qquad \text{ for } x\in S.
	\]
	 
	Using \eqref{eq:adjpoint of T} once more,  we obtain 
	\begin{equation}\label{eq:lem_524_e1}
		\begin{split}
			\bigl\langle \widetilde T_{\Xi, \mu}(\chi_{P\cap S} \, f_{Q,P}), \chi_{P\cap S}\, g_{R,S}\bigr\rangle &= c_P \, c_S\, \bigl\langle  \widetilde T_{\Xi, \mu}b, b\bigr\rangle=\frac{c_P \,c_S}{2}\, \bigl\langle  K_{\mu} b, b\bigr\rangle= \frac{1}{2}\bigl\langle  K_{\mu}(\Delta_Q f), \Delta_Rg\bigr\rangle.
		\end{split}
	\end{equation}
	
	Thus, \eqref{eq:lem_524_e1} and \eqref{eq:L2 bound of K} yield
	\[
	 	\bigl\langle \widetilde T_{\Xi, \mu}(\chi_{P\cap S} \, f_{Q,P}), \chi_{P\cap S}\, g_{R,S}\bigr\rangle\lesssim \|\Delta_Q f\|_{L^2(\mu)} \|\Delta_Q g\|_{L^2(\mu)},
	\]
	which is the only modification needed in the proof of \eqref{eq:lemma524_1}.
\end{proof}

\vv

We have made all the necessary adjustments to prove \eqref{eq:bound_supp_good_functions_tilde}, and thus Lemma \ref{lem:main_lemma_good_functions}, by adapting the proof of Nazarov, Treil, and Volberg as presented in \cite{To14}. For the remainder of the argument and the derivation of Theorem \ref{theorem:suppressed_Tb} from Lemma \ref{lem:main_lemma_good_functions}, we refer to \cite[pp. 176-192]{To14}.

\vvv

We are now ready to  prove  Theorem \ref{thm:big piecesTb-aux}.

\vv

\begin{proof}[Proof of  Theorem \ref{thm:big piecesTb-aux}]
	By the arguments of \cite[p.186]{To14}, there exists a  set 
	\[
		G \subset F \setminus \bigcap_{w \in \R^d} \Big( H \cup \mathfrak T_{\mathcal D(w)} \Big)
	\] 
	so that  $\mu(G)\geq \hat c \,\mu(F)$, which satisfies the property \eqref{eq:growth_G_B-1} and such that $\Phi(x)=0$ for any $x\in G$.   Since,  by Lemma \ref{lemma:estimates_suppressed_kernel_2},  for any $x \in G$,  it holds that  $T_{\Phi, \mu}f(x) = T_{\mu} f(x)$,   by \eqref{eq:bd_L2_norm_suppressed-0},  we deduce that
	\[
		\|  T_{\mu} f \|_{L^2(\mu|_G)} \lesssim \|f\|_{L^2(\mu|_G)},
	\]
	concluding the proof of the theorem.
\end{proof}

\vv

{
\begin{remark}
	The construction of the set $G$ in Theorem \ref{thm:big piecesTb-aux} can be carried out so that it does not depend on $\delta$, provided that this parameter is sufficiently small.
		
		We refer to \cite[Section 5.11.2]{To14}, where $G$ is defined in display (5.84) as
		\[
			G\coloneqq \bigl\{x\in F: p_0(x)\geq (1-\delta_1)/2\bigr\},
		\]
		where $p_0(x)$ is defined in terms of a suitable probability, and $\delta_1$ is a constant that depends only on the parameter $\delta_0$.
		We note that $\delta_1$, which was defined in \cite[p.140]{To14} as $\delta_1=(1+\delta_0)/2$, can instead be chosen sufficiently close to $\delta_0$. 
		
		If we define $\alpha\coloneqq \max\bigl(16 c_0 c_b, (2c_*/(\delta_1-\delta_0))\bigr)$, the bound \cite[(5.6)]{To14} still holds, and we additionally have
		\[
			 C_{\delta_1}\mu(F)\leq \mu(G)\leq \mu(F),
		\]
		where $C_{\delta_1}\coloneqq (1-\delta_1)/(1+\delta_1)$. Thus,  if $\delta_0$ is small enough,  we can pick $\delta_1$ also such that $(1-C_{\delta_1})$ is as small as we want.
\end{remark}
}
\vv

{
	Finally, we present a result which is crucial for the application to free-boundary problems.
\begin{theorem}\label{theorem:thm_bd_op_T_G}
		Let $\mu$ be a Radon measure on $\Rn1$ with compact support, and let $E\subset\Rn1$ be a $\mu$-measurable set with $\mu(E)>0$ such that
		\[
			E\subset \bigl\{x\in \Rn1: M_R\mu(x)<\infty \quad \text{ and }\quad T_*\mu(x)<\infty\bigr\}.
		\]
		Then, there exists a Borel set $G\subset E$ with $\mu(G)>0$ such that $\mu|_G\in M^n(\Rn1)$ and the operator $T_{\mu|_G}$ is bounded on $L^2(\mu|_G)$.
\end{theorem}
\begin{proof}
For the proof, it suffices to repeat verbatim the arguments of \cite[Theorem 8.13]{To14}, along with the pointwise estimates for suppressed kernels in Section \ref{sec:aux_operator}.
\end{proof}
}

\vvv

\section{The gradient of the single layer potential is a ``good" SIO }\label{sec:gradient_SLP_good}

{
	
	\subsection{Pointwise estimates for auxiliary kernels}
	The article \cite{MMPT23} relies on a three-step perturbation argument for the gradient of fundamental solutions. In the following lemmas, we quote the estimate corresponding to the first step and later adapt the second and third ones to the method used in the present paper.

	\begin{lemma}\label{lem:main_pw_estimate}
		Let  $A$ be a uniformly elliptic matrix in $\Rn1$, $n\geq 2$, satisfying $A\in \widetilde \DMO_{n-1}$.
		For $R_0>0$,  there exists $C=C(n, \Lambda, R_0)>0$ such that, for $x,  y \in \Rn1$ such that $0<|x-y|< R<R_0 $, and
		\begin{align*}
			\mathcal K^1_{\Theta}(x,y) &\coloneqq \nabla_1\Gamma_A(x,y)- \nabla_1\Theta\bigl(x,y; \bar A_{x,|x-y|/2}\bigr)\\
			\mathcal K^{1,*}_{\Theta}(x,y) &\coloneqq \nabla_2\Gamma_{A^T}(x,y)- \nabla_2 \Theta\bigl(x,y; (\bar A_{x,|x-y|/2})^T\bigr),
		\end{align*}
		we have
		\begin{equation}\label{eq:estimate_fund_sol_average_old}
			|\mathcal K^1_\Theta(x,y)| + |\mathcal K^{1,*}_{\Theta}(x,y)| \leq C \, \frac{\tau_A(r)}{r^{n}} + C \frac{\widehat \tau_A(R)}{R^n}\qquad \text{ for } r\coloneqq|x-y|/2,
		\end{equation}
		where  
		\begin{equation}\label{eq:definition_tau_A_new}
			\tau_A(r) \coloneqq  \mathfrak I_{\oomega_A}(r) + \mathfrak L^n_{\oomega_A}(r)=\int_0^r \oomega_A(t)\, \frac{dt}{t} +  r^{n}\int_r^\infty \oomega_A(t)\, \frac{dt}{t^{n+1}}
		\end{equation}
		and
		\[
		\widehat \tau_A(R)= \mathfrak I_{\oomega_A}(R) + \mathfrak L^{n-1}_{\oomega_A}(R) =  \int_0^R\oomega_A(t)\, \frac{dt}{t} + R^{n-1}\int_R^\infty \oomega_A(t) \, \frac{dt}{t^{n}}.
		\]
		{ In particular, if $R_0=1$, the constant $C$ only depends on ellipticity, dimension, and the Dini Mean Oscillation condition.}
	\end{lemma}
	
	\begin{proof}
		The proof of \eqref{eq:estimate_fund_sol_average_old} for $\mathcal{K}^1_\Theta$ can be found in \cite[Lemma 3.12]{MMPT23}. A similar but simpler argument establishes \eqref{eq:estimate_fund_sol_average_old} for $\mathcal{K}^{1,*}_\Theta$. In particular, we refer to the proof of \eqref{eq:estimate_fund_sol_average} for ${\mathfrak K_A^*}$, where we outline the main differences between the proof of a similar but more difficult estimate for ${\mathfrak K_A}$ and ${\mathfrak K_A^*}$, and address the only non-trivial difficulty.  We omit the details.
	\end{proof}
	
	\vvv
	Let $A$ be a uniformly elliptic matrix in $\Rn1$, $n\geq 2$, satisfying $A\in  \DMO_s$.
	If $\mathcal A$ is the uniformly continuous representative that agrees $\mathcal L^{n+1}$-a.e. with  $A$ with modulus of continuity $\mathfrak I_{\oomega_A}(\cdot)$ (see \cite[Appendix A]{HwK20}),  then,  by the Lebesgue differentiation theorem,  it holds that 
	\begin{equation}\label{eq:limit_matrix}
		\begin{split}
			\lim_{\delta \to 0}  \bar A_{x,\delta}&=\lim_{\delta \to 0} \avint_{B(x,\delta)}  A(y)\,dy= \lim_{\delta \to 0} \avint_{B(x,\delta)} \mathcal A(y)\,dy = \mathcal A(x), \qquad \textup{for all}\,\, x\in \R^{n+1}.
		\end{split}
	\end{equation}
	We also observe that, if  $A \in \widetilde \DMO_{n-1}$, the matrix $\mathcal {A}$ is uniformly elliptic and $\mathcal {A} \in \widetilde{\DMO}_{n-1}$.
	
	The explicit expression \eqref{eq:fund_sol_const_matrix} allows us to obtain pointwise estimates for the auxiliary kernels associated with $\mathcal {A}$.
	
	\begin{lemma}\label{lem:bound_av_diff_point}
		Let $A$ be a uniformly elliptic matrix in $\Rn1$, $n\geq 2$, satisfying $A\in  \DMO_{s}$, and let $\mathcal A$ be its uniformly continuous representative. 
		Let $r>0$, and define
		\begin{align*}
			\mathcal K^2_{\Theta}(x,y) &\coloneqq \nabla_1 \Theta \bigl(x, y;\bar A_{x,r/2}\bigr)-\nabla_1 \Theta\bigl(x, y;\mathcal A(x) \bigr)\\
			\mathcal K^{2,*}_{\Theta}(x,y) &\coloneqq \nabla_2 \Theta \bigl(x,y; (\bar A_{x,r/2})^T\bigr)-\nabla_2 \Theta\bigl(x, y; (\mathcal A(x))^T \bigr).
		\end{align*}
		For all $z\in \mathbb R^{n+1}\setminus\{0\}$, it holds
		\begin{equation}\label{eq:estim_diff_kernels_av}
			\bigl|\mathcal K^2_{\Theta}(z,0)\bigr| + \bigl|\mathcal K^{2,*}_{\Theta}(z,0)\bigr|\lesssim_{n,\Lambda} \frac{1}{|z|^{n}} \,\mathfrak I_{\oomega_A}(r).
		\end{equation}
		Moreover,  there exists $C=C(n,\Lambda)>0$ such that, for every $x, x'\in \Rn1$, we have
		\begin{equation}\label{eq:bound_av_diff_point}
			\bigl|\nabla_1\Theta\bigl(z,0;\mathcal A(x)\bigr)- \nabla_1\Theta\bigl(z,0;\mathcal A(x') \bigr)\bigr|\leq C \, \frac{\mathfrak I_{\oomega_A}(|x-x'|)}{|z|^n}, \qquad z\in\Rn1\setminus \{0\}.
		\end{equation}
	\end{lemma}

	\begin{proof}
		Since $\mathcal A$ is the uniformly continuous representative of $A$ with modulus of continuity $\mathfrak I_{\oomega_A}(\cdot)$ and we have $\bar A_{x,r}=\bar {\mathcal A}_{x,r}$ for every $x\in\Rn1$ and $r>0$, we obtain
		\begin{equation}\label{eq:bound_diff_A_delta_hat}
			\begin{split}
				\bigl|\bar A_{x,r} - \mathcal A(x)\bigr| &=\bigl|\bar {\mathcal A}_{x,r} - \mathcal A(x)\bigr|=\biggl|\avint_{B(x,r)}\bigl(\mathcal A(z)- \mathcal A(x)\bigr)\, dz\biggr|\leq  \avint_{B(x,r)} \mathfrak I_{\oomega_{ A}}(|z-x|)\, dz \leq \mathfrak I_{\oomega_{A}}(r).
			\end{split}
		\end{equation}
		In order to prove \eqref{eq:estim_diff_kernels_av} it is enough to repeat verbatim the proof of (3.69) in \cite[Lemma 3.13]{MMPT23} with \eqref{eq:bound_diff_A_delta_hat} replacing \cite[(3.68)]{MMPT23}.
		The same proof leads to \eqref{eq:bound_av_diff_point}, replacing \cite[(3.68)]{MMPT23} with the uniform continuity of $\mathcal {A}$, and, in particular, with
		\begin{equation*}
			\bigl|\mathcal {A}(x')-\mathcal {A}(x)\bigr|\leq \mathfrak{I}_{\oomega_A}(|x-x'|).\qedhere
		\end{equation*}
	\end{proof}
	
	\vv
	A variation of the same proof allows us to obtain a pointwise bound for the kernel associated with the next step of the perturbation argument.

	\begin{lemma}\label{lem:lem_2_prep_spherical_harmonics_new}
		Let $A$ be a uniformly elliptic matrix in $\Rn1$, $n\geq 2$, satisfying $A\in  \DMO_{s}$, and let $\mathcal  A$ be its uniformly continuous representative.  Assume that  $B \subset \Rn1$ is a ball.
		If we denote  $\bar A_{B}\coloneqq\avint_{B} A$ then, if $x \in B$,
		it holds
		\begin{equation}\label{eq:estim_diff_averages2_new}
			\bigl|\bar A_{B} - \mathcal A(x)\bigr|\lesssim \mathfrak{I}_{\oomega_A}(r(B))\leq \tau_A(r(B)).
		\end{equation}
		Furthermore, if we define
		\begin{align}\label{eq:definition_K_3}
			\mathcal K^3_\Theta(z,0)\coloneqq \nabla_1 \Theta\bigl(z,0;\mathcal A(x)\bigr)- \nabla_1 \Theta(z,0;\bar A_{B}),
		\end{align}
		for all $z\in \mathbb R^{n+1}\setminus\{0\}$, we have
		\begin{equation}\label{eq:estim_diff_kernels2_new}
			\bigl|\mathcal K^3_\Theta(z,0)\bigr|\lesssim_{n,\Lambda} \frac{1}{|z|^{n}} \mathfrak{I}_{\oomega_A}(r(B)).
		\end{equation}
	\end{lemma}
	
	\begin{proof}
		Let $\delta \in \bigl(0, \sqrt{n+1}\ell(Q)\bigr)$. The triangle inequality yields
		\begin{equation}\label{eq:bound_diff_prep_lemma}
			\bigl|\bar A_{B} - \mathcal A(x)\bigr|\leq \bigl|\bar A_{B} - \bar A_{x,\delta}\bigr| +\bigl|\bar A_{x,\delta} - \mathcal A(x)\bigr|\leq \mathfrak I_{\oomega_A}(r(B)) + \mathfrak I_{\oomega_A}(\delta),
		\end{equation}
		where the latter bound follows from \cite[(3.74)]{MMPT23} together with the properties of $\mathcal  A$.
		Hence, passing to the limit in \eqref{eq:bound_diff_prep_lemma} as $\delta\to 0$, we obtain \eqref{eq:estim_diff_averages2_new}.
		Using the pointwise bounds of Lemma \ref{lem:bound_av_diff_point}, the proof is analogous to \cite[Lemma 3.13]{MMPT23}; see also \cite[Lemma 3.14]{MMPT23}.
	\end{proof}
	
	\vv

	\begin{lemma}\label{lem:truncatedL2}
		Let $A$ be a uniformly elliptic matrix in $\Rn1$, $n\geq 2$, satisfying  $A\in \widetilde \DMO_{n-1}$.  Let $Q$ be a ball in $\Rn1$ with center $x_B$ and side-length $r(B) \lesssim 1$,  and let $\mu$  be a non-negative Radon measure such that  $\supp (\mu) \subset B$ and 
		\[
		\Theta^n_\mu(B(x,r)) \leq \mathfrak C_0 \,\Theta^n_\mu(B), \qquad  \text{ for every } x \in \supp(\mu) \textup{ and } r \in (0, r(B)).
		\]  
		Then $ T^{j}_{\mu}\colon L^2(\mu) \to L^2(\mu)$, for   $j=1,2$,  satisfies
		\begin{align*}
			\bigl\|  T^{1}_{\mu}  \bigr\|_{{L^2(\mu)}\to{L^2(\mu)} }& \lesssim \mathfrak C_0 \,\Theta^n_\mu(B) \, \bigl[ \mathfrak I_{\tau_A}(r(B) ) +\widehat \tau_A(r(B) ) \bigr], \\
			\bigl\|  T^{2}_{\mu}  \bigr\|_{{{L^2(\mu)}\to{L^2(\mu)} }} 	& \lesssim \mathfrak C_0 \,\Theta^n_\mu(B) \,  \mathfrak I_{\tau_A}(r(B)),
		\end{align*}
		where the implicit constants depend only on ellipticity, dimension, and the $\DMO$ condition\footnote{In fact, $A\in  \DMO_{s}$ is enough to obtain $ T^{2}_{\mu}\colon L^2(\mu) \to L^2(\mu)$.}.
	\end{lemma}
	\vv
	
	\begin{lemma}\label{lem:estimate_norm_K3_new}
		Let $A$ be a uniformly elliptic matrix in $\Rn1$, $n\geq 2$, satisfying  $A\in  \DMO_{s}$.
		Let $B$ be a ball in $\Rn1$ with center $x_B$ and side-length $r(B) \lesssim 1$,
		and let $\mu$  be a non-negative Radon measure such that  $\supp (\mu) \subset B$.
		For $\mathcal K^3_\Theta(x,z)$ given by \eqref{eq:definition_K_3}, we define 
		\[
		T^3_{\mu,\delta}f(x)\coloneqq \int_{|x-y|>\delta} \mathcal K^3_\Theta(x,x-y)f(y)\, d\mu(y),\qquad f\in L^1_{\loc}(\mu).
		\]
		Then, there exists a positive constant $C''=C''(n,\Lambda)$ such that
		\begin{align}\label{eq:main_lemma_estim2_old}
			\|  T^{3}_{\mu}  \|_{L^2(\mu)\to L^2(\mu)}&\leq  C'' \mathfrak I_{\oomega_A}(r(B))^{1/2}\|\mathcal R_{\mu}\|_{L^2(\mu)\to L^2(\mu)}.
			\end{align}
	If, in addition, it holds that
				\begin{equation}\label{eq:growth_mu_k3}
			\Theta^n_\mu(B(x,r)) \leq \mathfrak C_0 \,\Theta^n_\mu(B), \qquad  \text{ for every } x \in \supp(\mu) \textup{ and } r \in (0,  2 r(B)),
		\end{equation}
		then, there exist constants $c_s>0$, depending on $\mathfrak C_0$,  and $\tilde c_{s}>0$,  depending on $C''$,  such that 
			\begin{align}
			\| T^3_{\mu,*} \|_{L^2(\mu)\to L^2(\mu)}&\leq  c_s  \Theta^n_\mu(B) + \tilde c_s \mathfrak I_{\oomega_A}(r(B))^{1/2}\|\mathcal R_{\mu}\|_{L^2(\mu)\to L^2(\mu)}.\label{eq:main_lemma_estim2_new}
		\end{align}				
	\end{lemma}
	
	\begin{proof}
		The proof of \eqref{eq:main_lemma_estim2_old} is  the same as that of \cite[Lemma 3.16]{MMPT23}, while \eqref{eq:main_lemma_estim2_new} follows from \eqref{eq:main_lemma_estim2_old}, \eqref{eq:growth_mu_k3}, and Cotlar's inequality (see \cite[Theorem 2.21]{To14}). We omit the details. 
	\end{proof}
	
	\vvv

\subsection{The gradient of the single layer potential is ``good''}

\begin{lemma}\label{lem:pw_bounds_kernel_t}
	Let $A$ be a uniformly elliptic matrix in $\Rn1$, $n\geq 2$, satisfying $A\in  \DMO_{s}$, and let $\mathcal A$ be its uniformly continuous representative. We define
	\begin{equation}\label{eq:def_t_delta}
		t(x,y)\coloneqq \nabla_1 \Theta\bigl(x-y, 0; \mathcal A(x)\bigr)\, \qquad \text{ for } x,y\in \Rn1, x\neq y.
	\end{equation}
	There exists $C=C(n,\Lambda)>0$ such that
	\begin{align}
		|t(x,y)|    &\leq C|x-y|^{-n},\qquad\text{for} \,\,x,y\in \Rn1,  x\neq y; \label{eq:t-kernel-point}\\
		|t(y,x)-t(y,x')|&\leq C \frac{|x-x'|}{\max(|x-y|,|x'-y|)^{n+1}}, 
		\quad \text{and} \label{eq:t-kernel-cont-1}
		\\
		|t(x,y)-t(x',y)| &\leq C   \frac{|x-x'|}{\max(|x-y|,|x'-y|)^{n+1}} + C  \frac{\mathfrak I_{\oomega_A}(|x-x'|)}{\max(|x-y|,|x'-y|)^{n}}, \label{eq:t-kernel-cont-2}
	\end{align}
	for $x,  x', y \in \Rn1$ satisfying $2|x-x'|\leq \max(|x-y|,|x'-y|)$.
\end{lemma}
\begin{proof}
	The estimate \eqref{eq:t-kernel-point} readily follows from \eqref{eq:gradient_fund_sol_const_matrix} and the fact that the uniform ellipticity of $\mathcal A(x)$ is bounded uniformly on $x$.  In order to prove \eqref{eq:t-kernel-cont-2}, we write
	\begin{equation*}
		\begin{split}
			|t(x,y)-t(x',y)|&\leq \bigl|t(x,y) - \nabla_1\Theta\bigl(x-y, 0;\mathcal A(x')\bigr) \bigr|\\
			&\qquad \quad + \bigl|\nabla_1\Theta\bigl(x-y, 0; \mathcal A(x')\bigr)-t(x',y)\bigr|\eqqcolon I + II.
		\end{split}
	\end{equation*}
	By \eqref{eq:bound_av_diff_point} and the identity \eqref{eq:gradient_fund_sol_const_matrix},  we have that
	\[
	I + II\ \lesssim \frac{\mathfrak I_{\oomega_A}(|x-x'|)}{|x-y|^{n}} + \frac{|x-x'|}{\max(|x-y|,|x'-y|)^{n+1}},
	\]
	where the implicit constant depends on $n$ and the uniform ellipticity of $A$.	By analogous (but simpler) estimates, we can prove \eqref{eq:t-kernel-cont-1}, since the matrix  $\mathcal A(y)$ does not depend on $x$ and $x'$,  concluding the proof of the lemma.
\end{proof} 

\vvv

Given $\varepsilon >0$ and a Radon measure $\mu$ on $\Rn1$ we define the operators
\begin{equation}\label{eq:def_operator_T_delta}
	\widetilde{\mathcal T}_{\mu, \varepsilon} f(x)\coloneqq  \int_{|x-y|>\varepsilon} t(x,y)f (y)\,  d\mu(y), \qquad \text{ for } f\in L^1_{\loc}(\mu),
\end{equation}
and we denote by $\widetilde{\mathcal T}_{\Xi, \mu}$ the associated suppressed operator with kernel $t_\Xi$, see \eqref{eq:suppressed_kernel}.  It is easy to see that its adjoint operator is given by
\begin{equation}\label{eq:def_operator_T_delta-adjoint}
\widetilde{\mathcal T}^*_{\mu, \varepsilon} f(x)\coloneqq  \int_{|x-y|>\varepsilon} \nabla_2\Theta\bigl(x,y; (\mathcal{A}(y))^T\bigr) \, f (y)\,  d\mu(y), \qquad \text{ for } f\in L^1_{\loc}(\mu).
\end{equation}

\vvv

\begin{lemma}\label{lemma:quasi_antisymm_suppressed}
	Let $A$ be a uniformly elliptic matrix in $\Rn1$, $n\geq 2$, satisfying $A\in \DMO_{s}$, $\delta>0$, and $\mu \in M^n_+(\R^{n+1})$ with constant $c_0$ and $\diam(\supp(\mu))\leq R$. Then there is a bounded operator $ K_\mu\colon L^2(\mu)\to L^2(\mu)$ such that, for $f\in L^2(\mu)$, we have
	\begin{equation}\label{eq:quasi_antisymm_suppressed}
	\widetilde{\mathcal T}^{*}_{\Xi, \mu}f	= -\widetilde{\mathcal T}_{\Xi, \mu}f +  K_\mu f.
	\end{equation}
	More specifically, for any $f\in L^2(\mu)$, we have
	\begin{equation}
		\| K_\mu\|_{L^2(\mu)}\lesssim c_0\, \mathfrak I_{\oomega_A}(R)\, \|f\|_{L^2(\mu)}.
	\end{equation}
\end{lemma}
\begin{proof}
	Let $t^*(x,y)\coloneqq \nabla_2\Theta(x,y;(\mathcal A(y))^T)$ be the kernel of the operator $\widetilde{\mathcal T}^*_{\mu}$, and let $K_\mu$ be the integral operator associated with the kernel $\tilde k(x,y)\coloneqq t_\Xi(x,y)+t^*_\Xi(x,y)$.
	In particular, we have that
	\begin{equation*}\label{eq:quasi_antisym_sup1}
		\begin{split}
			|\tilde k(x,y)|&=\bigl|t_{\Xi}(x,y)+ t^*_{\Xi}(x,y)\bigr|\overset{\eqref{eq:suppressed_kernel}}{=}\biggl|\frac{t(x,y)+ t^*(x,y)}{1+ |x-y|^{-2n} \Xi(x)^n \Xi(y)^n} \biggr|\\
			&  \leq \bigl|t(x,y)+ t^*(x,y)\bigr|=\bigl|\nabla_1 \Theta(x,y; \mathcal A(x))+\nabla_2\Theta(y,x;\mathcal A(y)^T)\bigr|\\
			&=\bigl|\nabla_1 \Theta(x,y; \mathcal A(x))-\nabla_1\Theta\bigl(x,y;\mathcal A(y)\bigr)\bigr|\lesssim \frac{\mathfrak I_{\oomega_A}(|x-y|)}{|x-y|^n},
		\end{split}
	\end{equation*}
	where in the last line we use \eqref{eq:gradient_fund_sol_const_matrix-1} and \eqref{eq:bound_av_diff_point}.
	Therefore, by Lemma \ref{lem:lem_estim_dini_integr_growth},  we deduce that $\| K_\mu \|_{L^2(\mu) \to L^2(\mu)} \lesssim c_0\, \mathfrak I_{\oomega_A}(R)$, concluding the proof.
\end{proof} 
\vv

In the following lemma, we introduce the auxiliary operators that allow us to prove that $T_\mu$ is a ``good'' SIO, in the sense of Definition \ref{eq:good_SIO_definition}.
\begin{lemma}\label{lem:aux_lambda_operators}
	Let $A$ be a uniformly elliptic matrix in $\Rn1$, $n\geq 2$, satisfying $A\in \widetilde \DMO_{n-1}$, $\delta>0$, and $\mu \in M^n_+(\Rn1)$ with growth constant $c_0$ and compact support. Furthermore, for  $f\in L^1_{\loc}(\mu)$, we define
	\[
		\Lambda_\mu f \coloneqq {\mathcal T}_{\Xi, \mu}f- \widetilde{\mathcal T}_{\Xi, \mu}f  \qquad \text{ and }\qquad \Lambda^*_\mu f \coloneqq {\mathcal T}^{*}_{\Xi, \mu}f -\widetilde{\mathcal T}^*_{\Xi, \mu}f.
	\]
	The operators $\Lambda_\mu$ and $\Lambda^*_\mu$ have kernels $K^\Lambda_\Xi$ and $K^{\Lambda^*}_\Xi$, respectively, such that, for $x,y\in\Rn1$ with $0<|x-y|<R$,
	\begin{equation}\label{eq:pw_bound_kernels_lambda_lambda_star}
			|K^\Lambda_\Xi(x,y)|\lesssim \frac{\tau_A(|x-y|)}{|x-y|^{n}} +  \frac{\widehat \tau_A(R)}{R^n}\quad \text{ and }\quad \bigl|K^{\Lambda^*}_\Xi(x,y)\bigr|\lesssim \frac{\tau_A(|x-y|)}{|x-y|^{n}} +  \frac{\widehat \tau_A(R)}{R^n}.
	\end{equation}
	Moreover,  if  $\diam (\supp(\mu)) \leq R$, then,  for any $f\in L^2(\mu)$,
	\begin{equation}
		\|\Lambda_\mu f\|_{L^2(\mu)}+ \|\Lambda^*_\mu f\|_{L^2(\mu)}\lesssim c_0\, \bigl(\mathfrak F_{\oomega_A}(R) + \mathfrak L^{n-1}_{\oomega_A}(R) \bigr)\|f\|_{L^2(\mu)}. 
	\end{equation}
\end{lemma}

\begin{proof}
			Let $\mathcal K^1_{\Theta}, \mathcal K^2_{\Theta}, \mathcal K^{1,*}_{\Theta}$ and $\mathcal K^{2,*}_{\Theta}$ be as defined in Lemmas \ref{lem:main_pw_estimate} and  \ref{lem:bound_av_diff_point}, for $r=|x-y|$. Hence, for $K^{\Lambda}\coloneqq \mathcal K^1_{\Theta}+\mathcal K^2_{\Theta}$, the operator $\Lambda_\sigma$ has kernel $K^{\Lambda}_\Xi$ and, for any $x,y\in\Rn1$ with $0<|x-y|<R$,
			\[
				|K^{\Lambda}_\Xi(x,y)|\leq |K^{\Lambda}(x,y)|\leq |\mathcal K^1_{\Theta}(x,y)|+|\mathcal K^2_{\Theta}(x,y)|\lesssim \frac{\tau_A(|x-y|)}{|x-y|^{n}} + C \frac{\widehat \tau_A(R)}{R^n},
			\]
			where the latter bound holds because of Lemmas \ref{lem:main_pw_estimate} and  \ref{lem:bound_av_diff_point}. 
			By Lemma \ref{lem:truncatedL2}, for $f\in L^2(\mu)$, we have that $\|\Lambda_\sigma f\|_{L^2(\mu)}\lesssim c_0 \bigl(\mathfrak I_{\tau_A}(R) + \hat \tau_A(R)\bigr)\|f\|_{L^2(\mu)}$, where we recall that 
			\[
				\mathfrak I_{\tau_A}(R)= \mathfrak I_{\mathfrak I_{\oomega_A}}(R)+ \mathfrak I_{\mathfrak L^{n}_{\oomega_A}}(R)\overset{\eqref{eq:split_I_L_theta_merged}}{=}  \mathfrak I_{\mathfrak I_{\oomega_A}}(R) +\frac{1}{n}\bigl(\mathfrak I_{\oomega_A}(R) + \mathfrak L^{n}_{\oomega_A}(R)\bigr).
			\]
			Thus, since $\mathfrak L^{n}_{\oomega_A}(R) \leq \mathfrak L^{n-1}_{\oomega_A}(R)$, 
			\[
				\|\Lambda_\mu f\|_{L^2(\mu)}\lesssim c_0\,  \bigl(\mathfrak F_{\oomega_A}(R) + \mathfrak L^{n-1}_{\oomega_A}(R)\bigr)\|f\|_{L^2(\mu)}.
			\]
					
		To prove the $L^2(\mu)$-boundedness of the $\Lambda^*_\mu$, we first observe that it is the operator associated with the suppression of the kernel $K^{\Lambda^*}(x,y)\coloneqq \nabla_2 \Gamma_{A^T}(x,y)-\nabla_2\Theta\bigl(x,y; \mathcal {A}(y)^T\bigr)$. 
		Thus, via triangle inequality we obtain
		\begin{equation}\label{eq:pw_bound_perturbation_Lambda_2_1}
			\begin{split}
				\bigl|K^{\Lambda^*}(x,y)&\bigr|\leq \bigl|\nabla_2 \Gamma_{A^T}(x,y)-\nabla_2\Theta\bigl(x,y; \mathcal {A}(x)^T\bigr)\bigr| + \bigl|\nabla_2\Theta\bigl(x,y; \mathcal {A}(x)^T\bigr)-\nabla_2\Theta\bigl(x,y; \mathcal{A}(y)^T\bigr)\bigr|\\
				&\leq |\mathcal K^{1,*}_{\Theta}(x,y)|+ |\mathcal K^{2,*}_{\Theta}(x,y)| + \bigl|\nabla_2\Theta\bigl(x,y; \mathcal{A}(x)^T\bigr)-\nabla_2\Theta\bigl(x,y; \mathcal {A}(y)^T\bigr)\bigr|.
			\end{split}
		\end{equation}
		By \eqref{eq:bound_av_diff_point},  we have that
		\begin{equation}\label{eq:pw_bound_perturbation_Lambda_2}
			\bigl|\nabla_2\Theta\bigl(x,y; \mathcal {A}(x)^T\bigr)-\nabla_2\Theta\bigl(x,y; \mathcal {A}(y)^T\bigr)\bigr|\lesssim \frac{\mathfrak I_{\oomega_A}(|x-y|)}{|x-y|^n},
		\end{equation}
		which, together with  Lemmas \ref{lem:main_pw_estimate} and  \ref{lem:bound_av_diff_point}, yields the second bound in \eqref{eq:pw_bound_kernels_lambda_lambda_star}.
		Finally, gathering \eqref{eq:pw_bound_perturbation_Lambda_2_1}, \eqref{eq:pw_bound_perturbation_Lambda_2}, and Lemmas \ref{lem:lem_estim_dini_integr_growth},  we obtain
		\[
			\|\Lambda^*_\mu f\|_{L^2(\mu)}\lesssim c_0 \, \bigl(\mathfrak F_{\oomega_A}(R) + \mathfrak L^{n-1}_{\oomega_A}(R)\bigr)\|f\|_{L^2(\mu)},
		\]
		which concludes the proof of the lemma.
\end{proof}

\vv

\section{Gradient of the  single layer potential and  quantitative rectifiability}\label{sec:mean_oscillation_perturbation}

Let $A$ be a uniformly elliptic matrix in $\Rn1$, $n\geq 2$, satisfying $A\in \widetilde \DMO_{n-1}$ and let $\mathcal A$ be its uniformly continuous representative.  For a ball \( B \subset \mathbb{R}^{n+1} \), we define the kernels  
\begin{equation*}
	\begin{split}
		\mathfrak{K}(x,y) &\coloneqq \nabla_1 \Gamma_A(x,y) - \nabla_1 \Theta(x,y; \bar{A}_B), \\  
		\mathfrak{K}^*(x,y) &\coloneqq \nabla_2 \Gamma_{A^T}(x,y) - \nabla_2 \Theta\bigl(x,y; \bar{A}^T_B\bigr).
	\end{split}
\end{equation*}  

For a non-negative Radon measure \( \mu \) in \( \mathbb{R}^{n+1} \), we define the operator  
\[
\mathcal{S}_\mu f(x) = \int \mathfrak{K}(x,y) f(y)\, d\mu(y), \qquad \text{for } f \in L^1_{\mathrm{loc}}(\mu),
\]  
and its adjoint  
\(
\mathcal{S}^*_\mu \vec{g} \coloneqq \int \mathfrak{K}^*(x,y) \cdot \vec{g}(y)\, d\mu(y),
\)  
for \( \vec{g} \in L^1_{\mathrm{loc}}(\mu; \mathbb{R}^{n+1}) \).

\vv

Let $\mathcal{K}^i_\Theta(x,y)$ and $\mathcal{K}^{i,*}_\Theta(x,y)$, for $i=1,2,3$, denote the kernels defined in Lemma \ref{lem:main_pw_estimate}, Lemma \ref{lem:bound_av_diff_point}, and Lemma \ref{lem:lem_2_prep_spherical_harmonics_new}, respectively, with the choice $r = |x - y|$.
 In particular, we have
\begin{align}\label{eq:k=k*}
	\mathfrak K(x,y)&=\mathcal K^1_\Theta(x,y) + \mathcal K^2_\Theta(x,y) + \mathcal K^3_\Theta(x,y),\\
	\mathfrak K^*(x,y)&=\mathcal K^{1,*}_\Theta(x,y) + \mathcal K^{2,*}_\Theta(x,y) - \mathcal K^{3}_\Theta(x,y).\label{eq:k=k*_2}
\end{align}

For $\tau_A$ and $\widehat \tau_A$ as in Lemma \ref{lem:main_pw_estimate}, we define
\begin{equation}\label{eq:definition_alpha_modulus}
	\alpha(t) \coloneqq \tau_A(t) + \widehat{\tau}_A(r(B)), \qquad \text{for } t > 0,
\end{equation} 
and observe that
\begin{equation}\label{eq:comparability_alpha_tau}
	\widehat \tau_A(r(B))\leq \alpha(r(B))\leq 2 \, \widehat \tau_A(r(B)).
\end{equation}
Then, for \( R_0 > 0 \),  
\eqref{eq:estimate_fund_sol_average_old}, \eqref{eq:estim_diff_kernels_av}, \eqref{eq:estim_diff_kernels2_new}, \eqref{eq:k=k*}, and \eqref{eq:k=k*_2}  
imply the pointwise bound  
\begin{equation}\label{eq:pointwise_bound_kernel_frak_K}
	|\mathfrak{K}(x,y)| + |\mathfrak{K}^*(x,y)| \lesssim_{n,\Lambda,R} \frac{\alpha(|x - y|)}{|x - y|^n},  
	\qquad \text{for } 0 < |x - y| < R.
\end{equation}  

If we further assume that \( \supp(\mu) \subset B \) and that there exists \( \mathfrak{C}_0 > 0 \) such that  
\[
\Theta^n_\mu(B(x,r)) \leq \mathfrak{C}_0 \,\Theta^n_\mu(B),  
\qquad \text{for every } x \in B \text{ and } r \in (0, r(B)),
\]  
then Lemmas \ref{lem:truncatedL2} and \ref{lem:estimate_norm_K3_new} yield  
\begin{equation}\label{eq:L2Smu}
	\begin{split}
		\| \mathcal{S}_\mu \|_{L^2(\mu) \to L^2(\mu)} +	\| \mathcal{S}^*_\mu \|_{L^2(\mu) \to L^2(\mu)} \lesssim_{n, \Lambda, R, \mathfrak{C}_0}& \mathfrak{I}_{\oomega_A}(r(B))^{1/2} \, \|\mathcal{R}_{\mu}\|_{L^2(\mu) \to L^2(\mu)}\\
		& + \Theta^n_\mu(B) \bigl[ \mathfrak{I}_{\tau_A}(r(B)) + \widehat{\tau}_A(r(B)) \bigr].  
	\end{split}
\end{equation}

\vvv

\subsection{The H\"ormander smoothness  condition for the perturbation kernel}

\vv

\begin{lemma}\label{lemma:hormander_perturbation}
	Let  $A$ be a uniformly elliptic matrix in $\Rn1$, $n\geq 2$, satisfying $A\in  \widetilde \DMO_{1-\gamma}$, for some $\gamma \in (0,1)$.  If $B$ is a ball in $\Rn1$,  we set $\bar A_B=\avint_B A$,
	\[
	\mathfrak K_A(x,y) \coloneqq \nabla_1\Gamma_A(x,y)- \nabla_1\Theta\bigl(x,y; \bar A_B\bigr), \qquad \qquad \text{ for }x\neq y,
	\]
	and
	\[
	{\mathfrak K_A^*}(x,y) \coloneqq \nabla_1\Gamma_A(y,x)- \nabla_1\Theta\bigl(y,x; \bar A_B\bigr), \qquad \qquad \text{ for }x\neq y.
	\]
	Let $\mu$ be a non-negative Radon measure whose support is contained in a compact set and let $B \subset \Rn1$ be a ball centered at $y_B \in \supp(\mu)$ with radius $r(B)$.  Let us also assume that there exist $\gamma\in (0,1)$ and $c_{db}>0$ such that $B$ is $c_{db}$-$P_{\gamma,\mu}$-doubling.
	Then,  there exists $C_\gamma=C_\gamma(n, \Lambda, \gamma, c_{db})>0$ such that, for any $y\in B$, it holds that
	\begin{align}\label{eq:estimate_fund_sol_average}
		\int_{\R^{n+1} \setminus 2B}	\Bigl(\bigl| \mathfrak K_A(x,y)- \mathfrak K_A(x,y_B)\bigr| &+\bigl| \mathfrak K^*_A(x,y)- \mathfrak K^*_A(x,y_B)\bigr| \Bigr)\, d\mu(x)\\
		& \leq C_\gamma\, \Theta^n_\mu(B)\bigl(\mathfrak I_{\oomega_A}(r(B))+ \mathfrak L^{1-\gamma}_{\oomega_A}(r(B))\bigr).\notag
	\end{align}
\end{lemma}

\begin{proof}
	We assume  that $\bigl(\Rn1 \setminus 2B \bigr)\cap \supp (\mu) \neq \varnothing$  since, otherwise,  \eqref{eq:estimate_fund_sol_average} trivially holds. 
	If we set $\Theta(z,y)\coloneqq \Theta(z,y;\bar A_B)$, then,  if we identify $\bar{\mathcal A}_B$ with $\bar A_B$, by Lemma \ref{lem:MMPT_lemma_310},  we can write
	\begin{equation}\label{eq:lem_pointwise_111}
		\begin{split}
			& \mathfrak K_A(\cdot, y)- \mathfrak K_A(\cdot, y_B) =\int \nabla_1\nabla_2\Gamma(\cdot ,z)\bigl(\bar A_B-  A(z)\bigr)\Bigl(\nabla_1\Theta(z,y)- \nabla_1\Theta(z,y_B)\Bigr)\, dz\\
			&\qquad \qquad = \int \nabla_1\nabla_2\Gamma(\cdot ,z)\bigl(\bar{\mathcal A}_B -  \mathcal A(z)\bigr)\Bigl(\nabla_1\Theta(z,y)- \nabla_1\Theta(z,y_B)\Bigr)\, dz\\
			&\qquad \qquad  \eqqcolon \int \Phi(\cdot,y, y_B,z)\, dz = \int_B \Phi(\cdot,y, y_B,z)\, dz + \sum_{j\geq 0}\int_{2^{j+1}B\setminus 2^j B}\Phi(\cdot,y, y_B,z)\, dz \\
			&\qquad \qquad \eqqcolon I(\cdot)+\sum_{j\geq 0} I_j(\cdot).
		\end{split}
	\end{equation}
	
	By Lemma \ref{lem:estim_fund_sol}, for $z\neq y$ and $z\neq y_B$, we get
	\begin{equation}\label{eq:cz_diff_Thetas}
		\bigl|\nabla_1\Theta(z,y)-\nabla_1\Theta(z,y_B)\bigr|\leq\bigl|\nabla_1\Theta(z,y)\bigl|+\bigr|\nabla_1\Theta(z,y_B)\bigr|\lesssim \frac{1}{|y-z|^n} + \frac{1}{|y_B-z|^n}.
	\end{equation}
	For $x\in \Rn1\setminus 2B$ we use again Lemma \ref{lem:estim_fund_sol}, \eqref{eq:cz_diff_Thetas}, and  obtain
	\begin{equation*}
		\begin{split}			
			|I(x)|&\lesssim \int_{B}\bigl| \bar {\mathcal A}_B -  \mathcal A(z) \bigr|\frac{\bigl|\nabla_1\Theta(z,y)- \nabla_1\Theta(z,y')|}{|x-z|^{n+1}}\, dz\\
			&\lesssim \int_{B}\frac{\bigl|\bar {\mathcal A}_B -  \mathcal A(z)\bigr|}{|x-z|^{n+1}|y_B-z|^n}\, dz +\int_{B}\frac{\bigl|\bar {\mathcal A}_B -  \mathcal A(z)\bigr|}{|x-z|^{n+1}|y-z|^n}\, dz \eqqcolon J_{1}  + J_{2}.
		\end{split}
	\end{equation*}
	
	If we set 
	\begin{equation}\label{eq:annulus}
		\mathcal C_k(B)\coloneqq 2^{k+1}B \setminus 2^{k} B,\qquad \textup{for} \,\,k\in \mathbb Z,
	\end{equation}
	and $N\geq 2$ is such that $x \in  \mathcal C_N(B)$,  then
	\begin{equation}
		\begin{split}
			\bigl(2^N r(B)\bigr)^{n+1}	J_1&\lesssim \int_{B}\frac{\bigl|\bar{\mathcal A}_B- \mathcal A(z)\bigr|}{|y_B-z|^n}\, dz
			=\sum_{j=1}^{\infty}\int_{\mathcal C_{-j}(B)}\frac{\bigl|\bar{\mathcal A}_B- \mathcal A(z)\bigr|}{|y_B-z|^n}\, dz.
		\end{split}
	\end{equation}
	
	Thus, by \eqref{eq:pw_diff_average_unif_cont} and \eqref{eq:omega<dini}, we obtain
	\begin{equation}\label{eq:J_1_estimate}
		\begin{split}
			J_1&\lesssim \frac{r(B)}{(2^Nr(B))^{n+1}}\sum_{j\geq 0}{2^{-j}}\, \mathfrak I_{\oomega_A}(r(B))\approx \frac{r(B)}{(2^Nr(B))^{n+1}}\, \mathfrak I_{\oomega_A}(r(B)). 
		\end{split}
	\end{equation}

	Let us now analyze $J_2$.	
	Let $\widetilde B\coloneqq B(y, 2r(B))$. In particular, $B\subset \widetilde B$, $r(\widetilde B) \approx r(B)$, and
	\begin{equation}\label{eq:I_r2_estimate}
		\begin{split}
			J_2&\lesssim \frac{1}{(2^N r(B))^{n+1}}\int_{B}\frac{\bigl|\bar{\mathcal A}_B- \mathcal A(z)\bigr|}{|y-z|^n}\, dz\\
			&\lesssim \frac{1}{(2^N r(B))^{n+1}}\int_{\widetilde B}\Biggl(\frac{\bigl|\bar{\mathcal A}_{\widetilde B}- \mathcal A(z)\bigr|}{|y-z|^n} + \frac{\bigl|\bar{\mathcal A}_B- \bar{\mathcal A}_{\widetilde B}\bigr|}{|y-z|^n}\Biggr)\, dz.
		\end{split}
	\end{equation}
	
	Following the proof of \eqref{eq:J_1_estimate},  by \eqref{eq:pw_diff_average_unif_cont} and the doubling property of $\mathfrak I_{\oomega_A}$,  we get that
	\begin{equation}\label{eq:J_21_estimate}
		\begin{split}
			&\int_{\widetilde B}\frac{\bigl|\bar{\mathcal A}_{\widetilde B}- \mathcal A(z)\bigr|}{|y-z|^n} \, dz \lesssim {r(B)}\, \mathfrak I_{\oomega_A}(r(\widetilde B))\lesssim {r(B)}\, \mathfrak I_{\oomega_A}(r(B)).
		\end{split}
	\end{equation}
	
	Moreover, we have
	\begin{equation*}
		\begin{split}
			\bigl|\bar{\mathcal A}_B-\bar{\mathcal A}_{\widetilde B}\bigr|&\leq \avint_{B}|\mathcal A(z)-\bar{\mathcal A}_{\widetilde B}|\, dz \lesssim \avint_{\widetilde B}\bigl|\mathcal A(z)-\bar{\mathcal A}_{\widetilde B}\bigr|\, dz\lesssim \mathfrak I_{\oomega_A}(r(\widetilde B))\lesssim \mathfrak I_{\oomega_A}(r(B)).
		\end{split}
	\end{equation*}
	Thus
	\begin{equation*}
		\begin{split}
			&\int_{\widetilde B}\frac{\bigl|\bar{\mathcal A}_B-\bar{\mathcal A}_{\widetilde B}\bigr|}{|y-z|^n} \, dz\lesssim {\mathfrak I_{\oomega_A}(r(B))}\, \int_{\widetilde B}\frac{1}{|y-z|^n} \, dz\lesssim {r(B)}\, \mathfrak I_{\oomega_A}(r(B))
		\end{split}
	\end{equation*}
	which, together with \eqref{eq:I_r2_estimate} and \eqref{eq:J_21_estimate}, implies that
	\begin{equation}\label{eq:final_estimate_J_2}
		J_2\lesssim \frac{r(B)}{(2^N r(B))^{n+1}}\, \mathfrak I_{\oomega_A}(r(B)).
	\end{equation}
	
	We gather \eqref{eq:J_1_estimate}, \eqref{eq:final_estimate_J_2}, and obtain 
	\begin{equation}
		|I(x)|\lesssim \frac{r(B)}{(2^N r(B))^{n+1}}\, \mathfrak I_{\oomega_A}(r(B)).
	\end{equation}

	Since $P_{1, \mu}(B) \leq P_{\gamma,\mu}(B)$, we obtain
	\begin{equation}\label{eq:OD_bound_I}
		\begin{split}
			\int_{\Rn1\setminus 2B}|I(x)|\, d\mu(x)&= \sum_{N\geq 2}\int_{\mathcal C_N(B)}|I(x)|\, d\mu(x)\lesssim \sum_{N\geq 0}  \frac{r(B)}{(2^N r(B))^{n+1}}\, \mathfrak I_{\oomega_A}(r(B))\, \mu(2^NB)\\
			&=\mathfrak I_{\oomega_A}(r(B))\sum_{N\geq 0} 2^{-N}\Theta^n_\mu(2^NB) 
			= P_{1,\mu}(B) \mathfrak I_{\oomega_A}(r(B))\\
			&\leq P_{\gamma,\mu}(B) \mathfrak I_{\oomega_A}(r(B))\lesssim \Theta^n_\mu(B) \mathfrak I_{\oomega_A}(r(B)),
		\end{split}
	\end{equation}
	where in the last inequality we used the fact that $B$ is $P_{\gamma,\mu}$-doubling.
	
	The estimate of the sum on the right-hand side of \eqref{eq:lem_pointwise_111} is in turn more delicate. We denote by $\mathfrak j_0$ the only integer such that
	\[
	2^{\mathfrak j_0}r(B)\leq R\coloneqq \diam (\supp (\mu)) <2^{\mathfrak j_0 +1}r(B),
	\]
	and write
	\begin{equation}\label{eq:master_ineq_sums}
		\begin{split}
			\sum_{j\geq 0}\int_{\Rn1\setminus 2B}&|I_j(x)|\, d\mu(x)\lesssim \sum_{N=1}^{\mathfrak j_0}\sum_{j\geq 0}\int_{\mathcal C_N(B)}|I_j(x)|\, d\mu(x)\\
			&\leq \sum_{N=1}^{\mathfrak j_0}\sum^{N-2}_{j=0} + \sum_{N=1}^{\mathfrak j_0}\sum^{N+1}_{j=N-1} + \sum_{N=1}^{\mathfrak j_0}\sum^{\mathfrak j_0-3}_{j=N+2} + \sum_{N=1}^{\mathfrak j_0}\sum_{j\geq \mathfrak j_0-2}\\
			&\eqqcolon \mathfrak S_1 + \mathfrak S_2 + \mathfrak S_3 + \mathfrak S_4.
		\end{split}
	\end{equation}
	
	We study each term $ \mathfrak S_1, \ldots,  \mathfrak S_4$ separately.
	\vvv
	
	\textbf{Case 1. } Let us consider $\mathfrak S_1$. Let $N\geq 1$.
	If $0\leq j \leq  N-2$, for $w\in \Rn1$ we define
	\begin{equation}\label{eq:definition_v_j}
		v_j(w)\coloneqq \int_{\mathcal C_j(B)}\nabla_2\Gamma(w,z)\bigl(\mathcal A(z)-\bar{\mathcal A}_B\bigr)\bigl[\nabla_1\Theta(z,y)-\nabla_1\Theta(z,y_B)\bigr]\, dz.
	\end{equation}
	
	For $z\in \mathcal C_j(B)$ and $w\in \mathcal C_N(B)$,  it holds that $|w-z|\approx 2^{N}r(B)$ and $|z-y_B|\approx 2^jr(B)$, so by Lemma \ref{lem:estim_fund_sol}, we have
	\begin{equation}\label{eq:pw_bound_v_j_1}
		\begin{split}
			|v_j(w)|&\lesssim \int_{\mathcal C_j(Q)}\frac{\bigl|\mathcal A(z)-\bar{\mathcal A}_B\bigr|}{|w-z|^n}\frac{|y-y_B|}{|z-y_B|^{n+1}}\, dz\\
			&\lesssim \frac{r(B)}{(2^Nr(B))^{n}(2^jr(B))^{n+1}}\, \int_{\mathcal C_j(B)} \bigl|\mathcal A(z)-\bar{\mathcal A}_B\bigr|\, dz\\
			&\lesssim  \frac{r(B)}{(2^Nr(B))^{n}}\, \avint_{2^{j+1} B}\bigl|\mathcal A(z)-\bar{\mathcal A}_B\bigr|\, dz\overset{\eqref{eq:pw_diff_average_unif_cont}}{\lesssim} \frac{r(B)}{(2^Nr(B))^{n}}\mathfrak I_{\oomega_A}(2^jr(B)).
		\end{split}
	\end{equation}
	
	Let $x\in \mathcal C_N(B)$. Since $v_j$ is a weak solution to $L_Av_j=0$ in $2^jB$ for all $j\geq 2$, by \eqref{eq:Linftyest-r} and Caccioppoli's inequality we obtain
	\begin{equation}\label{eq:pw_bound_v_j_2}
		\begin{split}
			|\nabla v_j(x)|&\leq \sup_{w\in B(x, 2^{N-6}r(B))}|\nabla v_j(w)|\leq \biggl(\avint_{B(x, 2^{N-5}r(B))}|\nabla v_j|^2\biggr)^{1/2}\\
			&\lesssim \frac{1}{2^N r(B)}\biggl(\avint_{B(x, 2^{N-4}r(B))}| v_j|^2\biggr)^{1/2}{\lesssim} \frac{r(B)}{(2^Nr(B)))^{n+1}}\mathfrak I_{\oomega_A}(2^jr(B)),
		\end{split}
	\end{equation}
	where the last inequality follows from \eqref{eq:pw_bound_v_j_1}.
	The inequality \eqref{eq:pw_bound_v_j_2} implies
	\begin{equation}\label{eq:OD_bound_S_1}
		\begin{split}
			\mathfrak S_1&\leq\sum_{N=1}^{\mathfrak j_0} \sum_{j=0}^{N-2}\mu(\mathcal  C_N(B))|\nabla v_j(x)|\lesssim \sum_{N=1}^{\mathfrak j_0} \sum_{j=0}^{N-2}\frac{\mu(\mathcal C_N(B))\, r(B)}{(2^Nr(B))^{n+1}}\mathfrak I_{\oomega_A}(2^jr(B)).
		\end{split}
	\end{equation}
	
	By hypothesis, the ball $B$ is $P_{\gamma, \mu}(B)$ doubling for $\gamma \in (0,1)$ as in the statement,  so
	\begin{equation}\label{eq:P_doubling_S_1}
		\begin{split}
			\sum_{N=1}^\infty 2^{-N\gamma}\frac{\mu(\mathcal C_N(B))}{(2^Nr(B))^n}&\leq \sum_{N\geq 0} 2^{-N\gamma}\Theta^n_\mu(2^NB)=P_{\gamma,\mu}(B)\lesssim \Theta^n_\mu(B).
		\end{split}
	\end{equation}
	Hence, changing the order of summation in \eqref{eq:OD_bound_S_1}, we obtain
	\begin{equation}\label{eq:final_estimate_S_1}
		\begin{split}
			\mathfrak S_1&\lesssim \sum_{j=0}^{\mathfrak j_0-2}\sum_{N=j+2}^{\mathfrak j_0}\frac{\mu(\mathcal C_N(B))}{(2^Nr(B))^{n}}\frac{1}{2^N}\mathfrak I_{\oomega_A}\bigl(2^jr(B)\bigr)=\sum_{j=0}^{\mathfrak j_0-2}\sum_{N=j+2}^{\mathfrak j_0}\frac{\mu(\mathcal C_N(B))}{(2^Nr(B))^{n}}\frac{2^{-N\gamma}}{2^{N(1-\gamma)}}\mathfrak  I_{\oomega_A}\bigl(2^jr(B)\bigr)\\
			&\overset{\eqref{eq:P_doubling_S_1}}{\lesssim} \Theta^n_\mu(B)\sum_{j=0}^{\mathfrak j_0-2} \frac{r(B)^{1-\gamma}}{(2^jr(B))^{1-\gamma}}\mathfrak I_{\oomega_A}\bigl(2^jr(B)\bigr)\overset{\eqref{eq:mod_cont_sum_2}}{\lesssim} \Theta^n_\mu(B) \, \mathfrak L^{1-\gamma}_{\mathfrak I_{\oomega_A}}(r(B)).
		\end{split}
	\end{equation}
	
	\textbf{Case 2.}
	Let us study $\mathfrak S_2$.  We assume that $N\geq 1$ and $x\in \mathcal C_N(B)$. Hence, if we set $\wt {\mathcal C}_N(B)\coloneqq\bigcup_{j=N-1}^{N+1}\mathcal C_j(B)$,  it holds that $B(x, 2^{N-4}r(B)) \subset \wt {\mathcal C}_N(B)$.  Thus, we write
	\begin{equation}\label{eq:split_j_1_j_2}
		\begin{split}
			\wt I_N(x)&\coloneqq 	\sum_{j=N-1}^{N+1} I_j(x)=\int_{\wt {\mathcal C}_N(B)} \nabla_1\nabla_2\Gamma(x ,z)\bigl(\bar{\mathcal A}_B-  \mathcal A(z)\bigr)\Bigl(\nabla_1\Theta(z,y)- \nabla_1\Theta(z,y_B)\Bigr)\, dz\\
			&=\int_{B(x, 2^{N-4}r(B))} + \int_{\wt {\mathcal C}_N(B)\setminus  B(x, 2^{N-4}r(B))} \eqqcolon \widetilde{\mathcal J}_{N,1} (x)+ \widetilde{\mathcal J}_{N,2}(x).
		\end{split}
	\end{equation}
	
	To estimate $\widetilde{\mathcal J}_{N,1}(x)$, we define
	\[
	\varepsilon_{x,B}(z)=\bigl(\bar{\mathcal A}_B - \mathcal A(z)\bigr)\chi_{B(x,2^{N-4}r(B))}(z),\qquad \qquad \text{ for }z\in \Rn1,
	\]
	so that 
	\[
	\widetilde{\mathcal J}_{N,1}(x)= \int_{\wt {\mathcal C}_N(B)} \nabla_1\nabla_2\Gamma(x ,z)\, \varepsilon_{x,B}(z)\, \Bigl(\nabla_1\Theta(z,y)- \nabla_1\Theta(z,y_B)\Bigr)\, dz.
	\]
	
	Let $\eta\in (0,1/2)$ be a parameter chosen as in Remark \ref{remark:rem_grad_bound}, $K_\eta>0$ be the only integer such that $2^{-K_\eta-1}\leq \eta<2^{-K_\eta}$, $K\coloneqq 3(4/3)^{K_\eta}$, $r=2^{N-4}r(B)$, and $\rho\coloneqq 2^{N-3}r(B)/(K+1)$. 
	We remark  that $\rho\approx 2^{N}r(B)$. 
	
	Let $t\in (0,1]$ and $w\in B(x,t\rho)$. We define $F(z)\coloneqq \nabla_1\Theta(z,y)-\nabla_1\Theta(z,y_B)$ and remark that Lemma \ref{lem:estim_fund_sol} yields that 
	\begin{equation}\label{eq:Fbound}
		|F(z)|\lesssim r(B)\bigl(2^Nr(B)\bigr)^{-(n+1)}, \qquad  \textup{for all} \,\,z\in B(w,t\rho).
	\end{equation}
	Hence,  as $B(w, t \rho) \subset B(x, 2^{N-4}r(B))$ for any $w\in B(x,t\rho)$ and $ t \in (0,1]$,  we have that 
	\begin{equation}\label{eq:mrg_omega_1}
		\begin{split}
			&\avint_{B(w,t\rho)}\Bigl|(\mathcal A-\bar{\mathcal A}_B)F-\avint_{B(w,t\rho)}(\mathcal A-\bar{\mathcal A}_B) F\, \Bigr|\\
			&\qquad =\avint_{B(w,t\rho)}\bigl|\mathcal AF - \bar{\mathcal A}_B F - \overline{(\mathcal AF)}_{B(w,t\rho)} +  \bar{\mathcal A}_B\bar F_{B(w,t\rho)}\bigr|\\
			&\qquad =\avint_{B(w,t\rho)}\bigl| \left(\mathcal AF -\bar{\mathcal A}_{B(w,t\rho)}F \right) + \bigl(\bar{\mathcal A}_{B(w,t\rho)}-\bar{\mathcal A}_B\bigr)\bigl(F-\bar F_{B(w,t\rho)}\big)\\
			&\qquad\qquad\qquad\qquad\qquad+ \bar{\mathcal A}_{B(w,t\rho)} \bar F_{B(w,t\rho)} - \overline{(\mathcal AF)}_{B(w,t\rho)} \bigr|\\
			&\qquad \leq 2\avint_{B(w,t\rho)} \bigl|\bigl(\mathcal A-\bar{\mathcal A}_{B(w,t\rho)}\bigr)F \bigr| + \avint_{B(w, t\rho)}\Bigl|\bigl(\bar{\mathcal A}_{B(w,t\rho)}-\bar{\mathcal A}_B\bigr)\bigl(F-\bar F_{B(w,t\rho)}\big)\Bigr|\\
			&\qquad\lesssim \frac{r(B)}{(2^Nr(B))^{n+1}}\, \oomega_A(t\rho)+ \bigl|\bar{\mathcal  A}_{B(w,t\rho)}-\bar{\mathcal A}_B\bigr|\, \avint_{B(w,t\rho)}\bigl|F(z)-\bar F_{B(w,t\rho)}\bigr|\, dz\\
			&\qquad\overset{\eqref{eq:pw_diff_average_unif_cont}}{\lesssim} \frac{r(B)}{(2^Nr(B))^{n+1}}\,\oomega_A(t\rho)+ \mathfrak I_{\oomega_A}(2^Nr(B))\, \avint_{B(w,t\rho)}\bigl|F(z)-\bar F_{B(w,t\rho)}\bigr|\, dz,
		\end{split}	
	\end{equation}
	where the penultimate inequality is a consequence of \eqref{eq:Fbound}. We now apply Lemma \ref{lem:estim_fund_sol} and obtain
	\begin{equation}\label{eq:mrg_omega_2}
		\begin{split}
			\avint_{B(w,t\rho)}\bigl|&F(z)-\bar F_{B(w,t\rho)}\bigr|\, dz\lesssim \avint_{B(w,t\rho)}\avint_{B(w,t\rho)} |F(z)-F(u)|\,du\, dz\\
			&	\lesssim \avint_{B(w,t\rho)}\avint_{B(w,t\rho)} |z-u|\, \max_{v\in B(w,2^{N-2}r(B))}|\nabla F(v)|\, du\, dz\\
			&\lesssim \frac{t\rho \, r(B)}{(2^Nr(B))^{n+2}}\lesssim \frac{t \, r(B)}{(2^Nr(B))^{n+1}},
		\end{split}
	\end{equation}
	where the last inequality yields because $\rho\approx 2^Nr(B)$. Therefore, by  \eqref{eq:mrg_omega_1}  and \eqref{eq:mrg_omega_2},
	\begin{equation}\label{eq:final_bound_mathring_omega_g}
		\int^1_0\mathring\oomega^{x,K\rho}_g(t)\, \frac{dt}{t}\lesssim \frac{r(B)}{(2^Nr(B))^{n+1}}\, \mathfrak I_{\oomega_A}(2^Nr(B)).
	\end{equation}
	
	The function $w=L^{-1}_A \nabla\cdot \bigl(\varepsilon_{x,B} F\bigr)$ is a weak solution to $L_Aw=\div \bigl( (\mathcal A-\bar{\mathcal A}_B)F\bigr)$ in $B(x,(N+1)\rho)$ so, by Remark \ref{remark:rem_grad_bound}, we have that
	\begin{equation}\label{eq:tilde_J_1_top}
		\begin{split}
			|\widetilde{\mathcal J}_{N,1}(x)|&\leq \sup_{\zeta \in B(x,2\rho)}\Biggl|\int \nabla_1\nabla_2\Gamma(\zeta,z) \varepsilon_{x,B}(z)\bigl(\nabla_1\Theta(x,y)-\Theta(x,y_B)\bigr)\, dz\Biggr|\\
			&\lesssim \biggl(\frac{1}{\rho^{n+1}}\int_{B(x,4\rho)} \bigl| \nabla w(z)|^2\, dz\biggr)^{1/2} + \int_0^1 \mathring \oomega_g^{x,K\rho}(t)\, \frac{dt}{t}.
		\end{split}
	\end{equation}
	Moreover,  since the operator $(\nabla L_A^{-1}\nabla\cdot)$ is bounded from  $L^2(\mathcal L^{n+1})$ to $L^2(\mathcal L^{n+1})$ (see \cite[(3.10)-(3.11)]{HK07}) with constants depending only on the ellipticity and dimension, we obtain  
	\begin{equation}\label{eq:tilde_J_1_top2}
		\begin{split}
			\| \nabla w\|_{L^2(\mathcal L^{n+1})}&\lesssim \biggl(\int_{B(x,r)} |\varepsilon_{x,B} (z)|^2\bigl|\nabla_1\Theta(x,y)-\nabla_1\Theta(x,y_B)\bigr|^2\, dz\biggr)^{1/2}\\
			&\lesssim \frac{r(B)}{(2^Nr(B))^{n+1}} \, \biggl(\int_{B(x,r)} \bigl|\mathcal A(z)-\bar{\mathcal  A}_B\bigr|^2\biggr)^{1/2}\overset{\eqref{eq:pw_diff_average_unif_cont}}{\lesssim} \frac{r(B)}{(2^Nr(B))^{(n+1)/2}} \, \mathfrak I_{\oomega_A}(2^N r(B)).
		\end{split}
	\end{equation}
	
	By \eqref{eq:tilde_J_1_top} and \eqref{eq:tilde_J_1_top2},  as $\rho \approx 2^N \ell(Q)$, we  infer that
	\begin{equation}\label{eq:bound_J_j_1}
		|\widetilde{\mathcal J}_{N,1}(x)|\lesssim \frac{r(B)}{(2^Nr(B))^{n+1}}\mathfrak I_{\oomega_A}(2^N r(B)).
	\end{equation}
	
	We bound the term $\widetilde{\mathcal J}_{N,2}(x)$ in \eqref{eq:split_j_1_j_2} via Lemma \ref{lem:estim_fund_sol}, which yields
	\begin{equation}\label{eq:bound_tilde_mathcal_J_2}
		\begin{split}
			\bigl|\widetilde{\mathcal J}_{N,2}(x)\bigr|&\lesssim \int_{\wt{\mathcal C}_N(B)\setminus B(x, 2^{N-4}r(B))}\frac{\bigl|\mathcal A(z)-\bar{\mathcal A}_B\bigr|}{|x-z|^{n+1}}\frac{r(B)}{|z-y_B|^{n+1}}\, dz\\
			&\lesssim \frac{r(B)}{(2^Nr(B))^{n+1} (2^N r(B))^{n+1}}\int_{\wt{\mathcal C}_N(B)\setminus B(x, 2^{N-4}r(B))}\bigl|\mathcal A(z)-\bar{\mathcal A}_B\bigr|\, dz\\
			&\lesssim \frac{r(B)}{(2^N r(B))^{n+1}}\,  \avint_{2^N B}\bigl|\mathcal A(z)-\bar{\mathcal A}_B\bigr|\, dz	\overset{\eqref{eq:pw_diff_average_unif_cont}}{\lesssim} \frac{r(B)}{(2^Nr(B))^{n+1}}\,  \mathfrak I_{\oomega_A}(2^Nr(B)).
		\end{split}
	\end{equation}
	
	Therefore,  gathering \eqref{eq:bound_J_j_1} and \eqref{eq:bound_tilde_mathcal_J_2},  we get
	\begin{equation}\label{eq:final_pw_bound_I_j_case_2}
		\sup_{x\in \wt{\mathcal{C}}_N(B)} |\wt I_N(x)|\lesssim \frac{r(B)}{(2^Nr(B))^{n+1}}\,  \mathfrak I_{\oomega_A}(2^N(B)).
	\end{equation}
	
	We remark that, for $N-1 \leq j \leq  N+1$, it holds that $2^j r(B)\approx 2^Nr(B)$,  so we finally deduce that
	
	\begin{equation}\label{eq:OD_bound_S2}
		\begin{split}
			\mathfrak S_2&\leq \sum_{N=1}^{\mathfrak j_0} \sum_{j=N-1}^{N+2}\mu\bigl(\mathcal C_N(B)\bigr)\sup_{x\in \wt{\mathcal{C}}_N(B)} | \wt I_N(x)|\\
			&\overset{\eqref{eq:final_pw_bound_I_j_case_2}}{\lesssim}\sum_{N=1}^{\mathfrak j_0} \sum_{j=N-1}^{N+2}\mu\bigl(\mathcal C_N(B)\bigr)\frac{r(B)}{(2^Nr(B))^{n+1}}\,  \mathfrak I_{\oomega_A}(2^Nr(B))\\
			&\approx\sum_{N=1}^{\mathfrak j_0} \frac{\mu\bigl(\mathcal C_N(B)\bigr)}{(2^Nr(B))^n}\sum_{j=N-1}^{N+2}\frac{r(B)}{(2^jr(B))}\,  \mathfrak I_{\oomega_A}(2^Nr(B))\\
			&\overset{\eqref{eq:P_doubling_S_1}}{\lesssim} \sum_{N=0}^{\mathfrak j_0}2^{-N\gamma} \Theta^n_\mu(2^{N} B)\sum_{j=0}^\infty\frac{r(B)^{1-\gamma}}{(2^jr(B))^{1-\gamma}}\,  \mathfrak I_{\oomega_A}(2^jr(B))\\
			&\overset{\eqref{eq:mod_cont_sum_2}}{\lesssim}  P_{\gamma,\mu}(B)\mathfrak L^{1-\gamma}_{\mathfrak I_{\oomega_A}}(B)\lesssim \Theta^n_\mu(B)\mathfrak L^{1-\gamma}_{\mathfrak I_{\oomega_A}}(B),
		\end{split}
	\end{equation}
	where the last bound  follows from the fact that $B$ is  $P_{\gamma, \mu}$-doubling.
	
	\vvv
	
	\textbf{Case 3.}
	We now bound $\mathfrak S_3$, namely the case $N+2 \leq j <\mathfrak j_0-2$. By Lemma \ref{lem:estim_fund_sol} and \eqref{eq:pw_diff_average_unif_cont}, for $w\in B(x,2^{j-4}r(B))$  and $v_j$ as in \eqref{eq:definition_v_j} it holds that
	\begin{equation}\label{eq_pw_bd_v_j}
		\begin{split}
			|v_j(w)|&\lesssim \int_{\mathcal C_j(B)}\frac{|\mathcal A(z)-\bar{\mathcal A}_B|}{(2^jr(B))^n}\frac{r(B)}{(2^jr(B))^{n+1}}\, dz\lesssim \frac{r(B)}{(2^jr(B))^n}\, \mathfrak I_{\oomega_A}(2^jr(B)).
		\end{split}
	\end{equation}
	Hence, by Remark \ref{remark:rem_grad_bound}, Caccioppoli inequality, and \eqref{eq_pw_bd_v_j}, we have that
	\begin{equation}\label{eq:bound_nabla_v_j_2}
		\begin{split}
			|\nabla v_j(x)|&\leq \sup_{w\in B(x, 2^{j-6}r(B))}|\nabla v_j(w)|\lesssim  \biggl(\avint_{B(x, 2^{j-5}r(B))}|\nabla v_j|^2\biggr)^{1/2}\\
			&\lesssim \frac{1}{2^jr(B)}\, \biggl(\avint_{B(x, 2^{j-4}r(B))}| v_j|^2\biggr)^{1/2}\overset{\eqref{eq_pw_bd_v_j}}{\lesssim}\frac{r(B)}{(2^jr(B))^{n+1}}\, \mathfrak I_{\oomega_A}(2^jr(B)).
		\end{split}
	\end{equation}

	Thus, we can bound $\mathfrak S_3$ as 
	\begin{equation}\label{eq:OD_bound_S_3}
		\begin{split}
			\mathfrak S_3\leq&\sum_{N=1}^{\mathfrak j_0} \sum_{j=N+2}^{\mathfrak j_0-2}
			\mu(\mathcal C_N(B))\, \sup_{x\in \mathcal C_N(B)}|\nabla v_j(x)|\\
			&\overset{\eqref{eq:bound_nabla_v_j_2}}{\lesssim} \sum_{N=1}^{\mathfrak j_0} \sum_{j=N+2}^{\mathfrak j_0-2}\mu(\mathcal C_N(B))\frac{r(B)}{(2^jr(B))^{n+1}}\, \mathfrak I_{\oomega_A}(2^jr(B)) \\
			&\lesssim \sum_{N=1}^{\mathfrak j_0} \sum_{j=N+2}^{\mathfrak j_0-2}2^{-j\gamma}\Theta^n_\mu\bigl(2^NB\bigr)\, 2^{-j(1-\gamma)}\, \mathfrak I_{\oomega_A}(2^jr(B))\\
			&\leq \sum_{N\geq 0} \Theta^n_\mu\bigl(2^NB\bigr)\, 2^{-\gamma N} \sum_{j=N+2}^{\mathfrak j_0-2}\frac{r(B)^{1-\gamma}}{(2^jr(B))^{1-\gamma}} \mathfrak I_{\oomega_A}(2^jr(B))\\
			&\overset{\eqref{eq:mod_cont_sum_2}}{\lesssim}P_{\gamma,\mu}(B)\, \mathfrak L^{1-\gamma}_{\mathfrak I_{\oomega_A}}(B)\lesssim \Theta^n_\mu(B)\, \mathfrak L^{1-\gamma}_{{\mathfrak I_{\oomega_A}}}(B),
		\end{split}
	\end{equation}
	where the last inequality relies on the $P_{\gamma,\mu}$-doubling assumption on $B$.
	
	\vvv
	
	\textbf{Case 4.}
	Let us study $\mathfrak S_4$, which corresponds to the range $j\geq \mathfrak j_0-2$. 
	For $w\in B\bigl(x,2^{j-4}r(B)\bigr)$, Lemma \ref{lem:estim_fund_sol} yields
	\begin{equation*}
		\begin{split}
			|v_j(w)|&\lesssim \int_{\mathcal C_j} \frac{\bigl|\mathcal A(z)-\bar{\mathcal A}_B\bigr|}{(2^j r(B))^n}\, \frac{r(B)}{(2^jr(B))^{n+1}}\, dz\lesssim \frac{r(B)}{(2^jr(B))^n} \avint_{2^jB}\bigl|\mathcal A(z)-\bar{\mathcal A}_B\bigr|\, dz\\
			&\overset{\eqref{eq:pw_diff_average_unif_cont}}{\lesssim} \frac{r(B)}{(2^jr(B))^n}\, \mathfrak I_{\oomega_A}\bigl(2^jr(B)\bigr)
		\end{split}
	\end{equation*}
	and, by estimates analogous to \eqref{eq:bound_nabla_v_j_2}, we have that
	\begin{equation*}
		\begin{split}
			|\nabla v_j(x)|&\lesssim \frac{1}{R}\Bigl(\avint_{B(x,R/10)}|v_j|^2\Bigr)^{1/2}\lesssim \frac{\ell(Q)}{R\, (2^j\ell(Q))^n}\, \mathfrak I_{\oomega_A}\bigl(2^jr(B)\bigr).
		\end{split}
	\end{equation*}
	For $j\geq \mathfrak j_0-2 \geq N-2$ we also have that $2^{N}r(B)  \lesssim R \approx 2^{j_0}r(B) \lesssim 2^jr(B)$ hence, the estimates above imply that
	\begin{equation}\label{eq:OD_bound_S4}
		\begin{split}
			\mathfrak S_4&\leq \sum_{N=1}^{\mathfrak j_0}\sum_{j=\mathfrak j_0 -2}^\infty \frac{r(B)}{R\, (2^jr(B))^n}\, \mathfrak I_{\oomega_A}\bigl(2^jr(B)\bigr)\, \mu(\mathcal C_N(B))\\
			& \lesssim \sum_{N=1}^{\mathfrak j_0}\frac{\mu(\mathcal C_N(B))}{R}\sum_{j=\mathfrak j_0 -2}^\infty  \frac{r(B)\mathfrak I_{\oomega_A}\bigl(2^jr(B)\bigr)}{\bigl(2^jr(B)\bigr)^n}\\
			&\lesssim \sum_{N=1}^{\mathfrak j_0}\frac{\mu(\mathcal C_N(B))}{(2^Nr(B))^n}\sum_{j=\mathfrak j_0 -2}^\infty  \frac{1}{2^j}\mathfrak I_{\oomega_A}\bigl(2^jr(B)\bigr)\\
			&= \sum_{N=1}^{\mathfrak j_0}{2^{-N\gamma}}\frac{\mu(\mathcal C_N(B))}{(2^Nr(B))^n}\sum_{j=\mathfrak j_0 -2}^\infty  \frac{r(B)^{1-\gamma}}{(2^jr(B))^{1-\gamma}}\mathfrak I_{\oomega_A}\bigl(2^jr(B)\bigr)\\
			&\overset{\eqref{eq:mod_cont_sum_2}}{\lesssim}P_{\gamma,\mu}(B)\, \mathfrak L^{1-\gamma}_{\mathfrak I_{\oomega_A}}(r(B))\lesssim \Theta^n_\mu(B)\,  \mathfrak L^{1-\gamma}_{\mathfrak I_{\oomega_A}}(r(B)),
		\end{split}
	\end{equation}
	where the last inequality follows from the $P_{\gamma, \mu}$-doubling assumption on $B$.
	
	Finally, we gather \eqref{eq:OD_bound_I}, \eqref{eq:final_estimate_S_1}, \eqref{eq:OD_bound_S2}, \eqref{eq:OD_bound_S_3}, and \eqref{eq:OD_bound_S4}, and conclude the proof of \eqref{eq:estimate_fund_sol_average} for $\mathfrak K_A$ together with  \eqref{eq:split_I_L_theta_merged}.
	
	\vv
	Establishing \eqref{eq:estimate_fund_sol_average} for \( \mathfrak{K}^* \) follows a similar but simpler approach. Therefore, we will only outline the proof, emphasizing the key aspects that require a different treatment. We first observe that \( \mathfrak{K}^*_A(x, y) = \mathfrak{K}_{A^T} (x,y) \). Thus, if we define  
	\[
	\Theta(x,y) \coloneqq \Theta(x - y; \bar{A}^T_B),
	\]
	by Lemma \ref{lem:MMPT_lemma_310}, we obtain  
	\begin{multline}
		\mathfrak{K}^*_A(\cdot, y) - \mathfrak{K}^*_A(\cdot, y_B)  
		= \int \nabla_2 \Gamma_{A^T}(\cdot ,z) \bigl(\bar A^T_B - A^T(z)\bigr)  
		\Bigl(\nabla_1 \nabla_2 \Theta(z,y) - \nabla_1 \nabla_2 \Theta(z,y_B) \Bigr)\, dz.
	\end{multline}
	
	If \( A \in \widetilde{\DMO}_{1-\gamma} \), then clearly \( A^T \in \widetilde{\DMO}_{1-\gamma} \) as well and $ \mathfrak I_{\oomega_A}=  \mathfrak I_{\oomega_{A^T}}$. Now,  if we define
	\[
	F(z)\coloneqq \nabla_1 \nabla_2 \Theta(z,y) - \nabla_1 \nabla_2 \Theta(z,y_B),
	\]
	it is easy to verify that $F$ can be estimated using \eqref{eq:gradient_fund_sol_const_matrix-2} by either applying the pointwise bounds or using its Lipschitz regularity via the mean value theorem.   Following the previous argument closely, we can now derive the desired estimate. The main difference is that applying \eqref{eq:Linftyest-r} is no longer necessary.  However,   estimating \( \widetilde{\mathcal{J}}_{N,1}(x) \) is not trivial so we provide the argument for clarity. 
	
	We first recall that 	
	\[\varepsilon_{x,B}(z)=\bigl(\bar{\mathcal A}_B - \mathcal A(z)\bigr)\chi_{B(x,2^{N-4}r(B))}(z),\qquad \qquad \text{ for }z\in \Rn1.
	\]
	As the function $w=L^{-1}_A \nabla\cdot \bigl(\varepsilon_{x,B} \, F\bigr)$ is a weak solution to $L_{ A}w=\div \bigl( \varepsilon_{x,B}\, F \bigr)$  in  $\Rn1$,   we can use \cite[Theorem 6.1]{Mou23} applied to $\Omega=\Rn1$ and obtain
	\begin{align*}
		|w(x)| \leq \|w\|_{L^\infty(\Rn1)} \lesssim \| \varepsilon_{x,B}  F \|_{L^{n+1,1}(\Rn1)}.
	\end{align*}
	It remains to prove that
	\[
	\| \varepsilon_{x,B}  F \|_{L^{n+1,1}(\Rn1)}\\ \lesssim \frac{r(B)}{(2^{N}r(B))^{n+1}}   \, \mathfrak I_{\oomega_A}(2^Nr(B)).
	\] 
	To this end,  note that 
	\[
	| \varepsilon_{x,B}  F(z) | \leq  C\, \frac{r(B)}{(2^{N}r(B))^{n+2}} \,\bigl|\mathcal A^T(z)-\bar{\mathcal A}^T_B\bigr| \, \chi_{B(x,2^{N-4}r(B))}(z),
	\]
	which implies that
	\begin{align*}
		(n+1)^{-1} &\| \varepsilon_{x,B} F \|_{L^{n+1,1}(\Rn1)} =  \int_0^\infty \Bigl| \bigl\{z \in \Rn1:| \varepsilon_{x,B}(z)  F(z)|>t \bigr\} \Bigr|^{\frac{1}{n+1}} \,dt\\
		&\leq \int_0^\infty \Bigl| \Bigl\{z \in B\bigl(x,2^{N-4}r(B)\bigr) : \bigl|\mathcal A^T(z)-\bar {\mathcal A}^T_B\bigr| > (Cr(B))^{-1}\,(2^{N}r(B))^{n+2} t \Bigr\} \Bigr|^{\frac{1}{n+1}} \,dt\\
		&= \frac{C\, r(B)}{(2^{N}r(B))^{n+2}} \int_0^\infty \Bigl| \Bigl\{z \in B\bigl(x,2^{N-4}r(B)\bigr) : \bigl| {\mathcal A}^T(z)-\bar{\mathcal A}^T_B\Bigr| > t \Bigr\} \Bigr|^{\frac{1}{n+1}} \,dt\\
		&\leq \frac{C' \, r(B)}{(2^{N}r(B))^{n+2}}  \int_0^\infty \Bigl| \Bigl\{z \in B\bigl(x,2^{N-4}r(B)\bigr) : \mathfrak I_{\oomega_A}(2^Nr(B)) > t \Bigr\} \Bigr|^{\frac{1}{n+1}} \,dt\\
		&\lesssim \frac{r(B)}{(2^{N}r(B))^{n+1}}  \, \mathfrak I_{\oomega_A}(2^Nr(B)).\qedhere
	\end{align*}
\end{proof}

\vv

\subsection{The David-Mattila lattice}
The following lemma presents the properties of the David-Mattila lattice of dyadic cubes adapted to a Radon measure $\sigma$ on $\Rn1$.
\begin{lemma}[\cite{DM00}, Theorem 3.2]\label{lemma:DM_cubes}
	Let $\sigma$ be a Radon measure on $\mathbb{R}^{n+1}$ with compact support, and denote $W\coloneqq \supp(\sigma)$. For $K_0 > 1$ and $A_0 > 5000 K_0$, there exists a sequence of partitions $(\mathcal D_{\sigma,k})_{k\geq 0}$ of $W$ into Borel subsets $Q$, which we will refer to as $\sigma$-cubes, with the following properties:
	\begin{itemize}
		\item  For each integer $k\geq 0$, $W$ is the disjoint union of the $\sigma$-cubes $Q$, for $Q\in\mathcal D_{\sigma,k}$. If $k < \ell$,
		$Q \in \mathcal D_{\sigma,\ell}$, and $R\in\mathcal D_{\sigma,k}$ , then either $Q\cap R=\varnothing$ or $Q\subset R$.
		\item 
		For each $k\geq 0$ and each $\sigma$-cube $Q \in \mathcal D_{\sigma,k}$, there is a ball $B(Q) = B(z_Q ,r(Q))$ such that
		\begin{equation*}
			\begin{split}
				z_Q\in W, \,\,\qquad  A_0^{-k}\,\diam(W) \leq r(Q)\leq K_0\,A_0^{-k} \,\diam(W),\\
				W\cap B(Q)\subset Q \subset W\cap 28 B(Q)= W\cap B(z_Q,28 r(Q)),
			\end{split}
		\end{equation*}
		and the balls $5B(Q),Q\in\mathcal D_{\sigma,k}$ are disjoint.
		\item The $\sigma$-cubes $Q\in\mathcal D_{\sigma,k}$ have small boundaries. Namely, for $Q\in\mathcal D_{\sigma,k}$ and $\ell\geq 0$, if we denote
		\begin{equation*}
			\begin{split}
					N_\ell^{\textup{int}}\coloneqq   \bigl\{x\in Q: \dist (x,W\setminus Q)<A_0^{-k-\ell}\,\diam(W)\bigr\},\\
				N_\ell^{\ext}(Q)\coloneqq   \bigl\{x\in W\setminus Q:\dist(x,Q)<A_0^{-k-\ell}\,\diam(W)\bigr\},
			\end{split}
		\end{equation*}
		and
		\(
		N_\ell(Q)\coloneqq   N_\ell^{\textup{int}}(Q)\cup N_\ell^{\ext}(Q),
		\)
		we have
		\begin{equation*}
			\sigma(N_\ell(Q))\leq \bigl(C^{-1}K_0^{-3(n+1)-1}A_0\bigr)^{-\ell}\sigma\big(90B(Q)\big).
		\end{equation*}
		\item Denote by $\mathcal{D}^{db}_{\sigma,k}$ the family of $\sigma$-cubes \( Q \in \mathcal{D}_{\sigma,k} \) for which  
		\begin{equation*}
			\sigma\big(100B(Q)\big) \leq K_0 \sigma(B(Q)).
		\end{equation*}  
		For \( Q \in \mathcal{D}_{\sigma,k} \setminus \mathcal{D}^{db}_{\sigma,k} \), we have that \( r(Q) = A_0^{-k}\,\diam(W) \).  
		Moreover, for any \( \ell \geq 1 \) with \( 100^\ell \leq K_0 \) and \( Q \in \mathcal{D}_{\sigma,k} \setminus \mathcal{D}^{db}_{\sigma,k} \), we have  
		\begin{equation}\label{ndb}
			\sigma\big(100B(Q)\big) \leq K_0^{-1} \sigma\big(100^{\ell+1}B(Q)\big).
		\end{equation}
	\end{itemize}
\end{lemma}

\vvv
We denote \( \mathcal{D}_\sigma \coloneqq \bigcup_k \mathcal{D}_{\sigma, k} \), and choose \( A_0 \) large enough so that  
\begin{equation}\label{assumption_a_0}
	C^{-1} K_0^{-3(n+1)-1} A_0 > A_0^{1/2} > 10.
\end{equation}

Here, we list some quantities associated with $Q\in\mathcal D_\sigma$:  
\begin{itemize}  
	\item \( J(Q) \coloneqq k \), which we refer to as the \textit{generation} of \( Q \).  
	\item \( \ell(Q) \coloneqq 56 K_0 A_0^{-k}\,\diam(W) \), that we call \textit{side length}. We note that  
	\begin{equation*}
		\frac{1}{28} K_0^{-1} \ell(Q) \leq \diam(28 B(Q)) \leq \ell(Q),
	\end{equation*}  
	and that \( r(Q) \approx \diam(Q) \approx \ell(Q) \).  
	\item Denoting by \( z_Q \) the \textit{center} of \( Q \), we define \( B_Q \coloneqq 28 B(Q) = B(z_Q, 28r(Q)) \), so that  
	\begin{equation*}
		Q \cap \frac{1}{28} B_Q \subset Q \subset B_Q.
	\end{equation*}
	\item For $Q\in\mathcal D_\sigma,$ we define $\mathcal D_\sigma(Q)\coloneqq \{P\in\mathcal D_\sigma: P\subseteq Q\}$.
\end{itemize}

\vv

\subsection{$L^1$-buffered estimates for the perturbation operator}
The proof of the next lemma is inspired by \cite[p.~1262]{MMT18}, which was originally formulated for (true) cubes and we adapt to balls via David-Mattila cubes.

{
\begin{lemma}\label{lemma:MMT}
Let $\mu$ be a non-negative Radon measure on $\Rn1$ and let  $B\subset \Rn1$ be a ball centered at $\supp(\mu)$ of diameter $d(B)$ with 
	Let us define
	\begin{equation}\label{eq:def_F_Q}
		\mathcal F(B)\coloneqq \bigl\{Q\in \mathcal D_{\mu|_B}:\, 29.5 B(Q) \cap \partial B \neq \varnothing\bigr\}.
	\end{equation}
	Then, for any $y\in B\setminus \partial B$, it holds that
	\begin{equation}\label{eq:itnegral_estimate_MMT}
		\int_{2B\setminus B} \frac{d\mu(x)}{|x-y|^n}\lesssim \sum_{Q\in \mathcal F(B)}\frac{\mu(29 B(Q))}{r(Q)^n}\chi_Q(y).
	\end{equation}
\end{lemma}

\vv

\begin{proof}
For $y \in B$, it holds that 
	\begin{equation}
		\begin{split}
			\int_{2B \setminus B} \frac{d\mu(x)}{|x-y|^n}&\leq \sum_{k\geq 0}\int_{\bigl\{x\in\Rn1\setminus B: \,\frac{3}{2}A_0^{-k-1}d(B) \leq |x-y|< \frac{3}{2}A_0^{-k}d(B) \bigr\}} \frac{d\mu(x)}{|x-y|^n}\\
			&\leq \sum_{k\geq 0} \frac{\mu\bigl(B(y,\frac{3}{2}A_0^{-k} d(B))\setminus B \bigr)}{(\frac{3}{2}A_0^{-k-1}d(B))^n}\\
			&\leq  A_0^n\,\sum_{k\geq 0}\sum_{Q\in \mathcal D_{\mu, k}(B)} \chi_Q(y)\, \frac{\mu\bigl(B(z_Q,29.5\,r(Q))\setminus B\bigr)}{r(Q)^n}\\
			&\leq A_0^n\,\sum_{Q\in \mathcal F(B)} \frac{\mu(29.5 B(Q)) }{r(Q)^n}\, \chi_Q(y).\qedhere
		\end{split}
	\end{equation}
\end{proof}

We  define
\(
L^2(\mu,B)\coloneqq \bigl\{f\in L^2(\mu):\supp f\subset B \bigr\},
\)
and denote by $L^2(\mu,B;\Rn1)$ its vector-valued analogue.

\vvv

In the following lemma, we bound the $L^1$-norm of $\mathcal S^*_\mu$ and $\mathcal S_\mu$  in ``buffer regions'' of the type $2B\setminus B$, under appropriate assumptions on the underlying measure.

\begin{lemma}\label{lem:L1 buffer}
Let $A$ be a uniformly elliptic matrix in $\Rn1$, $n\geq 2$, satisfying  $A\in \widetilde \DMO_{n-1}$. 	Let $\mu$  be a non-negative Radon measure in $\Rn1$,  $n\geq 2$, whose support is contained in a compact set of $\Rn1$ such that $\diam(\supp(\mu)) \leq R$,  and let  $B\subset \Rn1$ be a $(2, \mathfrak C_1)$-doubling ball centered at $\supp(\mu)$ of radius $r(B)$ with $(2,\mathfrak C_3)$-thin boundary.  Assume that there exists a constant $\widetilde{\mathfrak C}_2>0$ such that
	\begin{equation}\label{eq:growth_F_Q}
		\Theta^n_\mu(B(x,r)) \leq \widetilde{\mathfrak C}_2 \,\Theta^n_\mu(B), \qquad  \text{ for every } x \in B \text{ and }  r\in (0, 60 r(B)).
	\end{equation}
	If $\vec g\in L^2(\mu, B;\Rn1)$, $\sigma\coloneqq \mu(B)^{-1}\,\mu|_B$, and for $\widehat \tau_A$ as in Lemma \ref{lem:main_pw_estimate}, we have that
	\begin{equation}\label{eq:lemma_L1_buffer_zone}
		\int_{2B\setminus B}\bigl|\mathcal S^*_\sigma \vec g(x)\bigr|\, d\mu(x)\lesssim_{n, \Lambda, R} A_0^n\,   {\widetilde{\mathfrak C}_2} \, (K_0\,\mathfrak C_1\,\mathfrak C_3)^{\frac{1}{2}}\, \Theta^n_\mu(B)\, \widehat \tau_A(r(B))) \, \| \vec g \|_{L^2(\sigma)}
	\end{equation}
	and, for $f\in L^2(\mu,B)$, it holds
	\begin{equation}\label{eq:lemma_L1_buffer_zone_bis}
		\int_{2B\setminus B}\bigl|\mathcal S_\sigma f(x)\bigr|\, d\mu(x)\lesssim_{n, \Lambda, R} A_0^n\, {\widetilde{\mathfrak C}_2} \, (K_0\,\mathfrak C_1\,\mathfrak C_3)^{\frac{1}{2}}\, \Theta^n_\mu(B)\, \widehat \tau_A(r(B)) \, \| f \|_{L^2(\sigma)}.
	\end{equation}
\end{lemma}

\begin{proof}
	Let $\alpha(\cdot)$ be as in \eqref{eq:definition_alpha_modulus}.
	By inequality \eqref{eq:pointwise_bound_kernel_frak_K}, the doubling property of $\alpha(\cdot)$, \eqref{eq:comparability_alpha_tau}, and Lemma \ref{lemma:MMT},
	\begin{equation*}
		\begin{split}
			\int_{2B\setminus B}&\bigl|\mathcal S^*_\sigma \vec g(x)\bigr|\, d\mu(x){\lesssim}\int_{2B\setminus B} \int_B \alpha(2r(B))\frac{|\vec g(y)|}{|x-y|^n}\, d\sigma(y)\, d\mu(x)\\
			&\overset{\eqref{eq:itnegral_estimate_MMT}}{\lesssim } A_0^n\,\frac{\alpha(r(B))}{\mu(B)}\int_B |\vec g(y)| \sum_{Q\in \mathcal F(B)}\frac{\mu(29.5B(Q))}{r(Q)^n}\chi_Q(y)\, d\mu(y)\\
			&\leq A_0^n\, \frac{\widehat \tau_A(r(B))}{\mu(B)}\sum_{Q\in \mathcal F(B)}\frac{\mu(30B(Q))}{r(Q)^n}\biggl(\int_{Q}|\vec g|\, d\mu\biggr)\\
			&\leq  \widetilde{\mathfrak C}_2(30A_0)^n \Theta^n_\mu(B)\,\frac{\widehat \tau_A(r(B))}{\mu(B)}\biggl[\sum_{Q\in\mathcal F(B)}\Bigl(\frac{r(Q)}{r(B)}\Bigr)^{\frac{1}{2}} \,\int_{Q}|\vec g|^2\, d\mu \biggr]^{\frac{1}{2}}\biggl[\sum_{Q\in\mathcal F(B)}\mu(Q)\Bigl(\frac{r(B)}{r(Q)}\Bigr)^{\frac{1}{2}} \biggr]^{\frac{1}{2}}\\
			&\eqqcolon \widetilde{\mathfrak C}_2\,  (30A_0)^n\, \Theta^n_\mu(B)\,\frac{\widehat \tau_A(r(B))}{\mu(B)} \,\mathcal I_1 \cdot \mathcal I_2,
		\end{split}
	\end{equation*}
	where the latter bound holds because of \eqref{eq:growth_F_Q} and the Cauchy-Schwarz inequality.
	
	To estimate $\mathcal I_1$,  we write
	\begin{equation}\label{eq:bound_Iprime_final}
		\begin{split}
			(\mathcal I_1)^2 &\leq\sum_{Q\in\mathcal D(B)}\Bigl(\frac{r(Q)}{r(B)}\Bigr)^{\frac{1}{2}} \int_{Q}|\vec g|^2\, d\mu\leq (56\, K_0)^{1/2} \,\sum_{k=0}^\infty\sum_{Q\in\mathcal D_k(B)}A_0^{-\frac{k}{2}}\biggl(\int_{Q}|\vec g|^2\, d\mu\biggr)\\
			&\leq (56\, K_0)^{1/2} \,\frac{A_0^{1/2}}{A_0^{1/2}-1} \, \int_B|\vec g|^2\, d\mu \leq 2 \,(56\, K_0)^{1/2} \, \int_B|\vec g|^2\, d\mu,
		\end{split}
	\end{equation}
	where, in the penultimate inequality, we used that $A^{1/2}_0-1\geq A^{1/2}/2$.
	
	For $k\in\mathbb N$, we define
	\(
	\mathcal F_k(B)\coloneqq  \mathcal F(B) \cap \mathcal{D}_{\mu|_B,k}
	\)
	and we estimate 
	\begin{equation}\label{eq:bound_Idbprime_2}
		\begin{split}
			(\mathcal I_2)^2&=\sum_{k=0}^\infty \sum_{Q\in\mathcal F_k(B)} \mu(Q)\Bigl(\frac{r(B)}{\ell(Q)}\Bigr)^{\frac{1}{2}}= (56\, K_0)^{-1/2} \,\sum_{k=0}^\infty A_0^{\frac{k}{2}}  \sum_{Q\in\mathcal F_k(B)} \mu(Q).
		\end{split}
	\end{equation}
	Notice that, for any $Q \in \mathcal F_k(B)$, we have that if $y \in Q$, then
	\begin{align*}
	\dist(y, \partial B) \leq |y-z_Q|+	\dist(z_Q, \partial B) \leq 28r(Q)+ 29.5 r(Q) \leq 58 r(Q)
	\end{align*}
	and so 
		\begin{displaymath}
		Q \subset \bigl\{y \in B \colon\, \dist(y, \partial B) \le 58 K_0\,A_0^{-k}d(B)\bigr\}.
	\end{displaymath}
	Using that $B$ has $(2,  \mathfrak C_3)$-small boundary with respect to $\mu$ we can now deduce that
	\begin{align}\label{eq:mu_buffer}
		\begin{split}
		{\sum_{Q \in \mathcal F_k(R)}} \mu(Q) & \leq \mu\bigl( \bigl\{y \in 2B\colon\, \dist(y, \partial B) \le 58K_0\, A_0^{-k}d(B) \bigr\}\bigr) \leq  116   K_0 \,    \mathfrak   C_3\,A_0^{-k}\mu(2B).
							\end{split}
	\end{align}
	If we plug \eqref{eq:mu_buffer} in 	\eqref{eq:bound_Idbprime_2} and use that $B$ is $(2,\mathfrak C_2)$-doubling, we obtain
	\begin{equation}\label{eq:I_pp_final}
		\begin{split}
				(\mathcal I_2)^2 &\leq  116   K_0 \,   \mathfrak   C_1  \mathfrak   C_3 \,   (56\, K_0)^{-1/2} \,\sum_{k=0}^\infty A_0^{-\frac{k}{2}}   \mu(B)\leq 232  K_0 \,   \mathfrak   C_1  \mathfrak   C_3 \,   (56\, K_0)^{-1/2}  \, \mu(B).	
					\end{split}
	\end{equation}
By \eqref{eq:bound_Iprime_final} and \eqref{eq:I_pp_final},  we obtain 
$$
(\mathcal I_1 \cdot \mathcal I_2)^2 \leq 464 K_0 \,   \mathfrak   C_1  \mathfrak   C_3 \,   \mu(B)  \, \int_B|\vec g|^2,
$$
 which readily implies \eqref{eq:lemma_L1_buffer_zone}. The proof of \eqref{eq:lemma_L1_buffer_zone_bis} is analogous.
\end{proof}
}
\vvv

\subsection{$L^2$-mean oscillation estimates for the perturbation operator}

We now apply the results of this section to derive the main bound for the $L^2$-mean oscillation of the operators $\mathcal{S}_\mu$ and $\mathcal{S}^*_\mu$.

\begin{lemma}\label{lem:mean oscillation perturb}
Let $A$ be a uniformly elliptic matrix in $\Rn1$, $n\geq 2$, satisfying  $A\in \widetilde \DMO_{1-\gamma}$, for some $\gamma \in (0,1)$.  	Let $\mu$  be a non-negative Radon measure in $\Rn1$,  $n\geq 2$, whose support is contained in a compact set of $\Rn1$ such that $\diam(\supp(\mu)) \leq R$,  and let  $B\subset \Rn1$ be a $\mathfrak C_1$-$P_{\gamma, \mu}$-doubling ball centered at $\supp(\mu)$ of radius $r(B)$ with $(2,\mathfrak C_3)$-thin boundary.  Assume that there exists a constant $\widetilde{\mathfrak C}_2>0$ such that
	\begin{equation}\label{eq:growth_F_Q-bis}
		\Theta^n_\mu(B(x,r)) \leq \widetilde{\mathfrak C}_2 \,\Theta^n_\mu(B), \qquad  \text{ for every } x \in B \text{ and }  r\in (0, 60 r(B)).
	\end{equation}
	If we set
	\[
	\tilde \alpha(t)\coloneqq \mathfrak F_{\oomega_A}(t) + \mathfrak L^{1-\gamma}_{\oomega_A}(t),	\qquad \text{ for } t>0,
	\]
	then, we have that
	\begin{equation}\label{eq:main_estimate_error_term_S}
		\begin{split}
			\avint_B \bigl|\mathcal S_\mu 1 - \avint_B \mathcal S_\mu 1\, d\mu\bigr|^2\, d\mu &\lesssim \mathfrak I_{\oomega_A}(r(B)) \,  \|\mathcal R_{\mu|_B}\|^2_{L^2(\mu|_B)\to L^2(\mu|_B)} +\Theta^n_\mu(B)^2\, \tilde \alpha(r(B))^2
		\end{split}
	\end{equation}
	and
	\begin{equation}\label{eq:main_estimate_error_term_S_adjoint}
		\begin{split}
			\avint_B \bigl|\mathcal S^*_\mu 1 - \avint_B \mathcal S^*_\mu 1\, d\mu\bigr|^2\, d\mu &\lesssim \mathfrak I_{\oomega_A}(r(B)) \,  \|\mathcal R_{\mu|_B}\|^2_{L^2(\mu|_B)\to L^2(\mu|_B)} +\Theta^n_\mu(B)^2\, \tilde \alpha(r(B))^2,
		\end{split}
	\end{equation}
	where the implicit constants in \eqref{eq:main_estimate_error_term_S} and \eqref{eq:main_estimate_error_term_S_adjoint} depend on $n, \Lambda, \mathfrak C_1,\mathfrak C_2, \mathfrak C_3, R,$ and $\gamma$.
\end{lemma}

\begin{proof}
	For brevity, for $f\in L^1_{\loc}(\mu)$, we write 
	\[
	m_Q(f, \mu)=\avint_B f\,d\mu. 
	\]
	We also define
	\[
	L^2_0(\mu,B)\coloneqq \bigl\{f\in L^2(\mu):\supp f\subset B \quad\text{ and } \quad m_B(f,\mu)=0\bigr\}
	\]
	and denote by $L^2_0(\mu,B;\Rn1)$ its vector-valued analogue.
	The space $L^2_0(\mu,B)$ endowed with the norm $\|\cdot\|_{L^2(\mu)}$ is a Hilbert space whose Banach dual is the space of functions in $L^2(\mu)$ 	modulo an additive constant and equipped with the norm $\|\cdot\|_{L^2(\mu)}$
	(see e.g. \cite[1.2.2, p. 143]{St93}). Moreover, for $f\in L^2(\mu,B)$ it holds
	\begin{equation}\label{eq:duality_L_0}
		\begin{split}
			\|f-m_B(f,\mu)\|_{L^2(\mu,B)}&\approx \sup_{g\in L^2_0(\mu,B),\atop \|g\|_{L^2(\mu)}=1}\int (f-m_B(f,\mu))g\, d\mu=	\sup_{g\in L^2_0(\mu,B),\atop\|g\|_{L^2(\mu)}=1}\int fg\, d\mu,		
		\end{split}
	\end{equation}
	where the second identity follows from $m_B(g,\mu)=0$.

	We denote $\sigma\coloneqq \mu(B)^{-1}\mu|_B$, and observe that, by duality,
	\begin{equation}
		\begin{split}
			\Biggl(\int_B \biggl|\mathcal S_\mu 1 - \int_B \mathcal S_\mu 1\, d\sigma\biggr|^2\, d\sigma\Biggr)^{1/2}\approx \sup_{\vec g\in L^2_0(\sigma; \Rn1),\atop \| \vec g\|_{L^2(\sigma; \Rn1)}=1} \biggl|\int \mathcal S_\mu 1\cdot \vec g\, d\sigma\biggr|.
		\end{split}
	\end{equation}
	
	Let $\vec g\in L_0^2(\sigma; \Rn1)$ such that  $\| \vec g\|_{L^2(\sigma; \Rn1)}=1$.  It holds that
	\begin{equation}\label{eq:lem_L2_perturb_1}
		\begin{split}
			\int \mathcal S_\mu 1\cdot \vec g\, d\sigma &=\int \mathcal S^*_\sigma \vec g\, d\mu
		\end{split}
	\end{equation}
	and we split
	\begin{equation*}
		\begin{split}
			\biggl|\int \mathcal S^*_\sigma \vec g\, d\mu\biggr|&\leq  \biggl|\int_B \mathcal S^*_\sigma \vec g\, d\mu\biggr|+\biggl|\int_{2B\setminus B} \mathcal S^*_\sigma \vec g\, d\mu\biggr|+\biggl|\int_{\Rn1\setminus 2B} \mathcal S^*_\sigma \vec g\, d\mu\biggr|\eqqcolon \mathcal I_1+ \mathcal I_2+ \mathcal I_3.
		\end{split}
	\end{equation*}
	In light of \eqref{eq:L2Smu}, duality,   \eqref{eq:split_I_L_theta_merged}, and the fact that $\mathfrak{L}^{d}_{\oomega_A}(B) \leq \mathfrak{L}^{1-\gamma}_{\oomega_A}(B)$ for any $d \leq 1-\gamma$,  we infer that
	\begin{equation}\label{eq:final_estim_I_1_auxS}
		\begin{split}
			\mathcal I_1 &\lesssim \mathfrak I_{\oomega_A}(r(B))^{1/2} \,  \|\mathcal R_{\mu|_B}\|_{L^2(\mu|_B)\to L^2(\mu|_B)} +\Theta^n_\mu(B)\,\bigl[ \mathfrak I_{\tau_A}(r(B)) +\widehat \tau_A( r(B) )\bigr]\\
			&\lesssim   \mathfrak I_{\oomega_A}(r(B))^{1/2} \,  \|\mathcal R_{\mu|_B}\|_{L^2(\mu|_B)\to L^2(\mu|_B)} +\Theta^n_\mu(B)\,\bigl[\mathfrak F_{\oomega_A}(R)   + \mathfrak{L}^{1-\gamma}_{\oomega_A}(B)  \bigr],
		\end{split}
	\end{equation}
	while,  by Lemma \ref{lem:L1 buffer},  we obtain
	\begin{equation}\label{eq:final_estim_I_2_auxS}
		\begin{split}
			\mathcal I_2 &\lesssim \Theta^n_\mu(B)\, \widehat \tau_A(r(B)) \leq  \Theta^n_\mu(B)\,\bigl[\mathfrak F_{\oomega_A}(R)   + \mathfrak{L}^{1-\gamma}_{\oomega_A}(B)  \bigr].
		\end{split}
	\end{equation}
	Finally,  since $\supp \vec g\subset B$ and $\int \vec g\,d\mu =0$,  if $y_B$ is the center of $B$,   it holds that 
	\begin{equation*}
		\mathcal I_3 \leq  \int_{\Rn1 \setminus 2B}  \int_B \left| \mathfrak K(y,x) - \mathfrak K(y_B,x) \right| \, |\vec g(y)|\,d\sigma(y)\,d\mu(x).
	\end{equation*}
	By Fubini's theorem,   Lemma \ref{lemma:hormander_perturbation}, the  Cauchy-Schwarz inequality,  and the facts that $\sigma(B)=1$ and  $\|\vec g\|_{L^2(\sigma; \Rn1)}=1$, we deduce that 
	\begin{equation}\label{eq:final_estim_I_3_auxS}
		\begin{split}
			\mathcal I_3 &\lesssim \Theta^n_\mu(B)\bigl[\mathfrak I_{\oomega_A}(r(B))+ \mathfrak 	L^{1-\gamma}_{\oomega_A}(B)\bigr] \leq \Theta^n_\mu(B)\,\bigl[\mathfrak F_{\oomega_A}(R)   + \mathfrak{L}^{1-\gamma}_{\oomega_A}(B)  \bigr].
		\end{split}
	\end{equation}

	We finally gather \eqref{eq:final_estim_I_1_auxS}, \eqref{eq:final_estim_I_2_auxS}, \eqref{eq:final_estim_I_3_auxS}, and conclude the proof of \eqref{eq:main_estimate_error_term_S}. The bound \eqref{eq:main_estimate_error_term_S_adjoint} can be proved analogously.
\end{proof}

\vvv

\subsection{Quantitative rectifiability for Radon measures}
The $L^2$-mean oscillation estimate of the perturbation operator allows us to reduce Theorem \ref{theorem:elliptic_GSTo} to its Riesz transform analogue \cite[Theorem 1.1]{GT18}.
\begin{proof}[Proof of Theorem \ref{theorem:elliptic_GSTo}]
	We first observe that, since $B$ is $\mathfrak C_1$-$P_{\gamma, \mu}$-doubling, it follows hypothesis (3) implies that there exists a constant $\widetilde{\mathfrak C}_2>\mathfrak C_2$ depending on $n, \gamma, \mathfrak C_1,$ and $\mathfrak C_2$ such that 
	\[
		\Theta^n_\mu(B(x,r))\leq \widetilde{\mathfrak C}_2\,  \Theta^n_\mu(B), \qquad \text{ for every }x\in B \text{ and }r\in (0,60 r(B)).
	\]
	
	The bound \eqref{eq:maineq} implies that there exists $C_{\mathcal R}>0$ depending on $n$, $\Lambda$, $\mathfrak C_2$, and $\diam(\supp(\mu))$, such that 
	\begin{equation}\label{eq:precise_bound_Riesz}
		\begin{split}
			\|\mathcal R_{\mu|_{B}}\|_{L^2(\mu|_B)\to L^2(\mu|_B)}&\leq C_{\mathcal R}  \,  \bigl(\Theta^n_\mu(B) + \|T_{\mu|_{B}}\|_{L^2(\mu|_B)\to L^2(\mu|_B)}\bigr)\overset{\eqref{eq:L^2_bd_2n_B}}{\leq} C_{\mathcal R}\, (1+\mathfrak C_4)   \, \Theta^n_\mu(B).
		\end{split}
	\end{equation}
	Thus, by triangle inequality, \eqref{eq:mean_osc_small_theorem}, and \eqref{eq:main_estimate_error_term_S}, we have
	\begin{equation*}
		\begin{split}
			\avint_B\big|&\mathcal R_\mu1(x)-\avint_B\mathcal R_\mu 1(x)\big|^2\, d\mu(x)\lesssim_n \avint_B \biggl|T_\mu 1 - \avint_B T_\mu 1\, d\mu\biggr|^2\, d\mu + \avint_B \biggl|\mathcal S_\mu 1 - \avint_B \mathcal S_\mu 1\, d\mu\biggr|^2\, d\mu\\
			&\qquad\qquad \lesssim \tau\,\Theta^n_\mu(B)^2+ \mathfrak I_{\oomega_A}(r(B)) \,  \|\mathcal R_{\mu}\|^2_{L^2(\mu|_B)\to L^2(\mu|_B)} +\tilde \alpha(r(B))^2\,\Theta^n_\mu(B)^2\\
			&\qquad\qquad\overset{\eqref{eq:precise_bound_Riesz}}{\lesssim} \Bigl(\tau + C_{\mathcal R}\,(1+\mathfrak C_4)\, \mathfrak I_{\oomega_A}(r(B)) + \tilde \alpha(r(B))^2\Bigr)\,\Theta^n_\mu(B)^2\eqqcolon (\tau+\varepsilon(B))\,\Theta^n_\mu(B)^2.
		\end{split}
	\end{equation*}
	
	We remark that $\bigl(\mathfrak I_{\oomega_A}(t)+ \tilde \alpha(t)\bigr)\to 0$ as $t\to 0$.
	Thus, for every $\tau_{\mathcal R}>0$, we can pick $\tau=\tau_\mathcal{R}/2$ and $r(B)$  small enough  so that $\varepsilon(B)\leq  \tau_{\mathcal R}/2$. This, in turn, yields
	\begin{equation}\label{eq:precise_bound_Riesz_oscillation}
		\avint_B\big|\mathcal R_\mu1(x)-\avint_B\mathcal R_\mu 1(x)\big|^2\, d\mu(x)\leq \tau_{\mathcal R} \,\Theta^n_\mu(B)^2.
	\end{equation}

	The bounds \eqref{eq:precise_bound_Riesz} and \eqref{eq:precise_bound_Riesz_oscillation} allow us to apply \cite[Theorem 1.1]{GT18} and find  $\theta \in (0,1)$ and a uniformly $n$-rectifiable set $\Gamma$ such that $\mu(B\cap\Gamma)\geq \theta\,\mu(B)$.
\end{proof}

\vvv

For the applications to free boundary problems for elliptic measure, we need the following variant of Theorem \ref{theorem:elliptic_GSTo}.

\begin{theorem} \label{teo1}
	Let	$A$ be a uniformly elliptic matrix satisfying  $A\in \widetilde \DMO_{1-\gamma}$,  for some $\gamma \in (0,1)$. Let $\mu$ be a Radon measure with compact support in $\Rn1$, $n\geq 2$, so that $\diam(\sup(\mu)) \leq R$,  and denote by $T_\mu$ the gradient of the single layer potential associated with $L_A$ and $\mu$.  Let also $B\subset \Rn1$ be a ball centered at $\supp(\mu)$ of radius $r(B)$.  Assume that  there is $G_B \subset B$	and some positive real numbers $ \mathfrak C_1,  \mathfrak C_2,   \mathfrak C_3,  \tau, \eta$, $\gamma, \delta$, and $\lambda$, the following properties hold:
	\begin{enumerate}
		\item $r(B)\leq\lambda$.
		\item $\mu\bigl(B\setminus G_B\bigr)\leq \eta \, \mu (B)$.
		\item $P_{\gamma, \mu}(B)\leq \mathfrak C_1 \,\Theta^n_\mu(B)$.
		\item $\Theta^n_\mu(B(x,r)) \leq \mathfrak C_2  \,\Theta^n_\mu(B)$,   for every  $x \in G_B$ and $r \in (0, 30 r(B))$.
		\item $B$ has $\mathfrak C_3$-thin boundary.
		\item We have that
		\[
		T_*(\chi_{2B}\,\mu)(x)\leq  \mathfrak C_4\,\Theta^n_\mu(B)\quad \mbox{ for all $x\in G_B$}
		\]
		\item There exists some $n$-plane $L$ passing through the center of $B$ such that $\beta^L_{\mu,1}(B)\leq \,\delta\Theta^n_\mu(B).$
		\item It holds that
		\[
		\int_{G_B}\bigl|T\mu(x)-m_{\mu,G_B}(T\mu)\bigr|^2\, d\mu(x)\leq \tau \Theta^n_\mu(B)^2\mu(B).
		\]
	\end{enumerate}
	Then, there exists $\theta \in (0,1)$ such that if $\delta$, $\eta$, and $\tau$ are small enough, possibly depending on $n$, $\Lambda$, $\mathfrak C_1$,  $\mathfrak C_2$,  $ \mathfrak C_3$,  $\mathfrak C_4$,  $\diam(\supp (\mu))$, ,  and if  $\lambda$ is small enough, depending on $n$, $\Lambda$, $\mathfrak C_1$,  $\mathfrak C_2$,  $ \mathfrak C_3$,  $\mathfrak C_4$,  $\diam(\supp (\mu))$, and $\tau$,    there is  a uniformly $n$-rectifiable set $\Gamma \subset G_B$ such that
	\[
	\mu(B\cap\Gamma)\geq \theta\mu(B).
	\]
	The UR constants of $\Gamma$ depend on all the constants above.  The same result holds for the adjoint operator $T^*$ as well. 
\end{theorem}

\begin{proof}
	Theorem \ref{teo1} follows from Theorem \ref{theorem:elliptic_GSTo} by adapting the proof of \cite[Theorem 3.3]{AMT17a} and incorporating the results from Sections \ref{sec:aux_operator} and \ref{big piecesTb}, which are essential to the proof. We leave the details to the interested reader.
\end{proof}

\vvv

\section{Limits of sequences of domains and elliptic measures} \label{section:limits}

In this section, we study the compactness argument that leads to Proposition~\ref{limlem}. We first present two lemmas, which are used in its proof.

\begin{lemma}\label{eq:Green-repr}
	Let \( \{ \Omega_j \}_{j \geq 1} \subset \mathbb{R}^{n+1} \) be a sequence of CDC domains with the same CDC constants  and \( \inf_j \operatorname{diam}(\partial \Omega_j) > 0 \). Suppose there exists a fixed ball \( B = B(y , r) \subset \Omega_j \) so that $\delta_{\om_j}(y) \approx r$ for all \( j \geq 1 \).  	Assume further that \( \{A_j\}_{j \in \mathbb{N}} \) is a sequence of Borel measurable, uniformly elliptic coefficient matrices, all with the same ellipticity constants,  for which  there exists a  uniformly elliptic matrix $A$ (with not necessarily constant coefficients) such that
	$$
	A_j \to A  \quad \text{ in } L^1_{\loc}(\Rn1) \text{ as } j \to \infty.
	$$ 
	Define
\[
h_j(x) =
\begin{cases}
	H_{\Gamma_{A_j}(\cdot, y)|_{\partial \Omega_j}}(x) & \text{if } x \in \Omega_j, \\
	\Gamma_{A_j}(x, y) & \text{if } x \in \mathbb{R}^{n+1} \setminus \Omega_j,
\end{cases}
\]
where \( \Gamma_{A_j}(\cdot, \cdot) \) is the fundamental solution associated with \( A_j \) and $H$ denotes Perron's solution.   Then there exists a $L_A$-solution \( h_y \in Y_{\loc}^{1,2}(\mathbb{R}^{n+1}) \cap C^\alpha_{\mathrm{loc}}(\mathbb{R}^{n+1}) \) such that, after passing to a subsequence, \( h_j \to h_y \) locally uniformly and weakly in \( Y_{\loc}^{1,2}(\mathbb{R}^{n+1}) \).
\end{lemma}

\begin{proof}
Since $\Gamma_{A_j}(\cdot,y) \geq 0$, by  the  maximum principle,  $h_j  \geq 0$ in $\om_j$. Moreover,  thanks to \eqref{eq:bound_above_green},  
$$
\sup_{\Rn1 \setminus \om_j} \Gamma_{A_j}(\cdot,y) \lesssim  r^{1-n},
$$
 and, by maximum principle,  it holds that
	\begin{equation*}
		\sup_{  \om_j } h_j \leq \sup_{ \partial \om_j} \Gamma_{A_j} (\cdot,y)  \lesssim r^{1-n}.
	\end{equation*}
Therefore,
		\begin{equation}\label{eq:F-rho global bound}
		\sup_{ x\in \Rn1 } h_j(x)  \lesssim r^{1-n}.
	\end{equation}
	Thus,  by Caccioppoli's inequality,  as $B(y,2r) \subset \om_j$,  it holds that
	\begin{align*}
		\int_{B(y,r)} | \nabla h_j(x)|^2\,dx &\lesssim \frac{1}{r^{2}}\int_{B(y,2r)} h_j(x)^2\,dx 
		\lesssim r^{1-n}.
	\end{align*}

	The latter bound along with \eqref{eq;green ident-Wiener} and \eqref{eq:bd_g_1} (both for the Green's function and the fundamental solution),   implies that 
	$$
	\int_{\om_j} | \nabla h_j(x)|^2\,dx \leq \int_{\om_j  \setminus B(y,r)} | \nabla h_j(x)|^2\,dx+\int_{B(y,r)} | \nabla h_j(x)|^2\,dx\lesssim r^{1-n}.
	$$
	Similarly, we can prove that 
	\[
	\| h_j\|_{L^{2^*}(\om_j)} \lesssim r^{\frac{1-n}{2}}.
	\]
	Using once again  \eqref{eq:bd_g_1}, we obtain $\| \Gamma_{A_j}(\cdot,y)\|_{Y^{1,2}(\Rn1 \setminus B(y, r))} \lesssim  r^\frac{1-n}{2}$ and so 
	\begin{equation}\label{eq:Y12-Frho-almost}
		\| h_j\|_{Y^{1,2}(\Rn1 \setminus \partial \om_j)} \lesssim r^\frac{1-n}{2}.  
	\end{equation}

Additionally,  since the extension of the Green's function by zero in $\om_j^c$ is a subsolution in any ball $B$ so that $y \not \in 2B$ and $\Gamma_{A_j}$ is a solution in  $2B$ as well,  by Lemma \ref{eq:Green-repr-wiener}, we have that $h_j$ is also a non-negative subsolution in  $2B$. Thus,  by Caccioppoli's inequality and \eqref{eq:F-rho global bound},	we have that
$$
	\int_B | \nabla h_j(x)|^2\,dx \leq r(B)^{-2} \int_{2B} | h_j(x)|^2\,dx \lesssim r(B)^{n-1} \,r^{2(1-n)}.
	$$
	The latter inequality combined with \eqref{eq:Y12-Frho-almost} readily implies that $h_j \in Y^{1,2}_{\loc}(\Rn1)$ with constants uniform in $j$.
	
	Let us assume that $K \subset \Rn1$ is a compact set such that $y \not \in K$.  Our first goal is to show that $h_j$ is H\"older continuous in $K$ uniformly in $j$.  Thanks to the compactness of $K$ we can cover it with a finite number of balls $B$ so that $y\not\in 36B$  and $r(B)<\frac{1}{4} \diam \d\Omega$.  We claim that, without loss of generality, the sequence of functions $h_j$ is  H\"older continuous on such balls $B$ with constants independent of $j$.  To prove it, we distinguish some cases.
	\begin{enumerate}
		\item[Case 1:] Let $x,w\in B\setminus \Omega_j$. Then since $y \not \in 36B$, we have that $\Gamma_{A_j}(\cdot,y) $ is $L_{A_j}$-solution in $4B$ and so,  by local H\"older continuity  away from $y$ and \eqref{eq:bound_above_green}, 
		$$
		|h_j(x)-h_j(w)|= | \Gamma_{A_j}(x,y) -\Gamma_{A_j}(w,y)| \lesssim \frac{|x-w|^\alpha}{r(B)^{\alpha}} \sup_{z \in 2B}   \Gamma_{A_j}(z,y) \lesssim \frac{|x-w|^\alpha}{r(B)^{n-1+\alpha}}.
		$$
		\item[Case 2:]  Let  $x\in B\cap \Omega_j$ and $w \in B\setminus \Omega_j$. If $z\in [x,w]\cap \d\Omega_j \subset B$ and $s=|x-z|$, then by Theorem \ref{thm:boundary oscillation} and arguing as in Case 1 for $\Gamma_{A_j}(\cdot, y)$,  we have that 
		\begin{align*}
			|h_j(x) &-h_j(w)|= | h_j(x) - \Gamma_{A_j}(w,y)| \\
			&\leq 	|h_j(x)-h_j(z)| + | \Gamma_{A_j}(z,y) -\Gamma_{A_j}(w,y)| \\
			& \lesssim  \osc_{B(z, 2s) \cap \Omega_j} h_j +  \frac{|z-w|^\alpha}{r(B)^{n-1+\alpha}}\\
			& \lesssim  \frac{|x-z|^\alpha}{r(B)^{\alpha}}\sup_{B(z,4 r(B)) \cap \Omega_j} h_j +  \osc_{B(z, 2s) \cap \Omega_j} \Gamma_{A_j}(\cdot,y) +  \frac{|z-w|^\alpha}{r(B)^{n-1+\alpha}}\\
			& \lesssim  \frac{|x-z|^\alpha}{r(B)^{\alpha}} [ r^{1-n} + r(B)^{1-n}] +  \frac{|z-w|^\alpha}{r(B)^{n-1+\alpha}} \\
			&\leq  |x-w|^\alpha  [ r^{1-n} + r(B)^{1-n}] \, r(B)^{-\alpha},
		\end{align*}
		where the last inequality follows from the fact that $ \max(|x-z|,|z-w| ) \leq |x-w|$.
		\item[Case 3:] Let $x, w \in \Omega_j\cap B$ and suppose there is $z\in \Omega_j^{c}$ with $\dist(z,[x,w])<4|x-w|$.  Thus, if $z'\in [x,y]$ is the point that realizes $\dist(z,[x,w])$, then 
		\begin{equation}
			\label{e:y-z}
			|w-z|\leq |w-z'|+|z'-z|
			< |x-w|+4|x-y|=5|x-w|\leq 10r(B).
		\end{equation}
		The same argument shows that $|x-z|\leq 10r(B)$. Using the latter inequalities and that $y \not\in 25B$, if we denote by $c_B$ the center of $B$, we have 
		\begin{align*}
			|y-z| &\geq  |y-c_B| -|z-w| -|w-c_B|        \\&>36 r(B) - 10r(B) - r(B) =25 r(B).
		\end{align*}
		In particular, if $B'=B(z,5r(B))$, then $y \not\in 5 B'$. Thus, as in the estimate of Case 2 we have 
		\begin{align*}
			|h_j(x) -h_j(w)| \lesssim  |x-w|^\alpha  [ r^{1-n} + r(B)^{1-n}] \, r(B)^{-\alpha}.
		\end{align*}
		\item[Case 4:] Let $x, w \in \Omega_{j}\cap B$ and suppose that $\dist(z,[x,w])\geq4 |x-y|$ for every $z\in \Omega_j^{c}$. Covering the segment $[x,w]$ with balls of radius comparable to $\lvert x-w\rvert$ and using the interior regularity of solutions to elliptic equations,  see Theorem \ref{thm:C_alpha_interior-inhom},  concludes the proof.
	\end{enumerate}

	Furthermore,  since $h_j$ is a $L_{A_j}$-solution in $\om_j$, by interior regularity estimates,  it is also locally H\"older continuous in any ball $B$ such that $2B \subset \om_j$ and $y \in B$.  Thus, 
	$$
	h_j \in C^\alpha_{\loc}(\Rn1) \quad \text{  uniformly in } j.
	$$
	
	To summarize, we have shown that   $h_j\in Y_{\loc}^{1,2}(\Rn1) \cap C^\alpha_{\loc}(\Rn1)$ with uniform bounds in $j$. Therefore,  by Arzelà-Ascoli,   Alaoglu theorem,  Eberlein–\v Smulian theorem, and   a diagonalization  argument,  we can show that there exists  a non-negative function $h_y \in Y^{1,2}(\Rn1) \cap C^\alpha_{\loc}(\Rn1)$ such that, up to a non-relabeled subsequence,
	\[
		h_j (\cdot,y) \to h_y \quad \text{ weakly in } Y_{\loc}^{1,2}(\Rn1) \text{ and locally uniformly in } \Rn1.\qedhere
	\]
\end{proof}

\vv

In the following lemma, we identify the limits of the fundamental solutions \( \Gamma_{A_j} \) with the fundamental solution associated with the limiting matrix.

\begin{lemma}\label{lem:limits fundamental}
	Assume  that \( \{A_j\}_{j \in \mathbb{N}} \) is a sequence of Borel measurable, uniformly elliptic coefficient matrices, all with the same ellipticity constants,  for which  there exists a uniformly elliptic matrix $A$ such that
	$$
	A_j \to A  \quad \text{ in } L^1_{\loc}(\Rn1) \text{ as } j \to \infty.
	$$
	Then,  for any $y \in \Rn1$,  if $\Gamma_{A_j}(\cdot, y)$ is the fundamental solution associated with $A_j$ and $p \in [1, n+1/n)$,   there exists 
	$$
	\Gamma_y \in C^\alpha_{\loc}(\Rn1 \setminus \{y\}) \cap Y^{1,2}(\Rn1 \setminus B(y, r)) \cap W^{1,p}(B(x_0, s)),
	$$
	for every $r > 0$ and $s > 0$,  such that, up to a non-relabeled subsequence, 
	$$
	\Gamma_{A_j}(\cdot, y) \to \Gamma_y(\cdot)
	$$
	locally uniformly on $\mathbb{R}^{n+1} \setminus \{x_0\}$, and also weakly in  $W^{1,p}(B(x_0, s))$ and in  $Y^{1,2}(\mathbb{R}^{n+1} \setminus B(x_0, r))$ for every $r > 0$ and $s > 0$.  Moreover, 
	$$
	\Gamma_y(x) = \Gamma_A(x,y), \quad \text{ for all } x, y \in \Rn1 \text{ with } x\neq y.
	$$
\end{lemma}

\begin{proof}
	
	For fixed $y \in \Rn1$ and  for any ball $B$ so that $y \not \in 2B$, we have that 
	$$
	\Gamma_{A_j}(x, y)\lesssim |x-y|^{1-n} \lesssim r(B)^{1-n}, \quad \text{ if } x \in B,
	$$
	with implicit constants independent of $j$.  Since the functions $\Gamma_{A_j}(\cdot, y)$ have uniformly bounded $C^\alpha$ seminorm, where $\alpha$ depends only on the ellipticity constants,  by Arzelà-Ascoli theorem, there exists a $C^\alpha$ function $\Gamma_y$ such that $\Gamma_{A_j}(\cdot, y)$ converges (up to a non-relabeled subsequence) locally uniformly to $\Gamma_y$ in $B$.  
	
	Also, thanks to \eqref{eq:bd_g_1},  \eqref{eq:bd_g_2},  and \eqref{eq:bd_g_3} (when $\om =\Rn1$),   we know that 
	\begin{equation}
		\begin{split}
			&\lVert \Gamma_{A_j}(\cdot, y) \rVert_{Y^{1,2}(\R^{n+1}\setminus B(y,r))}\leq C r^{1-\frac{n+1}{2}}\qquad\text{for every }r>0,\\
			&\lVert \Gamma_{A_j}(\cdot,y)\rVert_{W^{1,p}(B(y,s))}\leq C s^{1-\frac{n+1}{p'}}, \qquad\text{for every } s>0. 
			\label{stimeGammanorme}
		\end{split}
	\end{equation}
	where the constant $C$ depends only on the ellipticity constants and $n$.  Thus, by Alaoglu theorem, Eberlein–\v Smulian theorem, and   a diagonalization argument,  after passing to a subsequence,  we have that $\Gamma_{A_j}$ converges   to $\Gamma_y$ weakly  in $Y^{1,2}(\R^{n+1}\setminus B(y,r))$ and in $W^{1,p}(B(y,s))$ thanks to the uniqueness of the limit for every $r>0$ and $s>0$, and satisfies the same  bounds as in \eqref{stimeGammanorme}.

	Finally, let us note that for every $\phi \in C^\infty_c(\Rn1)$,
	\begin{equation}\label{eq:fund-sol-weak}
		\begin{split}
			\phi(y)=&\lim_{j\to \infty} \int A_j(z) \nabla \Gamma_{A_j}(z,y)\cdot\nabla \phi(z)\, dz.
		\end{split}
	\end{equation}
	
	We write
	\begin{align*}
		\int A_j(z) & \nabla \Gamma_{A_j}(z,y) \cdot \nabla \phi(z)\, dz = \int (A_j(z)-A(z)) \nabla \Gamma_{A_j}(z,y) \cdot \nabla \phi(z)\, dz  \\
		&+ \int A(z) ( \nabla \Gamma_{A_j}(z,y) - \nabla \Gamma_y(z)) \cdot \nabla \phi(z)\, dz +  \int A(z) \nabla \Gamma_{y}(z) \cdot \nabla \phi(z)\, dz  \\
		&\eqqcolon I_j + II_j + III.
	\end{align*}
	
	By H\"older's inequality,  if $\supp \phi \subset B(y,M)$ for some $M>1$ large enough, by \eqref{stimeGammanorme},
	\begin{align*}
		| I_j| &\leq  \| A_j -A \|_{L^{p'}(B(y,1))} \| \nabla \Gamma_{A_j}(z,y)\|_{L^{p}(B(y,1))}  \| \nabla \phi\|_{L^\infty(\Rn1)}\\
		&\qquad+  \| A_j -A \|_{L^{2}(B(y,M))} \| \nabla \Gamma_{A_j}(z,y)\|_{L^{2}(B(y,M) \setminus B(y,1) )}   \| \nabla \phi\|_{L^\infty(\Rn1)}\\
		&\lesssim_M  \Bigl[ \Lambda^\frac{p'-1}{p'} \, \| A_j -A \|_{L^{1}(B(y,1))}^{1/p'}  + \Lambda^\frac{1}{2} \, \| A_j -A \|_{L^{1}(B(y,M))}^{1/2} \Bigr]   \| \nabla \phi\|_{L^\infty(\Rn1)} \\
		& \to 0, \quad \text{ as } j \to \infty.
	\end{align*}
	
	The fact that $II_j \to 0$ as $j \to \infty$ follows from the  weak convergence of $\nabla \Gamma_{A_j}(\cdot,y) \to \nabla \Gamma_y(\cdot)$ in $L^1(B(y,1))$ and $L^2(\R^{n+1} \setminus B(y,1))$.  Thus,  by \eqref{eq:fund-sol-weak}, we have that 
	\[
		\phi(y)= \int A(z) \nabla \Gamma_{y}(z) \cdot \nabla \phi(z)\, dz.
	\]
	
	It remains to prove that $\Gamma_y(x) = \Gamma(x,y)$ for all $x \in \Rn1$ such that $x \neq y$.  To this end, let $f \in C^\infty_c(\Rn1)$ and note that, by  the weak convergence of $ \Gamma_{A_j}(\cdot,y) \to \Gamma_y(\cdot)$ in $L^1(B(y,1))$ and $L^{2^*}(\R^{n+1} \setminus B(y,1))$, it holds that
	\[
		\lim _{j \to \infty} w_j(y)\coloneqq \lim_{j \to \infty} \int \Gamma_{A_j}(x,y) f(x) \,dx =  \int \Gamma_{y}(x) f(x) \,dx.
	\]

	As  $\| w_j \|_{Y^{1,2}(\Rn1)} \lesssim \|f\|_{L^\infty(\Rn1)}$, with implicit constants depending only on the diameter of the support of $f$,  there exists $w \in Y^{1,2}(\Rn1)$ with  $\| w \|_{Y^{1,2}(\Rn1)} \lesssim \|f\|_{L^\infty(\Rn1)}$, such that, after passing to a subsequence, $ w_j \to w$ weakly in $Y^{1,2}(\Rn1)$.  Moreover, $w_j$ is locally H\"older continuous whose H\"older seminorm in each compact set is uniformly bounded in $j$ (see Theorem \ref{thm:boundary oscillation}).   Thus,  since, by \eqref{eq:Green_function_global bound} and \eqref{eq:lorentz<bounded}, it holds that  
	$$
	\sup_{z \in \Rn1} |w_j(z)| \lesssim \|f\|_{L^\infty(\Rn1)},
	$$  we may pass to a subsequence, so that $w_j \to w$ locally uniformly as $j \to \infty$. This implies that 
	\[
		w(y)=\int \Gamma_{y}(x) f(x) \,dx
	\]
	and, as it is easy to see that  $L_{A^T_j} w= f$ in $\Rn1$,  we infer  that 
	\[
		\int \Gamma_{y}(x) f(x) \,dx =  \int \Gamma_{A}(x,y) f(x) \,dx.
	\]

The latter  identity is true  for any $f \in C^\infty_c(\Rn1)$ and using that both $\Gamma_{y}(\cdot)$ and $\Gamma_{A}(\cdot,y)$  are continuous away from $y$,  we deduce that $ \Gamma_{y}(x) = \Gamma_{A}(x,y) $ for any $x \in \Rn1$ with $x \neq y$, concluding the proof of the lemma. 
\end{proof}

\vv

Using the previous lemmas, we are now ready to prove the main result of this section.

\begin{proof}[Proof of Proposition \ref{limlem}]
	We split the proof of the proposition into several steps.
	
	\vspace{2mm}
	\noindent \textbf{Convergence of Green's functions.} 
	
	\vspace{1mm}
	Let $u_j$ be defined by 
	\[
	u_j(x) \coloneqq 
	\begin{cases}
		G^{A^T_j}_{\Omega_{j}} (x, x_0) & \text{ if } x \in \Omega_j, \\
		0 & \text{ if } x \in \mathbb{R}^{n+1} \setminus \Omega_j.
	\end{cases}
	\]

	By Lemma~\ref{lemgreen*}, for any ball $B$ such that $2B \subset \mathbb{R}^{n+1} \setminus B_0$, it holds that
	\[
	\sup_{z \in B} u_j(z) \lesssim r_0^{1-n}.
	\]
	Following the argument of \cite[Lemma 2.9]{AMT20}, we obtain that $u_j \in C^{\alpha}_{\mathrm{loc}}(\mathbb{R}^{n+1})$ with Hölder constants uniform in $j$. Therefore, by Arzelà-Ascoli and a diagonalization argument, there exists a non-negative function $u_\infty^{x_0} \in C^\alpha_{\mathrm{loc}}(\mathbb{R}^{n+1})$ such that, up to a subsequence,
	\[
	u_j \to u_\infty^{x_0} \quad \text{uniformly on compact subsets of } \mathbb{R}^{n+1},
	\]
	and 
	\[
	\sup_{z \in \mathbb{R}^{n+1} \setminus B_0} u_j(z) \lesssim r_0^{1-n}.
	\]
	
	Furthermore, since each $u_j$ is a positive subsolution to $L_{A^T_j} w = 0$ in $B$, we may apply Caccioppoli's inequality to obtain
	\[
	\int_B |\nabla u_j|^2 \lesssim \frac{1}{r^2} \int_{2B} |u_j|^2 \lesssim \Bigl( \frac{r(B)}{r_0} \Bigr)^{n-1}.
	\]
	This implies that $\|u_j\|_{Y^{1,2}(B)} \lesssim \bigl( \frac{r(B)}{r_0} \bigr)^{n-1}$ uniformly in $j$. By the Banach-Alaoglu theorem and the Eberlein–\v{S}mulian theorem, we can extract a further subsequence such that
	\[
	u_j \rightharpoonup u_\infty^{x_0} \quad \text{weakly in } Y^{1,2}_{\mathrm{loc}}(\mathbb{R}^{n+1} \setminus B_0).
	\]
	
	In addition, by \eqref{eq:bd_g_2},  and \eqref{eq:bd_g_3}, we know that 
	$$
	\|  u_j \|_{L^p(2B_0)} \lesssim r_0^{2-\frac{n+1}{p'}} \quad \text{ and } \quad \| \nabla u_j \|_{L^p(2B_0)} \lesssim r_0^{1-\frac{n+1}{p'}}
	$$
	for $p \in [1, \frac{n+1}{n})$,  and hence, up to a further subsequence,
	\[
	u_j \rightharpoonup u_\infty^{x_0} \quad \text{weakly in } W^{1,p}(2 B_0).
	\]
	
	\vspace{2mm}
	\noindent \textbf{Definition of $\Omega_{\infty}^{x_{0}}$.}
	
	\vspace{1mm}
	We now define the limiting open set to be
	\[
		\Omega_{\infty}^{x_{0}} \coloneqq \{ x \in \mathbb{R}^{n+1} : u_\infty^{x_0}(x) > 0 \}.
	\]
	
	We claim that for every $\vphi \in C^\infty_c(\om_\infty^{x_0})$, there exists $j_0$ large enough, so that   $\vphi \in C^\infty_c(\om_j)$, for every $j \geq j_0$. Indeed, if this was not the case, there would exist $x \in \om^{x_0}_\infty$ so that $x \in \pom_j$ for every $j$.  In that case,  it would hold that $u_j(x)=0$ for every $j$ while $u_\infty^{x_0}(x)>0$.  However,  $0= \lim_{j\to\infty} u_j(x) = u(x)>0$, which leads to a contradiction and the claim follows.
	\vspace{2mm}
	
	\noindent {\bf 	Convergence of elliptic measures.}
	
	\vspace{1mm}
	
	Recall that the measures $\omega_j\coloneqq\omega_{\Omega_j}^{A_j, x_0}$ are probability measures, and form a tight sequence by Lemma \ref{l:holdermeasure}. Indeed, let $\xi_j \in \partial\Omega_j$ be a point realizing $\delta_{\Omega_j}(x_0)$, and note that $\delta_{\Omega_j}(x_0) \approx r_0$ uniformly in $j$. Then, given any $\varepsilon > 0$, we can choose $r \approx \varepsilon^{-1/\alpha} r_0$ and apply \eqref{e:wholder} to obtain
	\[
		\omega_j( B(\xi_j, 2r)^c )\lesssim  \Bigl(   \frac{|x_0 - \xi_j | }{r}\Bigr)^{\alpha} \approx  \Bigl(   \frac{r_0 }{r}\Bigr)^{\alpha}< \ve,
	\]
	where the implicit constants are  independent of $j$.  Thus,  we may pass to a subsequence,  so that $\omega_j$ converges weakly to a Radon measure $\omega_{\infty}$.

	Note that  $\supp (\omega_j)=\d\Omega_{j}$ simply because of Lemma \ref{l:bourgain}, the mutual absolute continuity of elliptic measures with different poles together, and  the fact that the sets $\Omega_j$ satisfy the CDC condition for every $j\in \mathbb N$.

	We will now prove that  $\supp (\omega_\infty)=\d\Omega_{\infty}^{x_{0}}$.  Let us first notice that $\supp (\omega_\infty)\subseteq \overline{\Omega_\infty^{x_0}}$. This is of course due to the fact that the sets $\{u_j>0\}$ converge in the Hausdorff metric to $\{u_\infty^{x_0}>0\}$ since $u_j$ converges  to $u_\infty^{x_0}$  locally uniformly. 
	
	Now, let $x\in \supp (\omega_\infty) \cap \Omega_\infty^{x_0}$ and let $r>0$ be so small that $B\coloneqq B(x,r)\subseteq \Omega_\infty^{x_0}$. Let $\phi$ be a positive smooth function supported on $B$ such that $\phi(x)>0$. Notice that for $j\in \mathbb N$ big enough we further have that $B\subseteq \Omega_j$ arguing as above. This however results in a contradiction since
	\[
	0<\int \phi \, d\omega_\infty=\lim_{j\to \infty}\int \phi\,  d\omega_j=0.
	\]
	Thus such a point cannot exist and so $\supp(\omega_\infty^{x_0})\subseteq \partial \Omega_\infty^{x_0}$. 
	
	Conversely, let $x \in \partial \Omega_\infty^{x_0} \setminus \supp (\omega_\infty)$, and let $r > 0$ be such that $B(x, r) \cap \supp \omega_\infty = \varnothing$, with the additional assumption that $x_0 \notin B(x, r)$. Then, for every smooth function $\phi$ supported in $B(x, r)$, we may apply Lemma~\ref{lemgreen*} to obtain that
	\begin{equation*}
		\begin{split}
			0=&\int \phi \, d\omega_\infty=\lim_{j\to \infty} \int \phi\,  d\omega_j=\lim_{j\to \infty} \int A^T_j \nabla u_j(y) \cdot \nabla\phi(y)\,  dy\\
			=&\lim_{j\to \infty} \int ( A^T_j  - A^T ) \nabla u_j(y) \cdot \nabla\phi(y)\,  dy + \int A^T \nabla u_j(y) \cdot\nabla \phi(y)\, dy = \lim_{j\to \infty} ( I_j +II_j).  
		\end{split}
	\end{equation*}
	By H\"older's inequality and the fact that $\| \nabla u_j \|_{L^2(B(x,r))} \lesssim  (r/r_0)^{n-1}$ and that $A_j  \to  A $  in $L^{2}(B)$-norm, (as $A_j$   and $A$ are in $L^\infty(\Rn1)$ with uniform  bounds),  we obtain
	$$
	|I_j| \lesssim \| A^T_j  - A^T\|_{L^{2}(B(x,r))} \| \nabla u_j \|_{L^2(B(x,r))} \| \nabla \phi\|_{L^\infty(B(x,r))} \to 0, \quad \text{ as } j \to \infty.
	$$
	Moreover, by the local weak $L^2$ convergence of $u_j $ to $u_\infty^{x_0}$ in $B(x,r)$, we get that 
	$$
	\lim_{j\to \infty} II_j = \int A \nabla u^{x_0}_\infty(y) \cdot\nabla \phi(y)\, dy.
	$$
	This implies that $L_{A^T}u^{x_0}_\infty=0$ in $B(x,r)$.  Thus,  since $u^{x_0}_\infty(x)=0$, by  maximum principle, $u^{x_0}_\infty=0$ in $B(x,r)$, which contradicts the fact that $x \in \partial \Omega_\infty^{x_0}$.

	\vspace{2mm}

	\noindent {\bf 	$\Omega_{\infty}^{x_{0}}$ satisfies the CDC.}	
	
	\vspace{1mm}
	
	Thanks to the  argument above and the fact that $\omega_j \warrow \omega_\infty$, we know that for any $\xi \in \pom^{x_0}_\infty$, there exists a sequence $\{\xi_j \}_{j\geq 1}$ such that  $\xi_j \in \d \om_j$ and $\xi_j \to \xi$.  The exact same argument as in the proof of \cite[Lemma 2.9]{AMT20} proves that $\Omega_{\infty}^{x_{0}}$ satisfies the CDC.

	\vspace{2mm}

	\noindent {\bf $u_\infty^{x_0}$  is the Green function for the operator  $L_{A^T}$ in $\Omega_\infty^{x_0}$ with pole at $x_0$.}  	
	
	\vspace{1mm}

	If $\phi \in C^\infty_c(\Omega_\infty^{x_0})$,  then it also holds that $\phi \in C^\infty_c(\Omega_j)$ for $j$ large enough and so
	$$
	\phi(x_0)= \int_{\om_j}  A^T_j \nabla G^{A^T_j}_{\Omega_{j}} (x,x_0)\cdot \nabla \phi(x) \,dx.
	$$
	Since $u_j= G_{A^T_j, \om_j}(\cdot,x_0)$ in $\Omega_j$,  by the weak convergence of $\nabla  u_j$ to $\nabla u_\infty^{x_0}$ in $L^p(B(x_0, 2r_0)$ and $ L^2( \Rn1 \setminus B(x_0,r_0) )$,  arguing as in the proof of \eqref{eq:fund-sol-weak}, we can show that
	$$
	\phi(x_0)= \lim_{j \to \infty} \int A^T_j \nabla u_j(x) \cdot\nabla \phi(x) \,dx= \int A^T \nabla u_\infty^{x_0}\cdot \nabla \phi(x) \,dx.
	$$
	
	Recall that, by Lemma \ref{eq:Green-repr-wiener},  for $x \in \mathbb{R}^{n+1} \setminus \{x_0\}$, we have
	\begin{equation}\label{eq:green-repr-j}
	u_j(x) = \Gamma_{A^T_j}(x, x_0) - h_j(x),
	\end{equation}
	where $h_j = H_{\Gamma_{A^T_j}(\cdot,x_0)}$ in ${\Omega}_j$ for the operator $L_{A^T_j}$ and $h_j = \Gamma_{A^T_j}(\cdot,x_0)$ in $\mathbb{R}^{n+1} \setminus \overline{\Omega}_j$.
	
	Moreover, by Lemma \ref{eq:Green-repr}   there exists $h_{x_0} \in Y_{\loc}^{1,2}(\mathbb{R}^{n+1}) \cap C^\alpha_{\mathrm{loc}}(\mathbb{R}^{n+1})$ such that, up to a subsequence, $h_j \to h_{x_0}$ locally uniformly and weakly in $Y^{1,2}(\mathbb{R}^{n+1})$.  Owing to   \ref{lem:limits fundamental}, we know that $\Gamma_{A^T_j}(x, x_0) \to \Gamma_{A^T}(x, x_0)$ locally uniformly on compact subsets of $\mathbb{R}^{n+1} \setminus \{x_0\}$.  
	
	We claim that $ L_{A^T} h_{x_0}= 0$. Indeed,  if $\phi \in C^\infty_c(\om^{x_0}_\infty)$, there exists $j_0$, so that  $\phi \in C^\infty_c(\om_j)$ for all $j \geq j_0$. Therefore,  as $L_{A_j^T} h_j =0$ in $\om_j$,  by the weak covergence of $h_j $ to $h_{x_0}$ in $Y^{1,2}_{\loc}(\Rn1)$, 
	\begin{align*}
0= \lim_{j \to \infty} \int A_j^T \nabla h_j \cdot\nabla \phi = \int A^T \nabla h_{x_0} \cdot\nabla \phi,
	\end{align*}
	proving our claim. 
	
We take limits in \eqref{eq:green-repr-j} as \( j \to \infty \) and obtain
\[
u_\infty^{x_0}(x) = \Gamma_{A^T}(x, x_0) - h_\infty(x) \quad \text{for all } x \in \mathbb{R}^{n+1} \text{ with } x \neq x_0.
\]
By the very definition of \( \Omega_\infty^{x_0} \), we have \( \Gamma_{A^T}(x, x_0) = h_\infty(x) \) for every \( x \in \mathbb{R}^{n+1} \setminus \Omega_\infty^{x_0} \). Hence, since \( \Omega_\infty^{x_0} \) is Wiener regular,
\[
h_\infty(x) = H_{\Gamma_{A^T}(\cdot, x_0)}(x) \quad \text{for all } x \in \Omega_\infty^{x_0},
\]
and an application of Lemma~\ref{eq:Green-repr} yields
\[
u_\infty^{x_0}(x) = G^{A^T}_{\Omega_\infty^{x_0}}(x, x_0) \quad \text{for all } x \in \Omega_\infty^{x_0} \text{ with } x \neq x_0.
\]

We remark that this immediately implies that \( \Omega_\infty^{x_0} \) is connected.

	\vspace{2mm}

	\noindent{\bf $\omega_{\infty}^{x_{0}}$ is the elliptic measure for $L_A$ in $\Omega_\infty^{x_0}$ with pole at $x_0$.}
	
	\vspace{1mm}

	Let $\phi\in C_{c}^{\infty}(\bR^{n+1})$ be such that $\phi(x_{0})=0$.   By the weak convergence of $u_j$ to $u_\infty^{x_0}$ in $Y^{1,2}_{\loc}(\Rn1)$, we have 
	\begin{equation}\label{phiom}
		\begin{split}
				\int \phi \,d\omega_{\infty} 
			&= \lim_{j\rightarrow\infty} \int \phi \,d\omega_j=\lim_{j\rightarrow\infty}\int A^T_j \nabla u_j\cdot \nabla \phi(y)\, dy \\
			&= \int_{\om^{x_0}_\infty} A^T(y)\nabla G_{\om_\infty^{x_0} }^{A^T}(x, x_0) \cdot \nabla \phi(y)\, dy = \int \phi \,d\omega_{\om^{x_0}_\infty}^{A, x_{0}}.
		\end{split}
	\end{equation}
	This concludes the proof by showing that $\omega_\infty$ is the elliptic measure of $L_A$ relative to the domain $\Omega_\infty^{x_0}$.   
\end{proof}

To conclude this section, we now prove an auxiliary lemma that is useful to implement the compactness argument under the hypothesis \eqref{eq:small_BMO_condition}.
\begin{lemma}\label{lem:BMOmatrix}
	Assume  that \( \{A_j\}_{j \in \mathbb{N}} \) is a sequence of Borel measurable, uniformly elliptic coefficient matrices, all with the same ellipticity constants.  For each \( j \in \mathbb{N} \), suppose there exists an infinitesimal sequence \( \{ \varepsilon_j \}_{j \in \mathbb{N}} \) such that
	\begin{equation}
		\sup_{0< r < 1} \sup_{x \in \mathbb{R}^{n+1}} \, \avint_{B(x,r)} \left| A_j(y) - \avint_{B(x,r)} A_j \right| \, dy \leq \varepsilon_j.
	\end{equation}
	Then there exists  a constant coefficients uniformly elliptic matrix $A$,  such that, after passing to a subsequence, 
	$$
	A_j \to A  \quad \text{ in } L^1_{\loc}(\Rn1) \text{ as } j \to \infty.
	$$
\end{lemma}

\begin{proof}

	Observe that for any ball \(B \subset \mathbb{R}^{n+1}\), since the matrices \(A_j\) are uniformly bounded and elliptic, by compactness and passing to a (non-relabeled) subsequence, we may assume the existence of a matrix \(C(B)\), with the same ellipticity constants, such that
	\[
	\lim_{j \to \infty} \avint_B A_j = C(B).
	\]
	From the vanishing mean oscillation condition in all balls, we also have
	\[
	\lim_{j \to \infty} \avint_B \left| A_j - \avint_B A_j \right| = \lim_{j \to \infty} \avint_B |A_j - C(B)| = 0.
	\]
	
	Now consider a countable, dense collection of balls \(\{B_k\}_{k \in \mathbb{N}}\) in \(\mathbb{R}^{n+1}\),  constructed using a countable dense set of centers and radii. By a diagonalization argument, we may extract a non-relabeled subsequence such that
	\[
	\lim_{j \to \infty} \avint_{B_k} A_j = C(B_k), \quad \text{for all } k \in \mathbb{N}.
	\]
	
	Next, take any two balls \(B_1, B_2\) from the collection such that \(  \mathcal{L}^{n+1}(B_1 \cap B_2)> 0 \), and let \(B_3 \subseteq B_1 \cap B_2\) be another ball from the collection. Since the mean of \(A_j\) converges to the same value on \(B_1\), \(B_2\), and \(B_3\), and because convergence in mean implies convergence almost everywhere up to subsequence, we deduce that
	\[
	C(B_1) = C(B_2) = C(B_3).
	\]
	
	Repeating this argument for all overlapping balls \(B_k\), we conclude that all the limits \(C(B_k)\) must coincide. Hence, there exists a single constant matrix \(A\) such that \(C(B_k) = A\) for all \(k\). It follows that \(A_j \to A\) in \(L^1_{\mathrm{loc}}(\mathbb{R}^{n+1})\), as desired.
\end{proof}
\vvv

\section{Applications to free boundary problems} \label{section:main_lemma}

\subsection{The proof of Theorem \ref{theorem:main_one_phase}.}

The results from Sections \ref{sec:aux_operator} and \ref{big piecesTb}, combined with the strategy used in the harmonic measure setting in \cite{AHMMMTV16}, allow us to prove Theorem~\ref{theorem:main_one_phase}. In particular, the proof of rectifiability in the qualitative one-phase problem for elliptic measure can be reduced to the following key lemma.

Before stating it, we introduce the \emph{maximal radial operator}
\[
\mathcal{M}_n \mu(x) \coloneqq \sup_{r>0} \frac{|\mu|(B(x,r))}{r^n},
\]
where $\mu$ is a Radon measure on $\mathbb{R}^{n+1}$.

\begin{lemma}\label{label:key_lemma_qualit_one-phase}
	Let $\Omega$, $p$, $A$, and $E$ be as in Theorem~\ref{theorem:main_one_phase}. Then, for $\omega \coloneqq \omega^{L_A, p}_\Omega$, it holds that
	\[
		\mathcal{M}_n \omega(x) + T_* \omega(x) < \infty, \qquad \text{for } \omega\text{-a.e.\ } x \in E.
	\]
\end{lemma}

\begin{proof}[Sketch of the proof.]
	Instead of the original proof, we hereby follow the more direct proof in \cite[Section 12]{PPT21}.
	Without loss of generality, we assume that $\Omega$ is a bounded domain. The fact that $\mathcal{M}_n\omega(x)<\infty$ follows directly from the Lebesgue differentiation theorem and the absolute continuity of $\omega|_E$ with respect to $\mathcal{H}^n|_E$, exactly as in \cite{PPT21}.
	
	To show that $T_*\omega(x)<\infty$ for $\omega$-a.e.\ $x\in E$, for each $k\geq 1$, we define
	\[
	E_k \coloneqq \left\{x \in E : \mathcal{M}_n\omega(x) \leq k \right\},
	\]
	and observe that $\omega\bigl(E \setminus \bigcup_{k\geq 1} E_k \bigr) = 0$.
	Let $\varphi$ be a smooth radial cut-off function satisfying $0\leq \varphi \leq 1$, $\varphi \equiv 1$ on $B(0,1)$, and $\varphi \equiv 0$ on $\mathbb{R}^{n+1} \setminus B(0,2)$. For $r>0$, define $\varphi_r(\cdot) \coloneqq r^{-1} \varphi(\cdot/r)$ and consider the truncated operator
	\[
	\widetilde{T}_r\omega(z) \coloneqq \int \nabla_2 \Gamma_A(y,z) \, \varphi_r(z - y) \, d\omega(y).
	\]
	
	Standard Calderón–Zygmund theory (which applies thanks to the pointwise bounds in Lemma \ref{lem:estim_fund_sol}) and elliptic PDE estimates, together with \eqref{eq:Linftyest-r} and  the representation formula \eqref{eq;green ident-Wiener} , can be invoked exactly as in \cite[Lemma 4.3]{AHMMMTV16} to conclude that $|\widetilde{T}_r\omega(x)| < \infty$ for $\omega$-a.e.\ $x \in E_k$ and all $r > 0$, which concludes the lemma by standard regularized kernel estimates.
\end{proof}

Applying Theorem~\ref{thm:big piecesTb-aux_intro}, we conclude that there exists a subset $G_0 \subset E$ with $\omega^p(G_0) > 0$ such that the operator $T_{\omega^p|_{G_0}}$ is bounded on $L^2(\omega^p|_{G_0})$. Since $\omega^p$ is absolutely continuous with respect to $\mathcal{H}^n|_{G_0}$, we may invoke \cite[Corollary 1.4]{MMPT23} to deduce that $G_0$ is $n$-rectifiable.
Hence, we can conclude the proof of Theorem \ref{theorem:main_one_phase} by a standard exhaustion argument.
\vv
\subsection{The proof of Theorem \ref{theorem:elliptic-one-phase_quant}.}

The proof of the quantitative one-phase result for elliptic measure follows almost verbatim the argument of \cite[Theorem 1.1]{MT20}. That result is based on the proof of a Main Lemma \cite[Main Lemma 4.1]{MT20}, which asserts that, under suitable assumptions on a $\mu$-doubling ball $B$ centered at $\supp(\mu)$ with thin boundary, if \eqref{eq:quant_mutual_ac} holds for a sufficiently small constant $\kappa$, then either there exists a properly rescaled sub-ball $B(x_B, \eta r)$ with $\mu(B(x_B, \eta r)) \gtrsim \mu(B)$, or there is a subset $G \subset B$ such that $\mu(G) \gtrsim \mu(B)$ and $T_{\mu|_G}$ is bounded on $L^2(\mu|_G)$.

This argument generalizes to the operators $L_A$, $A\in \widetilde \DMO_{n-1}$, with the minor modifications outlined below:

\begin{enumerate}
	\item The proof of \cite[Key Lemma 7.1]{MT20} carries over via the estimates for the gradient of the single layer potential, specifically the Calder\'on–Zygmund-type bounds provided in Lemma~\ref{lem:estim_fund_sol}, which replace the classical bounds for the Riesz kernel used in the harmonic case. 
	
	\item As a consequence, \cite[Main Lemma 4.1]{MT20} extends to elliptic operators $L_A$ with $A \in \widetilde{\DMO}_{n-1}$ using a $T1$ theorem analogous to \cite[Theorem 8.1]{MT20}, formulated in terms of the gradient of the single layer potential. The extension to the $\widetilde \DMO_{n-1}$ framework is made possible by our novel big pieces $Tb$ theorem, Theorem \ref{thm:big piecesTb-aux_intro}.
\end{enumerate}

With these tools in place, the remainder of the proof follows via a Corona decomposition and a good-$\lambda$ inequality argument, in the spirit of \cite[Section 9]{MT20}.  This part of the argument relies again only on the Calder\'on–Zygmund nature of the operator and does not require additional properties of the Riesz transform.

\vvv

\subsection{The proof of Theorem \ref{theorem:quantitative_two_phase}.}
For the proof, we combine the approach of \cite{AMT20} with our main results concerning the gradient of the single layer potential. We sketch the argument for the reader’s convenience and highlight the main modifications required.

We now assume that $\Omega_i$, $F$, and $\omega_i$, $i=1,2$, are as in the statement of Theorem \ref{theorem:quantitative_two_phase}, and that $A\in \widetilde \DMO_{1-\alpha}$ for some $\alpha\in (0,1)$.
{By Lemma \ref{lem:reduction_to_UC} and \cite[Appendix A]{HwK20}, if $\Omega_i$ are connected and we denote by $\mathcal A$ the uniformly continuous representative of $A$, we have that the elliptic measures $\omega^{L_A, x_i}_{\Omega_i}$ and $\omega^{L_{\mathcal  A}, x_i}_{\Omega_i}$ coincide. Thus, without loss of generality,  we may and will assume that the matrix $A$ in the statement of Theorem \ref{theorem:quantitative_two_phase} is \textit{uniformly continuous} and has modulus of continuity $\mathfrak I_{\oomega_A}$.

	\vvv
		In the following lemma, we merge the upper and lower bounds for the densities of elliptic measures. We remark that the upper bound is a consequence of the monotonicity formula \eqref{eq:monotonicity_formula}, while the lower bound holds, for $\mathfrak{C}_1$-$P_{\gamma,\omega_1}$-doubling balls, by Lemma \ref{lemma:control_density}. 
		\begin{lemma}\label{lemma:bound_density_two_phase_problem}
			Let $\Omega_1, \Omega_2\subset \Rn1$ be disjoint Wiener regular domains. Let also $$0<R<\min\bigl(\dist(x_1, \partial \Omega_1), \dist(x_2, \partial \Omega_2)\bigr).$$ For $0<r<R/4$, $\xi \in \partial \Omega_1\cap \partial \Omega_2$, and $J(\xi, r)$ as in Theorem \ref{theorem:ACF_DMO} it holds that
			\begin{equation}\label{eq:bound_density_two_phase_problem}
				\frac{\omega_1(B(\xi, r))}{r^{n}}\frac{\omega_2(B(\xi, r))}{r^{n}}\lesssim J(\xi, 2r)^{1/2}.
			\end{equation}
			Furthermore, if $B(x,r)$ is a $\mathfrak C_1$-$P_{\gamma,\omega_1}$-doubling ball, for $r_x>0$ as in Lemma \ref{lemma:lemma_balls_Hauss}, we have
			\begin{equation}\label{eq:lower bound_density_two_phase_problem}
			J(\xi, r)^{1/2}\lesssim \frac{\omega_1(B(\xi, 32 r))}{r^n}\frac{\omega_2(B(\xi, 32 r))}{r^n}, \qquad \text{ for }\quad 0<r< R/16.
			\end{equation}
		\end{lemma}

The proof of \eqref{eq:bound_density_two_phase_problem} can be found in \cite[Theorem 3.3]{KPT09}, while for \eqref{eq:lower bound_density_two_phase_problem},  we refer to \cite[Lemma 4.11]{AMT17a} and the comments after \cite[Lemma 12.4]{Pu19}.

}

\vv

Let us first recall a general definition.

\begin{definition}[$\textrm{weak-}A_\infty$]
	Let $\mu$ be a Radon measure on $\Rn1$ and $B\subset \Rn1$ be a ball. We say that a Radon measure $\nu$ belongs to \textit{$\textrm{weak-}A_\infty(\mu,B)$} if there exist $\varepsilon, \varepsilon'\in(0,1)$ such that, for all $E\subset B$ we have that
	\[
	\mu(E)<\varepsilon\, \mu(B)\qquad \implies\qquad \nu(E)<\varepsilon'\, \nu(2B).
	\]
\end{definition}
\vv

If $B$ is such that $\mu(B)\approx\mu(2B)$ and $\nu(B)\approx\nu(2B)$ then, by \cite[Lemma 2.12]{AMT20}, for $N=N(\varepsilon,\varepsilon')$ large enough, there is a set $G\subset B$ such that $\nu(B\setminus G)<2\varepsilon'\, \nu(B)$, and
\begin{equation}\label{eq:properties_of_G}
	N^{-1}\frac{\mu(B)}{\nu(B)}\leq \frac{\mu(B(x,r))}{\nu(B(x,r))}\leq N \frac{\mu(B)}{\nu(B)}, \qquad \text{ for all }x\in G, 0<r<r(B).
\end{equation}

By the same argument of \cite[Lemma 3.1]{AMT20}, we can assume that the poles $x_i$, $i=1,2$ in the statement of Theorem \ref{theorem:quantitative_two_phase} are also corkscrew points in $\Omega_i$ for the ball $B$.

Hence, we can apply Proposition~\ref{limlem}, which, together with Lemma~\ref{lem:BMOmatrix},  Lemma~\ref{l:nullboundary}, and Lemma \ref{l:nullboundary},  can be used in the same way as in~\cite[Lemma 3.2]{AMT20} to show that, for any $\eta, \tau > 0$, there exists a ball $B' \subset B$ with $\omega_1(\partial B')=\omega_2(\partial B')=0$  such that:
\begin{itemize}
	\item [(a)] $r(B')\leq \tau r(B)$.
	\item [(b)] $r(B')\gtrsim_{\eta, \tau} r(B)$.
	\item [(c)] $\omega_i (B')\gtrsim_{\tau,\eta} \omega_i(B)$, for $i=1,2$. 
	\item [(d)] $B'$ satisfies the flatness condition $\beta^n_{\omega_2,1}(B')<\eta \Theta^n_{\omega_2}(B')$.
\end{itemize}
It readily follows that $\omega_2\in\textrm{weak-}\tilde A_\infty(\omega_1,B')$ (see \cite[Lemma 3.3]{AMT20}).  Moreover,  by \cite[Remark 2.13]{AMT20} and (c),  we have that $\omega_i(B')\approx \omega_i(B)\approx 1$.

\begin{remark}
Possibly by adjusting the multiplicative constants in (a)-$\ldots$-(d), we may assume without loss of generality that the ball \( B' \) constructed above has a \textit{thin boundary} with respect to both \( \omega_1 \) and \( \omega_2 \).

Indeed, suppose that \( B' \) is a ball satisfying (a)-$\ldots$-(d), and that \( \omega_2 \in \text{weak-} \widetilde{A}_\infty(\omega_1, B') \) with constants \( \varepsilon, \varepsilon' \). By the proof of \cite[Lemma 9.43]{To14}, for \( 0 < \lambda \ll 1 \) sufficiently small (to be chosen later), there exist \( \mathfrak{C}_3 > 0 \) and a ball \( B'' = (1 + \lambda) B' \), such that \( B' \subset B'' \subset 1.1\, B' \), with \( \mathfrak{C}_3 \)-thin boundary with respect to both \( \omega_1 \) and \( \omega_2 \). It is immediate to see that:
\begin{itemize}
	\item[(a')] \( r(B'') \leq 1.1 \tau\, r(B) \).
	\item[(b')] \( r(B'') \geq r(B') \gtrsim_{\eta, \tau} r(B) \).
	\item[(c')] \( \omega_i(B'') \geq \omega_i(B') \gtrsim_{\tau, \eta} \omega_i(B)  \), for \( i = 1, 2 \).\\
	Moreover, since by \cite[Remark 2.13]{AMT20}, \( \omega_i(2B') \approx \omega_i(B') \approx 1 \), there exists \( C > 1 \) such that
	\begin{equation}\label{eq:wek_a_infty_new_0}
		\omega_i(2 B'') \leq 1 \leq C\, \omega_i(B'').
	\end{equation}
\end{itemize}

\begin{itemize}
	\item We now claim that \( \omega_2 \in \text{weak-} \widetilde{A}_\infty(\omega_1, B'') \). To prove this, let \( E \subset B'' \) be such that
	\begin{equation}\label{eq:wek_a_infty_new_hp}
		\omega_1(E) \leq \tilde{\varepsilon}\, \omega_1(B''),
	\end{equation}
	for some \( \tilde{\varepsilon} > 0 \) to be chosen below. Hence, by \eqref{eq:wek_a_infty_new_0} and \eqref{eq:wek_a_infty_new_hp}, we obtain
	\begin{equation*}
		\omega_1(E \cap B') < \tilde{\varepsilon}\, \omega_1(B'') \leq C\, \tilde{\varepsilon}\, \omega_1(B') < \varepsilon\, \omega_1(B'),
	\end{equation*}
	provided that \( \tilde{\varepsilon} < C^{-1} \varepsilon \). Thus, since \( \omega_2 \in \text{weak-} \widetilde{A}_\infty(\omega_1, B') \) with constants \( \varepsilon \) and \( \varepsilon' \),
	\begin{equation}\label{eq:weak_a_infty_new}
		\omega_2(E \cap B') < \varepsilon'\, \omega_2(B') \leq \varepsilon'\, \omega_2(B'').
	\end{equation}

		Since $\omega_2(\partial B')=0 $,  \( B'' \) has $\mathfrak C_3$-thin boundary,  and \( B' \) is doubling with respect to \( \omega_2 \), we have
	\begin{equation}\label{eq:flat_B_d_1}
		\omega_2(B'' \setminus B') \leq \omega_2\left(\left\{ x \in 2B' \setminus B' : \dist(x, \partial \Omega) < \lambda r(B') \right\}\right)
		\leq \mathfrak{C}_3\, \lambda\, \omega_2(2B') \leq C\, \mathfrak{C}_3\, \lambda\, \omega_2(B').
	\end{equation}
	
For \( \lambda \leq (C \mathfrak{C}_3)^{-1} \varepsilon' \),  we obtain
	\begin{equation*}
		\omega_2(E \setminus B') \leq \omega_2(B'' \setminus \overline{B'} ) \leq \varepsilon'\, \omega_2(B''),
	\end{equation*}
which, 	 combined with \eqref{eq:weak_a_infty_new},  concludes that \( \omega_2(E) < 2\, \varepsilon'\, \omega_2(B'') \).
	
	Hence, \( \omega_2 \in \text{weak-} \widetilde{A}_\infty(\omega_1, B'') \) with constants \( C^{-1} \varepsilon \) and \( 2\varepsilon' \).
	
	\item We are left with justifying that \( \omega_2 \) is flat at the level of \( B'' \); that is,
	\[
	\beta_{\omega_2,1}(B'') < \tilde{\eta}\, \Theta^n_{\omega_2, 1}(B''),
	\]
	for some constant \( \tilde{\eta} > 0 \) depending on \( \eta \) and \( n \), and for a suitable choice of \( \lambda \).

	If \( \lambda \leq \eta (C \mathfrak{C}_3)^{-1} \), then property (c) and \eqref{eq:flat_B_d_1} imply
	\begin{equation*}\label{eq:flat_B_d_2}
		\beta_{\omega_2, 1}(B'') \leq \beta_{\omega_2, 1}(B') + \frac{\omega_2(B'' \setminus B')}{r(B')^n}
		< \eta\, \Theta^n_{\omega_2}(B') + \frac{\omega_2(B'' \setminus B')}{r(B')^n}
		\leq 2\eta\, \Theta^n_{\omega_2}(B').
	\end{equation*}
	
	Thus, since \( B' \subset B'' \subset 2B' \), we conclude that
	\begin{equation}\label{eq:new flat}
		\beta_{\omega_2, 1}(B'') \leq 2^{n+1} \eta\, \Theta^n_{\omega_2}(B'').
	\end{equation}
\end{itemize}

Now observe that, since $\ve'$ and $\eta$ are sufficiently small, we may choose them such that $2\ve'$ and $\tilde{\eta}$ are also sufficiently small. For notational convenience, and without loss of generality, we will denote the \(\text{weak-} \widetilde{A}_\infty(\omega_1, B'')\) constants associated with $\omega_2$ by $\ve$ and $\ve'$, and the constant in~\eqref{eq:new flat} by~$\eta$.
\end{remark}

\vv

Let $\eta, \tau>0$ to be chosen later, and assume that $B=B(c_B,r(B))$. Hence, there exists $\tilde c>0$ such that the following properties hold:
\begin{enumerate}
	\item $\displaystyle \omega_i(B)\geq \tilde c$, for $i=1,2$.
	\item $\displaystyle r(B)\lesssim \tau |x_i-c_B|$, for $i=1,2$.
	\item $\displaystyle \omega_2$ belongs to $\text{weak-}\tilde A_\infty(\omega_1, B)$.
	\item $\displaystyle \beta_{\omega_2,1}(B)<\eta\, \Theta^n_{\omega_2}(B).$
\end{enumerate}
Thus, (1) yields that for all $\gamma \in [0,1]$,
\[
P_{\gamma,\omega_i}(B)\leq \sum_{j\geq 0} 2^{-j\gamma}\frac{1}{2^{jn}\, r(B)^n}\lesssim_{n,\gamma}\tilde {\mathfrak c}^{-1}\Theta^n_{\omega_i}(B), \qquad \text{ for } i=1,2.
\]
Furthermore, we denote by $G$ the subset of $B$ obtained as in \eqref{eq:properties_of_G}. The following lemma is an immediate generalization of \cite[Lemma 4.1]{AMT20} and  \cite[Lemma 4.2]{AMT20}.

\begin{lemma}\label{lemma:rect_assumptions_two_phase}
	There exist $\mathfrak c_1, \mathfrak c_2,\mathfrak c_3, \mathfrak c_4>0$ depending on $\tau, \eta,$ and $N$ such that the following properties hold:
	\begin{enumerate}
		\item $\displaystyle \Theta^n_{\omega_2}(B(x,r))\leq \mathfrak c_1 \Theta^n_{\omega_2}(B)$ for all $x\in G$ and $0<r\leq 2r(B)$.
		\item $\displaystyle T_*(\chi_{2B}\omega_i)(x)\leq \mathfrak c_2\Theta^n_{\omega_i}(B)$ for all $x\in G$ and $i=1,2$.
		\item $\displaystyle \omega_i|_G(B(x,r))\leq \mathfrak c_3 \Theta^n_{\omega_i}(B)r^n$ for all $x\in \Rn1$, all $r>0$, and $i=1,2$.
		\item For $i=1,2$, the operator $T_{\omega_i|_G}$ is bounded on $L^2(\omega_i|_G)$ and 
		\[
		\displaystyle \|T_{\omega_i|_G}\|_{L^2(\omega_i|_G)\to L^2(\omega_i|_G)}\leq \mathfrak c_4\Theta^n_{\omega_i}(B).
		\]
	\end{enumerate}
\end{lemma}
\vv

To prove the next result, we use the equivalence of the $L^2$-boundedness of the gradient of the single layer potential and of the Riesz transform, for $A\in \widetilde \DMO_{1-\alpha}$.

\begin{theorem}\label{theorem:aux_thm_NTV_pubmat}
	Let $n \geq 2$, and let $A$ be a uniformly elliptic matrix satisfying $A \in \widetilde \DMO_{1-\alpha}$ for $\alpha\in(0,1)$. There exists a constant $C > 0$ depending on $n$ such that, for all $p, q > 1$, the following holds.
	
	Let $\nu \in M^n_+(\Rn1)$ whose support is contained in a ball $B \subset \Rn1$, and define
	\[
	F_p \coloneqq \bigl\{x \in \supp(\nu) : \nu(B(x, r)) \geq p^{-1} r^n, \quad\text{ for } 0 < r \leq \diam(\supp(\nu)) \bigr\}
	\]
	and
	\[
	F_{p, q} \coloneqq \bigl\{x \in F_p : \nu(B(x, r) \cap F_p) \geq (pq)^{-1} r^n, \quad\text{ for } 0 < r \leq \diam(\supp(\nu)) \bigr\}.
	\]
	
	Then, if $T_\nu$ is bounded on $L^2(\nu)$, there exists a $C/(pq)$-Ahlfors regular measure $\sigma$ such that $\sigma|_{F_{p, q}} = \nu|_{F_{p, q}}$ and $T_\sigma$ is bounded on $L^2(\sigma)$.
\end{theorem}

\begin{proof}
	If $T_\nu$ is bounded on $L^2(\nu)$, \eqref{eq:maineq} yields that the Riesz transform $\mathcal{R}_\nu$ is bounded on $L^2(\nu)$. Hence, \cite[Theorem 2.1]{NToV14b} implies that there exists a $C/(pq)$-Ahlfors regular measure $\sigma$ such that $\sigma|_{F_{p, q}} = \nu|_{F_{p, q}}$; for a detailed statement, we also refer to \cite[Theorem 4.3]{AMT20}. Thus, we apply \eqref{eq:maineq} again and conclude that $T_\sigma$ is bounded on $L^2(\sigma)$, which proves the theorem.
\end{proof}

Furthermore, in the setting of Theorem \ref{theorem:aux_thm_NTV_pubmat}, \cite[Lemma 4.4]{AMT20} shows that, if $\nu(F_p) \geq \delta \nu(B)$ for some $\delta \in (0, 1/2)$, then there exists $q$ depending on $n,p$, and $\delta$ such that $\nu(F_{p, q}) \geq \nu(F_p)/2$.

\vv

We now split $G$ into a component where $\omega_1$ has \textit{big density} and another where it has \textit{small density}.
More specifically, let $0<\rho \ll 1$ be chosen later, and define
\begin{equation}\label{eq:definition_G_bd}
	G^{bd} \coloneqq \bigl\{x \in G : \Theta^n_{\omega_1}(x, r) \geq \rho\,  \Theta^n_{\omega_2}(B) \quad \text{for all } r \in (0, 2r(B)] \bigr\}
\end{equation}
and $G^{sd} \coloneqq G \setminus G^{bd}$.

\vv

\subsubsection{Study of $G^{bd}$} 
We assume that $\delta\in (0,1)$, to be chosen later, is such that
\begin{equation}\label{eq:big_density_big}
	\omega_1(G^{bd})\geq \delta\, \omega_1(B).
\end{equation}
Let $\nu =\frac{\omega_1|_G}{C\Theta^n_{\omega_1|_G}(B)}$, $F_p$ and $F_{p,q}$ be the set in Theorem \ref{theorem:aux_thm_NTV_pubmat}, and $\sigma$ the associated Ahlfors-regular measure such that $\sigma|_{F_{q,p}}=\nu|_{F_{q,p}}$ and the gradient of the single layer potential $T_\sigma$ is bounded on $L^2(\sigma)$.
Hence, we claim that $\sigma$ is uniformly $n$-rectifiable and $\Sigma_B\coloneqq \supp(\sigma)$ is such that
\begin{equation}\label{eq:covering_big_density}
	\omega_1(\Sigma_B\cap F\cap B)\geq \frac{\delta}{4}\omega_1(B).
\end{equation}
The proof of this result can be carried out using Theorem \ref{theorem:aux_thm_NTV_pubmat} and the argument in \cite[Section 6]{AMT20}. The only difference is that, in order to infer uniform rectifiability from the $L^2(\sigma)$-boundedness of $T_\sigma$, we use \cite[Corollary 1.3]{MMPT23} instead of \cite{NToV14}.
\vv

\subsubsection{Study of $G^{sd}$} Let us assume that \eqref{eq:big_density_big} does not hold, namely that
\[
\omega_1(G^{bd}) < \delta\, \omega_1(B),
\]
which, if $\delta < \varepsilon$, by \eqref{eq:hypothesis_A_infinity_type}, implies that $\omega_2(G^{bd}) < \varepsilon' \omega_2(B)$. So, by \eqref{eq:properties_of_G} we obtain
\[
\omega_2(B\setminus G)\leq 2 \varepsilon' \omega_2(B),
\]
which in turn implies that
\[
\omega_2(B\setminus G^{sd})\leq \omega_2(B\setminus G)+ \omega_2(B\setminus G^{bd})\leq 3\varepsilon' \omega_2(B).
\]
In the following lemma, we gather the main estimates for the gradient of the single layer potential.
\begin{lemma}\label{lemma:estimates_T_two_phase}
	Let $N$ be as in \eqref{eq:properties_of_G} and $\rho$ as in \eqref{eq:definition_G_bd}. There exist $\mathfrak C',\mathfrak C''$ depending on $n$, $a$, $N$, and the CDC constants  such that
	\begin{equation}\label{eq:bound_T_nabla}
		|T\omega_2(x)-\nabla_1\Gamma_A(x,x_2)|\leq \mathfrak C' \rho \, \Theta^n_{\omega_2}(B), \qquad \text{ for }x\in G^{sd},
	\end{equation}
	and, for $\beta\in(0,1)$ as in \eqref{eq:continuityGamma},
	\begin{equation}\label{eq:bound_L2_mean_osc_sd}
		\int_{G^{sd}}\bigl|T\omega_2(x)-m_{\omega_2,G^{sd}}(T\omega_2)\bigr|^2\, d\omega_2(x)\leq \mathfrak C'' \bigl(\rho+\tau^\beta + \mathfrak I_{\oomega_A}(\tau)\bigr)^2\Theta^n_{\omega_2}(B)^2\omega_2(B).
	\end{equation}
\end{lemma}
\begin{proof}
	For the proof of \eqref{eq:bound_T_nabla}, it suffices to repeat verbatim the argument in \cite[Lemma 6.3]{AMTV19}, as done in \cite[Lemma 6.1]{AMT20}. For the proof of \eqref{eq:bound_L2_mean_osc_sd} we observe that, if $x\in G^{sd}\subset B$, as $\omega_2(B)\approx 1$, we have that
	\begin{equation*}
		\begin{split}
			\bigl|T\omega_2(x)&-m_{\omega_2,G^{sd}}(T\omega_2)\bigr|\overset{\eqref{eq:bound_T_nabla}}{\lesssim}\rho\,  \Theta^n_{\omega_2}(B)+ \bigl|\nabla_1\Gamma_A(x,x_2)-m_{\omega_2,G^{sd}}(\nabla_1\Gamma_A(\cdot, x_2))\bigr|\\
			&\leq \rho\,  \Theta^n_{\omega_2}(B)+\sup_{y\in G^{sd}}|\nabla_1\Gamma_A(x,x_2)-\nabla_1\Gamma_A(y,x_2)|\\
			&\overset{\eqref{eq:continuityGamma}}{\lesssim} \rho \, \Theta^n_{\omega_2}(B) + \frac{\omega_2(B)}{|x_2-c_B|}\Bigl(\Bigl(\frac{r_B}{|x_2-c_B|}\Bigr)^\beta+\mathfrak I_{\oomega_A}\Bigl(\frac{r_B}{|x_2-c_B|}\Bigr)\Bigr)\\
			&\overset{}{\lesssim} \rho\,  \Theta^n_{\omega_2}(B) + \bigl(\tau^\beta + \mathfrak I_{\oomega_A}(\tau)\bigr)\Theta^n_{\omega_2}(B)=\bigl(\rho+\tau^\beta + \mathfrak I_{\oomega_A}(\tau)\bigr)\Theta^n_{\omega_2}(B),
		\end{split}
	\end{equation*}
	which proves \eqref{eq:bound_L2_mean_osc_sd}.
\end{proof}

Finally, we can choose $\eta$ and $\tau$ small enough so that the estimates in Lemma \ref{lemma:rect_assumptions_two_phase} and Lemma \ref{lemma:estimates_T_two_phase} imply that the assumptions of Theorem \ref{teo1} are satisfied for $\mu = \omega_2$, which we apply to find a uniformly $n$-rectifiable set $\Sigma_B$ such that $\omega_2(\Sigma_B \cap G^{sd}) \gtrsim \omega_2(B)$.

Furthermore, since $\frac{d\omega_2}{d\omega_1} \approx 1$ on $G^{sd}$, we have that
\[
\omega_1(\Sigma_B \cap F \cap B) \geq \omega_1(\Sigma_B \cap G^{sd}) \gtrsim \omega_1(B)
\]
which, together with \eqref{eq:covering_big_density}, implies \eqref{eq:thm13_a}. The final part of Theorem \ref{theorem:quantitative_two_phase}, if $x_1$ is a corkscrew point for $\frac{1}{4}B$, follows as in \cite[Sections 5 and 6]{AMT20}, which concludes the proof of the theorem.

}

\vvv
\subsection{The proof of Theorem \ref{theorem:main_two_phase_qualitative}.}

We assume without loss of generality that the set $E$ in Theorem \ref{theorem:main_two_phase_qualitative} is such that 
\[
\diam(E) \leq  \min\bigl(\diam(\Omega_1), \diam(\Omega_2)\bigr)/10.
\]
Let $\widetilde{B}$ be a ball centered at $E$, $i=1,2$, and choose the poles $x_i \in \Omega_i$ such that 
\(
x_i \in \Omega_i \cap 2 \widetilde{B} \setminus \widetilde{B}.
\)

Possibly by interchanging $\Omega_1$ and $\Omega_2$, we also assume that
	\[
	\mathcal H^{n+1}(\widetilde B\cap \Omega_1)\approx r(\widetilde B).
	\]
	By Bourgain's estimate Lemma \ref{l:bourgain},
	the argument at the beginning of \cite[Section 6]{AMTV19}, and \cite[Theorem 4.13]{AM19}, which repeats verbatim, this implies that 
	\begin{equation}\label{eq:omega_p_approx_1_s_delta_1}
		\omega_1\bigl(4\widetilde B\big)\approx 1.
	\end{equation}
	
	For $i=1,2$, by $u_i(\cdot)\coloneqq G^i(x_i, \cdot)$ the Green's functions with poles at $x_1$, extended by zero to $\Rn1\setminus \Omega_i$. By the proof of  \cite[Theorem 4.13]{AM19}, we obtain the following estimates.
	
	\begin{lemma}\label{lemma:lemma_balls_Hauss}
		Let $B$ be a ball centered at $E$ such that $x_1\not \in 10 B$, $\omega_1(8B)\leq C\omega_1(\tfrac{1}{2}B)$, and $\mathcal H^{n+1}\geq C^{-1}r(B)^{-1}$. Then we have
		\[
		\mathcal H^{n+1}\bigl(\Omega_1\cap 2 B\bigr)\gtrsim r(B)^{n+1}
		\]
		and
		\[
		\mathcal H^{n+1}\bigl( B\setminus \Omega_1\bigr)\approx \mathcal H^{n+1}\bigl(B\cap \Omega_2\bigr)\approx r(B)^{n+1}.
		\]
	\end{lemma}

}

{\vv

Let $\Omega\subset \Rn1$ be domain and $\mu$ a Radon measure on $\partial \Omega$. We define 
\begin{equation*}
\begin{split}
	\dim \mu &= \inf\bigl\{s>0:\text{ there is } F\subset \partial \Omega \text{ so that }\mathcal H^s(F)=0 \text{ and }\\
	&\qquad\quad \qquad \qquad \mu(F\cap K)=\mu(\partial \Omega \cap K) \text{ for all compact sets }K\subset \Rn1\}.
\end{split}
\end{equation*}

We denote by 
\[
\mathscr F\coloneqq\bigl\{c\, \mathcal H^n|_V: \, c>0\text{ and }V \text{ is an }n\text{-plane through the origin}\bigr\}
\]
the class of \textit{flat} measures. Given $a\in\Rn1$ and $r>0$, we define the map
\[
T_{a,r}(x)=\frac{x-a}{r}, \qquad \text{for } x\in\Rn1,
\]
and, for a Radon measure $\mu$ on $\Rn1$, we denote by $T_{a,r}[\mu]$ its image measure via $T_{a,r}$, namely,
\[
T_{a,r}[\mu](A)\coloneqq \mu(rA+ a), \qquad \text{ for }A\subset \Rn1.
\]
Finally, we say that a non-zero Radon measure $\nu$ on $\Rn1$ is a \textit{tangent measure} of $\mu$ at $\bar x\in \Rn1$ if there exist sequences $c_i>0$ and $r_i\searrow 0$ such that $c_iT_{\bar x, r_i}[\mu]$ converges weakly to $\nu$ as $i\to \infty$, and we write $\nu\in \Tan(\mu, \bar x)$.
A fine blow-up analysis by Azzam and the second named author made possible the proof of the following result for elliptic measure associated with uniformly elliptic matrices whose coefficients belong to Sarason's class of functions of vanishing mean oscillation, see \cite[Definition 1.5]{AM19}.

\begin{theorem}[\cite{AM19}, Theorem I]\label{theorem:AzzamMourgoglou_19}
Let $n\geq 2$, $\Omega_1, \Omega_2\subset \Rn1$ be two disjoint domains, and $L_A$ be an elliptic operator on $\Omega_1\cup\Omega_2$ such that $A$ is a uniformly elliptic $(n+1)\times(n+1)$-matrix with $A\in\VMO(\Omega_1\cup\Omega_2, \xi)$ at $\omega_1$-almost every $\xi \in E\subset\partial \Omega_1\cap \partial \Omega_2$ with respect to either $\Omega_i$, $i=1,2$. For $i=1,2$, denote $\omega_i=\omega^{L_A,x_i}_{\Omega_i}$ for some $x_i\in\Omega_i$. If $\omega_1$ and $\omega_2$ are mutually absolutely continuous on $E$, then for $\omega^i$-almost every $\xi \in E$, we have $\Tan(\omega_i, \xi)\subset \mathscr F$ and  $\dim\omega_i|_E=n$.
\end{theorem}

Using Theorem \ref{theorem:AzzamMourgoglou_19} and \cite[Theorem 9.43]{To14},  it is possible to find a sequence of arbitrarily small $P_{\gamma, \omega_1}$-doubling balls with thin boundaries with respect to $\omega_1$, at $\omega_1$-almost every point of $E$. Namely, a variation of \cite[Lemma 6.1]{AMT17a}, which uses Bourgain's bounds on the dimension of harmonic measure, allows us to prove that, under the assumptions of Theorem \ref{theorem:quantitative_two_phase}, for $\gamma\in(0,1)$, there exists a sufficiently large constant $\mathfrak{C}_1 = \mathfrak{C}_1(\gamma,n)>0$ such that, for $\omega_1|_E$-almost every $x\in E$, we can find a sequence of $\mathfrak{C}_1$-$P_{\gamma, \omega_1}$-doubling balls $B(x,r_j)$ with $r_j\searrow 0$ as $j\to \infty$.  
For further details, we refer the reader to \cite[Lemma 12.1]{Pu19}.
}
\vv

The proof of \cite[Section 6]{AM19} also implies that, under our assumptions, for $\omega_1$-almost every point of $E$, we have that 
\[{\beta_{\omega_1,1}(B(\xi,r))}/{\Theta^n_{\omega_1}(B(\xi,r))} \to 0, \qquad \text{ as } r \searrow 0,\] see also \cite[Lemma 5.11]{AMTV19}. Hence, we can perform a decomposition of the set $E$, analogously to \cite[Lemma 5.1]{AMTV19}.
Namely, for $\varepsilon \in (0,1/100)$, we denote by $E_m$ the set of $\xi \in E$ such that, for all $0<r<1/m$, the following properties hold for $i=1,2$:
\begin{itemize}
\item $\displaystyle\omega_i(B(\xi, 2r)) \leq m\, \omega_i(B(\xi, r))$.
\item $\displaystyle\mathcal H^{n+1}\bigl(B(\xi, r)\cap \Omega_i\bigr) \geq m^{-1}r^{n+1}$.
\item $\displaystyle\beta_{\omega_1, 1}(B(\xi, r)) < \varepsilon\, r^{-n} \omega_1(B(\xi, r))$.
\end{itemize}
Then
\[
\omega_i \biggl(E \setminus \bigcup_{m \geq 1} E_m\biggr) = 0.
\]

In particular, we choose $m$ such that $\omega_1(E_m) > 0$. The proof of \cite[Lemma 6.2]{AMT17a} applies without modification and yields the following result.

\begin{lemma}\label{lemma:control_density}
Let  $m\geq 1$.	For $\omega_1$-almost every $x\in E_m$ there is $r_x>0$ such that, for any $\mathfrak C_1$-$P_{\gamma, \omega_1}$-doubling ball $B(x,r)$ with $r\leq r_x$, there exists $G_m(x,r)\subset E_m\cap B(x,r)$ such that
\[
\Theta^n_{\omega_1}(B(x,t))\lesssim \Theta^n_{\omega_1}(B(x,r)),\qquad \text{ for all } z\in G_m(x,r), \, 0<t\leq 2r,
\]
which yields
\begin{equation}\label{eq:bd_G_m}
\omega_1\bigl(B(x,r)\setminus G_m(x,r)\bigr)\leq \delta \, \omega_1(B(x,r)).
\end{equation}
Moreover, if we denote by $\widetilde E_{m, \delta}$ the set of points $x\in E_m$ for which \eqref{eq:bd_G_m} holds, then
\(
\omega_1\bigl(\widetilde E_{m, \delta}\setminus E_m\bigr)=0.
\)
\end{lemma}

\vvv

We partition $G_m(x_0, r_0)$ into points of \textit{zero density} and \textit{positive density}. Specifically, for $m \geq 1$ and $x_0 \in \tilde{E}_{m,\delta}$, we define  
\[
G^{zd}_m(x_0, r_0) \coloneqq \bigl\{ x \in G_m(x_0, r_0) : \lim_{r \to 0} \Theta^n_{\omega_1}(B(x, r)) = 0 \bigr\},
\]
and  
\[
G^{pd}_m(x_0, r_0) \coloneqq \bigl\{ x \in G_m(x_0, r_0) : \limsup_{r \to 0} \Theta^n_{\omega_1}(B(x, r)) > 0 \bigr\}.
\]

\vv
In the following lemma, we gather the approach to study of these two sets. 

\begin{lemma}\label{lemma:key_lemma_pd_zd}
Let $m\geq 1$.  Let also  $x_0\in \tilde E_{m, \delta}$ and 
\[
0<r_0\leq \min\bigl(r_{x_0}, 1/m, c_1\, \dist(p_1, \partial\Omega_1)\bigr),
\]
for some $c_1>0$ small enough. If $\beta$ is as in \eqref{eq:continuityGamma}, $\gamma$ is as in \eqref{eq:extra_cond_A_DMO-gen},  $\bar \gamma \coloneqq\min\{\beta, \gamma\}$\footnote{or as in Lemma \ref{lem:holder_bound_I_omega} if we work under the hypothesis $M_{\oomega}<\infty$ as in \eqref{eq:extra_cond_A_DMO}.}, and we assume that $B=B(x_0,r_0)$  is $\mathfrak C_1$-$P_{\bar \gamma/2,\omega_1}$-doubling, then:
\begin{enumerate}
\item [(a)]	For any $x\in G_m(x_0,r_0)$
\begin{equation*}
	T_*\bigl(\chi_{2B}\,\omega_1\bigr)\lesssim \Theta^n_{\omega_1}(B).
\end{equation*} 
\item [(b)]	There is an $n$-rectifiable set $F(x_0,r_0)\subset G^{pd}_m(x_0,r_0)$ such that
\[
\omega_1\bigl(G^{pd}_m(x_0,r_0)\setminus F(x_0,r_0)\bigr)=0,
\]
and that $\omega_1|_{F(x_0,r_0)}$ is mutually absolutely continuous with respect to  $\mathcal H^n|_{F(x_0,r_0)}$.
\item [(c)] We have
\begin{equation}\label{eq:oscillation_zero_density}
	\begin{split}
		\int_{G^{zd}_m(x_0, r_0)}\bigl|&T\omega_1(x)-m_{\omega_1, G^{zd}_m(x_0,r_0)}(T\omega_1)\bigr|^2\, d\omega_1(x) \lesssim \Bigl(\frac{r_0}{|x_0-x_1|}\Bigr)^{\bar \gamma}\,  \Theta^n_{\omega_1}(B)^2\, \omega_1(B).
	\end{split}	
\end{equation}
\end{enumerate}
\end{lemma}
\begin{proof}
The proof of (a) and (b) is in turn a variant of Lemmas 6.3 and 6.4 in \cite{AMTV19}, respectively.
Let us prove \eqref{eq:oscillation_zero_density}. 
The same argument as in \cite[Lemma 6.5]{AMTV19} prove that
\begin{equation}\label{eq:omega_P_delta_approx_1}
\omega_1\bigl(B(x_0, 20 |x_0-x_1|)\bigr)\approx 1.
\end{equation}

By \eqref{eq:extra_cond_A_DMO-gen} there exists $s_0 \in (0,1)$ small enough and $C_\gamma>0$ such that
\(
	\oomega_A(s)\leq C_{\gamma}\, s^{\gamma}
\)
for all  $0< s\leq s_0,$
which implies that, for $0<s\leq s_0$,
\begin{equation}\label{eq:bd_I_hold_appl}
	\mathfrak I_{\oomega_A}(s)\leq C_\gamma\gamma^{-1}s^\gamma\eqqcolon \bar C_\gamma s^\gamma.
\end{equation}
In particular, we assume that $\tau\coloneqq r_0/20|x_0-x_1|\ll  (\mathfrak C_1  (1+ \bar C_{\gamma}))^{-1/\bar \gamma}$.
Hence, for $x\in G^{zd}_m(x_0,r_0)$, we have that
\begin{equation*}
\begin{split}
	\bigl|T\omega_1(x)&-m_{\omega_1, G^{zd}_m(x_0,r_0)}(T\omega_1)\bigr|\leq \sup_{y\in G^{zd}_m(x_0,r_0)}\bigl|\nabla_1 \Gamma_A(x, x_1)-\nabla_1\Gamma_A(y,x_1)\bigr|\\
	&\quad \overset{\eqref{eq:continuityGamma}}{\lesssim}\bigl(\tau^\beta + \mathfrak I_{\oomega_A}(\tau)\bigr)\frac{1}{|x_0-x_1|^n}\overset{\eqref{eq:omega_P_delta_approx_1}}{\approx}\bigl(\tau^\beta + {\mathfrak I_{\oomega_A}(\tau)	}\bigr) \Theta^n_{\omega_1}\bigl(B(x_0, 20 |x_0-x_1|)\bigr)\\
	&\quad \overset{\eqref{eq:bd_I_hold_appl}}{\leq } (1+ \bar C_{\gamma})\,  \tau^{\bar \gamma} \Theta^n_{\omega_1}\bigl(B(x_0, 20 |x_0-x_1|)\bigr)\leq (1+ \bar C_{\gamma})\,  \tau^{\bar \gamma/2} \, P_{\bar \gamma/2, \omega_1}(B)\\
	&\quad\leq   \tau^{{\bar \gamma}/2}\, (1+ \bar C_{\gamma})\,  \mathfrak C_1 \,  \Theta^n_{\omega_1}(B),
\end{split}
\end{equation*}
where the penultimate bound holds because of the definition of $P_{\bar \gamma/2, \omega_1}$, and the last inequality follows because of the $P_{\bar \gamma/2, \omega_1}$-doubling property of $B$.
Hence, we readily get \eqref{eq:oscillation_zero_density}.
\end{proof}

\vv
To complete the proof of Theorem \ref{theorem:main_two_phase_qualitative}, it suffices to use the rectifiability of the elliptic measure on the set of positive density, together with the quantitative rectifiability result for the set of zero density.

Indeed, zero-density points can be ruled out by combining Theorem \ref{teo1} with the same argument as in \cite[Section 6]{AMT17a}, which implies that $\omega_1\big(G^{zd}_m(x_0,r_0)\big) = 0$.
Finally, by Theorem \ref{lemma:key_lemma_pd_zd}-(c), we can apply Theorem \ref{theorem:main_one_phase} and thereby conclude the proof of Theorem \ref{theorem:main_two_phase_qualitative}.
\vvv

\frenchspacing
\bibliographystyle{alpha}

\newcommand{\etalchar}[1]{$^{#1}$}
\def\cprime{$'$}

\end{document}